%% file: RHC.tex
\documentclass{amsart}
\input{headart}

\usepackage{euscript}

\title{Rigidification of arithmetic $\pmb{\sD}$-modules and an overconvergent Riemann--Hilbert correspondence}

\usepackage{amsbsy}

\SetSymbolFont{stmry}{bold}{U}{stmry}{m}{n}
\usepackage{bm}

\author{Christopher Lazda}
       \address{Department of Mathematics\\ Harrison Building \\ Streatham Campus
 \\ University of Exeter \\ North Park Road \\ Exeter  \\ EX4 4QF \\ United Kingdom }
\email{c.d.lazda@exeter.ac.uk}

\begin{document}

\begin{abstract} In this article, I define triangulated categories of constructible isocrystals on varieties over a perfect field of positive characteristic, in which Le Stum's abelian category of constructible isocrystals sits as the heart of a natural t-structure. I then prove a Riemann--Hilbert correspondence, showing that, for objects admitting some (unspecified) Frobenius, this triangulated category is equivalent to the triangulated category of overholonomic $\sD^\dagger$-modules in the sense of Caro. I also show that the cohomological functors $f^!,f_+$ and $\otimes$ defined for $\sD^\dagger$-modules have natural interpretations on the constructible side of this correspondence. Finally, I use this to prove that, for any variety $X$ admitting an immersion into a smooth and proper formal scheme, rigid cohomology (with lisse coefficients) agrees with cohomology defined using arithmetic $\sD$-modules.
\end{abstract}

\maketitle 

\tableofcontents

\section*{Introduction}

Let $X/\C$ be a smooth algebraic variety, and $\sD_X$ the ring of differential operators on $X$. One of the (many!) forms of the Riemann--Hilbert correspondence states that there is an equivalence of categories
\[ \bD^b_c(X^{\an},\C) \leftrightarrows \bD^b_{\rm rh}(\sD_X) \]
between the bounded derived category of (algebraically) constructible sheaves of $\C$-modules on $X^{\an}$, and the bounded derived category of regular holonomic $\sD_X$-modules. Moreover, this equivalence matches up various natural cohomological operations on each side, namely the `six functors' of usual and extraordinary pushforward and pullback, duality, and tensor product.

From a differential geometer's perspective, this gives a way of understanding differential equations on $X$ by studying their associated constructible sheaves. But from an algebro-geometric point of view, it shows that the cohomology theory of constructible sheaves on $X^\an$ can be reconstructed completely algebraically, via the theory of regular holonomic $\sD$-modules on $X$. It therefore provides a template for understanding the cohomology of algebraic varieties in situations where taking the complex analytification is not possible, for example when working in characteristic $p$. 

This idea has been the one of the main driving forces behind two different approaches to studying the $p$-adic cohomology of varieties in characteristic $p$, both developed by Berthelot. The first of these to be introduced was the theory of rigid cohomology \cite{Ber96b}, generalising earlier work of Monsky--Washnitzer \cite{MW68}. Rigid cohomology works as a version of de\thinspace Rham cohomology on $p$-adic analytic varieties, and its coefficient objects, called overconvergent $F$-isocrystals, are therefore direct analogues of vector bundles with integrable connection. Such `lisse' coefficient objects cannot support a good formalism of cohomological operations, and so Berthelot introduced in \cite{Ber02} a theory of arithmetic $\sD$-modules on mixed characteristic formal schemes. This is analogous to the theory of regular holonomic $\sD$-modules on complex varieties, and was intended to play the same role in rigid cohomology as that played by the theory of constructible sheaves in $\ell$-adic \'etale cohomology. At least for objects admitting a Frobenius structure, it was proved by Caro--Tsuzuki in \cite{CT12} that the category of overholonomic $\sD^\dagger$-modules does indeed support a formalism of Grothendieck’s six operations.

Despite the fact that both approaches were inspired by the classical Riemann--Hilbert correspondence, they have a rather different flavour. The theory of rigid cohomology is based upon the de\thinspace Rham cohomology of $p$-adic analytic varieties in characteristic $0$, whereas arithmetic $\sD$-modules live on $p$-adic formal schemes of mixed characteristic $(0,p)$. This may seem like a fairly minor distinction, but it does significantly complicate comparisons between the two. For example, even though the theory of arithmetic $\sD$-modules was explicitly introduced in order to provide rigid cohomology with a `six operations' formalism, it has still remained open in general whether or not the rigid cohomology groups of a variety, with coefficients in an overconvergent $F$-isocrystal, coincide with the analogous cohomology groups computed using the theory of arithmetic $\sD$-modules. 

The two theories also have somewhat complementary strengths. It is rigid cohomology (and it's earlier version, Monsky--Washnitzer cohomology) that is generally more computable, and actually appears in algorithms \cite{Lau06,ABCMT19} used to calculate zeta functions of algebraic varieties. On the other hand, it is the theory of arithmetic $\sD$-modules that has a good cohomological formalism. It is therefore an important problem to relate these two approaches, and combine the strengths of both into one unified overconvergent cohomology theory.

A major step in this direction was taken by Le Stum, who in \cite{LS16} defined a category of `constructible isocrystals' on a variety in characteristic $p$. These are \emph{direct} generalisations of overconvergent $F$-isocrystals, and like them, live on $p$-adic analytic varieties in characteristic $0$. The definition is also very closely analogous to that of the category of constructible $\ell$-adic \'etale sheaves. He then conjectured that there should be a Deligne--Kashiwara correspondence between his category of constructible isocrystals and a certain category of `perverse arithmetic $\sD$-modules', and proved this conjecture in the case of smooth and proper curves \cite{LS14}, at least for objects which are of `Frobenius type'. One of the main results in this article proves Le Stum's conjecture for arbitrary smooth formal schemes, albeit with a slightly stronger `Frobenius type' hypothesis.

In fact, my approach will be to first establish a Riemann--Hilbert type correspondence on the level of derived categories, and then compare t-structures on both sides to deduce an equivalence between the respective hearts. The first goal of this article is therefore to define a derived analogue of Le Stum's category of constructible isocrystals. To explain the construction, let $\cV$ be a complete DVR with fraction field $K$ of characteristic $0$ and perfect residue field $k$ of characteristic $p>0$. The definition is the most na\"ive one: if $\fP$ is a smooth formal scheme over $\cV$, with generic fibre $\fP_K$, then a constructible complex on $\fP$ is a bounded complex of modules over the ring $\sD_{\fP_K}$ of (algebraic) differential operators on $\fP_K$, whose cohomology sheaves are constructible isocrystals in the sense of Le Stum. The category $\bD^b_{\cons}(\fP)$ of these objects is then viewed simply as a full subcategory of $\bD(\sD_{\fP_K})$. The hard work then consists of showing that this gives rise to a reasonable theory in characteristic $p$. Concretely, this amounts to showing that if $X\hookrightarrow \fP$ is a locally closed immersion, and $\fP$ is proper, then the full subcategory of $\bD^b_\cons(\fP)$ consisting of objects supported on $\tube{X}_\fP$ is independent of $\fP$.\footnote{Since I will be working with adic spaces throughout, the tube $\tube{X}_\fP$ considered here behaves very much like the tubes considered in Berkovich geometry, and in particular the formalism of overconvergent sections is in some sense absorbed into the topology of $\tube{X}_\fP$.} More generally, I show that this is the case provided that $\fP$ is smooth in a neighbourhood of $X$, and the closure of $X$ in $P$ is proper over $k$. This gives rise to a derived category $\bD^b_\cons(X)$ of constructible isocrystals on $X$, which is therefore a reasonable candidate for admitting a six operations formalism (at least for objects which are of Frobenius type).

The main goal of this article is then to prove an `overconvergent Riemann--Hilbert correspondence', showing that $\bD^b_\cons(X)$ is equivalent to the derived category $\bD^b_\hol(X)$ of overholonomic $\sD^\dagger$-modules on $X$ in the sense of Caro and Abe--Caro \cite{AC18a} (again, this will only hold for objects of Frobenius type). In fact, since I will be working with constructible objects it suffices to do this on the level of the smooth formal scheme $\fP$ in which $X$ has been embedded. 

In the classical Riemann--Hilbert correspondence, the functor from $\sD$-modules to constructible sheaves simply takes the (shifted) de\thinspace Rham complex of an algebraic $\sD$-module, but in the $p$-adic world even constructing the appropriate functor takes a little bit of work. In our previous article \cite{AL22} we constructed a functor $\sp_!$ \emph{from} constructible isocrystals \emph{to} overholonomic $\sD^\dagger$-modules. This was a rather complicated version of pushforward along $\sp\from \fP_K\to \fP$, and only worked on the level of abelian categories. Here, I construct a functor $\sp^!$ going in the other direction, which is a kind of completed pullback along $\sp$, shifted by the dimension of $\fP$ (the reason for this shift will be discussed below). Conceptually, the functor $\sp^!$ is much simpler than $\sp_!$, and works perfectly well on the level of derived categories. The main result of this article is then the following.
 
\begin{theoremu}[\ref{theo: riemann-hilbert}] Let $\fP$ be a smooth formal scheme over $\cV$, $\bD^b_{\cons,F}(\fP)\subset \bD^b_\cons(\fP)$ the full subcategory of objects of Frobenius type,\footnote{For the purposes of this article, I will take `of Frobenius type' to mean objects which are iterated extensions of those admitting some unspecified $p^n$-power Frobenius structure (where $n$ is allowed to vary). This coincides with the terminology from \cite{AL22}, though not with that from \cite{LS14}. It is what has previously been called `$F$-able' in the literature.} and $\bD^b_{\hol,F}(\fP)\subset \bD^b_\coh(\sD^\dagger_\fP)$ the full subcategory of overholonomic complexes of Frobenius type. Then $\sp^!$ induces an equivalence of categories
\[  \sp^!\from \bD^b_{\hol,F}(\fP) \isomto \bD^b_{\cons,F}(\fP). \]
\end{theoremu}

The very similar notation $\bD^b_{\hol,F}(\fP)$ and $ \bD^b_{\cons,F}(\fP)$ hides the fact that objects in the two categories are of a very different nature. On the one hand, we have overholonomic complexes of $\sD^\dagger$-modules on $\fP$, on the other we have complexes of modules with integrable connection on $\fP_K$ whose cohomology sheaves are constructible.
%

As in the case of varieties over $\C$, it is important not just to have such an equivalence, but to know how it behaves with respect to the natural cohomological operations on each side. Out of the six functors $f_+,f^+,f_!,f^!,\otimes$ and $\bD$ defined for overholonomic $\sD^\dagger$-modules, I give natural interpretations of $f^!,f_+$ and $\otimes$ as functors on the derived category of constructible complexes $\bD^b_{\cons}(X)$. The extraordinary pullback $f^!$ and tensor product $\otimes$ are easy: after taking an embedding $X\hookrightarrow \fP$ they just correspond to ordinary pullback of $\sD_{\tube{X}_\fP}$-modules, and ordinary tensor product over $\cO_{\tube{X}_\fP}$. For a smooth and proper morphism $u\from \fP\to \fQ$ of formal schemes, the functor $u_+$ for overholonomic $\sD^\dagger_{\fP\Q}$-modules corresponds (up to a shift) with the higher de \thinspace Rham pushfoward of constructible isocrystals. This has the following important consequence. 

\begin{corollaryu}[\ref{cor: smooth and proper preserves cons}] Let $u\from \fP\to \fQ$ be a smooth and proper morphism of smooth formal schemes. Then the functor $\bR u_{\dR*}$ maps $\bD^b_{\cons,F}(\fP)$ into $\bD^b_{\cons,F}(\fQ)$. 
\end{corollaryu}

For varieties $X/k$ embedded inside formal schemes, the correct interpretation of $f_+$ on the `constructible' side turns out to use the compactly supported de\thinspace Rham pushforwards introduced in \cite{AL20} (see \S\ref{sec: coh ops for vars} for details). Given the usual duality relations amongst the six functors, the only thing missing from a complete six functor formalism on the constructible side of the correspondence is a suitable interpretation of the duality functor, and I don't currently have a good candidate for this.

One can also also ask about the effect of $\sp^!$ on t-structures. The category $\bD^b_{\cons,F}(\fP)$ admits a natural t-structure coming from the inclusion $\bD^b_{\cons,F}(\fP)\to \bD(\sD_{\fP_K})$. On the other hand, there are three distinct t-structures on $\bD^b_{\hol,F}(\fP)$ that have been previously studied in the literature. The first is just the obvious one coming from the inclusion $\bD^b_{\hol,F}(\fP)\subset \bD^b_\coh(\sD^\dagger_\fP)$. The second is what is called the `constructible' t-structure. Namely, it is the analogue of the t-structure on $\bD^b_{\rm rh}(\sD_X)$ given by transporting the ordinary t-structure on $\bD^b_c(X^{\an},\C)$ along the classical Riemann--Hilbert correspondence. It is characterised by the fact that $f^+$ is t-exact. The third is what I call the `dual constructible' t-structure, and is defined as the Verdier dual of the constructible t-structure. It is therefore characterised by the fact that $f^!$ is t-exact. It turns out that $\sp^!$ matches up the ordinary t-structure on $\bD^b_{\cons,F}(\fP)$ with the dual constructible t-structure on $\bD^b_{\hol,F}(\fP)$.\footnote{This is the reason for the shift in the definition of $\sp^!$.} Denoting the heart of the dual constructible t-structure by ${\bf DCon}_F(\fP)\subset \bD^b_{\hol,F}(\fP)$, and the abelian category of constructible isocrystals on $\fP$ (of Frobenius type) by $\Isoc_{\cons,F}(\fP)$, I therefore deduce the following version of Le Stum's conjecture. 

\begin{corollaryu}[\ref{cor: del kash}] Let $\fP$ be a smooth formal scheme over $\cV$. Then $\sp^!$ induces an equivalence of categories
\[ {\bf DCon}_F(\fP)\isomto \Isoc_{\cons,F}(\fP).\] 
\end{corollaryu} 

This immediately implies a similar result for varieties over $k$. 

Finally, I use the equivalence $\sp^!$ to prove that rigid cohomology (with coefficients) coincides with its $\sD^\dagger$-module counterpart. Recall that in \cite[\S3]{Abe19}, Abe defines a functor $\rho_X\from \Isoc_F^\dagger(X)\to \bD^b_{\hol,F}(X)$ from the category of overconvergent isocrystals of Frobenius type on a variety $X$, to the category of overholonomic complexes on $X$, building on previous work of Caro. 

\begin{corollaryu}[\ref{theo: rigid and D-module cohomology}] Let $X$ be a strongly realisable variety, $f\from X\to \spec{k}$ the structure morphism, and $\sF$ an overconvergent isocrystal on $X$ of Frobenius type. Then $\sp^!$ induces an isomorphism
\[ f_+\rho_X(\sF) \isomto \bR\Gamma_{\rig}(X,\sF) \]
in $\bD^b(K)$. 
\end{corollaryu}

Let me now give an outline of contents the article. In \S\ref{sec: prelim} I recall various preliminary results that I need in the theory of analytic geometry and rigid cohomology. In particular, I state the main results that I will need on compactly supported cohomology of analytic varieties from \cite{AL20}. In \S\ref{sec: abelian constructible isocrystals} I carefully define the (abelian) category of constructible isocrystals on varieties and pairs, following Le Stum. Then in \S\ref{sec: derived constructible isocrystals} I define the analogous triangulated category, and prove that it satisfies all the expected invariance and functoriality properties. In \S\ref{sec: D-modules} I recall how the theory of overholonomic $\sD^\dagger$-modules on varieties and pairs works, and define the three natural t-structures that exist on the derived categories of these objects. In \S\ref{sec: rigidification of O-modules} I construct the `rigidification' functor $\bL\hat{\sp}^*$ on the level of $\cO$-modules, which is a kind of completed pullback along the specialisation morphism $\sp\from \fP_K\to \fP$ associated to a flat formal scheme $\fP$ over $\cV$. Then in \S\ref{sec: rigidification of D-modules} I upgrade this functor to include $\sD$-module structures, giving rise to the functor $\sp^!$ from overholonomic complexes of $\sD^\dagger$-modules to (complexes of) constructible isocrystals. I prove that this functor is t-exact for the dual constructible t-structure on the source and the natural t-structure on the target. In \S\ref{sec: logarithmic} I pave the way for proving the Riemann--Hilbert correspondence by establishing a key vanishing result in log rigid cohomology, and then in \S\ref{sec: oc RH} I prove my main result, that $\sp^!$ is an equivalence on objects of Frobenius type. Then in \S\ref{sec: cohomological operations} I describe the compatibility of $\sp^!$ on the various cohomological operations defined for constructible isocrystals and overholonomic $\sD^\dagger$-modules, and finally in \S\ref{sec: rigid cohomology} I prove the comparison theorem between rigid and $\sD^\dagger$-module cohomology.

\subsection*{Acknowledgements} This article has benefited enormously from conversations with Tomoyuki Abe, and could not have been written without his help. In particular, I learned from him the rigidification construction described in \S\ref{sec: rigidification of O-modules} below. I would also like to thank Bernard Le Stum for many helpful comments on this and earlier articles, as well as Atsushi Shiho for answering some of my questions about his work. Further acknowledgements to be added after the referee process. 

\subsection*{Notation and conventions}

\begin{itemize}
\item I will denote by $K$ a complete, discretely valued field of characteristic $0$, whose residue field $k$ is perfect of characteristic $p$. While the general formalism of rigid cohomology does not require $K$ to be discretely valued, or $k$ to be perfect, the theory of arithmetic $\mathscr{D}^\dagger$-modules is generally developed under these assumptions.\footnote{Although the requirement that $k$ is perfect has recently been lifted by \cite{Car19}.} Since my main point of interest is comparisons between the two, I will impose this hypothesis from the beginning. The absolute value on $K$ (or any valued extension thereof) will be normalised  so that $\norm{p}=p^{-1}$. 
\item I will write $\mathcal{V}$ for the ring of integers of $K$, $\mathfrak{m}$ for its maximal ideal, and $\varpi$ for a choice of uniformiser. I will fix a power $q=p^a$ of $p$, and assume that $K$ admits a lift $\sigma$ of the $q$ Frobenius on $k$. Frobenius will always mean the $q$-power Frobenius.
\item A \emph{variety} will mean a separated and finite type $k$-scheme, a \emph{formal scheme} will mean a separated and (topologically) finite type formal scheme over $\spf{\mathcal{V}}$, and an \emph{analytic variety} will mean an adic space, separated and locally of finite type over $\spa{K,\cV}$. Given a formal scheme $\fr{P}$, its generic fibre $\fP_K$ is therefore an analytic variety, and its special fibre $\fP_k$ is a variety. I will use fraktur letters to denote formal schemes, and the corresponding roman letters for their special fibres, for example $P=\fr{P}_k$.
\item Adjectives such as flat or smooth, when applied to varieties or formal schemes, should be understood to apply to the structure morphism to $\spec{k}$ or $\spf{\cV}$. Thus a smooth variety will be a variety that is smooth over $k$, and a flat formal scheme will be a formal scheme that is flat over $\cV$. Similar adjectives applied to analytic varieties should be understood to apply to the structure morphism to $\spa{K,\cV}$.
\item If $\rho\in\sqrt{\norm{K}}$, I will denote by $\D^d_K(0;\rho)$ the closed polydisc of radius $\rho$ over $K$,\footnote{via the normalisation $\norm{p}=p^{-1}$} and by $\D^d_K(0;\rho^-)$ the open polydisc of radius $\rho$. If $\scr{X}$ is an analytic variety, and $\#\in \{\emptyset,-\}$, I will denote by $\D^d_\scr{X}(0;\rho^{\#})$ the fibre product $\D^d_K(0;\rho^{\#})\times_{\spa{K,\cV}}\scr{X}$.
\item If $X$ is a topological space, I will denote by $\Sh(X)$ the category of (abelian) sheaves on $X$. If $\cO_X$ is a sheaf of (not necessarily commutative) rings on $X$, I will denote the category of coherent $\cO_X$-modules by $\Coh(\cO_X)$.\footnote{Recall that  an $\cO_X$-module $\sF$ is called \emph{coherent} if it is locally finitely generated, and the kernel of any map $\cO_{U}^{\oplus r} \to \sF\mid_U$ on any open subset $U\subset X$ is also locally finitely generated.}
\item If $\mathcal{A}$ is an abelian category I will denote by ${\bf D}^{\#}(\mathcal{A})$ the derived category with boundedness condition $\#\in \{\emptyset,+,-,b\}$. If $\mathcal{A}=\Sh(X)$ for some topological space $X$, I will usually write $\mathbf{Ch}^{\#}(X)$ and ${\bf D}^{\#}(X)$ instead, and if $\mathcal{A}$ is the category of $\mathcal{O}_X$-modules on a ringed space $(X,\mathcal{O}_X)$, I will write $\mathbf{Ch}^{\#}(\mathcal{O}_X)$ and ${\bf D}^{\#}(\mathcal{O}_X)$.
\item  If $A$ is an abelian group (or more generally, a sheaf of abelian groups on a topological space), I will write $A_{\Q}$ for $A\otimes_{\Z}\Q$. If $\mathcal{C}$ is an additive category, I will denote by $\mathcal{C}_{\Q}$ the corresponding isogeny category. Thus the objects of $\ca{C}_\Q$ are the same as $\ca{C}$, but the hom sets have been tensored with $\Q$: ${\rm Hom}_{\ca{C}_\Q}(X,Y)={\rm Hom}_{\ca{C}}(X,Y)\otimes_{\Z}\Q$.
\item For a morphism $f\from(X,\cO_X)\to (Y,\cO_Y)$ of ringed spaces (or more generally, ringed sites) I will use the formalism of $K$-injective and $K$-flat resolutions from \cite{Spa88} to define the functors $\bR f_*$ and $\bL f^*$ on unbounded derived categories of $\cO_X$ and $\cO_Y$-modules. In order for this construction to be well-behaved, the sites $X$ and $Y$ under consideration will need to have bases for their topologies which are of finite cohomological dimension, this will be the case for all sites considered in this article. In particular, if $Y$ is a point, then this gives the derived global sections functor $\bR\Gamma(X,-)$ for unbounded complexes. The internal hom functor
\[ \bR \underline{\rm Hom}_{\cO_X}(-,-) \from \bD(\cO_X)^{\rm op}\times \bD(\cO_X)\to \bD(\cO_X) \]
can be defined similarly. By then applying $\bR\Gamma(X,-)$, so can the functor $\bR{\rm Hom}_{\cO_X}(-,-)$ taking values in $\bD(\Gamma(X,\cO_X))$. 
\item I will use the notions of partition and stratification as defined in \cite[\href{https://stacks.math.columbia.edu/tag/09XY}{Tag 09XY}]{stacks}, although with slightly different terminology. Thus a partition of Noetherian topological space $X$ is a finite decomposition $X=\bigsqcup_{\alpha\in A }X_\alpha$ into locally closed subsets $X_\alpha$. It is called a stratification if $X_\alpha\cap \overline{X}_\beta \neq \emptyset \implies X_\alpha\subset \overline{X}_\beta$. This is in fact called a \emph{good} stratification in \cite[\href{https://stacks.math.columbia.edu/tag/09XY}{Tag 09XY}]{stacks}. Every partition can be refined to a stratification. 
\end{itemize}

\section{Preliminaries} \label{sec: prelim}

In this section I recall some general results and constructions I will need, mostly in rigid analytic geometry and rigid cohomology.

\subsection{Adic spaces} \label{sec: adic spaces}

In this article, analytic varieties will always be considered as adic spaces. I will therefore write $\mathbf{An}_K$ for the category of adic spaces separated and locally of finite type over $\spa{K,\mathcal{V}}$, and refer to such objects as \emph{analytic varieties} (over $K$). By \cite[\S1.1.11]{Hub96} there is an equivalence of categories
\begin{align*}
 (-)_0\colon\mathbf{An}_K &\rightarrow \mathbf{Rig}_K \\
 \mathscr{X} &\mapsto \mathscr{X}_0
\end{align*}
between $\mathbf{An}_K$ and the category of separated rigid analytic spaces over $K$ in the sense of Tate \cite{Tat71}. Denote a quasi-inverse to this functor by $r(-)$. If $\mathscr{X}_{\mathrm{an}}$ denotes the analytic site of $\mathscr{X}$ (that is, the category of open subsets of $\mathscr{X}$ equipped with its canonical topology), and $\mathscr{X}_{0,G}$ the $G$-site of $\mathscr{X}_0$ (that is, the category of admissible opens equipped with the topology of admissible open coverings), then the functor
\begin{align*}
\mathscr{X}_{\mathrm{an}} &\leftarrow \mathscr{X}_{0,G} \\
r(U) &\mapsfrom U
\end{align*}
induces an equivalence of toposes
\begin{align*}
 \mathbf{Sh}(\mathscr{X}) &\isomto \mathbf{Sh}_G(\mathscr{X}_0) \\
 \mathscr{F} &\mapsto \mathscr{F}_0,
\end{align*}
which is natural in $\mathscr{X}$ \cite[\S1.1.11]{Hub96}. In particular, it induces isomorphisms in cohomology
\[ {\rm H}^q(\mathscr{X},\mathscr{F}) \isomto {\rm H}^q(\mathscr{X}_0,\mathscr{F}_0)\]
for any abelian sheaf $\mathscr{F}$. Since $\left(\mathcal{O}_{\mathscr{X}}\right)_0\isomto \mathcal{O}_{\mathscr{X}_0}$, it also induces an equivalence of categories
\[ \mathbf{Coh}(\cO_\mathscr{X}) \isomto \mathbf{Coh}(\cO_{\mathscr{X}_0})\]
between coherent sheaves on $\mathscr{X}$ and $\mathscr{X}_0$.

\subsection{Frames and tubes} \label{subsec: frames and tubes}

The basic objects of rigid cohomology are frames and tubes. Since the theory is generally phrased in the language of rigid analytic spaces, I will briefly discuss here the changes that need to be made when using adic spaces instead.

\begin{definition} \begin{enumerate} 
\item A pair $(X,Y)$ consists of an open immersion $X\hookrightarrow Y$ of varieties.\footnote{Implicitly, these are always $k$-varieties, and hence a pair might be more precisely termed a `pair over $k$'.} 
\item A frame $(X,Y,\fr{P})$ consists of a pair $(X,Y)$ together with a closed immersion $Y\hto \fr{P}$ of formal schemes, such that $\fr{P}$ is flat.\footnote{Again, this flatness is over $\cV$, and a more precise term for a frame would be a `frame over $\cV$'.}
\end{enumerate}
There is an obvious notion of a morphism of pairs or of frames.
\end{definition}

The two formalisms of rigid cohomology and arithmetic $\sD$-modules for pairs and varieties work under slightly different hypotheses on the frames involved.

\begin{definition} Let $(X,Y,\fP)$ be a frame.
\begin{enumerate}
\item $\fP$ is smooth around $X$ if there exists an open subscheme $\fr{U}\subset \fP$, containing $X$, which is smooth over $\cV$.
\item $(X,Y,\fP)$ is an \emph{l.p. frame} if $\fP$ is smooth, and admits a locally closed immersion into a smooth and proper formal $\cV$-scheme. 
\end{enumerate}
\end{definition}

This leads to two different notions of `realisability' for pairs or varieties. 

\begin{definition} \label{defn: realisable} \begin{enumerate}
\item A pair $(X,Y)$ is said to be weakly realisable if there exists a frame $(X,Y,\fP)$ with $\fP$ smooth around $X$. 
\item A pair $(X,Y)$ is said to be strongly realisable if there exists an l.p. frame $(X,Y,\fP)$.
\item A variety $X$ is said to be weakly realisable if there exists a frame $(X,Y,\fP)$ with $Y$ proper and $\fP$ smooth around $X$. 
\item A variety $X$ is said to be strongly realisable if there exists an l.p. frame $(X,Y,\fP)$ with $Y$ proper.\footnote{Equivalently, $X$ admits a locally closed immersion into a smooth and proper formal scheme.}
\end{enumerate}
\end{definition}

If $\fP$ is a formal scheme, there is a continuous specialisation map
\[ \sp=\mathrm{sp}_\fP\from \colon\fr{P}_K \rightarrow \fr{P}\cong P. \]
Write $[\fP_K]$ for the separated quotient of $\fP_K$ in the sense of \cite[Chapter 0, \S2.3]{FK18}, this is the set of maximal points of $\fP_K$ equipped with the quotient topology via the map
\[ \sep=\sep_{\fP_K} \from \fP_K\to [\fP_K]\]
taking a point $x$ to its maximal generalisation $[x]$. Define a (non-continuous!) map
\[ [\sp]=[\sp_\fP]\from \fP_K \to P \]
as the composite 
\[ \fP_K \overset{\sep}{\lto} [\fP_K] \overset{\sp}{\lto} P \]
where the second map is the restriction of $\sp$ to the subset $[\fP_K]\subset \fP_K$. The map $[\sp]$ is in fact anti-continuous: the preimage of a closed subset is open and vice versa.

Recall that a subset $S\subset T$ of a topological space $T$ is called \emph{constructible} if it lies in the Boolean algebra generated by the open subsets $U$ of $T$ for which the inclusion $U\to T$ is a quasi-compact morphisms. In particular, if $T=P$ is a variety over $k$, then the quasi-compactness condition is automatically satisfied.

\begin{definition} For any constructible subset $S\subset P$, define the tube $\tube{S}_\fP:=[\sp]^{-1}(S)\subset \fP_K$.
\end{definition}

If $S$ is closed, then $\tube{S}_\fP$ is an open subspace of $\fP_K$, and if $S$ is open, then $\tube{S}_\fP$ is a closed subset of $\fP_K$. Since an implicit assumption on $\fP$ is that it is quasi-compact, it follows that $\tube{S}_\fP$ will be a finite union of locally closed subsets of $\fP_K$. However, it will not be constructible. For example, if $\fP=\widehat{\A}^1_\cV$ and $S=\{0\}\subset \A^1_k$, then $\tube{0}_{\widehat{\A}^1_\cV} \subset \D^1_K(0;1)$ is a non-quasi-compact open, and is therefore not constructible.

In general, $\tube{S}_\fP$ won't admit any natural structure as an adic space, unless $S$ is a closed subset of $P$. I can, however, always consider it as a locally ringed space by equipping it with the restriction
\[ \cO_{\tube{S}_\fP}:= \cO_{\fP_K}\!\!\mid_{\tube{S}_\fP} \]
of the structure sheaf on $\fP_K$. If $i \colon S\hookrightarrow S'$ is an inclusion of constructible subsets of $P$, I will generally abuse notation and also write $i \colon \tube{S}_\fr{P}\rightarrow \tube{S'}_\fr{P}$ for the induced morphism of tubes. This is then naturally a morphism of locally ringed spaces, and if $i\from S\hto S'$ is a locally closed immersion, then $i\from \tube{S}_\fP\to \tube{S'}_{\fP'}$ is topologically the inclusion of a locally closed subset. In this case the functor
\[ i_!\from \Sh(\tube{S}_\fP)\to \Sh(\tube{S'}_\fP)\]
of extension by zero along $i$ is defined purely topologically. A sheaf on $\tube{S'}_\fP$ is said to be \emph{supported on $\tube{S}_\fP$} iff it is in the essential image of this functor (note that we are not necessarily insisting on $\tube{S}_\fP$ being closed in $\tube{S'}_\fP$ when we say that a sheaf is supported on $\tube{S}_\fP$).  

Now, suppose that
\[ \xymatrix{ X'\ar[r]\ar[d]^f & Y' \ar[r]\ar[d]^g \ar[r]\ar[d] & \fP' \ar[d]^u \\ X \ar[r]& Y \ar[r]& \fP } \]
is a morphism of frames. Then there is an induced morphism
\[ \tube{f}_u \from\tube{X'}_{\fP'} \to \tube{X}_\fP \]
of locally ringed spaces, with associated module pullback functor $\tube{f}_u^*$.

\begin{remark} \label{rem: notation abuse morphisms of tubes} When $u=\mathrm{id}$ and $f$ is a locally closed immersion, this will often be denoted $f\from \tube{X'}_{\fP} \to \tube{X}_\fP$ instead of $\tube{f}_\mathrm{id}$. In general I will often simply write it as $u \from \tube{X'}_{\fP'} \to \tube{X}_\fP$, in a (possibly doomed) attempt to keep notational complexity within reasonable bounds. I will try my best to do this only when there is no possible scope for confusion.
\end{remark}

By definition of the structure sheaves on $\tube{X'}_\fP$ and $\tube{X}_\fP$, if $u=\mathrm{id}$ and $f$ is a locally closed immersion, the module pullback and sheaf pullback coincide: $f^*=f^{-1}$. The extension by zero functor $f_!\from {\bf Sh}(\tube{X'}_\fP)\to {\bf Sh}(\tube{X'}_\fP)$ can also be upgraded to a functor of $\cO$-modules, which coincides with the functor on underlying sheaves. The following simple `base change' result will be used constantly.

\begin{lemma} \label{lemma: u^*i_!} Let $(f,g,u):(X',Y',\fr{P}')\rightarrow (X,Y,\fr{P})$ be a morphism of frames, and $i:S\rightarrow X$ a locally closed immersion. Set $S':=f^{-1}(S)$, and write $i':S'\rightarrow X'$ for the induced locally closed immersion. Consider the (topologically Cartesian) diagram
\[ \xymatrix{ \tube{S'}_{\fr{P}} \ar[r]^-{i'}\ar[d]_{\tube{f}_u} & \tube{X'}_{\fr{P}'} \ar[d]^{\tube{f}_u} \\ \tube{S}_\fr{P} \ar[r]^-i & \tube{X}_\fr{P} } \]
of tubes. Then, for any $\mathcal{O}_{\tube{S}_\fr{P}}$-module $\mathscr{F}$, the base change map
\[ \tube{f}_u^*i_*\mathscr{F}\rightarrow i'_*\tube{f}_u^*\mathscr{F}\]
induces an isomorphism
\[ \tube{f}_u^*i_!\mathscr{F}\overset{\cong}{\longrightarrow} i'_!\tube{f}_u^*\mathscr{F}. \]
\end{lemma}

\begin{proof}
Set $T=X\setminus S$ and $T'=X'\setminus S'=f^{-1}(T)$ and let $j:T\rightarrow X$, $j':T'\rightarrow X'$ be the natural inclusions (of constructible subsets). Then the commutativity of the diagram
\[\xymatrix{\tube{T'}_{\fr{P}'} \ar[r]^-{j'}\ar[d]_{\tube{f}_u} & \tube{X'}_{\fr{P}'} \ar[d]^{\tube{f}_u} \\ \tube{T}_\fr{P}\ar[r]^-j & \tube{X}_\fr{P} } \]
shows that $j'^{-1}\tube{f}_u^*i_!\mathscr{F}=\tube{f}_u^*j^{-1}i_!\mathscr{F}=0$. First of all, this implies that the canonical map $\tube{f}_u^*i_!\mathscr{F}\rightarrow \tube{f}_u^*i_*\mathscr{F}\rightarrow i'_*\tube{f}_u^*\mathscr{F}$ factors through the subsheaf $i'_!\tube{f}_u^*\mathscr{F}\subset i'_*\tube{f}_u^*\mathscr{F}$. Secondly, since $j'^{-1}i'_!\tube{f}_u^*\mathscr{F}=0$ trivially, we can see that to prove
\[ \tube{f}_u^*i_!\mathscr{F}\rightarrow i'_!\tube{f}_u^*\mathscr{F} \]
is an isomorphism, it suffices to do so after restricting to $\tube{S'}_{\fr{P}'}$. But now we have
\[ i'^{-1}\tube{f}_u^*i_!\mathscr{F}=\tube{f}_u^*i^{-1}i_!\mathscr{F}=\tube{f}_u^*\mathscr{F},\;\;\;\;i'^{-1}i'_!\tube{f}_u^*\mathscr{F}=\tube{f}_u^*\mathscr{F},\]
and the morphism
\[ \tube{f}_u^*i_!\mathscr{F}\rightarrow i'_!\tube{f}_u^*\mathscr{F}  \]
restricts to the identity on $\tube{S'}_{\fr{P}'}$.
\end{proof}

The tube of a subset can be used to formulate the following rather weak notion of smoothness for frames, which will occasionally be useful.

\begin{definition} If $(X,Y,\fP)$ is a frame, I will say that $\fP$ is rig-smooth around $X$ if $\tube{X}_\fP$ admits an open neighbourhood $\tube{X}_\fP\subset V \subset \fP_K$ which is smooth over $K$. 
\end{definition}

The non-smooth locus of a morphism of analytic varieties is the subspace defined by a coherent ideal sheaf. Thus, if $\fP$ is smooth around $X$, then it is rig-smooth around $X$. It is easy to provide counter-examples to the converse statement.

Finally, as well as the open tube $\tube{Y}_\fr{P}$ of a closed subscheme $Y\hookrightarrow P$ defined above, I will also need the variants $\left[Y\right]_{\mathfrak{P}\eta}$ and $\tube{Y}_{\mathfrak{P}\eta}$, which are defined for $\eta<1$ sufficiently close to $1$. When $\fr{P}$ is affine, and $f_1,\ldots,f_r\in \Gamma(\fr{P},\mathcal{O}_\fr{P})$ are such that $Y=V(\varpi,f_1,\ldots,f_r)$, then
 \[ \tube{Y}_\fr{P} = \left\{\left. x\in \fr{P}_K \right\vert v_{[x]}(f_i)<1 \;\forall i \right\} \]
by \cite[Proposition II.4.2.11]{FK18}. Berthelot therefore defines the closed and open tubes
 \begin{align*} \left[Y\right]_{\mathfrak{P}\eta}&:= \left\{\left. x\in \fr{P}_K \right\vert v_{x}(f_i)\leq \eta \;\forall i \right\} \\  
 \tube{Y}_{\mathfrak{P}\eta}&:= \left\{\left. x\in \fr{P}_K \right\vert v_{[x]}(f_i)< \eta \;\forall i \right\}
 \end{align*}
of radius $\eta$. When $\norm{\varpi}<\eta<1$ these do not depend on the choice of the $f_i$, and hence glue together over an affine covering of $\fr{P}$, see \cite[\S1.1.8]{Ber96b}.

\subsection{Overconvergence}

Berthelot's functors $j^\dagger$ and $\underline{\Gamma}^\dagger$ of overconvergent sections and sections with support have a very natural interpretation  in the world of adic spaces. Let $(X,Y,\fP)$ be a frame, and write $j\from X\to Y$ for the given open immersion. Let $i\from Z\to Y$ be a complementary closed immersion.

\begin{definition} Define endofunctors
\begin{align*}
j_X^\dagger \colon \mathbf{Sh}(\tube{Y}_\fr{P}) &\rightarrow \mathbf{Sh}(\tube{Y}_\fr{P}) \\
\underline{\Gamma}^\dagger_Z \colon \mathbf{Sh}(\tube{Y}_\fr{P}) &\rightarrow \mathbf{Sh}(\tube{Y}_\fr{P})
\end{align*}
by $j_X^\dagger:=j_{*}j^{-1}$ and $\underline{\Gamma}^\dagger_Z:=i_{!}i^{-1}$.
\end{definition}

These are both exact, and there is an exact sequence
\[ 0\rightarrow \underline{\Gamma}^\dagger_Z \rightarrow \mathrm{id} \rightarrow j_X^\dagger  \rightarrow 0 \]
of endofunctors of $\mathbf{Sh}(\tube{Y}_\fr{P})$. More generally, if $j'\from U\to X\ot T\from i'$ are complementary open and closed immersions, then we have functors 
\begin{align*}
j_U^\dagger &:= j'_{*}j'^{-1} \from \Sh(\tube{X}_\fP) \to\Sh(\tube{X}_\fP) \\
\underline{\Gamma}^\dagger_T &:= i'_{!}i'^{-1} \from \Sh(\tube{X}_\fP) \to\Sh(\tube{X}_\fP)
\end{align*}  
sitting in a short exact sequence
\[  0\rightarrow \underline{\Gamma}^\dagger_T \rightarrow \mathrm{id} \rightarrow j_U^\dagger  \rightarrow 0. \]
To compare these with the original definitions of Berthelot, let $\mathfrak{P}_{K0}$ denote the rigid analytic generic fibre of $\fr{P}$ in the sense of \cite[\S0.2]{Ber96b}. In the notation of \S\ref{sec: adic spaces} above this is the rigid analytic space $(\fr{P}_K)_0$. Let $\tube{Y}_{\fr{P}0}\subset \fr{P}_{K0}$ denote the rigid analytic tube in the sense of Berthelot \cite[\S1.1]{Ber96b}, and $j_X^\dagger$, $\underline{\Gamma}^\dagger_Z$ the corresponding endofunctors of $\mathbf{Sh}_G(\tube{Y}_{\fr{P}0})$ as defined in \cite[\S2.1]{Ber96b}.

\begin{proposition} \label{prop: adic rig} There is a unique isomorphism $\left(\tube{Y}_\fr{P}\right)_0 \isomto  \tube{Y}_{\fr{P}0}$ of rigid analytic spaces over $K$, compatible with the natural open immersions of both sides into $\fP_{K0}$. Moreover, the diagrams
 \[ \xymatrix{ \mathbf{Sh}(\tube{Y}_\fr{P}) \ar[r]^-{(-)_0} \ar[d]_{j_X^\dagger} & \mathbf{Sh}_G(\tube{Y}_{\fr{P}0}) \ar[d]^{j_X^\dagger} & & \mathbf{Sh}(\tube{Y}_\fr{P}) \ar[r]^-{(-)_0} \ar[d]_{\underline{\Gamma}^\dagger_Z} & \mathbf{Sh}_G(\tube{Y}_{\fr{P}0}) \ar[d]^{\underline{\Gamma}^\dagger_Z}  \\ \mathbf{Sh}(\tube{Y}_\fr{P}) \ar[r]^-{(-)_0} & \mathbf{Sh}_G(\tube{Y}_{\fr{P}0}) & & \mathbf{Sh}(\tube{Y}_\fr{P}) \ar[r]^-{(-)_0} & \mathbf{Sh}_G(\tube{Y}_{\fr{P}0}) } \]
commute up to natural isomorphism.
\end{proposition}

\begin{proof}
Note that $\mathfrak{P}_{K0}$ can be identified with the set of \emph{rigid} points of $\mathfrak{P}_{K}$, and the functor $U\mapsto U\cap \mathfrak{P}_{K0}$ gives a one-to-one correspondence between admissible open subsets of $\mathfrak{P}_{K}$ (in the sense of \cite[Definition II.B.1.1]{FK18}) and admissible open subsets of $\mathfrak{P}_{K0}$ (in the sense of the $G$-topology). Since tube open subsets of $\fr{P}$ are admissible by \cite[Proposition II.B.1.7]{FK18}, the first claim is reduced to showing that $\tube{Y}_\fr{P} \cap \fr{P}_{K0}=\tube{Y}_{\fr{P}0}$ as subsets of $\fr{P}$. The question is now local on $\fr{P}$, which we may assume to be affine. Let $f_1,\ldots,f_r\in \Gamma(\fr{P},\mathcal{O}_\fr{P})$ be such that $Y=V(\varpi,f_1,\ldots,f_r)$. Then by \cite[Proposition II.4.2.11]{FK18} we can identify
 \[ \tube{Y}_\fr{P} = \left\{\left. x\in \fr{P}_K \right\vert v_{[x]}(f_i)<1 \;\forall i \right\}, \]
where $[x]$ denote the maximal generalisation of $x$. Since rigid points $x$ satisfy $x=[x]$, the claim now reduces to \cite[Proposition 1.1.1]{Ber96b}. For the second claim, there are exact sequences
\[ 0\rightarrow \underline{\Gamma}^\dagger_Z \rightarrow \mathrm{id} \rightarrow j_X^\dagger  \rightarrow 0 \]
 of functors on both $\mathbf{Sh}(\tube{Y}_\fr{P})$ and $\mathbf{Sh}_G(\tube{Y}_{\fr{P}0})$, it therefore suffices to consider $j_X^\dagger$. In this case, since the topos of an analytic variety is equivalent to that of the associated rigid space, the claim fact follows from the alternative definition of $j_X^\dagger$ given in \cite[Proposition 5.3]{LS07} (for the equivalence with Berthelot's definition, see \cite[Proposition 5.1.12]{LS07}).
\end{proof}

On the level of $\cO$-modules, Proposition \ref{prop: adic rig} shows that:
\begin{enumerate}
\item there is a canonical isomorphism $\left( j_X^\dagger\mathcal{O}_{\tube{Y}_\fr{P}}\right)_0 \cong j_X^\dagger\mathcal{O}_{\tube{Y}_{\fr{P}0}}$;
\item the functor $\mathscr{F}\mapsto \mathscr{F}_0$ induces an equivalence of categories
\[ \mathbf{Mod}(j_X^\dagger\mathcal{O}_{\tube{Y}_\fr{P}}) \isomto  \mathbf{Mod}(j_X^\dagger\mathcal{O}_{\tube{Y}_{\fr{P}0}}),\] 
preserving the full subcategories of coherent modules. 
\end{enumerate}

We can then use this to transport all results proved for overconvergent sheaves in the language of rigid analytic spaces into the adic context, for example the following.

\begin{proposition}[Proposition 2.1.10, \cite{Ber96b}] \label{prop: coherent sheaves colimit neighbourhoods} The inverse image functor induces an equivalence of categories
\[ \varinjlim_V \mathbf{Coh}(\mathcal{O}_V) \isomto \mathbf{Coh}(\mathcal{O}_{\tube{X}_\fr{P}})   \]
where $V$ ranges over all open neighbourhoods of $\tube{X}_\fr{P}$ in $\tube{Y}_\fr{P}$.  
\end{proposition}

In particular, this implies that $\mathcal{O}_{\tube{X}_\fr{P}}$ and $j_X^\dagger\cO_{\tube{Y}_\fr{P}}$ are coherent sheaves of rings on $\tube{X}_\fP$ and $\tube{Y}_\fP$ respectively. Since $j\from \tube{X}_\fP\to \tube{Y}_\fP$ is a closed immersion on the underling topological spaces, the following lemma is then elementary.

\begin{lemma} The functors 
\[ j_{*}\colon \mathbf{Coh}(\mathcal{O}_{\tube{X}_\fr{P}}) \leftrightarrows \mathbf{Coh}(j_X^\dagger\mathcal{O}_{\tube{Y}_\fr{P}}) \colon j^{-1} \]
are inverse equivalences of categories. If $\mathscr{E}$ is a coherent $j_X^\dagger\mathcal{O}_{\tube{Y}_\fr{P}}$-module, then
\[ {\rm H}^*(\tube{Y}_\fr{P},\mathscr{E}) \isomto {\rm H}^*(\tube{X}_\fr{P},j^{-1}\mathscr{E}). \]
\end{lemma}

\subsection{Germs and compactly supported de\thinspace Rham cohomology}

If $(X,Y,\fP)$ is a frame, then the tube $\tube{X}_\fP$ is not an adic space, but a closed subset of the adic space $\tube{Y}_\fP$. Moreover, it has the property that it is stable under generalisation inside $\tube{Y}_\fP$, in other words it is an \emph{overconvergent} closed subset. In our previous work on rigid cohomology and arithmetic $\sD$-modules \cite{AL20,AL22}, we found it useful to formalise this in the notion of an \emph{overconvergent germ}. 

\begin{definition} \begin{enumerate}
\item A pre-germ is a pair $(S,\mathscr{X})$ where $\mathscr{X}$ is an analytic variety, and $S\subset \mathscr{X}$ is a closed subset. 
\item A morphism $f\colon(S,\mathscr{X})\rightarrow (T,\mathscr{Y})$ of pre-germs is a morphism $f\colon\mathscr{X}\rightarrow \mathscr{Y}$ of analytic varieties such that $f(S)\subset T$.
\item A morphism $f\colon (S,\mathscr{X})\rightarrow (T,\mathscr{Y})$ of pre-germs is called a strict neighbourhood if $f\colon\mathscr{X}\hookrightarrow \mathscr{Y}$ is an open immersion inducing a homeomorphism $f\colon S\isomto T$.
\item The category of germs of analytic varieties is the localisation of the category of pre-germs at the class of strict neighbourhoods.
\item A germ $(S,\scr{X})$ is called \emph{overconvergent} if $S$ is stable under generalisation in $\scr{X}$. 
\end{enumerate} 
\end{definition}

I will generally suppress the ambient adic space $\scr{X}$ from the notation, and write a germ $(S,\scr{X})$ simply as $S$. Thus the motivating example $(\tube{X}_\fP,\tube{Y}_\fP)$ of an overconvergent germ will be written simply as $\tube{X}_\fP$. The category of germs admits fibre products, and contains the category of analytic varieties as a full subcategory. Thus it is possible to form spaces such as the relative (open or closed) polydisc $\D^r_S(0;\rho^{(-)})$ over a germ $S$. For example, in the language of germs, Berthelot's strong fibration theorem \cite{Ber96b} has the following form. 

\begin{theorem}[Berthelot] \label{theo: strong fibration theorem} Let $(f,g,u)\from (X',Y',\fP')\to (X,Y,\fP)$ be a morphism of frames, such that $f$ is an isomorphism, $g$ is proper and $u$ is smooth in a neighbourhood of $X'$, of relative dimension $d$. 
\begin{enumerate}
\item If $g$ is also an isomorphism, then, locally on $X$ and on $\fr{P}$, there exists an isomorphism
\[ \tube{X'}_{\fr{P}'} \isomto \tube{X}_\fr{P} \times_K \D^d_K(0;1^-) \]
of germs, identifying $\tube{f}_u$ with the first projection. 
\item If $d=0$, then
\[ \tube{f}_u\colon\tube{X'}_{\fr{P}'} \isomto \tube{X}_\fr{P} \]
is an isomorphism of germs.
\end{enumerate}
\end{theorem}

As with the case of tubes, any germ $S$ can be viewed as a ringed space by equipping it with the restriction
\[ \mathcal{O}_S:= \mathcal{O}_{\mathscr{X}}\!|_S\]
of the structure sheaf from its ambient adic space. Similarly, if $f\colon S\rightarrow T$ is a morphism in $\mathbf{Germ}_K$, the relative de\thinspace Rham complex $\Omega^\bullet_{S/T}$ is defined via restriction from the ambient adic space.

I will occasionally talk about morphisms of germs being smooth or partially proper, or other similar adjectives familiar from the theory of schemes and/or adic spaces. This should always be understood in the sense of \cite[\S1.10]{Hub96}. For example, being smooth means that there exists a representative $f\colon(S,\mathscr{X}) \rightarrow (T,\mathscr{Y})$ at the level of pre-germs such that $f\colon \mathscr{X}\rightarrow \mathscr{Y}$ is smooth and $S=f^{-1}(T)$. Partial properness is slightly more involved, see \cite[Definition 1.10.15]{Hub96}. In fact, more important for us than partial properness is what might be called partial properness in the sense of Kiehl. This is a morphism $f\from S\to T$ which, locally on the source and target, factors through a closed immersion $S \to \D^d_T(0;1^-)$ for some $d$. For a more detailed discussion of this property, see \cite[\S4]{AL20}.  

Note that being smooth in this sense is a very strong condition on a morphism of germs. For example, if $S$ is a germ which is \emph{not} an open subset of its ambient adic space, then the structure morphism $f\from S\to \spa{K,\cV}$ cannot be smooth. In particular, if $(X,Y,\fP)$ is a frame, then $\tube{X}_\fP$ cannot be smooth over $K$ if $X$ is not closed in $Y$. It will therefore be helpful to make the following weakening of the definition.

\begin{definition} A morphism $f\from (S,\scr{X})\to (T,\scr{Y})$ of germs is \emph{quasi-smooth} if there exists an open neighbourhood $S\subset V\subset \scr{X}$ which is smooth over $\scr{Y}$. 
\end{definition}

Note that this only depends on the induced morphism of germs, and implies that $\Omega_{S/T}$ is a locally finite free $\cO_S$-module. 

\begin{example}\label{exa: properties of morphisms between tubes} If $(X,Y,\fP)$ is a frame, then $\tube{X}_\fP$ is quasi-smooth over $K$ if and only if $\fP$ is rig-smooth around $X$.

More generally, let $(f,g,u)\from (X',Y',\fP')\to (X,Y,\fP)$ be a morphism of frames. If $g$ is proper, then $\tube{f}_u\from \tube{X'}_{\fP'} \to \tube{X}_\fP$ is partially proper in the sense of Kiehl. Indeed, the morphism $\tube{Y'}_{\fP'}\to \tube{Y}_\fP$ is a partially proper morphism of analytic varieties, and is hence partially proper in the sense of Kiehl by \cite[Remark 1.3.19]{Hub96}, and $\tube{X'}_{\fP'} \to \tube{Y'}_{\fP'}$ is a closed immersion. 

If $u$ is smooth in a neighbourhood of $X'$, then $\tube{f}_u$ is quasi-smooth, because the non-smooth locus of $\tube{g}_u$ is a closed analytic subspace of $\tube{Y'}_{\fP'}$. If, moreover, the given square
\[ \xymatrix{ X'\ar[r]\ar[d]_f & Y'\ar[d]^g \\ X\ar[r] & Y  } \]
of varieties is Cartesian, then $\tube{X'}_{\fP'}=\tube{g}_u^{-1}(\tube{X}_\fP)$, and so $\tube{f}_u$ is smooth. Up to possibly replacing $Y'$ by a closed subscheme containing $X'$, the above square will be Cartesian whenever both $g$ and $f$ are proper.
\end{example}

Quasi-smooth germs also satisfy analytic continuation.

\begin{lemma} \label{lemma: smooth connected injective} Let $S$ be a connected, quasi-smooth germ, and $\mathscr{F}$ a locally finite free $\mathcal{O}_{S}$-module. Then, for any non-empty open subset $U\subset S$, the natural map
\[ \Gamma(S,\mathscr{F})\rightarrow \Gamma(U,\mathscr{F}) \]
is injective. 
\end{lemma}

\begin{proof}
Fix a smooth ambient analytic space $\mathscr{X}$ for $S$, and let $s\in \Gamma(S,\mathscr{F})$ be a section. If $U,V\subset S$ are open subsets such that $s|_U=0$ and $s|_V=0$, then $s|_{U\cup V}=0$. We may therefore take $U$ maximal with the property that $s|_U=0$. 

Suppose for contradiction that $U\neq S$. Consider the collection $\mathcal{C}$ of all connected open subsets $V\subset S$ such that:
\begin{itemize}
\item $V= \bm{V}\cap S$ for a connected open $\bm{V}\subset \mathscr{X}$;
\item $\mathscr{F}|_V\cong \mathcal{O}_V^{\oplus n}$ where $n=\mathrm{rank}\,\mathscr{F}$;
\item the section $s|_V$ of $\mathscr{F}|_V\cong \mathcal{O}_V^{\oplus n}$ extends to a section of $\Gamma(\bm{V},\mathcal{O}_{\bm{V}}^{\oplus n})$.
\end{itemize}
Note that every point $x\in S$ admits a cofinal set of neighbourhoods satisfying the first and second of these, any sufficiently small element of which also satisfies the third. Hence $x$ at least one neighbourhood satisfying all three. Therefore, such open sets cover $S$, and since $S$ is connected, we deduce that there exists some $V\in \mathcal{C}$ such that $V\not\subset U$ and $U\cap V \neq\emptyset$. To contradict the maximality of $U$, it therefore suffices to show that $s|_V = 0$. But we know that $s|_{U\cap V}=0$, and since $s|_V$ extends to $\bm{V}$, it suffices to show that
\[ \Gamma(\bm{V},\mathcal{O}_{\bm{V}})\rightarrow \Gamma(U\cap V,\mathcal{O}_{\bm{V}}).\]
If we pick any point $v\in U\cap V$, it therefore suffices to show that
\[  \Gamma(\bm{V},\mathcal{O}_{\bm{V}}) \rightarrow \mathcal{O}_{\bm{V},v}\]
is injective. Since $\bm{V}$ is smooth and connected, this follows from \cite[Proposition 0.1.13]{Ber96b}.
\end{proof}

For any partially proper morphism of germs $f\colon S\rightarrow T$, we defined in \cite{AL20} a functor
\[ \mathbf{R}f_!\colon {\bf D}^+(S) \rightarrow  {\bf D}^+(T), \] 
having the following properties:
\begin{enumerate}
\item $f_!:=\mathcal{H}^0(\mathbf{R}f_!)$ is the functor of sections with proper support;
\item $\mathbf{R}f_!$ is the total derived functor of $f_!$;
\item there is a natural isomorphism $\mathbf{R}g_!\circ \mathbf{R}f_! \isomto \mathbf{R}(g\circ f)_!$ whenever $f,g$ are composable partially proper morphisms; 
\item $\mathbf{R}f_!=\mathbf{R}f_*$ whenever $f$ is proper;
\item $\mathbf{R}f_!= f_!$ is the usual extension by zero functor whenever $f$ is a (partially proper) locally closed immersion. 
\end{enumerate}
If $f:S \rightarrow \spa{K,\mathcal{V}}$ is the structure map of a germ, I will write $\mathbf{R}\Gamma_c(S,-)$ for $\mathbf{R}f_!$, and ${\rm H}^i_c(S,-)$ for its cohomology groups. 

The functor $\mathbf{R}f_!$ preserves module structures, and there is the following version of the projection formula.

\begin{lemma}[\cite{AL20}, Corollary 3.8.2] \label{lemma: proj form} Let $f:S\rightarrow T$ be a partially proper morphism of germs. For any locally free $\mathcal{O}_{T}$-module $\mathscr{F}$ of finite rank, and any complex $\mathscr{G}$ of $\mathcal{O}_{S}$-modules, there is a natural isomorphism
\[ \mathscr{F}\otimes_{\mathcal{O}_T}  \mathbf{R}f_! \mathscr{G} \isomto \mathbf{R}f_!(f^*\mathscr{F}\otimes_{\mathcal{O}_S} \mathscr{G})  \]
in $\bD^+(\cO_T)$.
\end{lemma} 

In general, the proper base change theorem for $\mathbf{R}f_!$ fails, but there is at least the following partial result. 

\begin{proposition}[\cite{AL20}, Lemma 3.5.2] \label{prop: proper base change} Let
\[ \xymatrix{ S' \ar[r]^{g'} \ar[d]_{f'} & S \ar[d]^f \\ T' \ar[r]^g & T } \]
by a Cartesian diagram of germs, such that $f$ is partially proper, and $g$ is one of the following:
\begin{enumerate}
\item a locally closed immersion onto a subspace which is closed under generalisation;
\item the inclusion of a maximal point of $T$.
\end{enumerate}
Then, for any $\mathscr{F}\in {\bf D}^+(S)$, the base change map
\[ g^{-1}\mathbf{R}f_!\mathscr{F} \rightarrow \mathbf{R}f'_! g'^{-1}\mathscr{F} \]
is an isomorphism.  
\end{proposition}

If $f\from S\to T$ is a quasi-smooth morphism of germs, de\thinspace Rham pushforwards with or without proper supports can be defined in the usual way. Namely, if $\sF$ is an $\cO_S$-module with integrable connection, then we form the de\thinspace Rham complex $\Omega^\bullet_{S/T}\otimes_{\cO_S} \sF$ and set
\begin{align*}
\bR f_{\dR*}\sF &:=\bR f_*(\Omega^\bullet_{S/T}\otimes_{\cO_S} \sF) \\
\bR f_{\dR!}\sF &:=\bR f_!(\Omega^\bullet_{S/T}\otimes_{\cO_S} \sF).
\end{align*}
as objects of $\bD^+(\cO_T)$. If, moreover, $T$ is quasi-smooth over $K$, then these can be upgraded to (complexes of) modules over the ring of differential operators on $T$. Indeed, if $\scr{X}$ and $\scr{Y}$ are ambient analytic varieties for $S$ and $T$ respectively, then we set
\[ \sD_T:=\sD_\scr{Y}\!\!\mid_T,\;\;\;\;\sD_S:=\sD_{\scr{X}}\!\!\mid_S.  \]
The transfer bimodule $\sD_{T\ot S}:= \sD_{\scr{Y}\ot\scr{X}}\!\!\mid_S$ is also defined via restriction, and there is the usual identification
\[ \Omega^\bullet_{S/T}\otimes_{\cO_S} \sF \isomto \sD_{T\ot S}\otimes^\bL_{\sD_S} \sF[-\dim f]. \]
The left $\sD_T$-module structure on $\sD_{T\ot S}$ therefore upgrades $\bR f_{\dR*}\sF$ and $\bR f_{\dR!}\sF$ to complexes of $\sD_T$-modules.

\subsection{The trace morphism}

Now suppose that $f\colon S\rightarrow T$ is a smooth morphism of germs,\footnote{recall that this is a very strong condition on $f$!} of relative dimension $d$, and partially proper in the sense of Kiehl. If $T$ is overconvergent, we constructed in \cite[\S5]{AL20} a trace map
\[ \mathrm{Tr}\colon \mathbf{R}^{2d}f_{\dR!} \cO_S  \to \mathcal{O}_{T}\]
having the following properties:
\begin{enumerate} 
\item \label{num: residue} if $T=\spa{R,R^+}$ is affinoid, and $S=\D^d_Y(0;1^-)$ is the relative open unit disc over $Y$, with co-ordinates $z_1,\ldots,z_d$, then $\Tr$ is induced by the residue map
\begin{align*} {\rm H}^d_c(S/T,\omega_{S/T}) = R\weak{z_1^{-1},\ldots,z_d^{-1}} \;d\log z_1\wedge \ldots \wedge d\log z_d &\rightarrow R \\ \sum_{i_1,\ldots,i_d\geq 0} r_{i_1,\ldots,i_d} z_1^{-i_1}\ldots z_d^{-i_d} \;d\log z_1\wedge \ldots \wedge d\log z_d&\mapsto r_{0,\ldots,0};  \end{align*}
\item \label{num: iso} whenever $S$ is locally either a $\D^d(0;1^-)$-bundle or an $\A^{d,\mathrm{an}}$-bundle over $T$, $\Tr$ is an isomorphism.
\end{enumerate} 
Moreover, we showed that $\mathbf{R}^qf_{\dR!}\cO_S=0$ if $q>2d$, and hence the trace map can be viewed as a morphism
\[ \mathrm{Tr}\colon \mathbf{R}f_{\dR!} \cO_S [2d] \rightarrow \mathcal{O}_{T}. \]
The projection formula then gives rise to a trace map
\[  \mathrm{Tr}_{\mathscr{F}}\colon \mathbf{R}f_{\dR!}f^*\sF [2d] = \mathbf{R}f_{\dR!}\cO_S \otimes_{\cO_T}\sF [2d] \to \mathscr{F} \]
for any finite locally free $\mathcal{O}_{T}$-module $\mathscr{F}$. Whenever $\mathscr{F}$ has the structure of a $\mathscr{D}_{T}$-module, this map is $\mathscr{D}_{T}$-linear, and whenever $S$ is locally either a $\D^d(0;1^-)$-bundle or an $\A^{d,\mathrm{an}}$-bundle over $T$, this map is an isomorphism.

\begin{corollary} \label{cor: trace isom from strong fibration}Let $(f,g,u)\from (X',Y',\fP')\to (X,Y,\fP)$ be a morphism of frames, such that $f$ is an isomorphism, $g$ is proper and $u$ is smooth in a neighbourhood of $X'$, of relative dimension $d$. Then the trace map 
\[  \Tr\colon\mathbf{R}\!\tube{f}_{\dR!}\cO_{\tube{X'}_{\fP'}}[2d] \rightarrow \mathcal{O}_{\tube{X}_\fr{P}} \]
is an isomorphism in $\bD^b(\sD_{\tube{X}_\fr{P}})$.
\end{corollary}

\begin{proof}
This follows from Theorem \ref{theo: strong fibration theorem} in exactly the same way as the proof that
\[ \mathcal{O}_{\tube{X}_\fr{P}}\rightarrow \mathbf{R}\!\tube{f}_{\dR*}\cO_{\tube{X'}_{\fP'}}  \]
is an isomorphism, see for example \cite[\S6.5]{LS07}. The first step is to show the result is true when $g$ is an isomorphism, which is a fairly direct consequence of Theorem \ref{theo: strong fibration theorem} and the above properties of the trace map. This then implies that the claim only depends on the given morphism of pairs $(X,Y')\to (X,Y)$ and not on the ambient morphism $u$ of formal schemes. 

The question is then local on $Y$, which can therefore be assumed quasi-projective, and one can use Chow's lemma to find a projective morphism $Y''\to Y'$ such that $Y''\to Y$ is also projective. Via a two-out-of-three argument, it therefore suffices to prove the result for both $(X,Y'') \to (X,Y')$ and $(X,Y'')\to (X,Y)$. This therefore reduces to the case when $g\from Y'\to Y$ is projective.  

Proposition \ref{prop: proper base change} together with Berthelot's resolution \cite[Proposition 2.1.8]{Ber96b} then implies that the question is local on $X$, and thanks to \cite[Lemma 6.5.1]{LS07} one can then assume that $u$ is actually \'etale in a neighbourhood of $X$. Finally, this case follows directly from Theorem \ref{theo: strong fibration theorem}.
\end{proof}

\subsection{Sheaves on diagrams of spaces} \label{subsec: diagram sheaves}

To make precise certain constructions, involving derived functors applied to diagrams of complexes, I will need to use the formalism of diagram toposes. So let $I$ be a category, and $X\from I\to {\bf Top}$ a diagram in the category of topological spaces. Then there is a natural site $X_{\rm Zar}$ associated to $X$ whose objects are pairs $(i,U)$ where $i\in I$ and $U\subset X_i$ is an open subset, and morphisms $(i,U)\to (j,V)$ consist of morphisms $\alpha\from i\to j$ such that the morphism $\alpha\from X_i\to X_j$ satisfies $U\subset \alpha^{-1}(V)$. A covering family is just a covering family on a single $X_i$. Sheaves on this site have the usual description as sheaves $\sF_i$ on each $X_i$, together with transition maps $\alpha^{-1}\sF_j\to\sF_i$ for any $\alpha\from i\to j$ in $I$, satisfying the usual compatibility condition for commutative triangles in $I$. The sheaf $\sF_i$ is called the  restriction of $\sF$ to $X_i$. I will usually abuse terminology and talk above sheaves on the diagram $X$ itself.

\begin{example} \begin{enumerate}
\item If $X$ is a diagram of ringed spaces, then $X_{\rm Zar}$ is naturally ringed by the sheaf of rings $\cO_X$ whose restriction to each $X_i$ is $\cO_{X,i}:=\cO_{X_i}$. 
\item If $X$ is a constant diagram, then the category of sheaves on $X$ is equivalent to the category of $I^{\rm op}$-shaped diagrams of sheaves on the single topological space $X$.
\end{enumerate} 
\end{example}

Now suppose that $Y\from J\to {\bf Top}$ is another diagram of topological spaces. Suppose furthermore that there is a functor $f^{-1}\from J\to I$, together with a morphism $f \from X\circ f^{-1} \to Y$ of $J$-shaped diagrams in ${\bf Top}$. There is then a natural functor
\begin{align*}
f^{-1} \from Y_{\rm Zar} &\to X_{\rm Zar} \\
(j,V) &\mapsto (f^{-1}(j), f^{-1}(V)).
\end{align*} 
Explicitly, if $(j,V)\in Y_{\rm Zar}$, with $V\subset Y_j$, then $f \from X_{f^{-1}(j)}\to Y_j$, and thus $f^{-1}(V)\subset X_{f^{-1}(j)}$.

\begin{lemma}\label{lemma: morphism of sites} Assume that for every $i\in I$, the category $i/J$ is cofiltered. Then $f^{-1}$ induces a morphism of sites $f\from X_{\rm Zar}\to Y_{\rm Zar}$. 
\end{lemma}

\begin{proof} The functor $f^{-1}$ induces, by the usual formulae, a functor $f^{-1}$ from sheaves on $Y_{\rm Zar}$ to those on $X_{\rm Zar}$, and the content of the lemma is that this functor is exact. 

Explicitly, if  $\sF$ is a sheaf on $Y_{\rm Zar}$, and $i\in I$, then for any $j\in i/J$ there is the given morphism $f_j\from X_i\to Y_j$, and the restriction of $f^{-1}\sF$ to $X_i$ is given explicitly by $\colim{(i/J)^{\rm op}} f_j^{-1}\sF_j$. (This can be seen by observing that the functor thus defined is adjoint to $f_*$.) If each $(i/J)^{\rm op}$ is filtered, then colimits of diagrams indexed by $(i/J)^{\rm op}$ are exact, and thus $f^{-1}$ is exact as required.
\end{proof}

Of course, if $X$ and $Y$ are diagrams of ringed spaces, and $f\from X\circ f^{-1} \to Y$ is a morphism of diagrams of ringed spaces, then $f$ can be naturally promoted to a morphism of ringed sites (at least under the assumptions of Lemma \ref{lemma: morphism of sites}. 

\begin{example} Suppose that $I=\{*\}$ is the category with a single object and morphism. Then a morphism $f \from X\to Y$ amounts to a compatible system of morphisms $f_j\from X\to Y_j$ for each $j$. The condition of Lemma \ref{lemma: morphism of sites} amounts to saying that $J$ is cofiltered. In this case, if $\sF=\{\sF_j\}_{j\in J}$ is a sheaf on $Y$, then $f^{-1}\sF=\colim{J^{\rm op}} f_j^{-1}\sF_j$.
\end{example}

It is also possible to form morphisms of toposes by taking limits, although of course this doesn't fit into the general formalism of Lemma \ref{lemma: morphism of sites}. To describe this construction, let $X$ be a topological space, $I$ a filtered category, and $X_I\from I\to {\bf Top}$ the constant diagram with value $X$. Thus a sheaf on $X_I$ is just an $I^{\rm op}$-shaped diagram of sheaves on $X$. The functor
\begin{align*}  \Sh(X_I) &\to \Sh(X)  \\
\{\sF_i\}_{i\in I} &\mapsto \lim{I}\sF_i 
\end{align*}
is then the pushforward for a morphism of toposes $X_I\to X$. The adjoint pullback functor
\begin{align*}  \Sh(X) &\to \Sh(X_I)  
\end{align*}
is, of course, the `constant diagram' functor.

\subsection{Limits and colimits}

For want of a better place to put them, I prove here a couple of simple lemmas on limits and colimits that will be useful later on. 

\begin{lemma} \label{lemma: inj prod colim} Let $A,B:I\rightarrow \mathbf{Ab}$ be filtered diagrams of abelian groups, and write $B^\N$ for the diagram given by $i\mapsto B(i)^{\N}:=\prod_{n\in \N} B(i)$. Suppose that we are given a morphism
\[ F: A\rightarrow B^{\N} \]
of diagrams. Assume that:
\begin{itemize}
\item $F$ is level-wise injective;
\item for each $i\in I$, the kernels of the transition maps $B(i)\rightarrow B(j)$ for $j\geq i$ eventually stabilise.\footnote{More formally, for each $i\in I$, there exists some $j_0\geq i$, such that for all $j\geq j_0$, the inclusion $\ker\left(B(i)\to B(j_0)\right)\to \ker\left(B(i)\to B(j)\right)$ is an equality.}
\end{itemize}
Then the natural map
\[ \mathrm{colim}_IA \rightarrow (\mathrm{colim}_I B)^\N\]
is also injective.
\end{lemma}

\begin{proof}
Since $F$ is levelwise injective, we see that the map
\[ \mathrm{colim}_IA \rightarrow \mathrm{colim}_I(B^{\N}) \]
is injective, it therefore suffices to prove that the natural map
\[ \mathrm{colim}_I(B^{\N}) \rightarrow (\mathrm{colim}_{I} B)^{\N} \]
is injective. An element of $\mathrm{colim}_I(B^{\N})$ is represented by an element $i\in I$ and a collection $(b_n)\in B(i)^{\N}$. Such an element mapping to zero in $(\mathrm{colim}_{I} B)^{\N}$ says that for each $n$, there exists $i_n\geq i$ such that $b_n\in \ker\left(B(i)\rightarrow B(i_n)\right)$. The hypothesis that the kernels eventually stabilise means that I can pick a single $j\in I$ such that $b_n\in \ker\left(B(i)\rightarrow B(j)\right)$ for all $n$, thus $(b_n)\in \ker\left(B(i)^{\N}\rightarrow B(j)^{\N} \right)$. This implies that $(b_n)$ represents the zero element in $\mathrm{colim}_I(B^{\N})$, and so
\[ \mathrm{colim}_I(B^{\N}) \rightarrow (\mathrm{colim}_{I} B)^{\N} \]
is injective as claimed.
\end{proof}

\begin{lemma} \label{lemma: pro zero} Let $X$ be a topological space, and $\mathscr{K}_n$ a uniformly bounded $\N$-indexed projective system of complexes on $X$. Suppose that for all $q\in \Z, n\in \N$, the transition maps in cohomology $\mathcal{H}^q(\mathscr{K}_{n+1})\rightarrow \mathcal{H}^q(\mathscr{K}_n)$ are zero. Then $\mathbf{R}\lim{n} \mathscr{K}_n=0$. 
\end{lemma}

\begin{proof}
By writing $\mathscr{K}_\bullet$ as an iterated extension of its cohomology sheaves, I can assume that in fact $\mathscr{K}_\bullet$ is concentrated in a single degree. The assumption that the transition maps $\mathscr{K}_{n+1}\rightarrow \mathscr{K}_{n}$ are zero implies that the natural morphism of projective systems
\[ \mathscr{K}_{\bullet+1}\to \scr{K}_{\bullet} \]
factors through the zero projective system. It follows that  the identity map on $\mathbf{R}\lim{n}\mathscr{K}_n$ factors through the zero map, and so $\mathbf{R}\lim{n}\mathscr{K}_n=0$. 
\end{proof}

\section{Constructible isocrystals} \label{sec: abelian constructible isocrystals}

Constructible isocrystals on varieties were introduced in \cite{LS16}, and their study continued in \cite{LS17,LS22}, where most results are phrased in terms of Le Stum's overconvergent site. It will be helpful to have a version of this formalism in the language of \emph{realisations} of constructible isocrystals on tubes. In this section, the aim is therefore to restate many of the foundational results of \cite{LS16}, but in the context of realisations.

For a frame $(X,Y,\fP)$, I want to be able to work locally on $\fP$, which makes it important to work without properness hypotheses on $Y$. This makes it necessary to develop a theory of constructible isocrystals not just for varieties, but for pairs $(X,Y)$ consisting of a variety $X$ and a partial compactification $Y$. In fact, this case is implicitly dealt with in \cite{LS16}, however, it will be more convenient here to phrase everything from the beginning in terms of realisations.

\subsection{Constructible modules}

I start by introducing the $\cO$-modules underlying constructible isocrystals. 

\begin{definition} Let $(X,Y,\fP)$ be a frame, and let $\sF$ be an $\cO_{\tube{X}_\fP}$-module. Then $\sF$ is called \emph{constructible} if there exists a stratification $\{i_\alpha \from X_\alpha \to X\}_{\alpha\in A}$ such that, for each $\alpha\in A$, $i_\alpha^{-1}\sF$ is a coherent $\cO_{\tube{X_\alpha}}$-module. It is called \emph{constructible locally free} if a stratification can be chosen such that each $i_\alpha^{-1}\sF$ is a locally finite free $\cO_{\tube{X_\alpha}_\fP}$-module.
\end{definition}

The following facts about constructible and constructible locally free modules can be verified immediately. 

\begin{enumerate}
\item If $\sF$ is an $\cO_{\tube{X}_\fP}$-module, and $\{i_\alpha \from X_\alpha \to X\}_{\alpha\in A}$ is a stratification of $X$, then $\sF$ is constructible (resp. constructible locally free) if and only if each $i_\alpha^{-1}\sF$ is.
\item If $i\from Z\to X$ is a locally closed immersion, and $\sF$ is an $\cO_{\tube{Z}_\fP}$-module, then $\sF$ is constructible (resp. constructible locally free) if and only if $i_!\sF$ is constructible (resp. constructible locally free) as an $\cO_{\tube{X}_\fP}$-module. The exact functors
\[ i_!\from {\bf Mod}(\cO_{\tube{Z}_\fP}) \leftrightarrows {\bf Mod}(\cO_{\tube{X}_\fP})\from i^{-1}  \]
induce an equivalence of categories between constructible (resp. constructible locally free) $\cO_{\tube{Z}_\fP}$-modules, and constructible (resp. constructible locally free) $\cO_{\tube{X}_\fP}$-modules which are supported on $\tube{Z}_\fP$. 
\item \label{num: cons properties 3} If $i\from Z\to X$ is a closed immersion, with open complement $j\from U\to X$, then every constructible (resp. constructible locally free) $\cO_{\tube{X}_\fP}$-module $\sF$ sits in a short exact sequence
\[ 0\to i_!i^{-1}\sF \to\sF\to j_*j^{-1}\sF\to 0 \]
where $i^{-1}\sF$ and $j^{-1}\sF$ are constructible (resp. constructible locally free) modules on $\tube{Z}_\fP$ and $\tube{U}_\fP$ respectively.
\item \label{num: cons properties 4} Every constructible (resp. constructible locally free) $\cO_{\tube{X}_\fP}$-module $\sF$ admits a finite composition  series
\[ 0 =\sF_0 \subset \sF_1 \subset \ldots \subset \sF_n =\sF, \]
such that, for each $1\leq \alpha \leq n$, there exists a locally closed subscheme $i_\alpha\from Z_\alpha\to X$, a coherent (resp. locally finite free) $\cO_{\tube{Z_\alpha}_\fP}$-module $\sG_\alpha$, and an isomorphism
\[ \sF_{\alpha}/\sF_{\alpha-1}\isomto i_{\alpha!}\sG_\alpha.  \]
\item The category of constructible (resp. constructible locally free) $\cO_{\tube{X}_\fP}$-module is an abelian subcategory of ${\bf Mod}(\cO_{\tube{X}_\fP})$, closed under extensions. In particular, every $\cO_{\tube{X}_\fP}$-module admitting a filtration as in (\ref{num: cons properties 4}) is constructible (resp. constructible locally free if each $\sG_\alpha$ is locally free). 
\item If $\sF$ and $\sG$ are constructible (resp. constructible locally free) then so is $\sF\otimes_{\cO_{\tube{X}_\fP}}\sG$.
\end{enumerate}

\begin{remark} Note that the functors $i_!i^{-1}$ and $j_*j^{-1}$ appearing in (\ref{num: cons properties 3}) are the functors $\underline{\Gamma}^\dagger_Z$ and $j_U^\dagger$ introduced in \S\ref{subsec: frames and tubes} above.
\end{remark}

\begin{lemma} Any constructible locally free $\cO_{\tube{X}_\fP}$-module is flat.
\end{lemma}

\begin{proof}
Let $\{i_\alpha \from X_\alpha \to X\}_{\alpha\in A}$ be a stratification such that each $i_\alpha^{-1}\sF$ is locally finite free. Note that the tubes $\tube{X_\alpha}_\fP$ cover $\tube{X}_\fP$, and, by definition, we have $\cO_{\tube{X_\alpha}_\fP}=i_\alpha^{-1}\cO_{\tube{X}_\fP}$. Since flatness can be checked on stalks, we can therefore replace $X$ by $X_\alpha$, and thus assume that $\sF$ is a locally finite free $\cO_{\tube{X}_\fP}$-module. In this case the claim is clear.
\end{proof}

\begin{remark} In general, constructible $\cO_{\tube{X}_\fP}$-modules which are not constructible locally free need not be flat.
\end{remark}

\begin{lemma} \label{lemma: cons Omod open} Suppose that $\sF\in \Mod(\cO_{\tube{X}_\fP})$. \begin{enumerate}
\item Let $\{\fr{P}_\alpha\}_{\alpha\in A}$ be a finite open cover of $\fr{P}$, and $X_\alpha:=X\cap \fP_\alpha$. Then $\sF$ is constructible (resp. constructible locally free) if and only if each $\sF|_{\tube{X_\alpha}_{\fr{P}_\alpha}}$ is. 
\item Let $\{X_\alpha\}_{\alpha\in A}$ be a finite open cover of $X$. Then $\sF$ is constructible (resp. constructible locally free) if and only if each $\sF|_{\tube{X_\alpha}_{\fr{P}}}$ is.
\end{enumerate}
\end{lemma} 

\begin{proof}
In both cases the `only if' direction is clear, I will prove the `if' direction.
\begin{enumerate}
\item Given stratifications of each $X_\alpha$, there exists a stratification of $X$ whose restriction to each $X_\alpha$ is a refinement of the given one of $X$. Working on each of the given strata, the question therefore amounts to showing that if $\sF$ is an $\cO_{\tube{X}_\fP}$-module, and each $\sF|_{\tube{X_\alpha}_{\fr{P}_\alpha}}$ is coherent (resp. locally finite free), then $\sF$ is coherent (resp. locally finite free). In this case, the tubes $\tube{X_\alpha}_{\fr{P}_\alpha}=\tube{X}_\fr{P}\cap \fr{P}_{\alpha K}$ form an open cover of $\tube{X}_{\fr{P}}$, and the claim then follows.
\item Arguing similarly, I need to show that if $\sF$ is an $\cO_{\tube{X}_\fP}$-module, and each $\sF|_{\tube{X_\alpha}_{\fr{P}}}$ is coherent (resp. locally finite free), then $\sF$ is coherent (resp. locally finite free). Thanks to \cite[Proposition 2.1.8]{Ber96b}, $\sF$ sits in an exact sequence
\[ 0 \to \sF\to \bigoplus_{\alpha\in A} j_{X_\alpha}^\dagger \sF \to \bigoplus_{\alpha,\beta} j_{X_\alpha\cap X_\beta}^\dagger \sF, \]
and applying \cite[Proposition 5.4.12]{LS07} shows that the unique such $\cO_{\tube{X}_\fP}$-module $\sF$ has to be coherent. If each $j_{X_\alpha}^\dagger \sF$ is moreover locally free, then following through the \emph{proof} of \cite[Proposition 5.4.12]{LS07}, which is essentially based upon Proposition \ref{prop: coherent sheaves colimit neighbourhoods} quoted above, proves that in fact $\sF$ is locally free. \qedhere
\end{enumerate}
\end{proof}

If $u\from (X',Y',\fP')\to (X,Y,\fP)$ is a morphism of frames, it is a straightforward check to show that the functor
\[ u^* \from {\bf Mod}(\cO_{\tube{X}_\fP})\rightarrow  {\bf Mod}(\cO_{\tube{X'}_{\fP'}}) \]
preserves both constructible and constructible locally free modules. I don't know in what generality, the projection formula holds for constructible $\cO_{\tube{X}_\fP}$-modules, but the following special cases will suffice.

\begin{proposition} \label{prop: weak proj formula for cons mod} Let $\fP$ be a flat formal scheme, $X\hookrightarrow \fP$ a locally closed immersion, $\sF$ a constructible locally free $\cO_{\tube{X}_\fP}$-module, and $f\from V\to \tube{X}_\fP$ a morphism of germs.
\begin{enumerate}
\item \label{num: weak prof formula for cons mod i} If $f$ is quasi-compact, then the map
\[ \bR f_* \cO_{V} \otimes^{\bL}_{\cO_{\tube{X}_\fP}} \sF \to \bR f_*\bL f^*\sF \]
is an isomorphism in $\bD(\cO_{\tube{X}_\fP})$. 
\item \label{num: weak prof formula for cons mod ii} If $f$ is partially proper, then there is a natural (in $\sF$) isomorphism
\[ \bR f_! \cO_{V} \otimes^{\bL}_{\cO_{\tube{X}_\fP}} \sF \isomto \bR f_!\bL f^*\sF \]
in $\bD^+(\cO_{\tube{X}_\fP})$. 
\end{enumerate}
\end{proposition}

\begin{proof}
The proofs in the two cases both ultimately rely on the same result in general topology: if $g\from S\to T$ is a coherent\footnote{quasi-compact and quasi-separated} morphism of locally coherent\footnote{coherent = quasi-compact, quasi-separated, and with a basis of quasi-compact opens} and sober\footnote{every irreducible subset has a unique generic point} topological spaces, $\sG$ is a sheaf on $S$, $t\in T$, and $G(t)\subset T$ is the set of generalisations of $t$, then
\[ (\bR^qg_*\sG)_t \isomto {\rm H}^q(g^{-1}(G(t)),\sG)\]
for all $q\geq 0$. This can be seen, for example, by taking a cofinal system $\{U_i\}_{i\in I}$ of neighbourhoods of $t$ in $T$, and applying \cite[Chapter 0, Proposition 3.1.19]{FK18} to the inverse system $g^{-1}(U_i)$. It then follows, by passing to stalks, that if $i\from T'\to T$ is the inclusion of a locally closed subset which is stable under generalisation, then the base change map
\[ i^{-1}\bR g_*\to \bR g'_* i'^{-1} \]
associated to the Cartesian square
\[ \xymatrix{ S'\ar[r]^-{i'} \ar[d]_{g'} & S \ar[d]^g \\ T'\ar[r]^-{i} & T } \]
is an isomorphism.

In case (\ref{num: weak prof formula for cons mod i}), the dependence on this base change result is explicit in the proof. In this case, since constructible locally free modules are flat, the tensor products and pullbacks in the statement above are really underived, thus I need to show that the map
\[  \bR f_* \cO_{V} \otimes_{\cO_{\tube{X}_\fP}} \sF \to \bR f_*f^*\sF \]
is an isomorphism. Choose a stratification $\left\{i_\alpha\from X_\alpha\to X\right\}_{\alpha \in A}$ such that each $\sF\!\mid_{\tube{X_\alpha}_\fP}$ is locally finite free. Letting
\[ f_\alpha\from V_\alpha:= f^{-1}\left(\tube{X_\alpha}_{\fP}\right) \to \tube{X_\alpha}_{\fP} \]
denote the pullback, the classical projection formula shows that
\[ \bR f_{\alpha*}\cO_{V_\alpha} \otimes_{\cO_{\tube{X_\alpha}_\fP}} (\sF\!\mid_{\tube{X_\alpha}_\fP}) \to \bR f_{\alpha*}f_\alpha^*(\sF\!\mid_{\tube{X_\alpha}_\fP})  \]
is an isomorphism. Hence, arguing via stalks, it suffices to show that for each $\alpha\in A$, the base change map
\[ (\bR f_*f^*\sF)\!\mid_{\tube{X_\alpha}_\fP} \to \bR f_{\alpha*}f_\alpha^*(\sF\!\mid_{\tube{X_\alpha}_\fP})  \]
is an isomorphism, together with the analogous result with $\sF$ replaced by $\cO_{\tube{X}_\fP}$. Since each $\tube{X_\alpha}_\fP\subset \tube{X}_\fP$ is stable under generalisation, this follows from the general topological result above.

In case (\ref{num: weak prof formula for cons mod ii}), the dependence on the above-mentioned `base change' result is hidden in the proof of Proposition \ref{prop: proper base change}. In this case, I first construct a morphism
\begin{equation} \label{eqn: proj form cons}\mathscr{F}\otimes_{\cO_{\tube{X}_\fr{P}}} \mathbf{R}f_{!}\cO_{V}\rightarrow \mathbf{R}f_{!}f^*\mathscr{F}.
\end{equation}
To do this, choose an injective resolution $\cO_{V}\rightarrow \mathscr{I}$. Then flatness of $\mathscr{F}$ over $\cO_{\tube{X}_\fr{P}}$ implies that $f^*\mathscr{F}\rightarrow f^*\mathscr{F}\otimes_{\cO_{V}} \mathscr{I}$ is a quasi-isomorphism. Next, choose an injective resolution $f^*\mathscr{F}\otimes_{\cO_{V}} \mathscr{I} \rightarrow \mathscr{J}$, so that $\mathscr{J}$ is also an injective resolution of $f^*\mathscr{F}$. It is therefore enough to construct a map
\[ \mathscr{F}\otimes_{\cO_{\tube{X}_{\fr{P}}}} f_{!}\mathscr{I} \rightarrow f_{!}\mathscr{J}. \]
But now, such a map can be found factoring through the map
\[ f_{!}(f^*\mathscr{F}\otimes_{\cO_{V}} \mathscr{I}) \rightarrow f_{!}\mathscr{J}\]
simply by noting that the map
\[  \mathscr{F}\otimes_{\cO_{\tube{X}_{\fr{P}}}} f_{*}\mathscr{I} \rightarrow f_{*}(f^*\mathscr{F}\otimes_{\cO_{V}} \mathscr{I})  \]
coming from adjunction sends 
\[ \mathscr{F}\otimes_{\cO_{\tube{X}_{\fr{P}}}} f_{!}\mathscr{I}\subset \mathscr{F}\otimes_{\cO_{\tube{X}_{\fr{P}}}} f_{*}\mathscr{I}\]
into
\[ f_{!}(f^*\mathscr{F}\otimes_{\cO_{V}} \mathscr{I})\subset  f_{*}(f^*\mathscr{F}\otimes_{\cO_{V}} \mathscr{I}).\]
The proof that (\ref{eqn: proj form cons}) is an isomorphism now proceeds exactly as in case (\ref{num: weak prof formula for cons mod i}), using d\'evissage and Proposition \ref{prop: proper base change} to reduce the projection formula for locally free sheaves, that is, to Lemma \ref{lemma: proj form}. 
\end{proof}

\subsection{Convergent stratifications and connections}

I now introduction convergent stratifications on $\cO_{\tube{X}_\fP}$-modules. If $(X,Y,\fP)$ is a frame, then embedding $Y$ in $\fP^2$ via the diagonal gives rise to a frame $(X,Y,\fP^2)$. Consider the two morphisms of frames
\[ p_0,p_1 \from (X,Y,\fP^2) \to (X,Y,\fP) \]
coming from the projection maps $p_i\from \fP^2\to \fP$, as well as the morphism
\[ \Delta\from (X,Y,\fP)\to (X,Y,\fP^2) \]
induced by the diagonal map $\Delta\from\fP\to\fP^2$, which is a common section to the $p_i$. As in Remark \ref{rem: notation abuse morphisms of tubes}, I will write
\[p_i\from \tube{X}_{\fP^2}\to\tube{X}_\fP,\;\;\;\;\Delta\from \tube{X}_\fP\to \tube{X}_{\fP^2} \]
for the induced morphisms of tubes.

\begin{definition} Let $(X,Y,\fP)$ be a frame with $\fP$ smooth around $X$, and $\sF$ an $\cO_{\tube{X}_\fP}$-module. Then a convergent stratification on $\sF$ is an isomorphism
\[ \epsilon \from p_1^*\sF \to p_0^*\sF \]
of $\cO_{\tube{X}_{\fP^2}}$-modules, restricting to the identity along $\Delta$, and satisfying the cocycle condition on $\tube{X}_{\fP^3}$.
\end{definition}

\begin{remark} \begin{enumerate}
\item Of course, the definition makes sense without any smoothness hypothesis on $\fP$, but will only be a reasonable one when $\fP$ is smooth around $X$.
\item Note that what I have called `convergent' here might more properly be termed `overconvergent along $Y\setminus X$'. Given the wide range of different kinds of isocrystals that will appear in this article, I have chosen the slightly simpler terminology to avoid multiplying adjectives beyond necessity. I will try to avoid any potential confusion this might cause. 
\end{enumerate} 
\end{remark}

Now let $\fP^{(n)}$ denote the $n$th infinitesimal neighbourhood of $\fP$ inside $\fP^2$. The restriction of $\epsilon$ to each $\tube{X}_{\fP^{(n)}}$ gives rise to a stratification on $\sF$ in the usual sense. Since $\fP$ is smooth around $X$, there exists an open neighbourhood $\tube{X}_\fP\subset V\subset \tube{Y}_\fP$ which is smooth, and hence any such stratification can be viewed as an integrable connection. This gives rise to a faithful (but not in general full) functor from modules with convergent stratification on $\tube{X}_{\fP}$ to modules with integrable connection on $\tube{X}_\fP$.

\begin{definition} \label{defn: isoc cons} A constructible isocrystal on $(X,Y,\fP)$ is a constructible $\cO_{\tube{X}_\fP}$-module $\sF$ together with a convergent stratification $\epsilon$ on $\sF$. The category of constructible isocrystals on $(X,Y,\fP)$ is denoted $\Isoc_\cons(X,Y,\fP)$. 
\end{definition}

In the special case that $X=Y=P$ (and therefore $\fP$ is everywhere smooth) I will generally write $\Isoc_\cons(\fP)$ instead of $\Isoc_\cons(P,P,\fP)$, and call these objects constructible isocrystals on $\fP$. 

\begin{lemma} \label{lemma: isoc loc free} The $\cO_{\tube{X}_\fP}$-module underlying any $\sF\in \Isoc_\cons(X,Y,\fP)$ is constructible locally free, and in particular, flat.
\end{lemma}

\begin{proof}
Since $\fP$ is smooth around $X$, it is also smooth around any locally closed subscheme of $X$. After possibly passing to a suitable stratification of $X$, I may assume that $\sF$ itself is coherent.

Applying Proposition \ref{prop: coherent sheaves colimit neighbourhoods} I can then choose an open neighbourhood $V$ of $\tube{X}_\fP$ such that $\sF$ extends to a coherent $\cO_V$-module. Since $\tube{X}_{\fP^{(n)}}$ has the same underlying topological spaces as $\tube{X}_\fP$, applying Proposition \ref{prop: coherent sheaves colimit neighbourhoods} on $\tube{X}_{\fP^{(n)}}$ shows that, after possibly shrinking $V$ further, the stratification (and hence the integrable connection) on $\sF$ also extends to $V$.

It therefore suffices to show that coherent $\cO_V$-modules admitting integrable connections are locally finite free. Since $V$ is smooth over $K$, this is well-known, see for example \cite[Lemma 3.3.4]{Ked06a}. 
\end{proof}

Inside $\Isoc_\cons(X,Y,\fP)$ there is a natural full subcategory consisting of objects whose underlying $\cO_{\tube{X}_\fP}$-module is coherent, or equivalently locally free. I will denote this subcategory by
\[ \Isoc(X,Y,\fP) \subset  \Isoc_\cons(X,Y,\fP) \]
and refer to its objects as locally free isocrystals on $(X,Y,\fP)$. Again, these objects have traditionally been called `partially overconvergent isocrystals', and I will try to avoid any confusion that may arise from the more simplistic terminology I've chosen to use here. 

\begin{proposition} \label{prop: full faithful conv strat} The forgetful functor from $\Isoc_\cons(X,Y,\fP)$ to $\cO_{\tube{X}_\fP}$-modules with integrable connection is fully faithful.
\end{proposition} 

\begin{proof}
This is \cite[Proposition 4.8]{LS16}.
\end{proof}

Thus $\Isoc_\cons(X,Y,\fP)$ can be viewed as a full subcategory of $\Mod(\sD_{\tube{X}_\fP})$. It is in fact an abelian subcategory by \cite[Proposition 4.3]{LS16}. It also follows relatively quickly from the definitions that if $\sF$ and $\sG$ are constructible isocrystals, then $\sF\otimes_{\cO_{\tube{X}_\fP}}\sG$ can be endowed with the structure of a constructible isocrystal, by Proposition \ref{prop: full faithful conv strat} this is uniquely determined by the fact that this is compatible (via the natural forgetful functor) with the tensor product (over $\cO_{\tube{X}_\fP}$) of $\sD_{\tube{X}_\fP}$-modules. Similarly, if 
\[ \xymatrix{ X'\ar[r]\ar[d] & Y' \ar[r]\ar[d] & \fP'\ar[d]\\ X\ar[r] & Y\ar[r] & \fP } \]
is a morphism of frames, with $\fP'$ smooth around $X'$ and $\fP$ smooth around $X$, then the pullback functor
\[ \tube{f}_u^*\from \Mod(\sD_{\tube{X}_\fP})\to \Mod(\sD_{\tube{X'}_{\fP'}})\]
sends $\Isoc_\cons(X,Y,\fP)$ into $\Isoc_\cons(X',Y',\fP')$, essentially because $\tube{f}_u^*$ commutes with the maps $p_i^*$ used to construct convergent stratifications. Of course, pullback commutes with tensor product in the obvious sense.

The analogues for constructible isocrystals of each of the `elementary facts' about constructible modules are true, but some of them are rather harder to prove. We start with the following.

\begin{lemma} \label{lemma: cons isoc locally closed} Let $i\from Z\to X$ be a locally closed immersion, and $\overline{Z}$ the closure of $Z$ in $P$. Then the exact functors
\[ i_!\from \Mod(\sD_{\tube{Z}_\fP}) \leftrightarrows \Mod(\sD_{\tube{X}_\fP}) \from i^{-1} \]
preserve constructible isocrystals, and induce an equivalence of categories between $\Isoc_\cons(Z,\overline{Z},\fP)$ and the full subcategory of $\Isoc_\cons(X,Y,\fP)$ consisting of objects supported on $\tube{Z}_\fP$.  
\end{lemma} 

\begin{remark} \label{rem: support Z ]Z[} In this situation, I will generally refer to constructible isocrystals being supported on $Z$, rather than being supported on $\tube{Z}_\fP$.
\end{remark}

\begin{proof}
The fact that each of $i^{-1}$ and $i_!$ preserve constructible isocrystals follows from the fact that each commutes with $p_i^*$. For $i^{-1}$ this is pretty much immediate from the definitions, for $i_!$ this is the content of Lemma \ref{lemma: u^*i_!}. The statement on the equivalence of categories then follows. 
\end{proof}

\begin{corollary} Let $i\from Z\to X$ be a closed immersion, with open complement $j\from U\to X$, and let $\overline{Z}$ be the closure of $Z$ inside $P$. Then every constructible isocrystal $\sF\in\Isoc_\cons(X,Y,\fP)$ sits in a short exact sequence
\[ 0\to i_!i^{-1}\sF \to\sF\to j_*j^{-1}\sF\to 0 \]
where $i^{-1}\sF\in \Isoc_\cons(Z,\overline{Z},\fP)$ and $j^{-1}\sF\in \Isoc_\cons(U,Y,\fP)$. 
\end{corollary}

\begin{corollary} \label{cor: cons isoc iterated extension} Let $\sF\in \Isoc_\cons(X,Y,\fP)$. Then $\sF$ admits a finite composition  series
\[ 0 =\sF_0 \subset \sF_1 \subset \ldots \subset \sF_n =\sF, \]
such that for each $1\leq \alpha \leq n$, there exists a locally closed subscheme $i_\alpha\from X_\alpha\to X$, with closure $\overline{X}_\alpha$ in $P$, a locally free isocrystal $\sG_\alpha\in \Isoc(X_\alpha,\overline{X}_\alpha,\fP)$, and an isomorphism
\[ \sF_{\alpha}/\sF_{\alpha-1}\isomto i_{\alpha!}\sG_\alpha.  \]
\end{corollary}

\begin{remark} I will call objects of the form $i_{\alpha!}\sG_\alpha$ `locally free isocrystals supported on $X_\alpha$'. Thus a more succinct expression of the corollary is that every constructible isocrystal on $X$ is an iterated extension of locally free isocrystals supported on locally closed subschemes of $X$.
\end{remark}

\begin{lemma} \label{lemma: cons isoc open} Suppose that $\sF\in \Mod(\sD_{\tube{X}_\fP})$. \begin{enumerate}
\item Let $\{\fr{P}_\alpha\}_{\alpha\in A}$ be a finite open cover of $\fr{P}$, and set $X_\alpha=X\cap \fP_\alpha$. Then $\sF$ is a constructible isocrystal if and only if each $\sF|_{\tube{X_\alpha}_{\fr{P}_\alpha}}$ is.
\item Let $\{X_\alpha\}_{\alpha\in A}$ be a finite open cover of $X$. Then $\sF$ is a constructible isocrystal if and only if each $\sF|_{\tube{X_\alpha}_{\fr{P}}}$ is.
\end{enumerate}
\end{lemma} 

\begin{proof}
As in the proof of Lemma \ref{lemma: cons Omod open}, in both cases the `only if' direction is clear, and I will concentrate on the `if' direction. In both cases, the constructibility was proved in \ref{lemma: cons Omod open}, so what is left to prove is the convergence of the stratification. The case of an open cover of $\fP$ is straightforward, since this induces an open cover of $\tube{X}_{\fr{P}^2}$ over which the convergent stratification $\epsilon$ can be constructed (that they coincide on the intersections follows from the full faithfulness result Proposition \ref{prop: full faithful conv strat}).

For the case of an open cover of $X$, I argue as in Lemma \ref{lemma: cons Omod open}, and write $\sF$ as the kernel of the map
\begin{equation}\label{eqn: cons isoc open}
\bigoplus_\alpha j_{X_\alpha}^\dagger \sF \rightarrow \bigoplus_{\alpha,\beta} j^\dagger_{X_\alpha\cap X_\beta}\sF.
\end{equation}
The fact that each $\sF|_{\tube{X_\alpha}_{\fr{P}}}$ is a constructible isocrystal, together with Lemma \ref{lemma: cons isoc locally closed}, implies that both terms in (\ref{eqn: cons isoc open}) are constructible isocrystals. Hence I can conclude using the fact that $\Isoc_\cons(X,Y,\fP)$ is an abelian subcategory of $\Mod(\sD_{\tube{X}_\fP})$.
\end{proof}

The final two properties to establish are the facts that constructible isocrystals can be detected over a stratification of $X$, and are stable under extensions. The proof of the latter of these depends on the former, which is a lot involved than the properties we have established so far, and is in fact the first result in this section that genuinely goes beyond what was proved in \cite{LS16}. The key calculation that we need in the proof (this calculation is Lemma \ref{lemma: inj cons formal} below) applies in the following situation.

\begin{setup} \label{setup: inj cons formal} Suppose that $(X,Y,\fr{P})$ is a frame with $\fP$ rig-smooth around $X$, and consider the morphism of frames \[ \xymatrix{  & & \widehat{\A}^d_\fP  \ar[d]^\pi \\ X \ar[r] & Y \ar[r]\ar[ur] & \fr{P}  } \]
where $Y$ is embedded into $\widehat{\A}^d_\fP$ via the zero section, and $\pi$ is the projection. Let $\fr{O}^{(n)}_\fP\subset \widehat{\A}^d_\fP$ denote the $n$th infinitesimal neighbourhood of the zero section, and write $\pi_n:\fr{O}^{(n)}_\fP \rightarrow \fr{P}$ for the projection. Let 
\[ i\from Z \to X \ot U\from j \]
be a complementary closed and open immersion. If $T=X,U$ or $Z$ we write
\begin{align*}
\pi &\from \tube{T}_{\widehat{\A}^d_\fP}\rightarrow \tube{T}_\fr{P} \\
\pi_n&\from \tube{T}_{\fr{O}_\fP^{(n)}} \rightarrow \tube{T}_\fr{P}
\end{align*} for the projections. Similarly, if $\fr{T}=\fr{P}$, $\widehat{\A}^d_\fP$ or $\fr{O}^{(n)}_\fr{P}$ we write
\[ i \from \tube{Z}_\fr{T} \to \tube{X}_\fr{T} \ot \tube{U}_\fr{T} \from j\]
for the immersions. 
\end{setup}

\begin{lemma} \label{lemma: inj cons formal} Consider Setup \ref{setup: inj cons formal}, and assume that, locally on $Y$, $U\subset Y$ is the complement of a hypersurface. Let $\sF$ be a coherent $\cO_{\tube{U}_{\widehat{\A}^d_\fP}}$-module, and $\mathscr{G}$ a constructible locally free $\mathcal{O}_{\tube{Z}_\fr{P}}$-module. Then the natural map
\[ {\rm Hom}_{\cO_{\tube{U}_{\widehat{\A}^d_\fr{P}}}}(\sF,j^{-1}i_{*}\pi^*\mathscr{G}) \rightarrow  \lim{n}{\rm Hom}_{\cO_{\tube{U}_{\fr{O}^{(n)}_\fr{P}}}}(\sF,j^{-1}i_{*}\pi^*\mathscr{G}) \]
is injective.
\end{lemma}

\begin{proof}
The question is local on $\fP$ (and hence $Y$), so we may assume that $\fP$ is affine, and that $U\subset Y$ is the complement of a hypersurface. The functors ${\rm Hom}(\sF,-)$, $j^{-1}$, $i_{*}$, $\pi^*$, $\lim{n}$, $\pi_n^*$ are all left exact, and so the claim is additive over short exact sequences in $\mathscr{G}$. I can therefore assume that there exists a locally closed immersion $a\from T\rightarrow Z$ and a locally free $\mathcal{O}_{\tube{T}_\fr{P}}$-module $\mathscr{H}$ such that $\mathscr{G}=a_!\mathscr{H}$. I will continue to use $\pi$ and $\pi_n$ to denote the natural maps
\begin{align*}
\pi &\from \tube{T}_{\widehat{\A}^d_\fr{P}}\rightarrow \tube{T}_\fr{P} \\
\pi_n &\from \tube{T}_{\fr{O}^{(n)}_\fr{P}}\rightarrow \tube{T}_\fr{P}.
\end{align*}
Since $a_!$ commutes with $\pi^*$ by Lemma \ref{lemma: u^*i_!}, the natural map $a_!\rightarrow a_*$ is injective, and $i_*$, $j^{-1}$ are left exact, there is a natural injective map
\[ j^{-1}i_{*}\pi^*\mathscr{G}=j^{-1}i_{*}\pi^*a_!\mathscr{H}=j^{-1}i_{*}a_!\pi^*\mathscr{H}\rightarrow j^{-1}i_{*}a_*\pi^*\mathscr{H}.\]
Similarly, there is an injective map
\[ j^{-1}i_{*}\pi_n^*\mathscr{G} \rightarrow j^{-1}i_{*}a_*\pi_n^*\mathscr{H}.\]
Again appealing to left exactness of ${\rm Hom}$ and ${\rm lim}$, it suffices to show that the map
\begin{equation} \label{eqn: inj open disc formal}
{\rm Hom}_{\cO_{\tube{U}_{\widehat{\A}^d_\fr{P}}}}(\sF,j^{-1}i_{*}a_*\pi^*\mathscr{H}) \rightarrow  \lim{n}{\rm Hom}_{\cO_{\tube{U}_{\fr{O}^{(n)}_\fr{P}}}}(\sF, j^{-1}i_{*}a_*\pi_n^*\mathscr{H}) 
\end{equation} 
is injective. Next, let $i'\from Z\rightarrow Y$ and $j':U\rightarrow Y$ denote the inclusions into $Y$, thus
\[ j'^{-1}i'_{*}=j^{-1}i_{*}. \]
This shows that I can replace $j^{-1}i_{*}a_*\pi^*\mathscr{H}$ with $j'^{-1}i'_{*}a_*\pi^*\mathscr{H}$, and similarly $j^{-1}i_{*}a_*\pi_n^*\mathscr{H}$ with $j'^{-1}i'_{*}a_*\pi_n^*\mathscr{H}$.

Now, the given claim is local on the tube $\tube{Y}_\fr{P}$ of $Y$, so replacing $\fr{P}$ by a formal model for the closed tubes $[Y]_{\fr{P}\eta}$ for varying $\eta$, I can assume that $Y=P$. The LHS of (\ref{eqn: inj open disc formal}) is now identified with
\[ {\rm Hom}_{\cO_{\D^d_K(0;1^-)\times_K\tube{U}_\fr{P}}}(\sF,j'^{-1}i'_{*}a_*\pi^*\mathscr{H} )=\lim{\rho<1} {\rm Hom}_{\cO_{\D^d_K(0;\rho)\times_K\tube{U}_\fr{P}}}(\sF,j'^{-1}i'_{*}a_*\pi^*\mathscr{H} ),  \]
so it is enough to show that the natural map
\[ {\rm Hom}_{\cO_{\D^d_K(0;\rho)\times_K\tube{U}_\fr{P}}}(\sF,j'^{-1}i'_{*}a_*\pi^*\mathscr{H} ) \rightarrow \lim{n}{\rm Hom}_{\cO_{\tube{U}_{\fr{O}^{(n)}_\fr{P}}}}(\sF, j^{-1}i_{*}a_*\pi_n^*\mathscr{H})  \]
is injective, for each $\rho<1$.

Let $\{V_\lambda\}_{\lambda \geq \lambda_0}$ be a cofinal sequence of neighbourhoods of $\tube{U}_\fr{P}$ in $\tube{Y}_\fr{P}=\fr{P}_K$. Since $\fP$ is affine, and $U$ is the complement of a hypersurface in $Y=P$, we can take each $V_\lambda$ to be affinoid. 

By quasi-compactness, $\D^d_K(0;\rho)\times_K V_\lambda$ (resp. $\fr{O}^{(n)}_{\fP,K} \times_{\fP_K} V_\lambda$) is a cofinal system of neighbourhoods of $\D^d_K(0;\rho)\times_K \tube{U}_\fr{P}$ (resp. $\fr{O}^{(n)}_{\fP,K}\times_{\fP_K}\tube{U}_\fr{P}$) inside $\D^d_K(0;\rho)\times_K \fr{P}_K$ (resp. $\fr{O}^{(n)}_{\fP,K}$). Thus, for all sufficiently large $\lambda$, $\sF$ extends to a coherent sheaf on $\D^d_K(0;\rho)\times_K V_\lambda$ by Proposition \ref{prop: coherent sheaves colimit neighbourhoods}. Since $\D^d_K(0;\rho)\times_K V_\lambda$ is affinoid, the extension of $\sF$ to $\D^d_K(0;\rho)\times_K V_\lambda$ is generated by finitely many global section. Hence $\sF$ is generated by finitely many global sections as a coherent $\cO_{\D^d_K(0;\rho)\times_K \tube{U}_\fr{P}}$-module. Thus choosing a surjection
\[ \cO_{\D^d_K(0;\rho)\times_K \tube{U}_\fr{P}}^{\oplus m} \twoheadrightarrow \sF,   \]
and once again appealing to left exactness, I can reduce to the case $\sF=\cO_{\D^d_K(0;\rho)\times_K \tube{U}_\fr{P}}$, in other words to showing that
\[ \Gamma(\D^d_K(0;\rho)\times_K\tube{U}_\fr{P},j'^{-1}i'_{*}a_*\pi^*\mathscr{H} ) \rightarrow \lim{n}\Gamma(\fr{O}_{\fP,K}\times_{\fP_K} \tube{U}_{\fr{P}}, j^{-1}i_{*}a_*\pi_n^*\mathscr{H})  \]
is injective. Writing things out explicitly, this is the map
\[ \colim{\lambda} \Gamma(\D^d_K(0;\rho) \times_K (V_\lambda\cap \tube{T}_\fr{P}),\pi^*\mathscr{H})\rightarrow \lim{n}\colim{\lambda} \Gamma(\fr{O}^{(n)}_{\fP,K} \times_{\fP_K} (V_\lambda\cap \tube{T}_\fr{P}),\pi_n^*\mathscr{H}). \]
By an explicit calculation, I can identify 
\[
\Gamma(\fr{O}^{(n)}_{\fP,K} \times_{\fP_K} (V_\lambda\cap \tube{T}_\fr{P}),\pi_n^*\mathscr{H}) =\Gamma(V_\lambda\cap \tube{T}_\fr{P},\mathscr{H}) \otimes_K \frac{K[z_1,\ldots,z_d]}{(z_1,\ldots,z_d)^{n+1}}. \]
Thus, by applying Lemma \ref{lemma: inj prod colim}, it suffices to show that:
\begin{enumerate}
\item for each sufficiently large $\lambda$, the map
\[  \Gamma(\D^d_K(0;\rho) \times_K (V_\lambda\cap \tube{T}_\fr{P}),\pi^*\mathscr{H})\rightarrow \lim{n}  \Gamma(\fr{O}^{(n)}_{\fP,K}\times_{\fP_K} (V_\lambda\cap \tube{T}_\fr{P}),\pi_n^*\mathscr{H}) \]
is injective;
\item for each $\lambda$, the kernels of the transition maps
\[ \Gamma(V_\lambda\cap \tube{T}_\fr{P},\mathscr{H})\rightarrow \Gamma(V_{\lambda'}\cap \tube{T}_\fr{P},\mathscr{H}) \]
eventually stabilise.
\end{enumerate}
Since each $V_\lambda\cap \tube{T}_\fr{P}$ has only finitely many connected components, the second of these follows from Lemma \ref{lemma: smooth connected injective}. 

For the first, note that the claim can be checked over an open cover of $\tube{T}_\fr{P}$. I am  therefore again free to replace the (non-quasi-compact) open tube $\tube{T}_\fr{P}$ by the (quasi-compact) closed tube $\tube{T}_\fr{P}\cap [\overline{T}]_{\fr{P}\eta}$, where $\overline{T}$ is the closure of $T$ in $Y$. 

Now choose a cofinal sequence $\{W_\delta\}_{\delta\geq \delta_0}$ of open neighbourhoods of $\tube{T}_\fr{P}\cap [\overline{T}]_{\fr{P}\eta}$ in $[\overline{T}]_{\fr{P}\eta}$. Then the spaces $\D^d_K(0;\rho)\times_K (V_\lambda \cap W_\delta)$ and $\fr{O}^{(n)}_{\fP,K}\times_{\fP_K} (V_\lambda \cap W_\delta)$ form cofinal sequences of neighbourhoods of
\[ \D^d_K(0;\rho)\times_K  (V_\lambda \cap \tube{T}_\fr{P}\cap [\overline{T}]_{\fr{P}\eta}) \text{ in }\D^d_K(0;\rho)\times_K  (V_\lambda \cap [\overline{T}]_{\fr{P}\eta})\]
and 
\[ \fr{O}^{(n)}_{\fP,K}\times_{\fP_K}  (V_\lambda \cap \tube{T}_\fr{P}\cap [\overline{T}]_{\fr{P}\eta}) \text{ in }\fr{O}^{(n)}_{\fP,K}\times_{\fP_K}  (V_\lambda \cap[\overline{T}]_{\fr{P}\eta}) \]
respectively. I may also assume, by increasing $\delta_0$ if necessary, that $\mathscr{H}$ extends to a locally free sheaf on $W_{\delta_0}$, and therefore on all $W_{\delta}$. I thus need to show that the map
\[
 \colim{\delta}\Gamma(\D^d_K(0;\rho) \times_K (V_\lambda\cap W_\delta),\pi^*\mathscr{H} )\rightarrow \lim{n}  \colim{\delta}\Gamma(\fr{O}^{(n)}_{\fP,K} \times_{\fP_K} (V_\lambda\cap W_\delta),\pi_n^*\mathscr{H} ) 
\]
is injective, at least for large enough $\lambda$. Again appealing to Lemma \ref{lemma: inj prod colim}, it is enough to prove that:
\begin{enumerate}
\item[(A)] the map 
\begin{equation} \label{eqn: second inj} \Gamma(\D^d_K(0;\rho) \times_K (V_\lambda\cap W_\delta),\pi^*\mathscr{H} )\rightarrow \lim{n} \Gamma(\fr{O}^{(n)}_{K} \times_K (V_\lambda\cap W_\delta),\pi_n^*\mathscr{H} )  \end{equation}
is injective;
\item[(B)] for each $\eta$, the kernels of the transition maps
\[ \Gamma(V_\lambda\cap W_\delta,\mathscr{H}) \rightarrow \Gamma(V_{\lambda}\cap W_{\delta'},\mathscr{H}) \]
eventually stabilise. 
\end{enumerate}
Again, the second of these follows from Lemma \ref{lemma: smooth connected injective}, since each  $V_{\lambda}\cap W_{\delta}$ has only finitely many connected components. For the first, take an open affinoid $\spa{R,R^+}\subset V_\lambda\cap W_\delta$, and set $M=\Gamma(\spa{R,R^+},\mathscr{H})$. Thus $M$ is a finite projective $R$-module. The given map can be identified with the map
\[ R\tate{\rho^{-1}\bm{z}}\otimes_R M \rightarrow R\pow{\bm{z}} \otimes_R M \]
from the restricted power series ring to the full power series ring, tensored with $M$. Since this map is injective for each such $\spa{R,R^+}$, the map \eqref{eqn: second inj} is injective as required.
\end{proof}

I can now show that constructible isocrystals can be detected over a stratification. 

\begin{proposition} \label{prop: cons isoc strat} Let $\{i_\alpha:X_\alpha\rightarrow X\}_{\alpha\in A}$ be a stratification of $X$, and let $\sF\in \mathrm{Mod}(\mathscr{D}_{\tube{X}_\fr{P}})$. Then $\sF$ is a constructible isocrystal if and only if, for each $\alpha\in A$, $i_\alpha^{-1}\mathscr{F}\in \mathrm{Mod}(\mathscr{D}_{\tube{X_\alpha}_\fr{P}})$ is a constructible isocrystal.
\end{proposition}

\begin{proof}
The `only if' direction has already been proved. For the `if' direction, by induction on $\norm{A}$, I may assume that there are precisely two strata $\{U,Z\}$ consisting of an open subscheme $j:U\rightarrow X$ and a closed complement $i:Z\rightarrow X$. Thanks to Lemma \ref{lemma: cons isoc open} the property of being a constructible isocrystal is local on both $\fr{P}$ and $X$, so I may assume that both are affine. Moreover, I can appeal to Noetherian induction to shrink $U$ if required. I can therefore assume that $U$ is the complement of a hypersurface in $Y$, and that $j^{-1}\mathscr{F}$ is a locally free isocrystal on $(U,Y,\fP)$.

In this case, I will follow closely the proof of \cite[Proposition 5.11]{LS14}. If $T=X,U$ or $Z$, I will write
\begin{align*}
p_i &\from \tube{T}_{\fr{P}^2}\rightarrow \tube{T}_\fr{P} \\ 
p_i^{(n)} &\from \tube{T}_{\fr{P}^{(n)}} \rightarrow \tube{T}_\fr{P}
\end{align*} for the projections, and if $\fr{T}=\fr{P},\fr{P}^2$ or $\fr{P}^{(n)}$, I will write
\[ i\from\tube{Z}_\fr{T}\to \tube{X}_\fr{T}\ot \tube{U}_\fr{T}\from j \]
for the immersions. There are therefore isomorphisms
\[ \epsilon_U:j^{-1}p_1^*\mathscr{F}\isomto j^{-1}p_0^*\mathscr{F},\;\;\;\; \epsilon_Z:i^{-1}p_1^*\mathscr{F}\isomto i^{-1}p_0^*\mathscr{F}\]
extending the Taylor isomorphisms on each $\tube{X}_{\fr{P}^{(n)}}$, and I need to show that these glue to an isomorphism $p_1^*\mathscr{F}\rightarrow p_2^*\mathscr{F}$ on $\tube{X}_{\fr{P}^2}$.

To do this, I will use the fact that that sheaves on $\tube{X}_{\fr{P}^2}$ are determined by their restriction to each of $\tube{U}_{\fr{P}^2}$ and $\tube{Z}_{\fr{P}^2}$, together with the natural adjunction morphism between then. Moreover, this construction gives rise to an equivalence of categories. Thus $\epsilon_U$ and $\epsilon_Z$ glue to give the required stratification if and only if the diagram
\[ \xymatrix{ j^{-1}p_1^*\mathscr{F}\ar[d]_{\epsilon_U} \ar[r] & j^{-1}i_{*}i^{-1}p_1^*\mathscr{F} \ar[d]^{\epsilon_Z} \\  j^{-1}p_0^*\mathscr{F} \ar[r] & j^{-1}i_{*}i^{-1}p_0^*\mathscr{F}  } \]
commutes. Equivalently, this happens if and only if the difference between the two maps
\[ j^{-1}p_1^*\mathscr{F} \rightarrow j^{-1}i_{*}i^{-1}p_0^*\mathscr{F} \]
is zero. Note that the analogous diagram
\[ \xymatrix{ j^{-1}p_1^{(n)*}\mathscr{F}\ar[d]_{\epsilon_U} \ar[r] & j^{-1}i_{*}i^{-1}p_1^{(n)*}\mathscr{F} \ar[d]^{\epsilon_Z} \\  j^{-1}p_0^{(n)*}\mathscr{F} \ar[r] & j^{-1}i_{*}i^{-1}p_0^{(n)*}\mathscr{F}  } \]
commutes, for all $n$, so it suffices to show that the natural map
\begin{equation}
\mathrm{Hom}_{\mathcal{O}_{\tube{U}_{\fr{P}^2}}}(j^{-1}p_1^*\mathscr{F} , j^{-1}i_{*}i^{-1}p_0^*\mathscr{F}) \rightarrow \lim{n}\mathrm{Hom}_{\mathcal{O}_{\tube{U}_{\fr{P}^{(n)}}}}(j^{-1}p_1^{(n)*}\mathscr{F} , j^{-1}i_{*}i^{-1}p_0^{(n)*}\mathscr{F})
\end{equation} 
is injective. Since $\fr{P}$ is smooth over $\mathcal{V}$ in a neighbourhood of $X$, I can now appeal to Berthelot's strong fibration theorem (see in particular the form stated in Theorem \ref{theo: strong fibration theorem} above). This says that, after localising on $X$ and $\fP$, I can replace the morphism of frames
\[ p_0:(X,Y,\fr{P}^2)\rightarrow (X,Y,\fr{P})\]
by the natural projection	
\[ \pi:(X,Y,\widehat{\A}^d_\fr{P})\rightarrow (X,Y,\fr{P}),\]
where $d$ is the relative dimension of $\fr{P}$ over $\mathcal{V}$ around $X$, and $Y$ is embedded into $\widehat{\A}^d_\fr{P}$ via the zero section. Since $\fr{P}^{(n)}$ now gets identified with the $n$th infinitesimal neighbourhood of the zero section $\fr{P}\rightarrow \widehat{\A}^d_\fr{P}$, the claim is now precisely the one proved in Lemma \ref{lemma: inj cons formal} above. 
\end{proof}

\begin{corollary} \label{cor: cons isoc weak Serre} The subcategory $\Isoc_\cons(X,Y,\fP)\subset \Mod(\sD_{\tube{X}_\fP})$ is closed under extensions.
\end{corollary} 

\begin{proof}
Let 
\[ 0 \to \sF\to \sG\to \mathscr{H}\to 0 \]
be a short exact sequence of $\sD_{\tube{X}_\fP}$-modules, with both $\sF,\mathscr{H}\in \Isoc_\cons(X,Y,\fP)$. Thanks to Proposition \ref{prop: cons isoc strat}, it suffices to find a stratification of $X$ over which $\mathscr{G}$ becomes a constructible isocrystal. To do so, choose stratifications of $X$ over which $\mathscr{F}$ and $\mathscr{H}$ become locally free isocrystals, and take a common refinement. This gives a stratification over which both $\mathscr{F}$ and $\mathscr{H}$ become locally free isocrystals, in which case it is proved in \cite[Proposition 1.2.2]{CLS99a} that $\mathscr{G}$ is a locally free isocrystal.
\end{proof}

I leave for now the question of to what extent the category $\Isoc_\cons(X,Y,\fP)$ is independent of $Y$ and $\fP$, since I shall shortly deduce it from a more general invariance result for the derived analogue of these categories.

\section{The derived category of constructible isocrystals}
\label{sec: derived constructible isocrystals}

In \S\ref{sec: abelian constructible isocrystals} above I introduced abelian categories of constructible isocrystals. The next task is to establish the basic properties of the analogous triangulated categories.

\subsection{Triangulated categories of constructible isocrystals} \label{sec: disoc triangle}

Let $(X,Y,\fr{P})$ be a frame, with $\fr{P}$ smooth around $X$. The following is the derived analogue of Definition \ref{defn: isoc cons}.

\begin{definition} \label{defn: db cons} A constructible complex on $(X,Y,\fr{P})$ is a complex $\mathscr{K} \in {\bf D}^b(\mathscr{D}_{\tube{X}_{\fr{P}}})$ whose cohomology sheaves are constructible isocrystals. We define ${\bf D}^b_\mathrm{cons}(X,Y,\fr{P})$ to be the corresponding full subcategory of ${\bf D}^b(\mathscr{D}_{\tube{X}_{\fr{P}}})$.
\end{definition}

\begin{remark} \begin{enumerate}
\item  Note that if $\fP$ is everywhere smooth, and $i:X\rightarrow P$ is the given immersion, then the extension by zero functor 
\[ i_!:{\bf D}^b(\mathscr{D}_{\tube{X}_{\fr{P}}})\rightarrow {\bf D}^b(\mathscr{D}_{\fr{P}_K})\]
is fully faithful, with essential image those constructible complexes supported on $X$ (see Remark \ref{rem: support Z ]Z[}). 
\item Again, if $X=Y=P$, I will generally write $\bD^b_\cons(\fP)$ instead of $\bD^b_\cons(P,P,\fP)$.
\end{enumerate}
\end{remark}

\begin{lemma} \label{lemma: basic properties Db cons} The subcategory $\bD^b_\cons(X,Y,\fP)\subset \bD^b(\sD_{\tube{X}_\fP})$ is triangulated. Moreover, it is the smallest triangulated subcategory containing all locally free isocrystals supported on locally closed subschemes of $X$. 
\end{lemma} 

\begin{proof}
This follows from Corollaries \ref{cor: cons isoc iterated extension} and \ref{cor: cons isoc weak Serre}. 
\end{proof}

Since constructible isocrystals are $\mathcal{O}_{\tube{X}_\fr{P}}$-flat by Lemma \ref{lemma: isoc loc free}, and stable under tensor product over $\cO_{\tube{X}_\fP}$, this tensor product descends to a t-exact functor
\[ -\otimes_{\mathcal{O}_{\tube{X}_\fr{P}}}-:{\bf D}^b_\mathrm{cons}(X,Y,\fr{P})\times {\bf D}^b_\mathrm{cons}(X,Y,\fr{P}) \rightarrow {\bf D}^b_\mathrm{cons}(X,Y,\fr{P}). \]
Similarly, if $(f,g,u):(X',Y',\fr{P}')\rightarrow (X,Y,\fr{P})$ is a morphism of frames, the pullback functor $\tube{f}_u^*$ from $\mathscr{D}_{\tube{X}_\fr{P}}$-modules to $\mathscr{D}_{\tube{X'}_{\fr{P}'}}$-modules preserves constructible isocrystals and derives trivially to give a t-exact functor
\[ \tube{f}_u^*:{\bf D}^b_\mathrm{cons}(X,Y,\fr{P}) \rightarrow {\bf D}^b_\mathrm{cons}(X',Y',\fr{P}'). \]
Pullback and tensor product commute in the obvious sense.

\subsection{Independence of the frame} \label{subsec: independence of the frame}

The key point will be to show that ${\bf D}^b_\mathrm{cons}(X,Y,\fr{P})$ enjoys the same functorial properties as classical categories of (locally free) isocrystals $\Isoc(X,Y,\fP)$. As in that case, the crucial invariance result is the following.

\begin{theorem}  \label{theo: disoc indep} Let
\[ \xymatrix{ & Y' \ar[r] \ar[d]^g & \fr{P}' \ar[d]^u \\ X \ar[ur]	 \ar[r] & \ar[r]  Y  & \fr{P} }\]
be a morphism of frames, such that $g$ is proper, and both $\fr{P}$ and $u$ are smooth around $X$. Let $d$ be the relative dimension of $u$, and write $u=\tube{\id}_u$ for the induced morphism $\tube{X}_{\fP'}\to \tube{X}_\fP$. Then the functor
\[ u^*\from {\bf D}(\mathscr{D}_{\tube{X}_\fr{P}})\to {\bf D}(\mathscr{D}_{\tube{X}_{\fr{P}'}}) \]
induces an t-exact equivalence of categories
\[ {\bf D}^b_\mathrm{cons}(X,Y,\fr{P}) \rightarrow {\bf D}^b_\mathrm{cons}(X,Y,\fr{P}'). \]
The functors
\begin{align*}
\mathbf{R}u_{\dR!}[2d] &\from {\bf D}(\mathscr{D}_{\tube{X}_{\fr{P}'}}) \to {\bf D}(\mathscr{D}_{\tube{X}_{\fr{P}}}) \\
\mathbf{R}u_{\dR*} &\from {\bf D}(\mathscr{D}_{\tube{X}_{\fr{P}'}}) \to {\bf D}(\mathscr{D}_{\tube{X}_{\fr{P}}})
\end{align*}
both induce quasi-inverses to $u^*$.
\end{theorem}

To begin the proof of Theorem \ref{theo: disoc indep}, I make the following simple observation. 

\begin{lemma} \label{lemma: disoc indep prelim 1} The functor $u^*$ sends ${\bf D}^b_\mathrm{cons}(X,Y,\fr{P})$ into ${\bf D}^b_\mathrm{cons}(X,Y',\fr{P}')$, and the essential image contains all locally free isocrystals supported on locally closed subschemes of $X$.
\end{lemma}

\begin{proof}
The first claim has already been observed in \S\ref{sec: disoc triangle} above. The second claim is then a consequence of the following classical result, proved for example in \cite[Theorem 7.1.8]{LS07}: for any locally closed subscheme $i:Z\hookrightarrow X$, with closures $\overline{Z}$ in $Y$ and $\overline{Z}'$ in $Y'$, the pullback functor
\[ u^*:\Isoc(Z,\overline{Z},\fr{P})\rightarrow  \Isoc(Z,\overline{Z}',\fr{P}') \]
induced by the (abusively denoted) morphism
\[ u\from \tube{Z}_{\fP'}\to \tube{Z}_\fP \]
on locally free isocrystals is an equivalence of categories. Indeed, if $\mathscr{F}$ is a locally free isocrystal on $(Z,\overline{Z}',\fr{P}')$, then $\mathscr{F}=u^*\mathscr{G}$ for some locally free isocrystal on $(Z,\overline{Z},\fr{P})$. It then follows from Lemma \ref{lemma: u^*i_!} that $i_!\mathscr{F}=i_!u^*\mathscr{G}=u^*i_!\mathscr{G}$ is in the essential image of $u^*$.
\end{proof}

Proving Theorem \ref{theo: disoc indep} eventually boils down to a local calculation in the same situation as that considered in Lemma \ref{lemma: inj cons formal}.

\begin{proposition} \label{prop: equiv rhom}
Consider Setup \ref{setup: inj cons formal}, and suppose that $\mathscr{F},\mathscr{G}\in\Mod(\mathscr{D}_{\tube{X}_\fr{P}})$ are constructible as $\mathcal{O}_{\tube{X}_\fr{P}}$-modules. 
\begin{enumerate}
\item \label{num: equiv rhom 1} The natural map $\mathscr{F}\rightarrow \mathbf{R}\pi_{\dR *}\pi^*\mathscr{F}$ is an isomorphism in ${\bf D}(\mathscr{D}_{\tube{X}_\fr{P}})$.
\item The natural map
\[ \mathbf{R}\mathrm{Hom}_{\mathscr{D}_{\tube{X}_\fr{P}}}(\mathscr{F},\mathscr{G})\rightarrow \mathbf{R}\mathrm{Hom}_{\mathscr{D}_{\tube{X}_{\widehat{\A}^d_\fr{P}}}}(\pi^*\mathscr{F},\pi^*\mathscr{G}) \]
is an isomorphism in ${\bf D}(K)$. 
\end{enumerate}
\end{proposition}

\begin{remark} Note that the hypotheses of Theorem \ref{theo: disoc indep} have been weakened slightly here. First of all, I only require that $\fr{P}$ is rig-smooth, rather than smooth, in a neighbourhood of $X$. Secondly, I only require that $\mathscr{F}$ and $\mathscr{G}$, are constructible, and not necessarily convergent. This implies that the full faithfulness part of Theorem \ref{theo: disoc indep} holds under similarly weaker hypotheses. The essential surjectivity, of course, does not.
\end{remark}

\begin{proof}
\begin{enumerate}
\item First of all, the claim is local on $\tube{Y}_\fr{P}$, and so, replacing $\fr{P}$ by a formal model for the closed tube $[Y]_{\fr{P}\eta}$, I can assume that $Y=P$. Letting $j\from X\rightarrow P$ denote the given open immersion, the claim for $\mathscr{F}$ on $\tube{X}_\fr{P}$ is equivalent to the claim for $j_*\mathscr{F}$ on $\fr{P}_K$ by Lemma \ref{lemma: u^*i_!}, so I can also assume that $X=P$.

Now, for any $\rho<1$, let $\pi_\rho:\D^d_{\fP_K}(0;\rho)\rightarrow \fr{P}_K$ denote the projection, thus
\[ \mathbf{R}\pi_{\dR*}\pi^*\mathscr{F} \isomto \mathbf{R}\lim{\rho}\bR\pi_{\rho\dR*}\pi_\rho^*\mathscr{F}. \]
Since $\pi_\rho$ is quasi-compact, it follows from Proposition \ref{prop: weak proj formula for cons mod} that
\[ \bR\pi_{\rho\dR*}\cO_{\D^d_{\fP_K}(0;\rho)} \otimes_{\cO_{\fP_K}}\mathscr{F} \isomto \bR\pi_{\rho\dR*}\pi_\rho^*\mathscr{F}, \]
thus 
\[ \mathbf{R}\pi_{\dR*}\pi^*\mathscr{F} \isomto \mathbf{R}\lim{\rho}\left(\bR\pi_{\rho\dR*}\cO_{\D^d_{\fP_K}(0;\rho)} \otimes_{\cO_{\fP_K}}\mathscr{F} \right). \]
Next, let $\pi_{\rho^-}:\D^d_{\fP_K}(0;\rho^-)\rightarrow \fr{P}_K$ denote the projection from the \emph{open} disc, thus
\[ \mathbf{R}\lim{\rho}\left(\bR\pi_{\rho\dR*}\cO_{\D^d_{\fP_K}(0;\rho)} \otimes_{\cO_{\fP_K}}\mathscr{F} \right) \isomto \mathbf{R}\lim{\rho}\left(\bR\pi_{\rho^-\dR*}\cO_{\D^d_{\fP_K}(0;\rho^-)} \otimes_{\cO_{\fP_K}}\mathscr{F} \right). \]
It therefore suffices to show that
\[ \cO_{\fP_K}\isomto \bR\pi_{\rho^-\dR*}\cO_{\D^d_{\fP_K}(0;\rho^-)}, \]
which is a well-known calculation in rigid analytic geometry.  
\item This is a simple consequence of adjunction and part (\ref{num: equiv rhom 1}):
\[ \mathbf{R}\mathrm{Hom}_{\mathscr{D}_{\tube{X}_{\widehat{\A}^d_\fr{P}}}}(\pi^*\mathscr{F},\pi^*\mathscr{G})=\mathbf{R}\mathrm{Hom}_{\mathscr{D}_{\tube{X}_{\fr{P}}}}(\mathscr{F},\mathbf{R}\pi_{\dR*}\pi^*\mathscr{G})=\mathbf{R}\mathrm{Hom}_{\mathscr{D}_{\tube{X}_{\fr{P}}}}(\mathscr{F},\mathscr{G}).  \qedhere \]
\end{enumerate}
\end{proof}

\begin{proof}[Proof of Theorem \ref{theo: disoc indep}] First consider the claim that $u^*$ induces an equivalence of categories 
\[ \bD^b_\cons(X,Y,\fP)\to \bD^b_\cons(X,Y',\fP'). \]
To prove this, it suffices to show that the map
\begin{equation} \label{eqn: rhom isom} \mathbf{R}\mathrm{Hom}_{\mathscr{D}_{\tube{X}_{\fr{P}}}}(\mathscr{F},\mathscr{G}) \rightarrow \mathbf{R}\mathrm{Hom}_{\mathscr{D}_{\tube{X}_{\fr{P}'}}}(u^*\mathscr{F},u^*\mathscr{G})
\end{equation}
is an isomorphism. Indeed, this immediately implies that $u^*$ is fully faithful, but in conjunction with Lemmas \ref{lemma: basic properties Db cons} and \ref{lemma: disoc indep prelim 1}, it also implies that $u^*$ is essentially surjective, since the essential image is triangulated and contains all locally free isocrystals supported on locally closed subschemes of $X$. 

To prove this `derived full faithfulness', I can clearly localise on $\fr{P}$, and I claim that I can also localise on $X$. To see this, cover $X$ by open subspaces $\{X_i\}_{1\leq i\leq n}$, and for any $I\subset \{1,\ldots,n\}$ set $X_I=\cap_{i\in I} X_i$. Then, as in \cite[Proposition 2.1.8]{Ber96b}, any sheaf $\mathscr{E}$ on $\tube{X}_\fr{P}$ has a canonical resolution 
\[ 0 \rightarrow \mathscr{E} \rightarrow \bigoplus_{i=1}^n j_{X_i}^\dagger\mathscr{E} \rightarrow \ldots \rightarrow j^\dagger_{X_{\{1,\ldots,n\}}}\mathscr{E}\rightarrow 0. \] 
Applying this to the complexes $\mathscr{F}$ and $\mathscr{G}$, it follows that we can deduce the isomorphy of (\ref{eqn: rhom isom}) from the isomorphy of 
\begin{equation}
\mathbf{R}\mathrm{Hom}_{\mathscr{D}_{\tube{X}_{\fr{P}}}}(j_{X_I}^\dagger\mathscr{F},j_{X_J}^\dagger\mathscr{G}) \rightarrow \mathbf{R}\mathrm{Hom}_{\mathscr{D}_{\tube{X}_{\fr{P}'}}}(u^*j_{X_I}^\dagger\mathscr{F},u^*j_{X_J}^\dagger\mathscr{G}) \label{eqn: rhom isom 2}
\end{equation}  
for any pair of non-empty open subset $I,J\subset \{1,\ldots,n\}$. Set $\mathscr{F}'=(j_{X_I}^\dagger\mathscr{F})|_{\tube{X_J}_\fr{P}}$, and $\mathscr{G}'=\mathscr{G}|_{\tube{X_J}_\fr{P}}$, these are constructible complexes on $(X_J,Y,\fr{P})$. Applying Lemma \ref{lemma: u^*i_!} identifies (\ref{eqn: rhom isom 2}) with the map
\[ \mathbf{R}\mathrm{Hom}_{\mathscr{D}_{\tube{X_J}_{\fr{P}}}}(\mathscr{F}',\mathscr{G}') \rightarrow \mathbf{R}\mathrm{Hom}_{\mathscr{D}_{\tube{X_J}_{\fr{P}'}}}(u^*\mathscr{F}',u^*\mathscr{G}'). \]
Hence the claim for each $X_J$ implies the claim for $X$.

If $g$ is an isomorphism, then by Theorem \ref{theo: strong fibration theorem}, we can see that, after localising on $X$ and $\fP$, there exists an isomorphism
\[ \tube{X}_{\fr{P}'}\isomto \D^d_K(0;1^-)\times_K \tube{X}_\fr{P} \]
of germs identifying $u$ with the second projection. Hence applying Proposition \ref{prop: equiv rhom} shows that (\ref{eqn: rhom isom}) is an isomorphism.

In general I now argue as in \cite[Theorem 7.1.8]{LS07}. Indeed, it follows from what I have already proved that if $(X,Y)$ is a weakly realisable pair (in the sense of Definition \ref{defn: realisable}), and $(X,Y,\fP)$ is a frame with $\fP$ smooth around $X$, then $\bD^b_\cons(X,Y,\fP)$ only depends on $(X,Y)$ and not on $\fP$, and is moreover functorial in the pair $(X,Y)$. I can therefore denote $\bD^b_\cons(X,Y,\fP)$ by $\bD^b_\cons(X,Y)$, and the claim amounts to showing that if $({\rm id},g)\from (X,Y')\to(X,Y)$ is a morphism of weakly realisable pairs, with $g$ proper, then $g^*\from \bD^b_\cons(X,Y)\to \bD^b_\cons(X,Y')$ is an equivalence, or indeed `derived fully faithful' in the sense that the analogue of (\ref{eqn: rhom isom}) is an isomorphism. The question is local on $Y$, which I may therefore assume to be quasi-projective, and thus by Chow's lemma there exists a projective morphism $g'\from Y''\to Y'$ such that $g\circ g'$ is also projective. Thus I can reduce to the case that $g$ is projective. Hence by \cite[Lemma 6.5.1]{LS07} $({\rm id},g)$ extends to a morphism of frames
\[ \xymatrix{ & Y'\ar[r]\ar[d]^g & \fP'\ar[d]^u \\ X\ar[ur]\ar[r] & Y \ar[r] & \fP } \]
such that $u$ is \'etale around $X$. Hence $u\from \tube{X}_{\fP'}\to \tube{X}_\fP$ is an isomorphism by Theorem \ref{theo: strong fibration theorem}, and so $u^*$ is trivially `derived fully faithful'.

Next, consider the claims that both $\bR u_{\dR*}$ and $\bR u_{\dR!}[2d]$ are quasi-inverses to $u^*$. To begin with,
\[ \bR u_{\dR*}\from \bD^b(\sD_{\tube{X}_{\fP'}}) \to \bD^b(\sD_{\tube{X}_{\fP}})\]
is right adjoint to $u^*$. Thus I can argue exactly as above (that is, following the proof of \cite[Theorem 7.1.8]{LS07}) to show that the map
\[ \mathrm{id}\isomto \bR u_{\dR*}u^* \]
is an isomorphism on objects of $\bD^b_\cons(X,Y,\fP)$, by reducing to Proposition \ref{prop: equiv rhom}(\ref{num: equiv rhom 1}). It thus follows that $\bR u_{\dR*}$ does indeed descend to a quasi-inverse
\[ \bD^b_\cons(X,Y',\fP')\to \bD^b_\cons(X,Y,\fP) \]
to $u^*$.

To prove the same is true for $\mathbf{R}u_{\dR!}[2d]$, Proposition \ref{prop: weak proj formula for cons mod} shows that there is, for any $\mathscr{F}\in {\bf D}^b_\mathrm{cons}(X,Y,\fr{P})$, an isomorphism
\begin{equation} \mathscr{F}\otimes_{\cO_{\tube{X}_\fr{P}}} \mathbf{R}u_{\dR!}\cO_{\tube{X}_{\fr{P}'}}\isomto \mathbf{R}u_{\dR!}u^*\mathscr{F}
\end{equation}
in ${\bf D}^b(\mathscr{D}_{\tube{X}_{\fr{P}}})$. It follows from Corollary \ref{cor: trace isom from strong fibration} that the trace map gives an isomorphism
\[ \Tr\colon \mathbf{R}u_{\dR!}\cO_{\tube{X}_{\fr{P}'}}[2d] \rightarrow \cO_{\tube{X}_{\fr{P}}}, \]
and therefore an isomorphism
\[ \mathbf{R}u_{\dR!}u^*\mathscr{F}[2d]\isomto \mathscr{F} \]
for any $\mathscr{F}\in {\bf D}^b_\mathrm{cons}(X,Y,\fr{P})$. This completes the proof. 
\end{proof}

\begin{remark} \label{rem: v^* quasi-inverse to u^*} In the situation of Theorem \ref{theo: disoc indep}, suppose in addition that the morphism of frames $u\from (X,Y',\fP')\to (X,Y,\fP)$ admits a section $v \from (X,Y,\fP)\to (X,Y',\fP')$. Then $v^*$ is a left inverse to $u^*$, and hence an inverse to $u^*$. In particular, it follows that
\[ v^*\cong \mathbf{R}u_{\dR*}\cong \bR u_{\dR!}[2d] \]
as functors $\bD^b_{\cons}(X,Y',\fP') \to \bD^b_{\cons}(X,Y,\fP)$.
\end{remark}

\begin{corollary} In the situation of Theorem \ref{theo: disoc indep}, the functor
\[ u^*\from \Isoc_\cons(X,Y,\fP)\rightarrow \Isoc_\cons(X,Y',\fP')\]
is an equivalence of categories, with quasi-inverse given by either $\bR u_{\dR*}$ or $\bR u_{\dR!}[2d]$.
\end{corollary}

\begin{corollary} The categories $\Isoc_\cons(X,Y,\fP)$ and $\bD^b_\cons(X,Y,\fP)$ are independent of $\fr{P}$ up to canonical equivalence, and functorial in $(X,Y)$. 
\end{corollary}

I may therefore denote then by $\Isoc_\cons(X,Y)$ and $\bD^b_\cons(X,Y)$ respectively. Similarly there is the full subcategory $\Isoc(X,Y)\subset \Isoc_\cons(X,Y)$ of locally free isocrystals.

\begin{corollary} If $(\mathrm{id},g)\from (X,Y')\to (X,Y)$ is a morphism of pairs, with $g$ proper, then
\[ g^*\from \bD^b_\cons(X,Y)\to \bD^b_\cons(X,Y') \]
is a t-exact equivalence of categories. It follows that if $Y$ is proper, then $\Isoc_\cons(X,Y)$ and $\bD^b_\cons(X,Y)$ are independent of $Y$ up to canonical equivalence, and functorial in $X$.
\end{corollary} 

I will therefore denote these categories by $\Isoc_\cons(X)$ and $\bD^b_\cons(X)$ respectively, again there is the full subcategory $\Isoc(X)\subset \Isoc_\cons(X)$ of locally free isocrystals. Note that these are all categories of \emph{overconvergent} objects on $X$. In particular, what I write as $\Isoc(X)$ is the category of \emph{overconvergent} isocrystals on $X$, and is more commonly written $\Isoc^\dagger(X)$. In my notation, the category of \emph{convergent} isocrystals on $X$ is written $\Isoc(X,X)$. \footnote{Another natural choice for the category of convergent isocrystals on $X$ might be $\Isoc^\circ(X)$, using the fact that passing from $\Isoc(X)$ to $\Isoc^\circ(X)$ involves restricting to the \emph{interior} of $\tube{X}_\fP$ (as a subset of $\tube{Y}_\fP$). I will use this notation later on for log convergent isocrystals.} Unfortunately, this does slightly clash with the notation $\Isoc_\cons(\fP)$ used above for smooth formal schemes $\fP$, thus
\[ \Isoc_\cons(\fP)=\Isoc_\cons(P,P)\neq \Isoc_\cons(P).\]
The reader should therefore bear in mind that when $\fP$ is a smooth formal scheme, $\Isoc_{(\cons)}(\fP)$ means the category of \emph{convergent} (constructible) isocrystals on $\fP$ (or equivalently $P$), whereas when $X$ is a variety, $\Isoc_{(\cons)}(X)$ will mean the category of \emph{overconvergent} (constructible) isocrystals on $X$.

Anyway, if $(f,g)\from (X',Y')\to (X,Y)$ is a morphism of pairs, I will usually write
\begin{align*}
 f^* \from &\Isoc_\cons(X,Y)\to \Isoc_\cons(X',Y') \\
  f^* \from &\bD^b_\cons(X,Y)\to \bD^b_\cons(X',Y')
\end{align*}
for the pullback functors, and similarly for morphisms of varieties over $k$. If $g$ is a closed immersion, I will also write
\begin{align*}
 f_!\from &\Isoc_\cons(X',Y')\to \Isoc_\cons(X,Y) \\
  f_!\from &\bD^b_\cons(X',Y')\to \bD^b_\cons(X,Y)
\end{align*}
for the (t-)exact functors which, on realisations, is the extension by zero along the inclusion $\tube{X'}_{\fP}\rightarrow \tube{X}_\fP$ of the locally closed subspace $\tube{X'}_{\fP}$ of $\tube{X}_{\fP}$. Note that when $f$ is an open immersion, these are right adjoint to $f^*$, and when $f$ is a closed immersion, they are left adjoint to $f^*$.

It is a straightforward check that the tensor product $\otimes_{\cO_{\tube{X}_\fP}}$ descends to a well-defined functor on either of $\bD^b_{\cons}(X,Y)$ or $\bD^b_\cons(X)$, which I will denote by $\otimes_{\cO_{X}}$.

\subsection{Frobenius structures}

Functoriality in $(X,Y)$ now means that it makes sense to talk about Frobenius structures on constructible isocrystals and complexes. Indeed, the absolute Frobenius morphism $F\from (X,Y)\to (X,Y)$ induces a $\sigma$-linear pullback functor
\[ F^*\from \Isoc_\cons(X,Y,\fP)\to \Isoc_\cons(X,Y,\fP). \] 

\begin{definition} \label{defn: F type}A Frobenius structure on an object $\sF\in \Isoc_\cons(X,Y,\fP)$ is an isomorphism
\[ \varphi\from F^{n*} \sF\isomto \sF\]
for some $n$. An object $\sF\in \Isoc_\cons(X,Y,\fP)$ is said to be of Frobenius type if it is an iterated extension of objects admitting a Frobenius structure. 
\end{definition}

The reader should be warned that again, this terminology is not completely standard, and what I have termed `of Frobenius type' has been previously referred to as `$F$-able' in the literature. The terminology here matches that used in \cite{AL22}, however, is different from that used in \cite{LS14}. Anyway, I will write $\Isoc_{\cons,F}(X,Y,\fP)\subset \Isoc_\cons(X,Y,\fP)$ for the full subcategory consisting of objects which are of Frobenius type. There are similarly defined categories $\Isoc_{\cons,F}(X,Y)$, $\Isoc_{\cons,F}(X)$ and $\Isoc_{\cons,F}(\fP)$, as well as the obvious analogues for categories of locally free isocrystals.

\begin{definition} A complex $\mathscr{K}\in \bD^b_\cons(X,Y,\fP)$ is said to be of Frobenius type if all of its cohomology sheaves are. I will denote by
\[ \bD^{b}_{\cons,F}(X,Y,\fP)\subset  \bD^b_\cons(X,Y,\fP) \]
the full subcategory consisting of objects of Frobenius type.
\end{definition}

I will similarly write $\bD^{b}_{\cons,F}(X,Y)$, $\bD^{b}_{\cons,F}(X)$ and $\bD^b_{\cons,F}(\fP)$ for the analogous categories associated to pairs, varieties, and formal schemes.

\subsection{Finite \'etale pushfowards and pullbacks}

There is one last general result I will need on constructible isocrystals, and that is a version of Theorem \ref{theo: disoc indep} in a slightly more general setting. Instead of having two different frames enclosing a single variety $X$, I will need to consider a morphism of frames inducing a finite \'etale cover of $X$. That is, a morphism
\[ \xymatrix{ X'\ar[r]\ar[d]^f & Y'\ar[r]\ar[d]^g & \fP'\ar[d]^u \\ X \ar[r] & Y \ar[r] & \fP } \]
such that $\fP$ is smooth around $X$, $u$ is smooth around $X'$ of relative dimension $d$, $g$ is proper, and $f$ is finite \'etale. For locally free isocrystals, the result I require can be proved rather abstractly.

\begin{proposition} \label{prop: finite etale special cons} The functor
\[ f^*\from \Isoc(X,Y)\rightarrow \Isoc(X',Y') \]
admits a simultaneous left and right adjoint $f_*$. Moreover, for any $\sF\in \Isoc(X,Y)$, $\sG\in \Isoc(X',Y')$ the compositions
\begin{align*}
&\sF \rightarrow f_*f^*\sF \rightarrow \sF \\
&\sG \rightarrow f^*f_*\sG \to \sG
\end{align*}
are the identity maps of $\sF$ and $\sG$ respectively.. 
\end{proposition}

\begin{proof}
Let $(X'',Y'')\to (X',Y')$ be a morphism of frames with $Y''\to Y'$ proper and $X''\to X'$ a Galois closure of $f$. Let $G$ be the Galois group of $X''/X$, and $H\leq G$ that of $X''/X'$. The morphism $(f,g)\from (X'',Y'')\rightarrow (X,Y)$ of pairs is then one of effective descent for locally free isocrystals: indeed, thanks to \cite[Theorem 4.1]{Laz22} this can be proved word for word the same as the corresponding result \cite[Theorem 5.1]{Laz22} when $Y$ is proper. 

It therefore follows in the usual manner that $\Isoc(X,Y)$ is equivalent to the category of $G$-equivariant objects in $\Isoc(X'',Y'')$. Similarly, $\Isoc(X',Y')$ is equivalent to the category of $H$-equivariant objects in $\Isoc(X'',Y'')$. I now define $f_*$ to be the composite
\[ \Isoc(X',Y')\isomto H\-\Isoc(X'',Y'')\overset{\mathrm{Ind}_H^G}{\lto}  G\-\Isoc(X',Y') \isomfrom \Isoc(X,Y). \]
The verification that it satisfies the claimed properties is a straightforward exercise in representation theory.
\end{proof}

As I said, this result was proved rather abstractly. With additional hypothesis on $(X,Y,\fP)$, however, it is possible to describe $f_*$ in the expected explicit manner, and generalise from locally free to constructible isocrystals. 

\begin{theorem} \label{theo: finite etale general cons} Assume that $\fP$ is affine, and that $X$ is the complement of a hypersurface in $Y$. 
\begin{enumerate}\item \label{num: fegc i} The functors
\[ \bR\!\tube{f}_{u\dR*},\bR\!\tube{f}_{u\dR!}[2d]\from \bD^b(\sD_{\tube{X'}_{\fP'}})\to \bD^b(\sD_{\tube{X}_{\fP}}) \]
induce isomorphic t-exact functors
\[ \bD^b_\cons(X',Y',\fP')\to \bD^b_\cons(X,Y,\fP),\]
which send $\Isoc(X',Y',\fP')$ into $\Isoc(X,Y,\fP)$.
\item \label{num: fegc ii} These functors are both left and right adjoints to
\[ \tube{f}_u^*\from \bD^b_\cons(X,Y,\fP)\to \bD^b_\cons(X',Y',\fP').\]
\item \label{num: fegc iii} For any objects 
\[ \sF\in \bD^b_\cons(X,Y,\fP),\;\;\;\; \sG\in \bD^b_\cons(X',Y',\fP'),\]
the compositions
\begin{align*}
&\sF \rightarrow \mathbf{R}\!\tube{f}_{u\dR*}\tube{f}_u^*\sF \cong \mathbf{R}\!\tube{f}_{u\dR!}\tube{f}_u^*\sF[2d] \rightarrow \sF \\
&\sG \rightarrow \tube{f}_u^*\mathbf{R}\!\tube{f}_{u\dR!}\sG[2d]\cong\tube{f}_u^*\mathbf{R}\!\tube{f}_{u\dR*}\sG  \to \sG
\end{align*}
are the identity maps.
\end{enumerate}
\end{theorem}

\begin{remark} The additional hypotheses here on $\fP$ and $X$ are surely unnecessary. However, it seems rather difficult to construct a global comparison morphism between the functors $\bR\!\tube{f}_{u\dR*}$ and $\bR\!\tube{f}_{u\dR!}[2d]$ in full generality, and these hypotheses provide a way of doing this. 
\end{remark}

The proof of Theorem \ref{theo: finite etale general cons} will be in several stages. 

\begin{lemma} It suffices to prove Theorem \ref{theo: finite etale general cons} under the additional assumptions that $u$ is proper, and \'etale around $X'$ (i.e. $d=0$).
\end{lemma}

\begin{proof}
Note that since $\fP$ is affine, so is $Y$, and since $X$ is the complement of a hypersurface in $Y$, so is $X$. Since $f$ is finite \'etale, it is proved in \cite[\href{https://stacks.math.columbia.edu/tag/00U9}{Lemma 00U9}]{stacks} that $f$ is a global complete intersection, for the reader's convenience I will recall the argument here. Choose a presentation of $X'$, that is, a closed immersion $X'\rightarrow \A^{n-1}_X$ for some $n$. Let $\cI$ denote the ideal of $X'$ inside $\A^{n-1}_X$. Since $f$ is \'etale, the conormal exact sequence then implies that
\[ \cI/\cI^2 \isomto \Omega^1_{\A^{n-1}_X/X} \otimes_{\cO_{\A^{n-1}_X}}\cO_{X'}, \]
thus the conormal sheaf of $X'$ inside $\A^{n-1}_X$ is free. Now choose functions $s_1,\ldots,s_{n-1}\in \cI$ lifting a basis of $\cI/\cI^2$. Then Nakayama's lemma implies that there exists some $t\in 1+\cI$ such that $\cI[1/t]=(s_1,\ldots,s_{n-1})[1/t]$. Letting $z_n$ be a new indeterminate, and setting $s_n=tz_n-1$, it now follows that $X'=V(s_1,\ldots,s_n)\subset \A^n_X$. 

The $s_i$ can now be successively lifted to provide a proper morphism $\fP''\to \fP$ lifting $f$, which \'etale around $X'$. First, homogenizing the $s_i$ makes $X'$ into a complete intersection inside $\P^n_X$, given by the zero locus of sections
\[ s_i\in\Gamma(\P^n_X,\cO_{\P^n_X}(n_i))\]
for some positive integers $n_i$. Chose $s\in \Gamma(Y,\cO_Y)$ so that $X=D(s)$. Now clear denominators to ensure that the $s_i$ extend to sections
\[ s'_i\in \Gamma(\P^n_Y,\cO_{\P^n_Y}(n_i)),
\]
and then lift the $s'_i$ to sections
\[ \tilde{s}'_i\in \Gamma(\widehat{\P}^n_\fP,\cO_{\widehat{\P}^n_\fP}(n_i)). \]
Note that the Jacobian criterion for smoothness implies that the zero locus
\[ V(\tilde{s}'_1,\ldots,\tilde{s}'_n) \subset \widehat{\P}^n_\fP \]
is \'etale over $\fP$ in a neighbourhood of $X'$.

I now let $\fP''$ denote the maximal closed formal subscheme of $V(\tilde{s}'_1,\ldots\tilde{s}'_n)$ which is flat over $\fP$. Thus the locally closed immersion $X'\to V(\tilde{s}'_1,\ldots\tilde{s}'_n)$ factors through $\fP''$. Let $Y''$ denote the closure of $X'$ inside $P''$. The upshot of all of this is that there exists morphism of frames
\[ \xymatrix{ X'\ar[r]\ar[d]^f & Y''\ar[r]\ar[d]^{g'} & \fP'' \ar[d]^{u'} \\ X \ar[r] & Y \ar[r] & \fP } \]
such that $u'$ is proper, and \'etale around $X'$. But now repeated applications of Theorem \ref{theo: disoc indep} imply, in the usual way, that proving the theorem for $(f,g',u')$ is equivalent to proving the theorem $(f,g,u)$.
\end{proof}

In the case where $u$ is \'etale, there is an obvious comparison map between $\bR\!\tube{f}_{u\dR*}$ and $\bR\!\tube{f}_{u\dR!}[2d]$, namely the `forget support' map.

\begin{lemma} \label{lemma: forget supports for fegc} In the situation of Theorem \ref{theo: finite etale general cons}, assume that $u$ is proper and that $d=0$. Then the natural `forget supports' map
\[ \bR\!\tube{f}_{u\dR!} \to \bR\!\tube{f}_{u\dR*} \]
is an isomorphism of functors $ \bD^b(\sD_{\tube{X'}_{\fP'}}) \to \bD^b(\sD_{\tube{X}_{\fP}})$.
\end{lemma}

\begin{proof}
The natural map $X'\rightarrow u^{-1}(X)$ is a closed immersion (since it is a proper, locally closed immersion), and since $u$ is \'etale around $X'$, there exists an open subscheme $U\subset u^{-1}(X)$ containing $X'$ which is \'etale over $X$. It follows that $X'\to U$ is both \'etale and a closed immersion, thus $X'$ is in fact open in $U$. Hence $X'$ is both open and closed in $u^{-1}(X)$. Hence $\tube{X'}_{\fP'}$ is both open and closed in $\tube{u^{-1}(X)}_{\fP'}=u^{-1}\tube{X}_\fP$. Since $u$ is proper, the `forget supports' map
\[ \bR u_{\dR!}\to \bR u_{\dR*}\]
is an isomorphism of functors from $ \bD^b(\sD_{u^{-1}\tube{X}_{\fP}})$ to $\bD^b(\sD_{\tube{X}_{\fP}})$. Since $\tube{X'}_{\fP'}$ is open and closed in $u^{-1}\tube{X}_\fP$, we therefore deduce the `forget supports' map
\[ \bR\!\tube{f}_{u\dR!}\to \bR\!\tube{f}_{u\dR*}\]
is an isomorphism as required.
\end{proof}

\begin{proof}[Proof of Theorem \ref{theo: finite etale general cons}\eqref{num: fegc i}]
We may assume that $u$ is proper, and that $d=0$, in which case
\[ \bR\!\tube{f}_{u\dR!} \to \bR\!\tube{f}_{u\dR*} \]
is an isomorphism. It remains to show that these functors induce t-exact functors 
\[ \bD^b_\cons(X',Y',\fP')\to \bD^b_\cons(X,Y,\fP),\]
or in other words that the functor $\bR\!\tube{f}_{u\dR*}$ sends any
\[ \sF\in \Isoc_\cons(X',Y',\fP')\subset \bD^b_\cons(X',Y',\fP') \]
to an object of
\[ \Isoc_\cons(X,Y,\fP)\subset \bD^b_\cons(X,Y,\fP). \]
This claim is additive over exact sequences in $\sF$, moreover, every stratification of $X'$ admits a refinement which is the pullback via $f$ of a stratification of $X$. We may therefore assume that there exists a locally closed subscheme $i\from Z\to X$, and a locally free isocrystal $\sG$ on $Z':=f^{-1}(X)$, such that, writing $i'$ for the induced immersion $i'\from Z'\to X'$, $\sF=i'_!\sG$. By Noetherian induction, I can always replace $Z$ by an open subscheme, hence I can ensure $Z$ shares the hypothesis with $X$ that it is the complement of a hypersurface inside its closure in $P$.

Writing $f'\from Z'\to Z$ for the induced morphism, transitivity of proper pushforwards implies that
\begin{align*}
 \bR\! \tube{f}_{u\dR!}\sF &= \bR\! \tube{f}_{u\dR!}i'_!\sG \\ 
 &= i_!\bR\! \tube{f'}_{u\dR!}\sG.
\end{align*}
Thus, replacing $X$ with $Z$, I can assume that $\sF$ itself is locally free on $\tube{X'}_{\fP'}$. Now, thanks to Proposition \ref{prop: finite etale special cons} above, there exists a locally free isocrystal $\sG$ on $(X,Y,\fP)$ such that $\sF$ is a direct summand of $\tube{f}_u^*\sG$. Thus it suffices to prove the claim for $\sF=\tube{f}_u^*\sG$, which, via the projection formula (Lemma \ref{lemma: proj form}), reduces to the case $\sF=\cO_{\tube{X'}_{\fP'}}$. 

To prove the result in this case, note that since $u$ is \'etale in a neighbourhood of $X'$, the non-\'etale locus of $u\from \fP'_K\to \fP_K$ is a closed analytic subspace of $\fP'_K$ disjoint from $\tube{X'}_{\fr{P}'}$. It follows that $\Omega^\bullet_{\tube{X'}_{\fr{P}'}/\tube{X}_\fP}=\cO_{\tube{X'}_{\fr{P}'}}$, thus $\bR\! \tube{f}_{u\dR!}\cO_{\tube{X'}_{\fP'}}=\bR\! \tube{f}_{u\dR*}\cO_{\tube{X'}_{\fP'}}= \bR\!\tube{f}_{u*}\cO_{\tube{X'}_{\fP'}}$, and I will show that $\bR \tube{f}_{u*}\cO_{\tube{X'}_{\fP'}}$ is a coherent $\cO_{\tube{X}_\fP}$-module sitting in degree $0$.

I first show that $\bR\! \tube{f}_{u*}\cO_{\tube{X'}_{\fP'}}$ has coherent cohomology sheaves. To see this, note that since $u$ is proper, $u:u^{-1}\tube{Y}_{\fr{P}} \rightarrow \tube{Y}_{\fP}$ is also proper, and hence by applying Proposition \ref{prop: proper base change} to the Cartesian diagram
\[ \xymatrix{ u^{-1}\tube{X}_{\fr{P}} \ar[r]\ar[d]_u & u^{-1}\tube{Y}_{\fr{P}} \ar[d]^u \\ \tube{X}_\fP \ar[r] & \tube{Y}_\fP } \]
shows that $\mathbf{R}u_*\cO_{u^{-1}\tube{X}_{\fr{P}}}$ has coherent cohomology sheaves on $\tube{X}_\fP$. Since $\tube{X'}_\fP$ is open and closed in $u^{-1}\tube{X}_\fP$ (as was shown during the proof of Lemma \ref{lemma: forget supports for fegc}), $\mathbf{R}\!\tube{f}_{u*}\cO_{\tube{X'}_{\fr{P}'}}$ is a direct summand of $\mathbf{R}u_*\cO_{u^{-1}\tube{X}_{\fr{P}}}$, and thus also has coherent cohomology sheaves.

I next show that the higher cohomology sheaves of $\bR\! \tube{f}_{u*}\cO_{\tube{X'}_{\fP'}}$ vanish. To see this, consider $\mathbf{R}^qu_*\cO_{u^{-1}\tube{Y}_{\fr{P}}}$ for some $q>0$, which is a coherent sheaf on $\tube{Y}_{\fP}$. After restricting to $\tube{X}_{\fP}$, this contains $\mathbf{R}^q\!\tube{f}_{u*}\cO_{\tube{X'}_{\fr{P}'}}$ as a direct summand. Now Proposition \ref{prop: coherent sheaves colimit neighbourhoods} implies that there exists an open neighbourhood $V$ of $\tube{X}_{\fP}$ in $\tube{Y}_{\fP}$, and a decomposition
\[ \left.\left(\mathbf{R}^qu_*\cO_{u^{-1}\tube{Y}_{\fr{P}}}\right)\right\vert_V \cong \mathscr{F}_1\oplus \mathscr{F}_2 \]
such that $\mathscr{F}_1|_{\tube{X}_{\fP}} = \mathbf{R}^q\!\tube{f}_{u*}\cO_{\tube{X'}_{\fr{P}'}}$. Let $\tube{X}_\fP^\circ$ and $\tube{X'}_{\fP'}^\circ$ denote the interiors of $\tube{X}_\fP$ and $\tube{X'}_{\fP'}$ respectively. Since $u$ is proper, and \'etale in a neighbourhood of $X'$, we know that the map
\[ u\from \tube{X'}^\circ_{\fP'}\to \tube{X}^\circ_{\fP} \]
of adic spaces is \'etale, and in fact every point $x\in \tube{X}^\circ_{\fP}$ has only finitely many preimages under this map. Hence, by Proposition \ref{prop: proper base change}, the stalk of $\mathbf{R}^q\!\tube{f}_{u*}\cO_{\tube{X'}_{\fr{P}'}}=\mathbf{R}^q\!\tube{f}_{u!}\cO_{\tube{X'}_{\fr{P}'}}$ at any maximal point of $\tube{X}^\circ_{\fP}$ is zero. The support of $\mathscr{F}_1$ is therefore a closed analytic subspace of $V$, not containing any maximal point of $\tube{X}^\circ_{\fP}$. It is therefore disjoint from $\tube{X}_{\fP}$, and so $\mathbf{R}^q\!\tube{f}_{u*}\cO_{\tube{X'}_{\fr{P}'}}=0$.

The last thing left to prove is that the connection on $\tube{f}_{u*}\cO_{\tube{X'}_{\fP'}}$ is overconvergent, which follows from \cite[Theorem 4.3.9]{LS07}, since any set of local co-ordinates on $(X,Y,\fP)$ is also a set of local co-ordinates on $(X',Y',\fP')$. Thus neither the derivations, nor the spaces of local sections $\Gamma(V^\lambda_{\delta},\ca{E})$, occurring in the statement of the Theorem, change on passing from $j^\dagger_{X'}\cO_{\tube{Y'}_{\fP'}}$ to $\tube{g}_{u*}j^\dagger_{X'}\cO_{\tube{Y'}_{\fP'}}$ (at least up to cofinality in $\lambda$).
\end{proof}

\begin{proof}[Proof of Theorem \ref{theo: finite etale general cons} \eqref{num: fegc ii} and  \eqref{num: fegc iii}]
Again, I can assume that $u$ is proper, and that $d=0$. Then $\bR\!\tube{f}_{u\dR*}$ is clearly a right adjoint to $\tube{f}_u^*$, since this is already true as a functor 
\[\bD(\sD_{\tube{X'}_{\fP'}}) \to \bD(\sD_{\tube{X}_{\fP}}). \]
I therefore need to prove that $\bR\!\tube{f}_{u\dR*}$ is also a left adjoint. To construct the counit
\[ \varepsilon_\sF\from \bR\! \tube{f}_{u\dR*}\tube{f}_u^*\mathscr{F}\rightarrow \sF, \]
I appeal to Proposition \ref{prop: weak proj formula for cons mod} (note that $\bR\! \tube{f}_{u\dR*}=\bR\! \tube{f}_{u\dR!}$ here) to show that
\[ \mathscr{F}\otimes^{\bL}_{\cO_{\tube{X}_\fP}} \bR\! \tube{f}_{u\dR*} \cO_{\tube{X'}_{\fP'}}  \isomto \bR\! \tube{f}_{u\dR*}\tube{f}_u^*\mathscr{F}. \]
Now, by uniqueness of (right) adjoints, I know that $\bR\!\tube{f}_{u\dR*}$ coincides on locally free objects with the abstract adjoint coming from Proposition \ref{prop: finite etale special cons}. The counit
\[ \varepsilon_{\cO_{\tube{X}_\fP}}\from \bR\! \tube{f}_{u\dR*} \cO_{\tube{X'}_{\fP'}} \to \cO_{\tube{X}_\fP} \]
may therefore be tensored with $\sF$ to provide the required morphism
\[ \varepsilon_\sF\from \bR\! \tube{f}_{u\dR*} \tube{f}_u^*\sF \to \sF. \]
I next construct the unit
\[ \eta_{\sG}\from \sG \to \tube{f}_u^*\bR\!\tube{f}_{u\dR*}\sG \]
in an entirely similar way. Indeed, I first construct a natural morphism
\begin{equation} 
\label{eqn: unit} 
 \tube{f}_u^*\bR\!\tube{f}_{u\dR*}\sG \to \sG \otimes^{\bL}_{\cO_{\tube{X'}_{\fP'}}} \tube{f}_u^*\bR\!\tube{f}_{u\dR*}\cO_{\tube{X'}_{\fP'}} 
\end{equation}
of endofunctors of $\bD^b_\cons(X',Y',\fP')$. To do this, note that by adjunction, it suffices to construct 
\[ 
\bR\!\tube{f}_{u\dR*}\sG \to \bR\tube{f}_{u\dR*}(\sG \otimes^{\bL}_{\cO_{\tube{X'}_{\fP'}}} \tube{f}_u^*\bR\tube{f}_{u\dR*}\cO_{\tube{X'}_{\fP'}} ).
\]
By the projection formula (note that $\bR\tube{f}_{u\dR*}\cO_{\tube{X'}_{\fP'}}$ is a locally free isocrystal) this amounts to producing a natural morphism
\[ \bR\!\tube{f}_{u\dR*}\sG \to \bR\!\tube{f}_{u\dR*}\sG \otimes^{\bL}_{\cO_{\tube{X}_{\fP}}} \bR\!\tube{f}_{u\dR*}\cO_{\tube{X'}_{\fP'}}.  \]
For this, we just take the natural map
\[ \cO_{\tube{X}_\fP}\to  \bR\!\tube{f}_{u\dR*}\cO_{\tube{X'}_{\fP'}} \]
and tensor with $\bR\!\tube{f}_{u\dR*}\sG$. Next, I claim that (\ref{eqn: unit}) is an isomorphism. Since $u^*$ commutes with extension by zero, as does $\bR\!\tube{f}_{u\dR!}=\bR\!\tube{f}_{u\dR*}$, the usual d\'evissage argument reduces to the case when $\sF$ is locally free. In this case, as remarked above, $\bR\!\tube{f}_{u\dR*}$ coincides with the functor $f_*$ from Proposition \ref{prop: finite etale special cons}, and the assertion then just reduces to elementary computations in the representation theory of finite groups. I can therefore define
\[ \eta_\sG\from \sG \to \tube{f}^*_u\bR\!\tube{f}_{u\dR*}\sG \]
by tensoring the unit
\[ \eta_{\cO_{\tube{X'}_{\fP'}}}\from \cO_{\tube{X'}_{\fP'}}\to \tube{f}^*_u\bR\!\tube{f}_{u\dR*}\cO_{\tube{X'}_{\fP'}} \]
from Proposition \ref{prop: finite etale special cons} with $\sG$.

Now, checking that $\varepsilon$ and $\eta$ give rise to an adjunction amount to checking that the natural transformations
\begin{align*}
\bR\!\tube{f}_{u\dR*}\sG &\overset{\eta}{\lto} \bR\!\tube{f}_{u\dR*}\tube{f}_u^*\bR\!\tube{f}_{u\dR*}\sG \overset{\varepsilon}{\lto} \bR\!\tube{f}_{u\dR*}\sG \\
\tube{f}_u^*\sF &\overset{\eta}{\lto} \tube{f}_u^*\bR\!\tube{f}_{u\dR*}\tube{f}_u^*\sF \overset{\varepsilon}{\lto}\tube{f}_u^*\sF
\end{align*}
are the identity. Unwinding the definitions, this follows from Proposition \ref{prop: finite etale special cons}.

The final claim, concerning the compositions of the units and counits giving the identity maps, reduces, by the construction, to the locally free case considered in Proposition \ref{prop: finite etale special cons}.
\end{proof}

\section{Overholonomic \texorpdfstring{$\sD^\dagger$}{Ddag}-modules} \label{sec: D-modules}

In this section, I will briefly recall the basic theory of overholonomic $\sD^\dagger$-modules on formal schemes and varieties, mostly following the exposition in \cite[\S1]{AC18a}. I will also introduce the dual constructible t-structure on the category of overholonomic complexes of $\sD^\dagger$-modules, which will be the one matching up with the natural t-structure on $\bD^b_\cons$ via the overconvergent Riemann--Hilbert correspondence.

\subsection{Cohomological formalism of overholonomic $\pmb{\sD^\dagger}$-modules: formal schemes}

Let $\fP$ be a smooth formal scheme, and $\sD^\dagger_{\fP\Q}$ Berthelot's ring of overconvergent differential operators of $\fP$. Then Caro has defined in \cite{Car09b} the notion of an overholonomic complex of $\sD^\dagger_{\fP\Q}$-modules, giving rise to a full, triangulated subcategory
\[ \bD^b_\hol(\fP)\subset \bD^b_\coh(\sD^\dagger_{\fP\Q}). \]
The category $\bD^b_\coh(\sD^\dagger_{\fP\Q})$ admits a $\sigma$-linear Frobenius pullback functor
\[ F^*\from \bD^b_\coh(\sD^\dagger_{\fP\Q}) \to \bD^b_\coh(\sD^\dagger_{\fP\Q}), \]
which is t-exact for the natural t-structure \cite[Th\'eor\'eme 4.2.4]{Ber00}. 

\begin{definition} An object $\cM\in \bD^b_\coh(\sD^\dagger_{\fP\Q})$ is said to be `of Frobenius type' if it's cohomology sheaves are iterated extensions of objects admitting some Frobenius structure. 
\end{definition}

Again, the terminology here is slightly non-standard. I will denote by
\[  \bD^b_{\hol,F}(\fP)\subset\bD^b_\hol(\fP)\]
the full subcategory on objects which are of Frobenius type. Thanks to work of Caro and Caro--Tsuzuki, this category admits a good formalism of cohomological operations, which I now describe. First of all, there are the duality and tensor product functors
\begin{align*}
 \mathbf{D}_{\fP} \from &\bD^b_{\mathrm{hol},F}(\fP)^{\mathrm{op}} \rightarrow \bD^b_{\mathrm{hol},F}(\fP) \\
 -\otimes_{\cO_\fP} - \from &\bD^b_{\mathrm{hol},F}(\fP) \times \bD^b_{\mathrm{hol},F}(\fP) \rightarrow \bD^b_{\mathrm{hol}}(\fP)
\end{align*}
defined in \cite[\S4.3.10]{Ber02} and \cite[\S2.1]{Car15b}. These preserve overholonomicity by \cite[Corollaire 3.4]{Car09b} and \cite[Th\'eor\`eme 4.2.4]{Car15b} respectively. Caro then defines
\[ \wotimes_{\cO_\fP} := \otimes_{\cO_\fP}[-\dim \fP] \]
which will turn out to match up better with the tensor product of constructible isocrystals.

For any locally closed subscheme $X\hookrightarrow P$, Caro has defined the idempotent endofunctor
\[ \bR\underline{\Gamma}^\dagger_X\from\bD^b_{\hol,F}(\fP) \to \bD^b_{\hol,F}(\fP) \]
of sections with support on $X$, see \cite[\S1.1.5]{AC18a} and the references given there. This works as follows. If $X$ is closed in $P$, then we can write it as an intersection of divisors $X=\cap_{i=1}^n D_i$. For each $D_i$ Berthelot defined functor $\bR\underline{\Gamma}^\dagger_{D_i}$ in \cite[\S4.4.4-4.4.5]{Ber02}, and Caro defines
\[ \bR\underline{\Gamma}^\dagger_X:= \bR\underline{\Gamma}^\dagger_{D_1} \circ \ldots \circ \bR\underline{\Gamma}^\dagger_{D_n}. \]
He shows that this doesn't depend on the choice of the $D_i$, and that the natural morphism
\[ \bR\underline{\Gamma}^\dagger_X\rightarrow \mathrm{id} \]
is an isomorphism when evaluated on overholonomic complexes supported (set theoretically) on $X$. He further shows that the morphism
\[ \bR\underline{\Gamma}^\dagger_X\rightarrow \mathrm{id} \]
has a functorial cone, denoted $(^\dagger X)$. If $X$ is locally closed, with closure $Y$ in $P$, Caro then defines
\[\bR\underline{\Gamma}^\dagger_X := (^\dagger Y\setminus X) \circ \bR\underline{\Gamma}^\dagger_Y. \]
The order of composition here can be reversed, and the result is this same if $Y$ and $Y\setminus X$ are replaced by arbitrary closed subschemes $Z$ and $T$ of $P$ such that $X=Z\setminus T$. The functor $\bR\underline{\Gamma}^\dagger_X$ only depends on $X_{\rm red}$.

If $u\from \fP\rightarrow \fQ$ is a morphism of smooth formal schemes, then Berthelot defined in \cite[\S4.3]{Ber02} the functors
\begin{align*}
u^!\from &\bD^b_{\coh}(\sD^\dagger_{\fQ\Q}) \to  \bD^b(\sD^\dagger_{\fP\Q})\\
u_+\from &\bD^b_{\coh}(\sD^\dagger_{\fP\Q}) \to  \bD^b(\sD^\dagger_{\fQ\Q}).
\end{align*}
Caro showed in \cite[Th\'eor\`eme 3.8]{Car09b} that $u^!$ preserves overholonomicity, and moreover in \cite[Th\'eor\'eme 3.9]{Car09b} that so does $u_+$ whenever $u$ is proper. In this case, there is the natural duality isomorphism
\[  u_+\circ \bD_\fP \isomto \bD_\fQ \circ u_+  \]
by \cite[Corollaire 7.3]{Vir04}, and $(u_+,u^!)$ form an adjoint pair by \cite[Th\'eor\`eme 7.4, and following N.B. i)]{Vir04}. For arbitrary $u$, $u^!$ is compatible with tensor product in the sense that 
\[ u^!(-\wotimes_{\cO_\fQ}-) \isomto u^!(-)\wotimes_{\cO_\fP} u^!(-), \]
see \cite[Proposition 2.1.9]{Car15b}. If
\[ \xymatrix{ \fP'\ar[r]^{f'}\ar[d]_{u'} & \fP \ar[d]^u \\ \fQ' \ar[r]^f & \fQ } \]
is a Cartesian diagram of smooth formal schemes, with $u$ proper, then there is a natural base change isomorphism
\[ f^!u_+ \isomto u'_+f'^!.\]
Indeed, the general case can be reduced to the two special cases where $f$ is either smooth or a closed immersion. When $f$ is smooth this is \cite[Proposition 3.1.8]{Car04}, and when $f$ is a closed immersion it amounts to proving that 
\[ \bR\underline{\Gamma}^\dagger_{Q'} \circ u_+ \isomto u_+\circ \bR\underline{\Gamma}^\dagger_{u^{-1}(Q')}  \]
which was shown in \cite[Th\'eor\`eme 2.2.17]{Car04}.

\subsection{Cohomological formalism of overholonomic $\pmb{\sD^\dagger}$-modules: pairs and varieties}

To obtain a theory for pairs and varieties, Caro makes the following definition.

\begin{definition} Let $(X,Y,\fP)$ be a frame, with $\fP$ smooth. We say that $\cM\in \bD^b_\hol(\fP)$ is supported on $X$ if there exists an isomorphism
\[ \cM \cong \mathbf{R}\underline{\Gamma}_X^\dagger\cM.\]
Let 
\[ \bD^b_{\hol,F}(X,Y,\fP)\subset \bD^b_{\hol,F}(\fP) \]
denote the full subcategory of objects supported on $X$. 
\end{definition}

\begin{remark} \begin{enumerate} \item Note that if $X$ is closed, then the natural direction of this isomorphism is $\mathbf{R}\underline{\Gamma}_X^\dagger\cM\to \cM$, and if $X$ is open it is $\cM\to  \mathbf{R}\underline{\Gamma}_X^\dagger\cM$.
\item If $\cM\in\bD^b_{\hol,F}(\fP)$ is any object, then $ \mathbf{R}\underline{\Gamma}_X^\dagger\cM$ is supported on $X$, since $\mathbf{R}\underline{\Gamma}_X^\dagger$ is idempotent. 
\end{enumerate}
\end{remark}

These categories admit similar cohomological operations to the versions for formal schemes, although for a theory that is provably independent of $\fP$, it is necessary to work with l.p. frames.

\begin{proposition}\label{prop: indep Ddag} Let
\[ \xymatrix{ & Y' \ar[d]^g \ar[r] & \fP' \ar[d]^u \\ X \ar[ur] \ar[r] & Y\ar[r] & \fP } \]
be a morphism of l.p. frames such that $g$ is proper and $u$ is smooth. Then the functors $u_+$ and $\mathbf{R}\underline{\Gamma}^\dagger_{X}\circ u^!$ induce inverse equivalences of categories
\[ \mathbf{R}\underline{\Gamma}^\dagger_{X} \circ u^! \from \bD^b_\hol(X,Y,\fP)\leftrightarrows \bD^b_\hol(X,Y',\fP')\from u_+.  \] 
\end{proposition}

\begin{proof}
This is \cite[Lemma 2.5]{Car12}.
\end{proof}

It follows that $\bD^b_{\hol,F}(X,Y,\fP)$ depends only on the pair $(X,Y)$, and not on the choice of l.p. frame $(X,Y,\fP)$ enclosing it, and may therefore be denoted $\bD^b_{\hol,F}(X,Y)$. Similarly, if $Y$ is proper, then $\bD^b_{\hol,F}(X,Y,\fP)$ only depends on $X$, and is this denoted $\bD^b_{\hol,F}(X)$. Again, let me repeat an earlier warning that $\bD^b_{\hol,F}(X)$ consists of `overconvergent' objects on $X$, the analogous category of `convergent' objects on $X$ would be denoted $\bD^b_{\hol,F}(X,X)$.

I now explain how the cohomological operations on $\bD^b_{\hol,F}(X,Y)$ are defined, following \cite[\S1]{AC18a}. First of all, there are the dual and tensor product functors
\begin{align*}
\mathbf{D}_{X} \from& \bD^b_{\mathrm{hol},F}(X,Y)^{\mathrm{op}} \rightarrow \bD^b_{\mathrm{hol},F}(X,Y) \\
 -\wotimes_{\cO_X} - \from & \bD^b_{\mathrm{hol},F}(X,Y) \times \bD^b_{\mathrm{hol},F}(X,Y) \rightarrow \bD^b_{\mathrm{hol},F}(X,Y)
 \end{align*}
which are defined by taking an l.p. frame $(X,Y,\mathfrak{P})$ and setting
\begin{align*} \mathbf{D}_{X} &:= \mathbf{R}\underline{\Gamma}^\dagger_X \circ \mathbf{D}_\mathfrak{P} \\
 \mathcal{M} {\wotimes}_{\cO_X} \mathcal{N} &:= \mathcal{M} {\wotimes}_{\mathcal{O}_{\mathfrak{P}}} \mathcal{N}.
 \end{align*}
The resulting functors only depend on $(X,Y)$ up to canonical isomorphism \cite[\S1.1.6]{AC18a}. Note that the notation here only refers to $X$ and not the full pair $(X,Y)$.

If $(f,g)\from (X',Y')\rightarrow (X,Y)$ is a morphism of (strongly realisable) pairs then there are functors
\[ f^!,f^+: \bD^b_{\mathrm{hol},F}(X,Y)\rightarrow \bD^b_{\mathrm{hol},F}(X',Y'), \]
and if $g$ is proper there are functors
\[ f_!,f_+: \bD^b_{\mathrm{hol},F}(X',Y') \rightarrow \bD^b_{\mathrm{hol},F}(X,Y). \]
These are defined as follows: choose a morphism
\[ (f,g,u)\from  (X',Y',\mathfrak{P}') \rightarrow (X,Y,\mathfrak{P}) \]
of l.p. frames extending $u$, and set
\begin{align*} f^! : = \mathbf{R}\underline{\Gamma}_{X'}^\dagger \circ u^!,\;\;&\;\; f^+ = \mathbf{D}_{X'} \circ  f^! \circ  \mathbf{D}_{X}\\
f_+=u_+ \;\;&\;\; f_! = \mathbf{D}_{X} \circ f_+ \circ  \mathbf{D}_{X'}.
\end{align*}
There is a morphism of functors
\[ f_! \rightarrow f_+ \]
which is an isomorphism whenever $f$ (as well as $g$) is proper \cite[\S1.1.9]{AC18a}. Again, let me emphasize that these functors are all defined for morphisms of (strongly realisable) pairs, even if the notation only refers to the morphism $f\from X'\to X$. There is also the compatibility
\[ f^!(-\wotimes_{\cO_X}-) \isomto f^!(-)\wotimes_{\cO_{X'}} f^!(-) \]
as in \cite[\S1.1.9]{AC18a}, as well as the projection formula
\[ f_+(f^!(-)\wotimes_{\cO_{X'}}(-)) \isomto (-)\wotimes_{\cO_X} f_+(-) \]
by \cite[Proposition A.6]{AC18a} Both $(f^+,f_+)$ and $(f_!,f^!)$ are adjoint pairs \cite[Lemma 1.1.10]{AC18a}, and if
\[\xymatrix{ (U',V') \ar[r]^-{(s',t')} \ar[d]_{(f',g')} & (X',Y')\ar[d]^{(f,g)} \\ (U,V) \ar[r]^-{(s,t)} & (X,Y) } \]
is a Cartesian morphism of pairs, with $g$ proper, then by \cite[Lemma 1.3.10]{AC18a} there is a natural isomorphism
\[ s^!f_+ \cong f'_+s'^! \]
of functors
\[ \bD^b_{\mathrm{hol},F}(X',Y') \rightarrow  \bD^b_{\mathrm{hol},F}(U,V) .\]
There is of course a similar formalism for strongly realisable varieties, obtained by choosing a strongly realisable pair $(X,Y)$ with $Y$ proper, and setting $\bD^b_{\mathrm{hol},F}(X)=\bD^b_{\mathrm{hol},F}(X,Y)$.

\subsection{Relation with locally free isocrystals}\label{sec: sp+}

If $(X,Y,\fP)$ is a frame, with $\fP$ smooth over $\cV$ and $X$ smooth over $k$, then Caro defined in \cite{Car11} a fully faithful functor
\[ \mathrm{sp}_{X+}\from \Isoc_F(X,Y) \to \bD^b(\sD^\dagger_{\fP\Q}), \]
and it is the main result of \cite{CT12} that this lands inside $\bD^b_{\hol,F}(\fP)$. In fact, for reasons explained in \cite{AL22} (see also \S\ref{sec: rigid cohomology} below) I will want to rename this functor $\widetilde{\sp}_{X+}$, and then define
\[ \mathrm{sp}_{X+}:=\widetilde{\sp}_{X+}[-\dim X] \]
to be the shifted version of Caro's functor. Caro also showed in \cite{Car11} that $\widetilde{\sp}_{X+}$ is compatible with duality and pullback, in the following sense.

For compatibility with duality, let $(-)^\vee$ denote the dual functor for locally free isocrystals. It was proved in \cite[Proposition 5.2.7]{Car11} that there is a natural isomorphism
\[  \widetilde{\sp}_{X+} \circ (-)^\vee \isomto \bR\underline{\Gamma}^\dagger_X \circ \bD_\fP\circ \widetilde{\sp}_{X+}.\]
Note that if $(X,Y,\fP)$ is an l.p. frame, then $\bR\underline{\Gamma}^\dagger_{X}\circ \bD_\fP$ is precisely the definition of $\bD_X$, but I do not want to make this assumption on the frame here. Compatibility with duality might therefore be slightly abusively written as 
\[  \widetilde{\sp}_{X+} \circ (-)^\vee \isomto  \bD_X\circ \widetilde{\sp}_{X+}.\]
For compatibility with pullback, suppose that
\[ \xymatrix{ X'\ar[r] \ar[d]^f & Y' \ar[d]^g \ar[r] & \fP' \ar[d]^u \\ X \ar[r] & Y\ar[r] & \fP } \]
is a morphism of frames, with $\fP,\fP',X$ and $X'$ all smooth. Thanks to the compatibility with duality already quoted, it was proved in \cite[Proposition 4.2.4]{Car11}, that there is a natural isomorphism
\[ \bR\underline{\Gamma}^\dagger_{X'}\circ u^! \circ \widetilde{\sp}_{X+}[\dim X]  \isomto \widetilde{\sp}_{X'+}  \circ f^*[\dim X']. \]
Again, if both of these frames are l.p. frames, then $\bR\underline{\Gamma}^\dagger_{X'}\circ u^!$ is precisely the definition of $f^!$, but, again, I don't want to make this assumption here.

The following definition of the `dual' of $\sp_+$ was given in \cite{AL22}.

\begin{definition} Let $(X,Y,\fP)$ be a frame with $\fP$ smooth over $\cV$ and $X$ smooth over $k$. Then define
\[ \mathrm{sp}_{X!}:=\widetilde{\sp}_{X+}[\dim X] = \bR\underline{\Gamma}^\dagger_X  \circ \bD_\fP \circ \sp_{X+} \circ (-)^\vee.  \]
\end{definition}

\begin{remark} it may seem rather odd to have defined both $\sp_+$ and $\sp_!$ separately as shifts of Caro's functor $\widetilde{\sp}_+$. The point is that the definitions are only this simple when $X$ is smooth. Both $\sp_+$ and $\sp_!$ generalise to the case when $X$ is not necessarily smooth, however, they are no longer just shifts of one another. Instead, it is the duality relation between $\sp_+$ and $\sp_!$ which persists. The analogy to bear in mind is that of a lisse $\ell$-adic sheaf $\sF$ on $X$. If $X$ is smooth, then the Verdier dual $\bD_X(\sF)$ is just a shift of $\sF^\vee$, however, this is no longer necessarily the case when $X$ is singular. 
\end{remark}

Thus compatibility with pullbacks can be (again, slightly abusively) written as
\[ f^! \circ \sp_{X!} \isomto \sp_{X'!}  \circ f^*. \]
As expected, everything in $\bD^b_{\hol,F}(X,Y,\fP)$ is generically in the essential image of $\widetilde{\sp}_{X+}$.

\begin{proposition} \label{prop: overhol complexes generically isocrystals} Let $\fP$ be a smooth formal scheme, and $\cM\in \bD^b_{\hol,F}(\fP)$ supported on a reduced closed subscheme $Y\hookrightarrow P$. Then there exists a divisor $D\subset P$ such that $X:=Y\setminus D$ is smooth, non-empty, and each $\mathcal{H}^q(\cM(^\dagger D))$ is in the essential image of
\[ \widetilde{\sp}_{X+} \from \Isoc_F(X,Y) \rightarrow \bD^b_{\hol,F}(\fP). \]
\end{proposition}

\begin{remark} Implicit in the statement is the fact that $\widetilde{\sp}_{X+}$ lands in the category of $\sD^\dagger_{\fP\Q}$-modules, not just complexes. Also note that since $D$ is a divisor, the functor $(^\dagger D)$ is exact for the natural t-structure on $\bD^b_\hol(\fP)$, and in fact $\mathcal{H}^q(\cM(^\dagger D))\isomto \sD^\dagger_{\fP\Q}(^\dagger D)\otimes_{\sD^\dagger_{\fP\Q}}\cH^q(\cM)$. It's also worth pointing out that $Y$ is not assumed to be irreducible, but $X$ need not necessarily be dense in $Y$ (just non-empty).
\end{remark}

\begin{proof}
This follows from \cite[Lemme 6.2.1]{Car11}.
\end{proof}

This can be used to deduce the following important `conservativity' result for extraordinary stalks. 

\begin{proposition} \label{prop: i^! conservative} Let $\fP$ be a smooth formal scheme, and $\cM\in \bD^b_{\hol,F}(\fP)$. If $\bR\underline{\Gamma}^\dagger_x\cM=0$ for all closed points $x\in P$, then $\cM=0$.
\end{proposition}

\begin{proof}
Let $Y$ denote the support of $\cM$ (in the set-theoretic sense), and choose $D$ as in Proposition \ref{prop: overhol complexes generically isocrystals} above. By Noetherian induction on $Y$, I can assume that $\mathbf{R}\underline{\Gamma}_D^\dagger\cM=0$, or in other words that $\cM=\cM(^\dagger D)$. 

Set $\mathfrak{U}:=\fP\setminus D$. Since $\mathbf{R}\underline{\Gamma}_D^\dagger\cM=0$, it suffices to show that $\cM\!\!\mid_\mathfrak{U}=0$. Indeed, if $\cM\!\!\mid_\mathfrak{U}=0$, then $\cM$ is supported on $D$, in which case $\cM=\mathbf{R}\underline{\Gamma}_D^\dagger\cM=0$. After replacing $\fP$ with $\mathfrak{U}$, I can therefore assume that the support $Y$ of $\cM$ is smooth, and that each $\cH^q(\cM)$ is in the essential image of 
\[  \widetilde{\sp}_{Y+} \from \Isoc_F(Y,Y) \rightarrow \bD^b_{\hol,F}(\fP), \]
say $\cH^q(\cM)=\widetilde{\sp}_{Y+}\sF_q$. 
If $i_x:x\rightarrow Y$ denotes the inclusion of a closed point of $Y$, then
\[ \bR\underline{\Gamma}^\dagger_{x} \circ \widetilde{\sp}_{Y+}[\dim Y]  \isomto \widetilde{\sp}_{x+}  \circ i_x^*.  \]
Hence the fact that $\bR\underline{\Gamma}^\dagger_{x}\cM=0$ implies that $i_x^*\sF_q=0$ for all $q$, and all closed points $x\in Y$. Therefore $\sF_q=0$, and so $\cH^q(\cM)=0$ as required.
\end{proof}

\subsection{t-structures on overholonomic $\pmb{\sD^\dagger}$-modules}

For any prime $\ell\neq p$, and any variety $X$, there are two well-known t-structures on the triangulated category ${\bf D}^b_c(X_\et,\Q_\ell)$ of bounded constructible complexes of $\ell$-adic sheaves on $X_\et$: the ordinary t-structure, with heart the category $\Con(X_\et,\Q_\ell)$ of constructible $\ell$-adic sheaves on $X$, and the perverse t-structure on ${\bf D}^b_c(X_\et,\Q_\ell)$, with heart the category $\Perv(X_\et,\Q_\ell)$ of perverse sheaves on $X$.

The perverse t-structure is self dual under the Verdier duality functor $\bD_X$, however, the constructible t-structure is not. This means that there is a third natural t-structure on ${\bf D}^b_c(X_\et,\Q_\ell)$, the dual constructible t-structure, whose heart $\DCon(X_\et,\Q_\ell)$ is canonically anti-equivalent to $\Con(X_\et,\Q_\ell)$, via $\mathbf{D}_X$. This third t-structure appears less often in the literature than the other two, and is, in a sense, of lesser importance, since most of its relevant properties can be deduced from that of the constructible t-structure by dualising. However, it will be important for us, since it is the analogue of the \emph{dual constructible} t-structure on $\bD^b_{\hol,F}$ that will match up with the natural t-structure on $\bD^b_{\cons,F}$. 

These three t-structures all have analogues in the world of overholonomic $\sD^\dagger$-modules, which I will now describe in the case of a smooth formal scheme $\fP$. 

\subsubsection{Holonomic t-structure} 

Let $\fP$ be a smooth formal scheme. The first t-structure on ${\bf D}^b_{\mathrm{hol},F}(\fP)$ is simply the natural one coming from the inclusion ${\bf D}^b_{\mathrm{hol},F}(\fP)\hookrightarrow {\bf D}^b(\mathscr{D}^\dagger_{\fr{P}\Q})$. That this is indeed a t-structure follows from the fact that a complex of $\sD^\dagger_{\fP\Q}$-modules is overholonomic iff its cohomology sheaves are.  We denote by ${\bf D}^{\geq q},{\bf D}^{\leq q}$ the full subcategories of objects concentrated in degrees $\geq q$ and $\leq q$ respectively, and by $\tau^{\geq q},\tau^{\leq q}$ the truncation functors. By combining \cite[Proposition 3.3 4)]{Car09b} with \cite[\S4]{Vir00}, we see that the holonomic t-structure is self-dual, in the sense that 
\begin{align*}
 \mathcal{M}\in {\bf D}^{\geq0} &\iff \mathbf{D}_{X}\mathcal{M}\in  {\bf D}^{\leq0}.
\end{align*}
I will denote the heart of the holonomic t-structure by $\mathbf{Hol}_F(\fP)$, and refer to its objects as holonomic modules on $\fP$. Cohomology functors will be denoted by
\[ \mathcal{H}^q: {\bf D}^b_{\mathrm{hol},F}(\fP) \rightarrow \mathbf{Hol}_F(\fP). \]
These are just the restriction to ${\bf D}^b_{\mathrm{hol},F}(\fP)$ of the natural cohomology functors
\[ \mathcal{H}^q\from \bD^b_\coh(\sD^\dagger_{\fP\Q}) \to \Mod(\sD^\dagger_{\fP\Q}). \]

\subsubsection{Constructible t-structure} The constructible t-structure on ${\bf D}^b_{\mathrm{hol},F}(\fP)$ is analogous to the t-structure on the derived category ${\bf D}^b_\mathrm{rh}(\mathscr{D}_X)$ of regular holonomic $\mathscr{D}$-modules on a smooth complex variety $X$ induced by the ordinary t-structure on ${\bf D}^b_c(X^\mathrm{an},\C)$ via the (covariant) Riemann--Hilbert correspondence
\[ \mathcal{M}\mapsto  \omega_{X^\mathrm{an}/\C} \otimes^{\mathbb{L}}_{\mathscr{D}_{X^\mathrm{an}}}\mathcal{M}^\mathrm{an} \cong \Omega^\bullet_{X^\mathrm{an}} \otimes_{\mathcal{O}_{X^\mathrm{an}}} \mathcal{M}^{\mathrm{an}}[\dim X].\]
For overholonomic $\sD^\dagger$-modules, this t-structure was defined for curves in \cite{LS14} and in general in \cite{Abe18a}. Concretely, if $\mathcal{M}\in \Hol_F(\fP)$, define the support $\mathrm{Supp}(\mathcal{M})$ of $\cM$ to be the smallest closed subscheme $Z\subset P$ such that $\mathcal{M}|_{P\setminus Z}=0$. Then define a pair of subcategories $({}^{\rm c}{\bf D}^{\geq0},{}^{\rm c}{\bf D}^{\leq0})$ of ${\bf D}^b_{\mathrm{hol},F}(\fP)$ as follows:
\begin{align*}
\mathcal{M} \in {}^\mathrm{c}{\bf D}^{\geq0} &\iff  \dim \mathrm{Supp}\, \mathcal{H}^n(\mathcal{M}) \leq n\;\;\forall n\geq 0, \;\;\mathcal{H}^n(\mathcal{M})  = 0 \;\;\forall n < 0 \\
\mathcal{M} \in {}^\mathrm{c}{\bf D}^{\leq0} &\iff \mathcal{H}^{-n}( \bR\underline{\Gamma}^\dagger_Z \bD_\fP\mathcal{M}) =0 \text{ for any closed subscheme }Z\hookrightarrow P\text{ with }\dim Z < n.
\end{align*}
I will postpone for now the proof that this is a t-structure, since I will deduce it from the corresponding claim for the dual constructible t-structure, which I will introduce next.

\subsubsection{Dual constructible t-structure} Whereas the holonomic t-structure is self-dual, the constructible t-structure is \emph{not}. There is therefore a third t-structure $({}^{\mathrm{dc}}{\bf D}^{\geq0} ,{}^{\mathrm{dc}}{\bf D}^{\leq0} )$ on $\bD^b_{\hol,F}(\fP)$  by setting
\begin{align*}
 \mathcal{M}\in {}^{\mathrm{dc}}{\bf D}^{\geq0} &\iff \mathbf{D}_{\fP}\mathcal{M}\in  {}^{{\rm c}}{\bf D}^{\leq0}  \\
  \mathcal{M}\in {}^{\mathrm{dc}}{\bf D}^{\leq0} &\iff \mathbf{D}_{\fP}\mathcal{M}\in  {}^{{\rm c}}{\bf D}^{\geq0}.
\end{align*}
Explicitly,
\begin{align*}
 \mathcal{M}\in {}^{\mathrm{dc}}{\bf D}^{\geq0} &\iff \mathcal{H}^{-n}(\bR\underline{\Gamma}^\dagger_Z\cM)=0,\;\;\forall Z\hookrightarrow P,\;\;\dim Z<n  \\
  \mathcal{M}\in {}^{\mathrm{dc}}{\bf D}^{\leq0} &\iff \dim \mathrm{Supp}\,\mathcal{H}^{-n}(\cM)\leq n\;\;\forall n\geq0,\;\;\mathcal{H}^{-n}(\cM)=0\;\;\forall n<0.
\end{align*}
Note that the condition for $\cM$ to lie in ${}^{\mathrm{dc}}{\bf D}^{\geq0}$ may be tested on irreducible $Z$. To prove that this is indeed a t-structure, I follow the approach of \cite[Proposition 1.3.3]{Abe18a}. 

\begin{lemma}\label{lemma: t exact sections with support}  Let $Y\to P$ be a closed subscheme. Then the functors $(^\dagger Y)$ and $\bR\underline{\Gamma}^\dagger_Y$  both preserve ${}^{\mathrm{dc}}{\bf D}^{\geq0}$ and ${}^{\mathrm{dc}}{\bf D}^{\leq0}$.
\end{lemma}

\begin{proof}
The claims for ${}^{\mathrm{dc}}{\bf D}^{\geq0}$ are relatively straightforward. Indeed, the fact that $\bR\underline{\Gamma}^\dagger_Y$ preserves ${}^{\mathrm{dc}}{\bf D}^{\geq0}$ simply follows from the fact that $\bR\underline{\Gamma}^\dagger_Z\bR\underline{\Gamma}^\dagger_Y=\bR\underline{\Gamma}^\dagger_{Y \cap Z}$, and $\dim Y\cap Z\leq \dim Z$. 

For $(^\dagger Y)$, take $Z\hookrightarrow P$ irreducible with $\dim Z<n$, $\cM\in  {}^{\mathrm{dc}}{\bf D}^{\geq0}$, and consider the exact sequence\[ \mathcal{H}^{-n}(\mathbf{R}\underline{\Gamma}^\dagger_Z\cM) \rightarrow \mathcal{H}^{-n}(\mathbf{R}\underline{\Gamma}^\dagger_Z\cM(^\dagger Y)) \rightarrow 
\mathcal{H}^{-n+1}(\mathbf{R}\underline{\Gamma}^\dagger_{Z\cap Y}\cM ).   \]
The left hand term here is zero, I need to show that so is the right hand term. If $\dim Z\cap Y<\dim Z$, then $\dim Z\cap Y<n-1$, and thus $\mathcal{H}^{-n}(\mathbf{R}\underline{\Gamma}^\dagger_Z\cM(^\dagger Y))=0$. Since $Z$ is irreducible, the only way that this can fail to happen is if $Z\subset Y$, in which case $\mathbf{R}\underline{\Gamma}^\dagger_Z\cM(^\dagger Y)=0$.

The hardest part is to prove that $\bR\underline{\Gamma}^\dagger_Y$ preserves ${}^{\mathrm{dc}}{\bf D}^{\leq0}$. To show this, write $Y$ as an intersection of divisors, this reduces to the case when $Y$ itself is a divisor. In this case, the functor $(^\dagger Y)$ is t-exact for the ordinary t-structure, and hence $\bR\underline{\Gamma}^\dagger_Y$ has cohomological amplitude $[0,1]$. Write $\cH^{i,\dagger}_Y$ for the cohomology sheaves of $\bR\underline{\Gamma}^\dagger$. If $\cM\in {}^{\mathrm{dc}}{\bf D}^{\leq0}$, it then follows that $\mathcal{H}^{-n}(\bR\underline{\Gamma}^\dagger_Y\cM)=0$ whenever $n<-1$. When $n\geq -1$, there is an exact sequence
\[ 0\rightarrow \mathcal{H}^{1,\dagger}_Y(\mathcal{H}^{-(n+1)}(\cM)) \rightarrow \mathcal{H}^{-n}(\bR\underline{\Gamma}^\dagger_X\cM) \rightarrow \mathcal{H}^{0,\dagger}_Y(\mathcal{H}^{-n}(\cM)) \to 0. \]
Since $\cM\in {}^{\rm dc}\bD^{\leq 0}$, we know that
\[ \dim\mathrm{Supp}\,\mathcal{H}^{-n}(\cM) \leq n,\;\;\;\; \dim\mathrm{Supp}\,\mathcal{H}^{-(n+1)}(\cM))\leq n+1 \]
and hence trivially 
\[\dim\mathrm{Supp}\,\mathcal{H}^{0,\dagger}_Y(\mathcal{H}^{-n}(\cM))\leq n,\;\;\;\;\dim\mathrm{Supp}\,\mathcal{H}^{1,\dagger}_Y(\mathcal{H}^{-(n+1)}(\cM)) \leq n+1.\]
What is needed is to show that
\[ \dim\mathrm{Supp}\,\mathcal{H}^{1,\dagger}_Y(\mathcal{H}^{-(n+1)}(\cM)) \leq n \]
(which should be interpreted as saying that $\mathcal{H}^{1,\dagger}_Y(\mathcal{H}^{0}(\cM))=0$ when $n=-1$). It is therefore enough to show that if $\cN\in \Hol_F(\fP)$, then
\[ \dim \mathrm{Supp}(\cN) \leq n+1\implies \dim \mathrm{Supp}\,\mathcal{H}^{1,\dagger}_Y(\cN)\leq n,\]
for all $n\geq -1$ (again, meaning that $\mathcal{H}^{1,\dagger}_Y(\cN)$=0 when $n=-1$). To prove this, write $\mathrm{Supp}(\cN)=Z_1\cup Z_2$, where $Z_1\subset Y$, and no irreducible component of $Z_2$ is contained in $Y$. Then $\dim Y\cap Z_2 \leq n$, and I claim that $\mathcal{H}^{1,\dagger}_Y(\cN)$ is supported on $Y\cap Z_2$. Indeed, $\mathcal{H}^{1,\dagger}_Y(\cN)$ is clearly supported on $Y$, so it will suffice to show that it is zero on $P\setminus Z_2$. But after restricting to $P\setminus Z_2$, we have $\mathrm{Supp}(\cN)\subset Y$, whence $\mathbf{R}\underline{\Gamma}^\dagger_Y\cN\isomto \cN$, and so $\mathcal{H}^{1,\dagger}_Y(\cN)=\mathcal{H}^1(\cN)=0$.

Finally, the fact that $(^\dagger Y)$ preserves ${}^{\mathrm{dc}}{\bf D}^{\leq0}$ now follows from taking cohomology sheaves of the exact triangle
\[ \mathbf{R}\underline{\Gamma}^\dagger_Y \to \mathrm{id} \to (^\dagger Y) \overset{+1}{\lto} \]
and using the already proved result for $\mathbf{R}\underline{\Gamma}^\dagger_Y$.
\end{proof}

\begin{theorem} The pair of full subcategories $({}^{\mathrm{dc}}{\bf D}^{\geq0},{}^{\mathrm{dc}}{\bf D}^{\leq0})$ defines a t-structure on $\bD^b_{\hol,F}(\fP)$.
\end{theorem}

\begin{proof} To set up a Noetherian induction, I will actually prove that $({}^{\mathrm{dc}}{\bf D}^{\geq0},{}^{\mathrm{dc}}{\bf D}^{\leq0})$ defines a t-structure on $\bD^b_{\hol,F}(Y,Y,\fP)\subset \bD^b_{\hol,F}(\fP)$ for any closed subscheme $Y\hookrightarrow P$, which I may as well assume to be reduced. 

I first reduce to the case that $Y$ is irreducible. Indeed, if $Y=Y_1\cup Y_2$ is a union of proper, non-empty closed subschemes, then any $\cM\in \bD^b_{\hol,F}(Y,Y,\fP)$ sits in an exact triangle
\[ \bR\underline{\Gamma}^\dagger_{Y_1\cap Y_2}\cM \to  \bR\underline{\Gamma}^\dagger_{Y_1}\cM\oplus\bR\underline{\Gamma}^\dagger_{Y_2}\cM  \to \cM\overset{+1}{\to}. \]
Assuming that $({}^{\mathrm{dc}}{\bf D}^{\geq0},{}^{\mathrm{dc}}{\bf D}^{\leq0})$ defines a t-structure on $\bD^b_{\hol,F}(Y_i,Y_i,\fP)$ for $i=1,2$, I can check the axioms  \cite[D\'efinition 1.3.1]{BBD82} explicitly as follows. First of all (ii) is clear, and (i) follows from the fact that if two out of the three morphisms in a morphism of exact triangles are zero, then so is the third. For (iii), I can appeal to \cite[Proposition 1.1.11]{BBD82} to form the commutative diagram
\[ \xymatrix{ \tau_{\leq0}\bR\underline{\Gamma}^\dagger_{Y_1\cap Y_2}\cM  \ar[r]\ar[d] & \tau_{\leq0}\bR\underline{\Gamma}^\dagger_{Y_1}\cM\oplus\tau_{\leq0}\bR\underline{\Gamma}^\dagger_{Y_2}\cM \ar[d] \ar[r] & \ca{L}\ar[d] \ar[r]^-{+1} & \\ 
\bR\underline{\Gamma}^\dagger_{Y_1\cap Y_2}\cM \ar[r]\ar[d]&  \bR\underline{\Gamma}^\dagger_{Y_1}\cM\oplus\bR\underline{\Gamma}^\dagger_{Y_2}\cM  \ar[r]\ar[d]& \cM \ar[d]\ar[r]^-{+1} &\\
\tau_{>0}\bR\underline{\Gamma}^\dagger_{Y_1\cap Y_2}\cM  \ar[r]\ar[d]^-{+1} & \tau_{>0}\bR\underline{\Gamma}^\dagger_{Y_1}\cM\oplus\tau_{>0}\bR\underline{\Gamma}^\dagger_{Y_2}\cM  \ar[r]\ar[d]^-{+1} & \ca{N} \ar[r]^-{+1}\ar[d]^-{+1} & \\ & & &  } \]
with all rows and columns exact triangles. It follows directly from the definitions that $\mathcal{L}\in {}^{\mathrm{dc}}{\bf D}^{\leq0}$. Since the condition to lie in ${}^{\mathrm{dc}}{\bf D}^{>0}$ can be checked on irreducible closed subschemes $Z\hookrightarrow P$, the fact that $\cN\in {}^{\mathrm{dc}}{\bf D}^{>0}$ follows from the fact that both $\underline{\bR}\Gamma^\dagger_{Y_1}\cN$ and $\underline{\bR}\Gamma^\dagger_{Y_2}\cN$ are in ${}^{\mathrm{dc}}{\bf D}^{>0}$.
 
I may therefore assume that $Y$ is irreducible. In this case, take a divisor $D\subset P$, with smooth, non-empty complement $X=Y\setminus D$ on $Y$, and define
\[ T(Y,D):= \left\{ \left.\cM\in\bD^b_{\hol,F}(Y,Y,\fP) \,\right\vert \,\mathcal{H}^q((^\dagger D)\cM) \in \widetilde{\sp}_+(\Isoc_F(X,Y))\;\;\forall q \right\}. \]
Then by Proposition \ref{prop: overhol complexes generically isocrystals},
\[ \bD^b_{\hol,F}(Y,Y,\fP) = \colim{D} T(Y,D),\]
and it therefore suffices to show that $({}^{\mathrm{dc}}{\bf D}^{\geq0},{}^{\mathrm{dc}}{\bf D}^{\leq0})$ defines a t-structure on $T(Y,D)$. Now set set $Z:=(Y \cap D)_\mathrm{red}$,  by Noetherian induction I can assume that $({}^{\mathrm{dc}}{\bf D}^{\geq0},{}^{\mathrm{dc}}{\bf D}^{\leq0})$ defines a t-structure on $T(Z):=\bD^b_{\hol,F}(Z,Z,\fP)$. I now set $T(X):=T(Y,D)\cap \bD^b_{\hol,F}(X,Y,\fP)$, thus
\begin{align*}
 (^\dagger D)=(^\dagger Z) &\from T(Y,D) \to T(X) \\
 \bR\underline{\Gamma}^\dagger_D = \bR\underline{\Gamma}^\dagger_Z &\from T(Y,D)\to T(Z),
\end{align*}
and the localisation triangle 
\[ \bR\underline{\Gamma}_D^\dagger  \to \mathrm{id}\to (^\dagger D) \overset{+1}{\to}, \] 
together with Lemma \ref{lemma: t exact sections with support}, shows that $\cM\in T(Y,D)$ is in ${}^{\mathrm{dc}}{\bf D}^{\geq0}$ or ${}^{\mathrm{dc}}{\bf D}^{\leq0}$ if and only if both $\bR\underline{\Gamma}_D^\dagger\cM$ and $\cM(^\dagger D)$ are. Moreover, since $D$ is a divisor, and $X$ is smooth, it follows that on $T(X)$, $({}^{\mathrm{dc}}{\bf D}^{\geq0},{}^{\mathrm{dc}}{\bf D}^{\leq0})$ is simply the shift of the ordinary t-structure $({\bf D}^{\geq0},{\bf D}^{\leq0})$ by the dimension of $X$, and therefore defines a t-structure on $T(X)$.

I now check the axioms \cite[D\'efinition 1.3.1]{BBD82}. Of course (ii) is straightforward. To prove (i), the fact that $({}^{\mathrm{dc}}{\bf D}^{\geq0},{}^{\mathrm{dc}}{\bf D}^{\leq0})$ defines a t-structure on both $T(X)$ and $T(Z)$ means that I can apply the localisation triangle
\[ \bR\underline{\Gamma}_D^\dagger  \to \mathrm{id}\to (^\dagger D) \overset{+1}{\to}, \] 
twice to reduce to showing that
\[ \mathrm{Hom}(\bR\underline{\Gamma}^\dagger_D\cM ,\cN(^\dagger D)) =0,\;\;\;\;\mathrm{Hom}(\cM(^\dagger D) ,\bR\underline{\Gamma}^\dagger_D\cN) =0\]
 whenever $\cM\in T(Y,D)\cap {}^{\mathrm{dc}}{\bf D}^{<0}$ and $\cN\in T(Y,D)\cap {}^{\mathrm{dc}}{\bf D}^{\geq0}$. The first is straightforward, since
 \[ \mathrm{Hom}(\bR\underline{\Gamma}^\dagger_D\cM ,\cN(^\dagger D))=\mathrm{Hom}(\bR\underline{\Gamma}^\dagger_D\cM ,\bR\underline{\Gamma}^\dagger_D\cN(^\dagger D)) = \mathrm{Hom}(\bR\underline{\Gamma}^\dagger_D\cM ,0)=0 .\]
For the second, note that $ \cM(^\dagger D)\in T(X)\cap {}^{\mathrm{dc}}{\bf D}^{<0}$, and hence $\cH^{n}(\cM(^\dagger D) )=0$ if $n \geq - \dim X$. On the other hand, $\bR\underline{\Gamma}^\dagger_D\cN\in T(Z)\cap {}^{\mathrm{dc}}{\bf D}^{\geq0}$, and hence $\cH^n(\bR\underline{\Gamma}^\dagger_D\cN)=0$ if $n< -\dim Z$. Thus $\mathrm{Hom}(\cM(^\dagger D) ,\bR\underline{\Gamma}^\dagger_D\cN)=0$ as required.

Finally, to prove (iii), I consider, for any $\cM\in T(Y,D)$, the shifted localisation triangle
\[ \cM \to \cM(^\dagger D) \to \bR\underline{\Gamma}^\dagger_D\cM[1] \overset{+1}{\to}. \]
Since $({}^{\mathrm{dc}}{\bf D}^{\geq0},{}^{\mathrm{dc}}{\bf D}^{\leq0})$ defines a t-structure on both $T(X)$ and $T(Z)$, I can extend this to a diagram
\begin{equation} \label{eqn: diagram for t-structures}
 \xymatrix{ & \tau_{\leq 0}\cM(^\dagger D) \ar[d] & \left(\tau_{\leq 0}\bR\underline{\Gamma}^\dagger_D\cM \right)[1] \ar[d] \\
\cM \ar[r] & \cM(^\dagger D) \ar[r] \ar[d] & \bR\underline{\Gamma}^\dagger_D\cM[1] \ar[d] \ar[r]^-{+1} & \\ & \tau_{>0}\cM(^\dagger D) \ar[d]^-{+1} & \left(\tau_{>0}\bR\underline{\Gamma}^\dagger_D\cM \right)[1]\ar[d]^-{+1} \\ & &  }
\end{equation}
Now consider the morphism 
\begin{equation}
\label{eqn: morphism for t structures}
 \tau_{\leq 0}\cM(^\dagger D) \to \left(\tau_{>0}\bR\underline{\Gamma}^\dagger_D\cM \right)[1] 
\end{equation}
Since $\tau_{\leq 0}\cM(^\dagger D)\in {}^{\mathrm{dc}}{\bf D}^{\leq 0}\cap T(X)$, it follows that
\[ \cH^n(\tau_{\leq 0}\cM(^\dagger D))=0,\;\;\;\;\forall n>-\dim X \]
On the other hand, since $(\tau_{>0}\bR\underline{\Gamma}^\dagger_D\cM)[1] \in {}^{\mathrm{dc}}{\bf D}^{\geq 0}\cap T(Z)$ it follows that 
\[ \cH^n(( \tau_{>0}\bR\underline{\Gamma}^\dagger_D\cM)[1]) =0,\;\;\;\;\forall n < -\dim Z.  \]
Since $Y$ is irreducible, $\dim Z<\dim X$, and it immediately follows that
\[ {\rm Hom}(\tau_{\leq 0}\cM(^\dagger D),(\tau_{>0}\bR\underline{\Gamma}^\dagger_D\cM )[1] )=0.\]
Hence the diagram (\ref{eqn: diagram for t-structures}) can be completed (and rotated) to obtain a diagram 
\[ \xymatrix{ \tau_{\leq 0}\bR\underline{\Gamma}^\dagger_D\cM \ar[r]\ar[d] & \ca{L} \ar[r]\ar[d] &  \tau_{\leq 0}\cM(^\dagger D)  \ar[d]  \ar[r]^-{+1}&\\
\bR\underline{\Gamma}^\dagger_D\cM\ar[r]\ar[d] & \cM \ar[d] \ar[r] & \cM(^\dagger D) \ar[r] \ar[d] \ar[r]^-{+1} & \\ \tau_{>0}\bR\underline{\Gamma}^\dagger_D\cM  \ar[r] \ar[d]^-{+1} & \cN \ar[d]^-{+1} \ar[r] & \tau_{>0}\cM(^\dagger D) \ar[d]^-{+1} \ar[r]^-{+1} & \\ & & & }\]
with all rows and columns exact triangles. Then $\ca{L}\in {}^{\mathrm{dc}}{\bf D}^{\leq0}$ and $\cN\in {}^{\mathrm{dc}}{\bf D}^{>0}$, completing the proof.
\end{proof}

It follows that the pair $({}^{\rm c}{\bf D}^{\geq0},{}^{\rm c}{\bf D}^{\leq0})$ also defines a t-structure on $\bD^b_{\hol,F}(\fP)$. I will denote the hearts by of these two t-structures by $\Con_F(\fP)$ and $\DCon_F(\fP)$, and cohomology functors by ${}^{\rm c}\cH^q$ and ${}^{\rm dc}\cH^q$. The truncation functors will be ${}^{\rm c}\tau$ and ${}^{\rm dc}\tau$, and objects of the hearts will be called constructible and dual constructible modules on $\fP$ respectively.

\begin{example} \label{exa: sp in hearts} Let $(X,Y,\fP)$ be a frame, with $\fP$ and $X$ smooth, and suppose $\sF\in \Isoc_F(X,Y)$. Then $\sp_{!}\sF\in \DCon_F(\fP)$. If there exists a divisor $D\subset \fP$ such that $X=Y\setminus D$, then $\widetilde{\sp}_+\sF\in \Hol_F(\fP)$. If $X=Y$ then $\sp_+\sF\in \Con_F(\fP)$.
\end{example}

As part of the proof that $({}^{\rm dc}{\bf D}^{\geq0},{}^{\rm dc}{\bf D}^{\leq0})$ is indeed a t-structure, we saw that the functors $\bR\underline{\Gamma}^\dagger_{Y}$ and $(^\dagger Y)$ for a closed subscheme $Y\hto P$ are t-exact for the dual constructible t-structure, I will abbreviate this as being `dct-exact'. If $u\from \fP\to\fQ$ is a closed immersion of smooth formal schemes, it is easy to check that $u_+$ is dct-exact.

\begin{proposition} \label{prop: u^! dct exact formal schemes} Let $u\from \fP\to \fQ$ be any morphism of smooth formal schemes. Then $u^!$ is dct-exact. 
\end{proposition}

\begin{remark} This is the dual of the fact that $u^+$ is t-exact for the constructible t-structure.
\end{remark}

\begin{proof}
If $u$ is a closed immersion, then $u_+u^!\isomto \bR{\underline{\Gamma}}^\dagger_P$. Hence the dct-exactness of $u^!$ follows from that of $\bR\underline{\Gamma}^\dagger_P$ and $u_+$. If $K'\subset K''$ are finite unramified extensions of $K$, with rings of integers $\cV'\subset \cV''$, then the proposition clearly holds for the morphism 
$u\from \spf{\cV''}\to \spf{\cV'}$. Hence the proposition holds whenever $u$ is a $\cV'$-valued point of $\fQ$, for any such $\cV'$.

But using Proposition \ref{prop: i^! conservative}, the general case follows from the particular case of $\cV'$-valued points for $\cV'/\cV$ unramified. Indeed, taking $\cM\in {}^{\rm dc}\bD^{\geq0}$ (resp. $\cM\in {}^{\rm dc}\bD^{\leq0}$), to prove that $u^!\cM\in {}^{\rm dc}\bD^{\geq0}$ (resp. $\cM\in {}^{\rm dc}\bD^{\leq0}$), it suffices to show that ${}^{\rm dc}\tau_{<0}u^!\cM=0$ (resp. ${}^{\rm dc}\tau_{>0}u^!\cM=0$). But now, taking any $\cV'$-valued point $i\from \spf{\cV'}\to \fP$, I can just calculate
\[ i^!{}^{\rm dc}\tau_{<0}u^!\cM= {}^{\rm dc}\tau_{<0}u(i)^!\cM =u(i)^!{}^{\rm dc}\tau_{<0}\cM = 0\]
(resp. $i^!{}^{\rm dc}\tau_{>0}u^!\cM=0$), and therefore conclude using Proposition \ref{prop: i^! conservative}.
\end{proof}

A similar method to Proposition \ref{prop: u^! dct exact formal schemes} can be used to prove that $\wotimes_{\cO_\fP}$ is dct-exact. Indeed, dct-exactness can be checked after taking extraordinary stalks, and $\wotimes_{\cO_\fP}$ commutes with extraordinary pullback. This therefore reduces to the trivial case when $\fP=\spf{\cV'}$ for $\cV'/\cV$ unramified.

If $u\from\fP\to \spf{\cV}$ is the structure morphism, then $\cO^{*}_\fP:=u^!\cO_{\cV\Q}\in \DCon_F(\fP)$. I will call this the \emph{constant} dual constructible module on $\fP$. Be warned that $\cO^{*}_{\fP}$ is \emph{not} the $\sD^\dagger_{\fP\Q}$-module $\cO_{\fP\Q}$, instead it is a shift of this module by the dimension of $\fP$. This shifted version of the `constant' module will instead match up with the constant locally free isocrystal on $\fP$ via the overconvergent Riemann--Hilbert correspondence.

\subsection{Dual constructible modules on pairs and varieties} \label{subsec: dc mod on pairs}

For pairs $(X,Y)$, I will only define the dual constructible t-structure. This is done by taking an l.p. frame $(X,Y,\fP)$ and simply restricting the dual constructible t-structure on $\bD^b_{\hol,F}(\fP)$ to $\bD^b_{\hol,F}(X,Y,\fP)$.The heart of this structure will be denoted $\DCon_F(X,Y)$, and referred to as the category of dual constructible modules on $(X,Y)$. 

\begin{remark}
The analogous definition is not the correct one for either the holonomic or constructible t-structures, and this fact largely explains why it is the dual constructible t-structure which matches up with the natural one on the category of constructible complexes. In fact, the increasing list of hypotheses in Example \ref{exa: sp in hearts} can be removed by replacing $\Hol_F(\fP)$ and $\Con_F(\fP)$ by appropriate analogues for the pair $(X,Y)$.
\end{remark} 

If $(f,g)\from (X,Y)\to (\spec{k},\spec{k})$ denotes the structure morphism, I define
\[ \cO_{(X,Y)}^{*}:=f^!\cO_{\spf{\cV}\Q}\]
to be the constant dual constructible module on $(X,Y)$. If $(X,Y,\fP)$ is an l.p. frame, then $\cO_{(X,Y)}^{*}=\mathbf{R}\underline{\Gamma}^\dagger_X\cO_{\fP}^{*}$. If $Y$ is proper over $k$, then $\cO^{*}_{(X,Y)}\in \bD^b_{\hol,F}(X,Y)=\bD^b_{\hol,F}(X)$ is independent of the choice of $Y$, in which case I will write it as $\cO^{*}_{X}$. Also be warned that this is \emph{not} the same as Caro's object defined in \cite[\S4.23]{Car09b}. Indeed, his is a constant holonomic module (that is, it lies in the heart of the holonomic t-structure on $\bD^b_{\hol,F}(X,Y,\fP)$, that we haven't defined), whereas ours is a constant dual constructible module. Our $\cO^{*}_{X}$ is the direct analogue of the Verdier dual $\bD_X(\underline{\Q}_{\ell,X})$ of the constant sheaf in $\ell$-adic \'etale cohomology.

\begin{lemma} The object $\cO^*_{(X,Y)}\in \bD^b_{\hol,F}(X,Y)$ is a unit for the tensor product $\wotimes_{\cO_{(X,Y)}}$.
\end{lemma}

\begin{proof} Let $(X,Y,\fP)$ be an l.p. frame. In the case $X=Y=P$, it is clear that $\cO^*_{\fP}$ is a unit for the tensor product $\wotimes_{\cO_{\fP}}$. If $X=Y$, then this provides a natural transformation
\[ \cO^*_{(Y,Y)}\wotimes_{\cO_{(Y,Y)}} (-)\to (-)\]
of endofunctors of $\bD^b_{\hol,F}(Y,Y)$. To check that this natural transformation is an isomorphism, it suffices to do so on extraordinary stalks, which reduces to the trivial case
\[ (Y,Y)=(\spec{k'},\spec{k'})\]
for $k'/k$ a finite extension. In general, this in turn induces a natural transformation
\[ (-)\to \cO^*_{(X,Y)}\wotimes_{\cO_{(X,Y)}} (-) \]
of endofunctors of $\bD^b_{\hol,F}(X,Y)$ which is proved to be an isomorphism in the same way.
\end{proof}

As in the case of formal schemes, the tensor product $\wotimes_{\cO_X}$ is dct-exact. If $(f,g)\from (X',Y')\to (X,Y)$ is a morphism of pairs, then $f^!$ is always dct-exact. If $g$ is a closed immersion, so $f$ is a locally closed immersion, then $f_+$ is dct-exact. In particular, if $(X,Y)$ is a pair,
\[ j\from U\to X \ot Z\from i \]
are complementary open and closed subschemes, and $\cM\in \DCon_F(X,Y)$, then the localisation triangle
\begin{equation}
\label{eqn: localisation triangle} i_+i^!\cM\to \cM\to j_+j^+ \cM \overset{+1}{\to} 
\end{equation}
can be viewed as a short exact sequence of dual constructible modules. There is then the following version of d\'evissage for dual constructible modules. 

\begin{proposition} Every $\cM\in \DCon_F(X,Y)$ admits a finite composition  series
\[ 0 =\cM_0 \subset \cM_1 \subset \ldots \subset \cM_n =\cM, \]
such that for each $1\leq \alpha \leq n$, there exists a smooth locally closed subscheme $i_\alpha\from X_\alpha\to X$, with closure $\overline{X}_\alpha$ in $P$, a locally free isocrystal $\sF_\alpha\in \Isoc_F(X_\alpha,\overline{X}_\alpha)$, and an isomorphism
\[ \cM_{\alpha}/\cM_{\alpha-1}\isomto i_{\alpha+}\mathrm{sp}_{X_\alpha!}\sF_\alpha. \] 
\end{proposition}

\begin{proof}
Thanks to the localisation exact sequence (\ref{eqn: localisation triangle}), this follows from Proposition \ref{prop: overhol complexes generically isocrystals}.
\end{proof}

I will end this section with the following $\sD^\dagger$-module analogue of Theorem \ref{theo: finite etale general cons}.

\begin{proposition} \label{prop: finite etale adjoints dct exact} Let
\[ \xymatrix{ X'\ar[r] \ar[d]^f & Y' \ar[d]^g\\ X \ar[r] & Y } \] 
be a morphism of strongly realisable pairs, such that $g$ is proper, $f$ is finite \'etale, and $X$ is smooth. Then the natural morphism $f_!\to f_+$ of functors $\bD^b_{\hol,F}(X',Y')\to \bD^b_{\cons,F}(X,Y)$ is an isomorphism, and there exists a trace isomorphism
\[ f^+\isomto f^! \]
of functors $\bD^b_{\hol,F}(X,Y)\to \bD^b_{\cons,F}(X',Y')$. All four functors $(f^+,f_+,f_!,f^!)$ are dct-exact.
\end{proposition}

\begin{remark} The hypothesis that $X$ is smooth is almost certainly unnecessary, but I will only need the result under this assumption, which makes the proof marginally simpler.
\end{remark}

\begin{proof}
That $f_!\isomto f_+$ follows from the fact that $f$ is proper. To construct the trace morphism
\[ f^+\to f^!, \]
I will construct the adjoint
\[ \mathrm{id}\to f_+f^!. \]
Indeed, since $\cO^*_{(X,Y)}$ is a unit for the tensor product, the projection formula shows that
\[ f_+f^!\cM \isomto f_+f^!\cO^{*}_{(X,Y)} \wotimes_{\cO_{X}} \cM, \]
so it suffices to construct
\[ f_+f^!\cO^*_{(X,Y)} \to \cO^*_{(X,Y)}. \]
But now $\cO^*_{(X,Y)}$ clearly extends to an object of the \emph{overconvergent} category $\DCon_F(X)\subset \bD^b_{\hol,F}(X)$, and so I can just restrict the trace morphism constructed in \cite[\S1.5]{Abe18a} from the overconvergent to the partially overconvergent category. 

To prove that $f^+\to f^!$ is an isomorphism, \cite[Lemma 1.2.3]{AC18a} shows that I can replace $(X,Y)$ by $(X,X)$, in other words I can work in the convergent category. Thus, after localising on $X$, I can assume that it lifts to a smooth formal scheme $\fX$, and that finite \'etale cover $f\from X'\to X$ lifts to a finite \'etale cover $u\from  \fX'\to \fX$. In this case, it is a straightforward computation that $u^!\circ \bD_{\fX} \cong \bD_{\fX'}\circ u^!$, and the claim follows. 

For the dct-exactness claims, the case of $f^!\cong f^+$ was handled in Proposition \ref{prop: u^! dct exact formal schemes}, and the case of $f_+\cong f_!$ then follows because it is both a left and a right adjoint to $f^!\cong f^+$.
\end{proof}

\begin{remark} \label{rem: D module finite etale direct summand} Since $f^!=f^+$ and $f_!=f_+$, it follows that any $\cM\in \DCon_F(X,Y)$ is a direct summand of $f_+f^+\cM\in \DCon_F(X,Y)$, and any $\cN\in\DCon_F(X',Y')$ is a direct summand of $f^+f_+\cN\in \DCon_F(X',Y')$.
\end{remark}

Needless to say, there are analogues of all of the results in \S\ref{subsec: dc mod on pairs} with the strongly realisable pair $(X,Y)$ replaced by a strongly realisable variety $X$. 

\section{Quasi-coherent complexes and rigidification} \label{sec: rigidification of O-modules}

A key tool in the comparison between $\bD^b_{\cons}$ and $\bD^b_{\hol}$ will be a completed version of the module pullback functor along the specialisation morphism
\[ \sp \from  \fP_K \to \fP \]
for $\fP$ a flat formal scheme. This will only work for complexes which are quasi-coherent in the sense of Berthelot (I will recall the definition below), and the goal in this section is to describe this construction.

\subsection{Quasi-coherent complexes}  \label{subsec: quasi-coherent}

Let $\fr{P}$ be a flat formal scheme, and set $P_n:=\fr{P}\times_\mathcal{V} \mathcal{V}/\fr{m}^{n+1}$. Thus $P_0=P$. I will write ${\bf D}_\mathrm{qc}(\mathcal{O}_{\fr{P}})$ for the derived category of complexes of $\mathcal{O}_\fr{P}$-modules which are \emph{quasi-coherent} in the sense of \cite[\S3.2]{Ber02}. Thus a complex of $\cO_\fP$-modules $\mathcal{M}$ is quasi-coherent iff: 
\begin{enumerate}
\item $ \mathcal{O}_{P_0} \otimes^\mathbf{L}_{\mathcal{O}_{\fr{P}}} \mathcal{M}$ is a quasi-coherent complex of $\mathcal{O}_{P_0}$-modules;
\item the natural map $\mathcal{M} \rightarrow \mathbf{R}\lim{n} \mathcal{O}_{P_n} \otimes^\mathbf{L}_{\mathcal{O}_{\fr{P}}} \mathcal{M}$ is an isomorphism in ${\bf D}(\cO_\fr{P})$.
\end{enumerate}
It follows by induction on $n$ that $\mathcal{M}_n:=\mathcal{O}_{P_n} \otimes^\mathbf{L}_{\mathcal{O}_{\fr{P}}} \mathcal{M}$ is a quasi-coherent complex of $\mathcal{O}_{P_n}$-modules for all $n$. In fact, Berthelot phrases the definition in terms of the topos $\fP_\bullet$ of $\N$-indexed projective systems of sheaves on $\fP$. This is ringed via the projective system $\cO_{P_\bullet}=\{\cO_{P_n}\}_{n\in \N}$, and there is a morphism of ringed toposes
\[ l_\fP \from (\fP,\cO_{P_\bullet}) \to (\fP,\cO_\fP),\]
where the pushforward functor takes the inverse limit, and the (module) pullback functor tensors over $\cO_{\fP}$ with $\cO_{P_\bullet}$. A complex $\cM$ is then quasi-coherent iff $(\bL l^*_{\fP}\cM)_0 \in \bD(\cO_{P_0})$ is quasi-coherent, and the natural map $\cM\to \bR l_{\fP*}\bL l_{\fP}^*\cM$ is an isomorphism.

\begin{example} Any (possibly unbounded) complex with coherent cohomology sheaves is quasi-coherent. Thus ${\bf D}_\mathrm{coh}(\cO_\fr{P})\subset {\bf D}_\mathrm{qc}(\cO_\fr{P})$ as a full subcategory.
\end{example}

\begin{remark} If $\ca{A}$ is an $\cO_\fP$-algebra, a complex of $\ca{A}$-modules will be called quasi-coherent if it is so as a complex of $\cO_\fP$-modules. The (derived) category of quasi-coherent complexes of $\ca{A}$-modules will be denoted ${\bf D}_\mathrm{qc}(\ca{A})$, and is viewed as a full subcategory of $\bD(\ca{A})$. Note that I do not assume that $\ca{A}$ is itself quasi-coherent as a complex of $\cO_\fP$-modules. For example, if $\fP$ is smooth, I will later want to take $\ca{A}=\widehat{\sD}^{(m)}_\fP$ and $\ca{A}=\sD^{(m)}_{\fP}$, although of course in this case the forgetful functor ${\bf D}_\mathrm{qc}(\widehat{\sD}^{(m)}_\fP)\to {\bf D}_\mathrm{qc}(\sD^{(m)}_\fP)$ is an equivalence.
\end{remark}

One important way of constructing quasi-coherent complexes is the following lemma.

\begin{lemma} \label{lemma: inv qc} Suppose that $\left\{\mathcal{M}_n\right\}_{n\in \N}$ is an inverse system of complexes of $\mathcal{O}_\fr{P}$-modules, such that:
\begin{enumerate}
\item each $\mathcal{M}_n$ is quasi-coherent complex of $\mathcal{O}_{P_n}$-modules;
\item \label{num: inv qc 2} for each $n$, the induced map
\[ \mathcal{O}_{P_n} \otimes^\mathbf{L}_{\mathcal{O}_{P_{n+1}}} \mathcal{M}_{n+1}\rightarrow \mathcal{M}_n  \]
is an isomorphism in $\bD(\cO_{P_n})$.
\end{enumerate}
Then $\mathcal{M}:=\mathbf{R}\lim{n} \mathcal{M}_n\in {\bf D}(\mathcal{O}_\fr{P})$ is quasi-coherent. 
\end{lemma}

\begin{proof}
I need to show that the map \[ \mathcal{O}_{P_n} \otimes^\mathbf{L}_{\mathcal{O}_{\fr{P}}} \mathcal{M} \rightarrow \mathcal{M}_n \]
is an isomorphism in ${\bf D}(\cO_{P_n}
)$. By tensoring both sides over $\mathcal{O}_{P_n}$ with the exact sequence
\[ 0 \rightarrow \mathcal{O}_{P_{n-1}} \overset{\times \varpi}{\longrightarrow} \mathcal{O}_{P_n} \rightarrow \mathcal{O}_{P_0} \rightarrow 0 \]
and using condition (2), I can argue by induction on $n$ to reduce to the case $n=0$. Since $\mathcal{O}_{P_0}$ is a perfect complex of $\mathcal{O}_{\fr{P}}$-modules, I can then calculate the LHS as
\begin{align*}   \mathcal{O}_{P_0} \otimes^\mathbf{L}_{\mathcal{O}_{\fr{P}}}\mathbf{R}\lim{n} \mathcal{M}_n 
&\cong \mathbf{R}\lim{n} \left( \mathcal{O}_{P_0} \otimes^\mathbf{L}_{\mathcal{O}_{\fr{P}}} \mathcal{M}_n \right)  \\
&\cong \mathbf{R}\lim{n} \left( \mathcal{O}_{P_0} \otimes^\mathbf{L}_{\mathcal{O}_{\fr{P}}} \mathcal{O}_{P_n} \otimes^\mathbf{L}_{\mathcal{O}_{P_n}}  \mathcal{M}_n \right) \\ &\cong \mathbf{R}\lim{n} \left( \mathcal{O}_{P_n}\otimes^\mathbf{L}_{\mathcal{O}_{\fr{P}}} \mathcal{M}_0 \right) . 
\end{align*}
Now each complex $\mathcal{O}_{P_n}\otimes^\mathbf{L}_{\mathcal{O}_{\fr{P}}} \mathcal{M}_0$ is quasi-isomorphic to the mapping cone of
\[ \mathcal{M}_0 \overset{0}{\rightarrow} \mathcal{M}_0,\]
and moreover the transition maps $\mathcal{O}_{P_{n+1}}\otimes^\mathbf{L}_{\mathcal{O}_{\fr{P}}} \mathcal{M}_0\rightarrow \mathcal{O}_{P_n}\otimes^\mathbf{L}_{\mathcal{O}_{\fr{P}}} \mathcal{M}_0$ are realised by the commutative diagram
\[ \xymatrix{ \mathcal{M}_0 \ar[r]^0\ar[d]_0 & \mathcal{M}_0 \ar[d]^{\mathrm{id}} \\ \mathcal{M}_0 \ar[r]^0 & \mathcal{M}_0.  }  \]
It therefore follows that
\[ \mathbf{R}\lim{n} \left( \mathcal{O}_{P_n}\otimes^\mathbf{L}_{\mathcal{O}_{\fr{P}}} \mathcal{M}_0 \right) \cong \mathcal{M}_0\]
as required.   
\end{proof}

Mapping complexes between quasi-coherent complexes have the following straightforward description.

\begin{lemma} \label{lemma: rhom} Suppose that $\mathcal{M},\mathcal{N}\in \bD_\mathrm{qc}(\mathcal{O}_\fr{P})$. Then the natural map
\[ \mathbf{R}\mathrm{Hom}_{\mathcal{O}_{\fr{P}}}(\mathcal{M},\mathcal{N}) \rightarrow \mathbf{R}\lim{n}\mathbf{R}\mathrm{Hom}_{\mathcal{O}_{P_n}}(\mathcal{M}_n,\mathcal{N}_n)   \]
is an isomorphism. 
\end{lemma}

\begin{proof}
This is a simple calculation:
\begin{align*} 
 \mathbf{R}\mathrm{Hom}_{\mathcal{O}_{\fr{P}}}(\mathcal{M},\mathcal{N}) &= \mathbf{R}\mathrm{Hom}_{\mathcal{O}_{\fr{P}}}(\mathcal{M},\mathbf{R}\lim{n}\mathcal{N}_n) \\
 &= \mathbf{R}\lim{n}\mathbf{R}\mathrm{Hom}_{\mathcal{O}_{\fr{P}}}(\mathcal{M},\mathcal{N}_n) \\
  &= \mathbf{R}\lim{n}\mathbf{R}\mathrm{Hom}_{\mathcal{O}_{P_n}}(\mathcal{O}_{P_n}\otimes^\mathbf{L}_{\mathcal{O}_{\fr{P}}}\mathcal{M},\mathcal{N}_n) \\
  &= \mathbf{R}\lim{n}\mathbf{R}\mathrm{Hom}_{\mathcal{O}_{P_n}}(\mathcal{M}_n,\mathcal{N}_n). \qedhere
\end{align*}
\end{proof}

Following \cite[\S3.4]{Ber02}, I can define a completed tensor product
\begin{align*}
 -\widehat{\otimes}^\mathbf{L}_{\mathcal{O}_\fr{P}}- &:\bD_\mathrm{qc}(\mathcal{O}_{\fr{P}}) \times \bD_\mathrm{qc}(\mathcal{O}_{\fr{P}}) \rightarrow \bD_\mathrm{qc}(\mathcal{O}_{\fr{P}}) \\
  \mathcal{M} \widehat{\otimes}_{\mathcal{O}_\fr{P}} \mathcal{N} &:= \mathbf{R}\lim{n} \left(  \mathcal{M}_n \otimes^\mathbf{L}_{\mathcal{O}_{P_n}} \mathcal{N}_n \right),
\end{align*} 
the result is a quasi-coherent complex by Lemma \ref{lemma: inv qc}. Formally, the definition is
\[ \cM \widehat{\otimes}^\mathbf{L}_{\mathcal{O}_\fr{P}} \cN := \bR l_{\fP*}(\bL l_\fP^* \cM \otimes^{\bL}_{\cO_{P_\bullet}} \bL l_\fP^*\cN). \]
Similarly, if $\pi:\fr{P}'\rightarrow \fr{P}$ is a morphism of flat formal schemes, with induced maps $\pi_n:P'_n\rightarrow P_n$ for each $n$, there is a functor
\begin{align*}  \mathbf{L}\hat{\pi}^*&: \bD_\mathrm{qc}(\mathcal{O}_{\fr{P}}) \rightarrow \bD_\mathrm{qc}(\mathcal{O}_{\fr{P}'}) \\
\mathbf{L}\hat{\pi}^*\mathcal{M}&:= \mathbf{R}\lim{n}\mathbf{L}\pi_n^* \mathcal{M}_n,
\end{align*}
again, it follows from Lemma \ref{lemma: inv qc} that this is indeed a quasi-coherent complex. The formal definition is given by extending $\pi$ to a morphism
\[ \pi_\bullet \from (\fP',\cO_{P_\bullet'}) \to (\fP,\cO_{P_\bullet}) \]
and then defining
\[ \bL\hat{\pi}^*:=\bR l_{\fP'*} \circ \bL\pi^*_\bullet\circ \bL l_{\fP}^*. \]
Note that
\[ \mathbf{L}\hat{\pi}^*\mathcal{M} \widehat{\otimes}^{\mathbf{L}}_{\mathcal{O}_{\fr{P}'}} \mathbf{L}\hat{\pi}^*\mathcal{N}\isomto \mathbf{L}\hat{\pi}^*(\mathcal{M} \widehat{\otimes}^{\mathbf{L}}_{\mathcal{O}_\fr{P}} \mathcal{N}). \]
Moreover, letting $\mathbf{L}\pi^*$ denote abstract module pullback along $\pi$, then the natural maps 
\[ \mathbf{L}\pi^*\mathcal{M} \rightarrow \mathbf{L}\pi^*_n\mathcal{M}_n\]
for $n\geq 0$ induce a map
\[ \mathbf{L}\pi^*\mathcal{M} \rightarrow \mathbf{L}\hat{\pi}^*\mathcal{M}\]
in ${\bf D}(\mathcal{O}_{\fr{P}'})$. Of course, $\bL\pi^*$ won't preserve quasi-coherence in general.

\begin{lemma}\label{lemma: qc adj}  The functor $\mathbf{R}\pi_*:{\bf D}(\mathcal{O}_{\fr{P}}) \rightarrow {\bf D}(\mathcal{O}_{\fr{P}'})$ preserves quasi-coherence, and 
\[ \mathbf{L}\hat{\pi}^*: {\bf D}_\mathrm{qc}(\mathcal{O}_{\fr{P}}) \leftrightarrows {\bf D}_\mathrm{qc}(\mathcal{O}_{\fr{P}'}): \mathbf{R}\pi_*  \]
form an adjoint pair.
\end{lemma} 

\begin{proof}
Suppose that $\mathcal{N}\in {\bf D}_\mathrm{qc}(\mathcal{O}_{\fr{P}'})$. Then
\begin{align*}
 \mathbf{R}\pi_*\mathcal{N}  &\cong  \mathbf{R}\pi_* \mathbf{R}\lim{n} \mathcal{N}_n \\
 &\cong   \mathbf{R}\lim{n} \mathbf{R}\pi_{n*}\mathcal{N}_n.
\end{align*} 
Now, since each $\mathcal{N}_n$ is quasi-coherent as an $\mathcal{O}_{P'_n}$-module, each $\mathbf{R}\pi_{n*}\mathcal{N}_n$ is quasi-coherent as an  $\mathcal{O}_{P_n}$-module. Since $\mathbf{R}\pi_{n*}$ has finite cohomological dimension, the base change formula
\[  \mathcal{O}_{P_n}\otimes^\mathbf{L}_{\mathcal{O}_{P_{n+1}}}\mathbf{R}\pi_{n+1*} \mathcal{N}_{n+1} \isomto \mathbf{R}\pi_{n*} \mathcal{N}_n  \]
can be proved by reducing to the case of bounded complexes and applying \cite[IV, Proposition 3.1.0]{SGA6}. Thus I can apply Lemma \ref{lemma: inv qc} to deduce that $\mathbf{R}\pi_*\mathcal{N}$ is quasi-coherent. Now, thanks to Lemma \ref{lemma: rhom}, the chain of identifications
\begin{align*}
\mathbf{R}\mathrm{Hom}_{\mathcal{O}_{\fr{P}'}}(\mathbf{L}\hat{\pi}^*\mathcal{M},\mathcal{N}) &= \mathbf{R}\lim{n}\mathbf{R}\mathrm{Hom}_{\mathcal{O}_{P'_n}}(\mathbf{L}\pi_n^*\mathcal{M}_n,\mathcal{N}_n)  \\
&=\mathbf{R}\lim{n}\mathbf{R}\mathrm{Hom}_{\mathcal{O}_{P_n}}(\mathcal{M}_n,\mathbf{R}\pi_{n*}\mathcal{N}_n) \\
 &= \mathbf{R}\mathrm{Hom}_{\mathcal{O}_{\fr{P}}}(\mathcal{M},\mathbf{R}\pi_*\mathcal{N})
\end{align*} 
shows that $(\bL\hat{\pi}^*,\bR\pi_*)$ do indeed form an adjoint pair as claimed.
\end{proof}

For open immersions, the abstract module pullback is already complete.

\begin{lemma} \label{lemma: open} Let $j:\fr{U}\rightarrow \fr{P}$ be an open immersion of flat formal schemes. Then  the natural morphism of functors $j^{-1}\rightarrow \mathbf{L}\hat{j}^*$ is an
isomorphism.
\end{lemma}

\begin{proof}
First note that $j^{-1}$ has a left adjoint $j_!$, and thus commutes with limits. The fact that $j_!$ is exact means that $j^{-1}$ moreover commutes with derived limits, since $j^{-1}$ preserves $K$-injective complexes in the category of inverse systems. I then simply compute
\[ j^{-1}\mathcal{M} \isomto j^{-1}\mathbf{R}\lim{n} \mathcal{M}_n = \mathbf{R}\lim{n} j_n^{-1}\mathcal{M}_n \isomto \mathbf{R}\lim{n} \mathbf{L}j_n^*\mathcal{M}_n = \mathbf{L}\hat{j}^*\mathcal{M} \]
as required. 
\end{proof} 

\begin{remark} \label{rem: qc local} The observation that $j^{-1}$ commutes with derived limits (together with the analogous assertion that $j^{-1}$ commutes with derived tensor products) implies that quasi-coherence of a complex $\mathcal{M}\in {\bf D}(\cO_\fr{P})$ can be checked locally on $\fr{P}$. 
\end{remark}

Another situation in which there is no need to complete is the following. 

\begin{lemma} \label{lemma: finite morphisms derived complete} Let $\pi\from \fP'\to \fP$ be a finite morphism of flat formal schemes, which is an isomorphism on the underlying topological spaces. Then the morphism
$\bL\pi^*\to \bL\hat{\pi}^*$ of functors $\bD(\cO_\fP)\to \bD(\cO_{\fP'})$ is an isomorphism.
\end{lemma}

\begin{remark} It seems reasonable to suppose that the lemma holds for more general finite morphisms, but I will only need this special case.
\end{remark}

\begin{proof}
The lemma amounts to showing that, for any $\cM\in \bD_\qc(\cO_\fP)$, the natural morphism
\[ \cO_{\fP'} \otimes^{\bL}_{\cO_{\fP}} \bR\lim{n} \cM_n \rightarrow \bR\lim{n} (\cO_{P_n'} \otimes^{\bL}_{\cO_{P_n}}\cM_n ) \]
is an isomorphism. This can be checked locally on $\fP$, so I can assume that both $\fP$ and $\fP'$ are affine. In this case, $\cO_{\fP'}$ admits a (possibly infinite, but bounded above) resolution by finite free $\cO_{\fP}$-modules. By truncating, and using the fact that $\bR\lim{n}$ has finite cohomological dimension (at least for systems $\{\cM_n\}_{n\in \N}$ with each $\cM_n$ quasi-coherent) I can therefore replace $\cO_{\fP'}$ by $\cO_\fP^{\oplus m}$, and each $\cO_{P'_n}$ by $\cO_{P_n}^{\oplus m}$, in which case the claim is clear.
\end{proof}

Quasi-coherent complexes satisfy the following version of the projection formula.

\begin{lemma} \label{lemma: proj} Let $\pi:\fr{P}'\rightarrow \fr{P}$ be a morphism of flat formal schemes. Then for any $\mathcal{M}\in \bD_\mathrm{qc}(\mathcal{O}_{\fr{P}})$ the map
\[ \mathcal{M} \widehat{\otimes}^\mathbf{L}_{\mathcal{O}_{\fr{P}}} \mathbf{R}\pi_* \mathcal{O}_{\fr{P}'} \rightarrow \mathbf{R}\pi_*\mathbf{L}\hat{\pi}^*\mathcal{M} \]
is an isomorphism.
\end{lemma}

\begin{proof}
It suffices to show that the map
\[ \mathcal{M}_n \otimes^\mathbf{L}_{\mathcal{O}_{P_n}} \mathbf{R}\pi_{n*} \mathcal{O}_{P_n'} \rightarrow \mathbf{R}\pi_{n*}\mathbf{L}\pi_n^*\mathcal{M}_n  \]
is an isomorphism, for each $n$. This question is local on $P_n$, which I can therefore assume to be affine. Now, the standard construction of $K$-flat resolutions gives a resolution of $\cM_n$ with terms direct sums of sheaves of the form $j_!\cO_{U}$ for an open subscheme $j\from U\to P_n$. But since $\cM_n$ is quasi-coherent, and $P_n$ is affine, I can actually find a $K$-flat resolution whose terms are all free $\cO_{P_n}$-modules. Now using the fact that $\bR\pi_{n*}$ has finite cohomological dimension and commutes with filtered colimits, I can reduce to the tautological case $\cM_n=\cO_{P_n}$.
\end{proof}

I will let ${\bf D}_{\mathrm{qc},\Q}(\cO_\fr{P})$ denote the isogeny category of quasi-coherent complexes, and ${\bf D}_{\mathrm{coh},\Q}^b(\cO_\fr{P})$ its full subcategory spanned by bounded, coherent complexes. There is a functor
\[ {\bf D}_{\mathrm{qc},\Q}(\cO_\fr{P}) \rightarrow {\bf D}(\cO_{\fr{P}\Q}) \]
defined on objects by $\cM\mapsto \cM_{\Q}$. In general, this need not be fully faithful, but it will be after restricting to ${\bf D}_{\mathrm{coh},\Q}^b(\cO_\fr{P})$.

\begin{lemma}
 Let $\fr{P}$ be a flat formal scheme, and $\cM,\cN\in \mathbf{Coh}(\cO_\fr{P})$ coherent $\cO_\fr{P}$-modules. Then the natural map
\[ \mathrm{Ext}^q_{\cO_\fr{P}}(\cM,\cN) \otimes_{\Z}\Q \rightarrow  \mathrm{Ext}^q_{\cO_{\fr{P}\Q}}(\cM_{\Q},\cN_{\Q}) \]
is an isomorphism, for all $q\geq 0$.
\end{lemma}

\begin{proof}
Since $\fP$ is quasi-compact, the question is local on $\fP$, which I may therefore assume to be affine. In this case, $\cM$ admits a resolution (possibly unbounded below) by finite free $\cO_\fr{P}$-modules. This reduces to the case $\cM=\cO_{\fP}$, which is clear.
\end{proof}

\begin{corollary} \label{cor: rhom tensor Q} Let $\fr{P}$ be a flat formal scheme. Then the functor
\[ {\bf D}^b_{\Q,\mathrm{coh}}(\cO_\fr{P})\rightarrow {\bf D}^b(\cO_{\fr{P}\Q}), \]
defined on objects by $\cM\mapsto \cM_{\Q}$, is fully faithful.
\end{corollary}

I can now deduce an invariance result for quasi-coherent complexes under admissible blowups.

\begin{proposition} \label{prop: adm bl ff} Let $\pi:\fr{P}'\rightarrow \fr{P}$ be an admissible blowup of flat formal schemes. Then, for any $\mathcal{M}\in {\bf D}_\mathrm{qc}(\mathcal{O}_{\fr{P}})$, the map
\[  \mathcal{M} \rightarrow \mathbf{R}\pi_*\mathbf{L}\hat{\pi}^*\mathcal{M} \]
is an isogeny. That is, it becomes invertible in ${\bf D}_{\mathrm{qc},\Q}(\mathcal{O}_{\fr{P}})$. 
\end{proposition}

\begin{proof}
Thanks to the projection formula (Lemma \ref{lemma: proj} above), the given map can be identified with
\[  \mathcal{M} \rightarrow \mathcal{M} \widehat{\otimes}^\mathbf{L}_{\mathcal{O}_{\fr{P}}} \mathbf{R}\pi_*\mathcal{O}_{\fr{P}'}. \]
Since isogenies are preserved by applying the functor $\mathcal{M}\widehat{\otimes}^\mathbf{L}_{\mathcal{O}_{\fr{P}}}-$, it suffices to treat the case $\mathcal{M}=\cO_\fr{P}$. In this case, after tensoring with $\Q$,
\begin{align*}
 \cO_{\fr{P}\Q}&=\mathbf{R}\mathrm{sp}_{\fr{P}*}\cO_{\fr{P}_K} \\
  \mathbf{R}\pi_*\cO_{\fr{P}'\Q} &=\mathbf{R}\pi_*\mathbf{R}\mathrm{sp}_{\fr{P}'}\cO_{\fr{P}'_K}=\mathbf{R}\mathrm{sp}_{\fr{P}*}\mathbf{R}\pi_{*}\cO_{\fr{P}'_K}.
\end{align*}
Since $\pi:\fr{P}'_K\rightarrow \fr{P}_K$ is an isomorphism, it follows that 
\[ \mathcal{O}_\fr{P} \rightarrow \mathbf{R}\pi_*\mathcal{O}_{\fr{P}'} \]
is a morphism of bounded complexes on $\fr{P}$, with coherent cohomology sheaves, which becomes a quasi-isomorphism after applying $-\otimes_{\Z}\Q$. It therefore follows from Corollary \ref{cor: rhom tensor Q} that it is an isogeny.
\end{proof}

%
%

I will also need to consider inductive systems of quasi-coherent complexes, the formalism of which works exactly as in \cite[\S4.2]{Ber02}. Thus $\fP^{(\bullet)}$ (resp. $\fP^{(\bullet)}_\bullet$) will denotes the topos of $\N$-indexed inductive systems of sheaves on $\fP$ (resp. on $\fP_\bullet$), as in \S\ref{subsec: diagram sheaves}, which is ringed via the constant ind-object $\cO_{\fP}$ (resp.\ $\cO_{P_\bullet}$). Berthelot then defines a double localisation $\underrightarrow{{\bf LD}}_{\Q}(\mathcal{O}_{\fr{P}})$ of the derived category of $\cO_{\fP}$-modules on $\fP^{(\bullet)}$. Roughly speaking this corresponds to tensoring with $\Q$ and then taking the colimit over the inductive system. If I need to emphasize the fact that I am considering categories of inductive systems of complexes, rather than just complexes, I will write the rings on $\fP^{(\bullet)}$ and $\fP_\bullet^{(\bullet)}$ as $\cO_{\fP^{(\bullet)}}$ and $\cO_{P^{(\bullet)}_\bullet}$ respectively, thus $\underrightarrow{{\bf LD}}_{\Q}(\mathcal{O}_{\fr{P}})$ is a localisation of the category $\bD(\cO_{\fP^{(\bullet)}})$.

As in \cite[\S4.2]{Ber02}, I will denote by $\underrightarrow{{\bf LD}}_{\Q,\mathrm{qc}}(\mathcal{O}_{\fr{P}})\subset \LD_{\Q}(\cO_\fP)$ the full subcategory on objects which are levelwise quasi-coherent. Exactly as above, any morphism $\pi\from \fP'\to \fP$ of flat formal schemes gives rise to a commutative square
\[ \xymatrix{  (\fP'^{(\bullet)}_{\bullet},\cO_{P'_{\bullet}})  \ar[r]^-{l_{\fP'^{(\bullet)}}}\ar[d]_{\pi^{(\bullet)}_\bullet} & (\fP'^{(\bullet)},\cO_{\fP'}) \ar[d]^{\pi^{(\bullet)}}  \\
(\fP^{(\bullet)}_{\bullet},\cO_{P_{\bullet}}) \ar[r]^-{l_{\fP^{(\bullet)}}} & (\fP^{(\bullet)},\cO_{\fP})     } \]
of ringed toposes, and $\bR l_{\fP'^{(\bullet)}*} \circ \bL\pi^{(\bullet),*}_\bullet\circ \bL l^{*}_{\fP^{(\bullet)}}$ descends to a functor
\[ \bL\hat{\pi}^*\from \underrightarrow{{\bf LD}}_{\Q,\mathrm{qc}}(\mathcal{O}_{\fr{P}}) \to \underrightarrow{{\bf LD}}_{\Q,\mathrm{qc}}(\mathcal{O}_{\fr{P}'})\]
on the localised categories. Informally, this just applies $\bL\hat{\sp}^*$ levelwise, and then passes to the localisation. There is of course a similar definition of the completed tensor product 
\[ -\widehat{\otimes}^\mathbf{L}_{\mathcal{O}_{\fr{P}}}- \colon \underrightarrow{{\bf LD}}_{\Q,\mathrm{qc}}(\mathcal{O}_{\fr{P}}) \times\underrightarrow{{\bf LD}}_{\Q,\mathrm{qc}}(\mathcal{O}_{\fr{P}}) \rightarrow \underrightarrow{{\bf LD}}_{\Q,\mathrm{qc}} (\mathcal{O}_{\fr{P}}). \]
Note that the formula
\[ \mathbf{L}\hat{\pi}^*\mathcal{M} \widehat{\otimes}^{\mathbf{L}}_{\mathcal{O}_{\fr{P}'}} \mathbf{L}\hat{\pi}^*\mathcal{N}\isomto \mathbf{L}\hat{\pi}^*(\mathcal{M} \widehat{\otimes}^{\mathbf{L}}_{\mathcal{O}_{\fr{P}}} \mathcal{N}) \]
still holds for objects of $\underrightarrow{{\bf LD}}_{\Q,\mathrm{qc}}$. 

\subsection{The rigidification functor} \label{subsec: rigidification O-mod}

The categories $\bD_{\Q,\qc}$ and $\LD_{\Q,\qc}$ are now the natural source of the functor of completed pullback along the specialisation map. Roughly speaking, for any flat formal scheme $\fr{P}$, I define
\begin{align*}
 \mathbf{L}\hat{\mathrm{sp}}^*=\mathbf{L}\hat{\mathrm{sp}}_{\fr{P}}^*&:{\bf D}_\mathrm{qc}(\mathcal{O}_{\fr{P}}) \rightarrow {\bf D}(\mathcal{O}^+_{\fr{P}_K})  \\
 \mathbf{L}\hat{\mathrm{sp}}_{\fr{P}}^*\mathcal{M}&:= \colim{\pi:\fr{P}'\rightarrow \fr{P}} \mathrm{sp}_{\fr{P}'}^{-1}\mathbf{L}\hat{\pi}^*\mathcal{M}
\end{align*}
where the colimit is over all admissible blowups $\pi:\fr{P}'\rightarrow \fr{P}$. Tensoring with $\Q$ and then passing to the colimit gives a functor
\begin{equation}
\label{eqn: Lsp for O-modules}
 \mathbf{L}\hat{\mathrm{sp}}^*:\underrightarrow{{\bf LD}}_{\Q,\mathrm{qc}}(\mathcal{O}_{\fr{P}}) \rightarrow {\bf D}(\mathcal{O}_{\fr{P}_K}).
\end{equation}
As above, the formal definition of $\bL\hat{\sp}^*$ requires the use of sheaves on diagrams of spaces. I first let $I_\fP$ denote the category of admissible blowups of $\fP$, and
\[ \ca{B}_\fP \from I_\fP\to {\bf FSch}\]
the tautological diagram in the category of formal schemes. I then consider:
\begin{itemize}
\item the topos ${\ca{B}}^{(\bullet)}_\fP$ of $\N$-indexed inductive systems of sheaves on the diagram $\ca{B}_\fP$, ringed via the sheaf $\cO_{{\ca{B}}_\fP}$ whose restriction to each $\fP'\in I_\fP$ is the constant inductive system $\cO_{\fP'}$;
\item the topos ${\ca{B}}^{(\bullet)}_{\fP_\bullet}$ of $\N$-indexed inductive systems of $\N$-indexed projective systems of sheaves on $\ca{B}_\fP$, ringed via the sheaf $\cO_{{\ca{B}}_{P_\bullet}}$ whose restriction to each $\fP'\in I_\fP$ is the constant inductive system $\cO_{P'_\bullet}$.
\end{itemize}
As in the case of a single formal scheme $\fP$, if I need to emphasize the fact that I am considering categories of inductive systems, I will sometimes write these two rings as $\cO_{{\ca{B}}^{(\bullet)}_\fP}$ and $\cO_{{\ca{B}}^{(\bullet)}_{P_\bullet}}$ respectively, for example in considering the derived category $\bD(\cO_{{\ca{B}}^{(\bullet)}_\fP})$ (and in particular to distinguish it from $\bD(\cO_{{\ca{B}}_\fP})$). There are then natural morphisms
\begin{align*}
l_{\ca{B}^{(\bullet)}_{\fP}}&\from (\ca{B}^{(\bullet)}_{\fP_\bullet},\cO_{\ca{B}_{P_\bullet}}) \to (\ca{B}^{(\bullet)}_{\fP},\cO_{\ca{B}_{\fP}}) \\
\pi & \from ({\ca{B}}_{\fP},\cO_{{\ca{B}}_{\fP}}) \to (\fP,\cO_\fP) \\
\pi^{(\bullet)}_\bullet &\from (\ca{B}^{(\bullet)}_{\fP_\bullet},\cO_{\ca{B}_{P_\bullet}}) \to (\fP_\bullet^{(\bullet)},\cO_{P_\bullet}).
\end{align*}
Just as in the case of a single admissible blowup $\fP'\to \fP$, I can then define the completed pullback functor
\[ \bL\hat{\pi}^*\from \bD(\cO_{\fP^{(\bullet)}})\to  \bD(\cO_{{\ca{B}}^{(\bullet)}_\fP})  \]
as the composite $\bR l_{\ca{B}^{(\bullet)}_{\fP}*} \circ \bL\pi^{(\bullet)*}_\bullet \circ \bL l^*_{\fP^{(\bullet)}}$. Finally, there is a morphism
 \[
\sp_{{\ca{B}}^{(\bullet)}_\fP} \from (\fP_K,\cO_{\fP_K}) \to ({\ca{B}}^{(\bullet)}_\fP,\cO_{{\ca{B}}_\fP})
\]
whose inverse image functor takes the colimit over both $I_\fP^{\rm op}$ and $\N$. The functor (\ref{eqn: Lsp for O-modules}) I am after is then defined by (easily!) checking that the composite functor
\begin{align*}
\bL\hat{\sp}^*=\bL\hat{\sp}^*_\fP &\from \bD(\cO_{\fP^{(\bullet)}}) \to \bD(\cO_{\fP_K}) \\
\bL\hat{\sp}^*&:=\bL \sp_{{\ca{B}}^{(\bullet)}_\fP}^*\circ \bL\hat{\pi}^*
\end{align*} 
descends to the quotient $\underrightarrow{{\bf LD}}_{\Q,\qc}(\mathcal{O}_{\fr{P}})$ of $\bD_{\qc}(\cO_{\fP^{(\bullet)}})$. Note that the functor $\bL \sp_{{\ca{B}}^{(\bullet)}_\fP}^*$ appearing in the definition applies $\sp_{\fP'}^{-1}$ on each admissible blowup $\fP'\to \fP$, then takes the colimit over both $I_\fP^{\rm op}$ and $\N$, and then finally tensors with $\Q$. In particular, all of the composite factors of $\bL \sp_{{\ca{B}}^{(\bullet)}_\fP}^*$ are exact and derive trivially. The natural maps
\[ \bL\pi^* \to \bL\hat{\pi}^* \]
defined for each $\pi\from \fP'\to \fP$ fit together to give a morphism of functors
\begin{equation}
\label{eqn: sp^* to completed sp^*} \bL\sp^* \to \bL\hat{\sp}^* 
\end{equation}
with source the abstract module pullback along $\sp\from (\fP_K,\cO_{\fP_K})\to (\fP,\cO_\fP)$.  

\begin{example} If $\cM$ is a coherent $\cO_\fP$-module such that $\cM_{\Q}$ is a locally projective $\cO_{\fP\Q}$-module, then the morphism (\ref{eqn: sp^* to completed sp^*}) induces an isomorphism
\[ \sp^*\cM \isomto \bL\hat{\sp}^*\cM \]
in $\bD(\cO_{\fP_K})$. 
\end{example}

\subsection{Functoriality of $\bm{\bL\hat{\sp}^*_\fP}$}

Now suppose that $u\from\fP\to\fQ$ is a morphism of flat formal schemes over $\cV$. Then taking the strict transform induces a functor
\[ I_\fQ\to I_\fP, \]
and it is straightforward to check that the hypotheses of Lemma \ref{lemma: morphism of sites} are satisfied. This therefore gives rise to a morphism of (ringed) sites
\[ u \from  \ca{B}_{\fP} \to \ca{B}_{\fQ}. \]
Of course, there are analogous morphisms for the sites of inductive and projective systems of sheaves on $\fP$ and $\fQ$. The commutative squares
\[ \xymatrix{ \fP_K \ar[r]^-{\sp}\ar[d]^-{u} &  \ca{B}^{(\bullet)}_{\fP} \ar[r]^-{\pi}\ar[d]^-{u} & \fP \ar[d]^-{u} \\ \fQ_K \ar[r]^{\sp} & \ca{B}^{(\bullet)}_{\fQ}  \ar[r]^-{\pi} & \fQ } \]
now give rise to a base change morphism
\begin{equation}
\label{eqn: base change sp^* u_*} \bL\hat{\sp}_\fQ^*\circ \bR u_* \to \bR u_* \circ\bL\hat{\sp}^*_\fP
\end{equation}
of functors
\[ \LD_{\Q,\qc}(\cO_{\fP}) \to \bD(\cO_{\fQ_K}). \]
When $\fQ=\spf{\cV}$ is a point, $u\from \fP\to \spf{\cV}$ is the structure morphism, and $\cM\in\LD_{\Q,\qc}(\cO_{\fP})$, I then define
\[ \bR\Gamma(\fP,\cM):= \bL\hat{\sp}^*_{\spf{\cV}}\bR u_*\cM \in \bD(K). \]
Concretely, if $\cM=\{\cM^{(m)}\}_{m\in\N}\in \LD_{\Q,\qc}(\cO_{\fP})$, then
\[ \bR\Gamma(\fP,\cM)= \mathbf{R}\Gamma(\fr{P},\underset{m}{\mathrm{colim}}\,\mathcal{M}^{(m)}\otimes_{\Z}\Q)= \underset{m}{\mathrm{colim}}\,\mathbf{R}\Gamma(\fr{P},\mathcal{M}^{(m)})\otimes_{\Z}\Q. \]
Before proving that (\ref{eqn: base change sp^* u_*}) is an isomorphism, I need to show that $\bL\hat{\sp}^*$ is compatible with open immersions. 

\begin{lemma} \label{lemma: Lsp^* and open immersions} Let $j:\fr{U}\rightarrow \fr{P}$ denote an open immersion of flat formal schemes. Then the natural map
\[ j^{-1}\mathbf{L}\hat{\mathrm{sp}}_{\fr{P}}^* \rightarrow \mathbf{L}\hat{\mathrm{sp}}_{\fr{U}}^*j^{-1} \]
is an isomorphism of functors $\LD_{\Q,\qc}(\cO_{\fP})\to \bD(\cO_{\fr{U}_K}^+)$.
\end{lemma}

\begin{proof}
It suffices to prove the analogous statement with $\underrightarrow{{\bf LD}}_{\Q,\mathrm{qc}}(\mathcal{O}_{\fr{P}})$ replaced by ${\bf D}_{\mathrm{qc}}(\mathcal{O}_{\fr{P}})$. The functor $\fr{P}'\mapsto \fr{P}'\times_\fr{P}\fr{U}$ from admissible blowups of $\fr{P}$ to those of $\fr{U}$ is cofinal, so I can compute
\begin{align*}
\mathbf{L}\hat{\mathrm{sp}}_\fr{U}^*\mathbf{L}\hat{j}^*\mathcal{M} &= \underset{\pi\colon \fr{U}'\rightarrow \fr{U}}{\mathrm{colim}}\, \mathrm{sp}_{\fr{U}'}^{-1}\mathbf{L}\hat{\pi}^*\mathbf{L}\hat{j}^*\mathcal{M} \\
&= \underset{\pi\colon \fr{U}'\rightarrow \fr{U}}{\mathrm{colim}}\, \mathrm{sp}_{\fr{U}'}^{-1}j^{-1}\mathbf{L}\hat{\pi}^*\mathcal{M} \\
&=\underset{\pi\colon \fr{P}'\rightarrow \fr{P}}{\mathrm{colim}}\, \mathrm{sp}_{\pi^{-1}(\fr{U})}^{-1}j^{-1}\mathbf{L}\hat{\pi}^*\mathcal{M}  \\
&=j^{-1}\underset{\pi\colon \fr{P}'\rightarrow \fr{P}}{\mathrm{colim}}\,\mathrm{sp}_{\fr{P}'}^{-1}\mathbf{L}\hat{\pi}^*\mathcal{M} \\
&=j^{-1}\mathbf{L}\hat{\mathrm{sp}}^*_{\fr{P}}\mathcal{M}
\end{align*}
using Lemma \ref{lemma: open} for the second (canonical) isomorphism.
\end{proof}

\begin{theorem} \label{theo: base change sp^* u_*}
Let $u\from \fP\to \fQ$ be a flat morphism of flat formal schemes. Then
\[\bL\hat{\sp}_\fQ^*\circ \bR u_* \to \bR u_* \circ\bL\hat{\sp}^*_\fP \]
is an isomorphism of functors $\LD_{\Q,\qc}(\cO_{\fP}) \to \bD(\cO_{\fQ_K})$.
\end{theorem}

\begin{proof}
Since both $\fr{P}$ and $\fr{P}_K$ are quasi-compact, cohomology commutes with filtered colimits. It therefore suffices to prove the analogous statement with $\underrightarrow{{\bf LD}}_{\Q,\mathrm{qc}}(\cO_{\fP})$ replaced by ${\bf D}_{\mathrm{qc}}(\mathcal{O}_{\fr{P}})$.

First suppose that $\fQ=\spf{\cV}$ is a point. Then the base change morphism amounts to a morphism 
\[ \mathbf{R}\Gamma(\fr{P},\mathcal{M}) \rightarrow \mathbf{R}\Gamma(\fr{P}_K,\mathbf{L}\hat{\mathrm{sp}}^*_\fP\mathcal{M})  \]
in ${\bf D}(K)$. To prove that this is an isomorphism, recall from \cite[\S4]{Hub94} that, as topological spaces,
\[ \fr{P}_K=\lim{I_\fP}\fr{P}'.\]
In this case, since each topological space $\fr{P}'$ is coherent and sober, and the transition maps are quasi-compact, it follows from \cite[Chapter 0, Proposition 3.1.19]{FK18} that the map
\[ \underset{\fr{P}'}{\mathrm{colim}}\,\mathbf{R}\Gamma(\fr{P}',\mathbf{L}\hat{\pi}^*\mathcal{M}) \rightarrow \mathbf{R}\Gamma(\fr{P}_K,\mathbf{L}\hat{\mathrm{sp}}^*\mathcal{M}) \]
is an isomorphism. In fact, the result in \cite{FK18} is stated for sheaves, rather than complexes, but it is straightforward to extend it to complexes using $K$-injective resolutions. It now simply suffices to observe that, by Proposition \ref{prop: adm bl ff}, the map
\[\mathbf{R}\Gamma(\fr{P},\mathcal{M}) \rightarrow \mathbf{R}\Gamma(\fr{P}',\mathbf{L}\hat{\pi}^*\mathcal{M}) \]
is an isomorphism for all $\fr{P}'$.

For the general case, it suffices to show that the given map induces an isomorphism on derived global sections for any quasi-compact open $V\subset \fQ_K$. To prove this, I can replace $\fQ$ by an admissible blowup without changing either side, thus I can assume that $V=\fr{V}_K$ for an open subscheme $\fr{V}\subset \fQ$. By Lemma \ref{lemma: Lsp^* and open immersions}, I can then replace $\fQ$ by $\fr{V}$, and hence assume that $V=\fQ_K$. In this case, the base change map fits into the following  commutative diagram.
\[ \xymatrix{ & \mathbf{R}\Gamma(\fQ_K,\bR u_*\bL \hat{\sp}^*_\fP\cM) \ar[r] & \bR\Gamma(\fP_K,\bL \hat{\sp}^*_\fP\cM)  \\ \mathbf{R}\Gamma(\fQ_K,\bL\hat{\sp}^*_\fQ \bR u_* \cM )\ar[ur]   & & \\ \bR\Gamma(\fQ,\bR u_*\cM) \ar[rr]\ar[u] & & \bR\Gamma(\fP,\cM)\ar[uu] } \]
The two horizontal maps are clearly isomorphisms, and the two vertical maps are isomorphisms by the already treated case $\fQ=\spf{\cV}$. Hence the top left diagonal map is an isomorphism. 
\end{proof}

It's worth explicitly stating the special case $\fQ=\spf{\cV}$.

\begin{corollary} Let $\fP$ be a flat formal scheme, and $\cM\in \LD_{\Q,\qc}(\cO_\fP)$. Then the natural map
\[ \bR\Gamma(\fP,\cM)\to \bR\Gamma(\fP_K,\bL\hat{\sp}_\fP^*\cM)  \]
is an isomorphism in $\bD(K)$.
\end{corollary}

Eventually, I will want to upgrade Theorem \ref{theo: base change sp^* u_*} to include $\mathscr{D}$-module structures, and in order to do so I will need a slightly different description of the functor $\bR u_*$ that works under the additional assumption that $u$ is flat. Restricting $\ca{B}_\fr{P}$ along the functor $I_\fr{Q}\to I_\fr{P}$ gives rise to a diagram $\ca{B}_{\fr{P}/\fr{Q}}\from I_\fr{Q} \to \bf{FSch}$ which by the flatness hypothesis on $u$ sends an admissible blowup $\fr{Q}'\to \fr{Q}$ to the admissible blowup $\fr{Q}'\times_\fr{Q}\fr{P}\to\fr{P}$. This gives rise to a category $\LD_{\Q,\qc}(\cO_{\mathcal{B}_{\fr{P}/\fr{Q}}})$ of inductive systems of quasi-coherent complexes on $\ca{B}_{\fr{P}/\fr{Q}}$ as above, and restriction along $I_{\fr{Q}}\to I_{\fr{P}}$ induces a `forgetful' functor
\[ \LD_{\Q,\qc}(\cO_{\mathcal{B}_{\fr{P}}}) \to \LD_{\Q,\qc}(\cO_{\mathcal{B}_{\fr{P}/\fr{Q}}})  \]
Of course, the morphism of diagrams $u\from \mathcal{B}_{\fr{P}/\fr{Q}}\to \mathcal{B}_{\fr{Q}}$ induces a functor
\[ \bR u_* \from\LD_{\Q,\qc}(\cO_{\mathcal{B}_{\fr{P}/\fr{Q}}}) \to \LD_{\Q,\qc}(\cO_{\mathcal{B}_{\fr{Q}}}).  \]

\begin{lemma} \label{lemma: alternative Ru*} The diagram
\[ \xymatrix{  \LD_{\Q,\qc}(\cO_{\mathcal{B}_{\fr{P}}})\ar[r] \ar[dr]_{\bR u_*} & \LD_{\Q,\qc}(\cO_{\mathcal{B}_{\fr{P}/\fr{Q}}}) \ar[d]^{\bR u_*} \\ & \LD_{\Q,\qc}(\cO_{\mathcal{B}_{\fr{Q}}})  } \]
commutes upto $2$-isomorphism.
\end{lemma}

\begin{proof} This is simply a matter of unwinding the definitions.
\end{proof}

I end this section with a brief discussion of the compatibility of $\bL\hat{\sp}^*$ with tensor products. If $\fr{P}$ is a flat formal scheme, and $\pi:\fr{P}'\rightarrow \fr{P}$ is an admissible blowup, there is a natural map
\[ \mathbf{L}\hat{\pi}^*\mathcal{M}\otimes_{\mathcal{O}_{\fr{P}'}}^\mathbf{L} \mathbf{L}\hat{\pi}^*\mathcal{N} \rightarrow \mathbf{L}\hat{\pi}^*\mathcal{M}\widehat{\otimes}_{\mathcal{O}_{\fr{P}'}}^\mathbf{L} \mathbf{L}\hat{\pi}^*\mathcal{N} = \mathbf{L}\hat{\pi}^*(\mathcal{M}\widehat{\otimes}_{\mathcal{O}_{\fr{P}}}^\mathbf{L} \mathcal{N}).  \]
Applying $\mathrm{sp}^{-1}_{\fr{P}'}$ gives
\[ \mathrm{sp}^{-1}_{\fr{P}'}\mathbf{L}\hat{\pi}^*\mathcal{M}\otimes_{\mathrm{sp}^{-1}_{\fr{P}'}\mathcal{O}_{\fr{P}'}}^\mathbf{L} \mathrm{sp}^{-1}_{\fr{P}'}\mathbf{L}\hat{\pi}^*\mathcal{N} \rightarrow   \mathrm{sp}^{-1}_{\fr{P}'}\mathbf{L}\hat{\pi}^*(\mathcal{M}\widehat{\otimes}_{\mathcal{O}_{\fr{P}}}^\mathbf{L} \mathcal{N}), \]
and passing to the colimit in $\fr{P}'$ gives a map
\[ \mathbf{L}\hat{\mathrm{sp}}^*\mathcal{M}\otimes_{\mathcal{O}_{\fr{P}^+_K}}^\mathbf{L} \mathbf{L}\hat{\mathrm{sp}}^*\mathcal{N} \rightarrow  \mathbf{L}\hat{\mathrm{sp}}^*(\mathcal{M}\widehat{\otimes}_{\mathcal{O}_{\fr{P}}}^\mathbf{L} \mathcal{N}) \]
in ${\bf D}(\mathcal{O}^+_{\fr{P}_K})$. Tensoring with $\Q$ therefore gives, for any $\mathcal{M},\mathcal{N}\in {\bf D}_{\Q,\mathrm{qc}}(\mathcal{O}_\fr{P})$, a map
\begin{equation} \label{eqn: tensor product} \mathbf{L}\hat{\mathrm{sp}}^*\mathcal{M}\otimes_{\mathcal{O}_{\fr{P}_K}}^\mathbf{L} \mathbf{L}\hat{\mathrm{sp}}^*\mathcal{N} \rightarrow   \mathbf{L}\hat{\mathrm{sp}}^*(\mathcal{M}\widehat{\otimes}_{\mathcal{O}_{\fr{P}}}^\mathbf{L} \mathcal{N}) 
\end{equation}
in ${\bf D}(\mathcal{O}_{\fr{P}_K})$. Passing to the colimit gives a similar map whenever $\mathcal{M},\mathcal{N}\in  \underrightarrow{{\bf LD}}_{\Q,\mathrm{qc}}(\mathcal{O}_\fr{P})$, I leave it to the reader to give a precise construction of this map. In general (\ref{eqn: tensor product}) won't be an isomorphism, however, I will show later on that this will be the case whenever $\mathcal{M},\mathcal{N}$ underly overholonomic $\mathscr{D}^\dagger_{\fr{P}\Q}$-modules.

\subsection{Functions with overconvergent singularities}

Now suppose that $\fP$ is smooth, and consider a divisor $D\subset P$. In this situation, Berthelot has defined in \cite[\S4.4]{Ber02} the $\cO_\fP$-algebra $\widehat{\cB}_\fP(D,r)$ for any $r\in \N$. Locally, if there exists some $f\in \cO_\fP$ such that $D=P\cap V(f)$, then
\[ \widehat{\cB}_\fP(D,r) = \frac{\cO_\fP\tate{X}}{(f^rX-p)}. \]

\begin{lemma} $\widehat{\mathcal{B}}_{\fr{P}}(D,r)$ is quasi-coherent as a complex of $\mathcal{O}_{\fr{P}}$-modules.
\end{lemma}

\begin{proof}
Quasi-coherence can be checked locally, by Remark \ref{rem: qc local}, so I can assume that $\fr{P}$ is affine and that $D \subset P$ is the vanishing locus of a single function $f\in\mathcal{O}_\fr{P}$, which is a non-zero divisor in $\mathcal{O}_{P}$. In this case 
\[ \widehat{\mathcal{B}}_\fP(D,r)=\frac{\mathcal{O}_\fr{P}\tate{X}}{(f^rX-p)} \]
is $p$-torsion-free, and hence flat over $\mathcal{O}_\fr{P}$. Thus
\[ \mathcal{O}_{P_n} \otimes^\mathbf{L}_{\mathcal{O}_{\fr{P}}}\frac{\mathcal{O}_\fr{P}\tate{X}}{(f^rX-p)} \cong \mathcal{O}_{P_n} \otimes_{\mathcal{O}_{\fr{P}}}\frac{\mathcal{O}_\fr{P}\tate{X}}{(f^rX-p)} = \frac{\mathcal{O}_{P_n}[X]}{(f^rX-p)}  \]
is quasi-coherent, and the transition maps $ \frac{\mathcal{O}_{P_{n+1}}[X]}{(f^rX-p)}  \rightarrow  \frac{\mathcal{O}_{P_n}[X]}{(f^rX-p)} $ are surjective. Hence
\[ \mathbf{R}\underset{n}{\mathrm{lim}}\,  \frac{\mathcal{O}_{P_n}[X]}{(f^rX-p)}  \cong \underset{n}{\mathrm{lim}}\,\frac{\mathcal{O}_{P_n}[X]}{(f^rX-p)}  \cong \frac{\mathcal{O}_\fr{P}\tate{X}}{(f^rX-p)} \]
as required.
\end{proof}

Berthelot then defines $\widehat{\cB}^{(m)}_\fP(D):=\widehat{\cB}_\fP(D,p^{m+1})$, which form an inductive system as $m$ varies. This gives rise to an object 
\[ \widehat{\cB}^{(\bullet)}_\fP(D)\in \LD_{\Q,\qc}(\cO_{\fP}).\]

\begin{theorem} \label{theo: jdag in terms of LDqc}
Let $\fP$ be a smooth formal scheme, $D\subset P$ a divisor, and $U:=P\setminus D$. Suppose that $\cM\in \LD_{\Q,\qc}(\cO_{\fP})$. Then the morphism
\[ \cM\to \widehat{\cB}^{(\bullet)}_\fP(D)\widehat{\otimes}^{\bL}_{\cO_{\fP}}\cM \]
induces an isomorphism
\[ j_U^\dagger \bL\hat{\sp}^* \cM  \isomto  \bL\hat{\sp}^*(\widehat{\cB}^{(\bullet)}_\fP(D)\widehat{\otimes}^\bL_{\cO_{\fP}}\cM) \]
in $\bD^(\cO_{\fP_K})$.
\end{theorem}

\begin{proof}
Set $r_m=p^{m+1}$, $\lambda_m:=p^{-1/r_m}$, let $\tube{D}_{\lambda_m}\subset \fr{P}_K$ be the open tube of radius $\lambda_m$, $\overline{D}_m$ its closure in $\fP_K$, and write $V_{m}:= \fr{P}_K\setminus \overline{D}_m$. Let $j_{m}:V_{m}\rightarrow \fr{P}_K$ denote the inclusion. Thus, if $D$ is defined by $f=0$ inside $P$, for some $f\in \mathcal{O}_\fr{P}$, then $V_{m}$ is defined inside $\fP_K$ by $\norm{f}\geq \lambda_m$. Therefore, for any complex $\mathscr{K}$ on $\fP_K$, 
\[ j_U^\dagger\mathscr{K} \isomto \colim{m}\bR j_{m*}j_m^{-1}\mathscr{K}. \]
The map 
\[ \bL\hat{\sp}^*\cM\to \bL\hat{\sp}^*(\widehat{\cB}^{(\bullet)}_\fP(D)\widehat{\otimes}^{\bL}_{\cO_{\fP}}\cM) \]
is the colimit of the maps
\[ \bL\hat{\sp}^*\cM\to \bL\hat{\sp}^*(\widehat{\cB}^{(m)}_\fP(D)\widehat{\otimes}^{\bL}_{\cO_{\fP}}\cM). \]
It therefore suffices to show that
\[\bL\hat{\sp}^*\cM\to \bL\hat{\sp}^*(\widehat{\cB}^{(m)}_\fP(D)\widehat{\otimes}^{\bL}_{\cO_{\fP}}\cM)  \]
factors through an isomorphism
\[ \bR j_{m*}j_m^{-1}\bL\hat{\sp}^*\cM\to \bL\hat{\sp}^*(\widehat{\cB}^{(m)}_\fP(D)\widehat{\otimes}^{\bL}_{\cO_{\fP}}\cM) \]
(which is then necessarily unique). To prove this, consider the admissible blowup
\[ \pi\from \fP':=\mathrm{Bl}_{(p,f^{r_m})}\fP \to \fP.\]
There exists a unique open subset $j_m\from \mathfrak{V}_m'\to \fP'$ such that $V_m=\sp^{-1}_{\fP'}(\mathfrak{V}_m')$. Now set $\cM':=\bL\hat{\pi}^*\cM$ and $\widehat{\cB}^{(m)}_{\fP'}(D):=\bL\hat{\pi}^*\widehat{\cB}^{(m)}_\fP(D)$, thus what I need to show is that the map
\[\bL\hat{\sp}^*_{\fP'}\cM'\to \bL\hat{\sp}^*_{\fP'}(\widehat{\cB}^{(m)}_{\fP'}(D)\widehat{\otimes}^{\bL}_{\cO_{\fP'}}\cM'),  \]
factors through an isomorphism
\[ \bR j_{m*}j_m^{-1}\bL\hat{\sp}^*_{\fP'}\cM'\to \bL\hat{\sp}^*_{\fP'}(\widehat{\cB}^{(m)}_{\fP'}(D)\widehat{\otimes}^{\bL}_{\cO_{\fP'}}\cM')\]
(again, this isomorphism is then necessarily unique). Applying Theorem \ref{theo: base change sp^* u_*}, it therefore suffices to prove that the map
\[ \cM'\to \widehat{\cB}^{(m)}_{\fP'}(D)\widehat{\otimes}^{\bL}_{\cO_{\fP'}}\cM' \]
factors through an isomorphism
\[ \bR j_{m*}j_m^{-1}\cM'\to \widehat{\cB}^{(m)}_{\fP'}(D)\widehat{\otimes}^{\bL}_{\cO_{\fP'}}\cM'\]
in $\LD_{\Q,\qc}(\cO_{\fP'})$. By Lemma \ref{lemma: proj} I can reduce to the case $\cM'=\cO_{\fP'}$, in other words I need to show that
\[ \cO_{\fP'}\rightarrow \widehat{\cB}^{(m)}_{\fP'}(D) \]
factors through an isomorphism
\[  \bR j_{m*}\cO_{\fr{V}_m'}\rightarrow \widehat{\cB}^{(m)}_{\fP'}(D) \]
in $\bD_{\Q,\qc}(\cO_{\fP'})$. To see this, I will construct the inverse isogeny
\[ \widehat{\cB}^{(m)}_{\fP'}(D) \rightarrow \bR j_{m*}\cO_{\fr{V}_m'}. \]
Locally on $\fP$,
\[ \fP' = \Proj{\cO_{\fP}}{\frac{\cO_{\fP}[X,Y]}{(f^{r_m}X-pY)}}, \]
where $D=P\cap V(f)$. This is covered by the two open formal subschemes
\begin{align*}
\fr{V}_m' &= \spf{\frac{\cO_{\fP}\tate{X}}{(f^{r_m}X-p)}} \\
\fr{U} &= \spf{\frac{\cO_{\fP}\tate{Y}}{(f^{r_m}-pY)}}
\end{align*}
In particular, locally on $\fP$, $\fr{V}_m'=D(Y)$ is the complement of a hypersurface in $\fP'$, and hence $\bR j_{m*}\cO_{\fr{V}'_m}=j_{m*}\cO_{\fr{V}'_m}$ is a module (and not just a complex). Thus I am allowed to work locally on $\fP$, and therefore assume that there is indeed such an $f$.

I can then calculate
\[ \widehat{\cB}^{(m)}_{\fP'}(D)\!\mid_{\fr{V}'_m} = \bR\lim{n} \frac{\cO_{P_n}[T]}{(f^{r_m}T-p)} \otimes^{\bL}_{\cO_{P_n}} \frac{\cO_{P_n}[X]}{(f^{r_m}X-p)}.  \]
Since $f$ is not a zero divisor in $\cO_{P_n}$, a direct calculation then shows that
\[ \cO_{P_n}[X] \overset{f^{r_m}X-p}{\longrightarrow} \cO_{P_n}[X] \]
is a flat resolution of $\frac{\cO_{P_n}[X]}{(f^{r_m}X-p)}$. Thus
\[ \frac{\cO_{P_n}[T]}{(f^{r_n}T-p)} \otimes^{\bL}_{\cO_{P_n}} \frac{\cO_{P_n}[X]}{(f^{r_m}X-p)} = \left[   \frac{\cO_{P_n}[X,T]}{(f^{r_m}T-p)} \overset{f^{r_m}X-p}{\longrightarrow} \frac{\cO_{P_n}[X,T]}{(f^{r_m}T-p)} \right]. \]
Since each $\frac{\cO_{P_n}[X,T]}{(f^{r_m}T-p)}$ is a quasi-coherent $\cO_{P_n}$-module, and the transition maps 
\[ \frac{\cO_{P_{n+1}}[X,T]}{(f^{r_m}T-p)}\to \frac{\cO_{P_n}[X,T]}{(f^{r_m}T-p)}\]
are surjective, 
\[ \bR\lim{n} \frac{\cO_{P_n}[X,T]}{(f^{r_m}T-p)} = \lim{n}\frac{\cO_{P_n}[X,T]}{(f^{r_m}T-p)} = \frac{\cO_{\fP}\tate{X,T}}{(f^{r_m}T-p)}.\]
Thus
\[ \widehat{\cB}^{(m)}_{\fP'}(D)\!\mid_{\fr{V}'_m} = \left[   \frac{\cO_{\fP}\tate{X,T}}{(f^{r_m}T-p)} \overset{f^{r_m}X-p}{\longrightarrow} \frac{\cO_{\fP}\tate{X,T}}{(f^{r_m}T-p)} \right]. \]
A direct calculation shows that this map is injective, thus
\[ \widehat{\cB}^{(m)}_{\fP'}(D)\!\mid_{\fr{V}'_m} = \frac{\cO_{\fP}\tate{X,T}}{(f^{r_m}T-p,f^{r_m}X-p)}. \]
On $\fr{V}'_m$, the factorisation
\[ \cO_{\fP'}\to \bR j_{m*}\cO_{\fr{V}'_m} \to \widehat{\cB}^{(m)}_{\fP'}(D) \]
I am after is simply the inclusion 
\[ \frac{\cO_{\fP}\tate{X}}{(f^{r_m}X-p)} \to \frac{\cO_{\fP}\tate{T,X}}{(f^{r_m}T-p,f^{r_m}X-p)}.\]
The inverse map
 \[ \widehat{\cB}^{(m)}_{\fP'}(D) \to \bR j_{m*}\cO_{\fr{V}'_m} \]
 is then the multiplication 
\[ \frac{\cO_{\fP}\tate{T,X}}{(f^{r_m}T-p,f^{r_m}X-p)} \to \frac{\cO_{\fP}\tate{X}}{(f^{r_m}X-p)}. \]
which is surjective, with kernel $(T-X)$ annihilated by $p$. Similarly, on $\fr{U}$, 
\begin{align*} \widehat{\cB}^{(m)}_{\fP'}(D)\!\mid_{\fr{U}} &= \bR\lim{n} \frac{\cO_{P_n}[T]}{(f^{r_m}T-p)} \otimes^{\bL}_{\cO_{P_n}} \frac{\cO_{P_n}[Y]}{(f^{r_m}-pY)} \\
 &= \frac{\cO_{\fP}\tate{T,Y}}{(f^{r_m}T-p,f^{r_m}-pY)}.
 \end{align*}
The inverse map
 \[ \widehat{\cB}^{(m)}_{\fP'}(D) \to \bR j_{m*}\cO_{\fr{V}'_m} \]
 is therefore the map
 \[ \frac{\cO_{\fP}\tate{T,Y}}{(f^{r_m}T-p,f^{r_m}-pY)} \to  \frac{\cO_{\fP}\tate{Y,Y^{-1}}}{(f^{r_m}-pY)} \]
given by $T \mapsto Y^{-1}$. Again, this map is surjective, with kernel $(1-TY)$ annihilated by $p$. 

Putting this together, I have constructed a map
\[ \widehat{\cB}^{(m)}_{\fP'}(D) \to j_{m*}\cO_{\fr{V}'_m} \]
which is surjective, and whose kernel is annihilated by $p$. It follows that it is an isogeny, and hence induces an isomorphism
\[ \widehat{\cB}^{(m)}_{\fP'}(D) \isomto \bR j_{m*}\cO_{\fr{V}'_m} \]
in $\bD_{\Q,\qc}(\cO_{\fP'})$. To identify this with the inverse of an isomorphism
\[ \bR j_{m*}\cO_{\fr{V}'_m}  \isomto \widehat{\cB}^{(m)}_{\fP'}(D) \]
factoring the natural map
\[ \cO_{\fr{P}'}  \isomto \widehat{\cB}^{(m)}_{\fP'}(D) \]
I can therefore restrict to $\fr{V}'_m$, in which case everything has already been worked out explicitly.
\end{proof}

\section{Rigidification of \texorpdfstring{$\mathscr{D}^\dagger$}{Ddag}-modules and constructibility} \label{sec: rigidification of D-modules}

Suppose now that the formal scheme $\fP$ is smooth over $\mathcal{V}$. Thus Berthelot defined in \cite[\S4.2]{Ber02} the category
\[ \LD_{\Q,\qc}(\widehat{\sD}^{(\bullet)}_{\fP}) \]
of inductive systems of quasi-coherent complexes of $\widehat{\sD}^{(\bullet)}_{\fP}$-modules, up to isogeny. In this section, my goal is to upgrade $\bL\hat{\sp}^*$ to a functor
\[ \bL\hat{\sp}^*\from \LD_{\Q,\qc}(\widehat{\sD}_{\fP}^{(\bullet)}) \to \bD(\sD_{\fP_K}),\]
and prove that if $\cM$ is an overholonomic complex of Frobenius type, then $\bL\hat{\sp}^*\cM$ is a constructible complex on $\fP$.

\subsection{$\pmb{\sD^{(k,0)}}$-modules on admissible blowups}

The functor $\bL\hat{\sp}^*_\fP$ passes through the category of quasi-coherent complexes of $\cO$-modules on the diagram of all admissible blowups of $\fP$, and the same will be therefore need to be true for its $\sD$-module enhancement. In trying to do this directly, one quickly runs into the following problem: if $\pi\from \fP'\to\fP$ is an admissible blowup of a smooth formal scheme $\fP$, then $\pi^{-1}\sD_{\fP}^{(0)}$ does not necessarily act on $\cO_{\fP'}$. Thus $\pi^*\sD_{\fP}^{(0)}$ does not have any natural ring structure over which the pullback $\pi^*\cM$ of a $\sD^{(0)}_\fP$-module $\cM$ could be a module. 

To get around this problem, I will need to use the rings of differential operators with congruence level from \cite{HSS21}. In that article, the authors define, for any $k\geq 0$, a subring
\[  \sD^{(k,0)}_\fP\subset \sD^{(0)}_\fP  \] 
consisting of `differential operators of congruence level $k$'. In terms of local co-ordinates $z_1,\ldots,z_d$, with corresponding derivations $\partial_1,\ldots,\partial_d$, there is the usual description
\[ \sD^{(0)}_\fP = \left\{\left.\sum_{\underline{i}=(i_1,\ldots,i_d)\in \N^d } a_{\underline{i}} \partial_1^{i_1}\ldots \partial_d^{i_d} \;\right\vert\; a_{\underline{i}} \in \cO_\fP,\;\text{only finitely many }\neq0 \right\}, \]
and then $\sD^{(k,0)}_\fP$ consists of those elements for which $a_{\underline{i}}\in \fr{m}^{k\norm{\underline{i}}}=\fr{m}^{k(i_1+\ldots+i_d)}$. In particular, $\sD^{(k,0)}_{\fP\Q}=\sD^{(0)}_{\fP\Q}$. 

If $\pi\from \fP'\to \fP$ is an admissible blowup, then $\pi^{-1}\sD^{(k,0)}_{\fP\Q}=\pi^{-1}\sD^{(0)}_{\fP\Q}$ acts on $\cO_{\fP'\Q}$. Since $\fP'$ is flat over $\cV$, it therefore makes sense to ask whether or not this induces an action of $\pi^{-1}\sD^{(k,0)}_{\fP}$ on $\cO_{\fP'}$. In \cite[Corollary 2.1.15]{HSS21} the authors show that this will be the case for all sufficiently large $k$. In particular, this will be true if $\fP'={\rm Bl}_{\cI}(\fP)$ and $\varpi^k\in \cI$. 

\begin{definition} If $\pi\from \fP'\to \fP$ is an admissible blowup, with $\fP$ smooth, define $k_{\fP'}:=\min\{k \mid \pi^{-1}\sD^{(k,0)}_{\fP}\text{ acts on }\cO_{\fP'}\}$. If $\rho\from \fP''\to \fP'$ is a morphism of admissible blowups, define $k_\rho:=\max\{k_{\fP'},k_{\fP''}\}$.
\end{definition}

Thus, for any $k\geq k_{\fP'}$ it is possible to form the sheaf of rings $\sD^{(k,0)}_{\fP'}:=\pi^*\sD^{(k,0)}_{\fP}$. If $\cM$ is a $\sD^{(0)}_\fP$-module, then $\pi^*\cM$ is naturally a $\sD^{(k,0)}_{\fP'}$-module, for all $k\geq k_{\fP'}$. More generally, if $\rho\from \fP''\to \fP'$ is a morphism of admissible blowups, and $\cM$ is a $\sD^{(k_{\fP'},0)}_{\fP'}$-module, then $\rho^*\cM$ is naturally a $\sD^{(k,0)}_{\fP'}$-module, for all $k\geq k_\rho$. 

\begin{definition} \label{defn: D-mods on blowups} Let $\cB_\fP$ denote the diagram of admissible blowups of $\fP$. A (left) $\sD^{(0)}_{\cB_{\fP}}$-module consists of:
\begin{itemize}
\item for each admissible blowup $\fP'\to \fP$, a (left) $\sD^{(k_{\fP'},0)}_{\fP'}$-module $\cM_{\fP'}$;
\item for each morphism $\rho\from \fP'\to \fP$, a morphism $\chi_\rho\from \rho^*\cM_{\fP'} \to \cM_{\fP''}$ of (left) $\sD^{(k_{\rho},0)}_{\fP''}$-modules,
\end{itemize}
subject to the condition that, for any commutative triangle
\[ \xymatrix{ \fP''' \ar[rr]^-\tau  \ar[dr]_-{\rho\circ \tau}&& \fP'' \ar[dl]^-\rho  \\ & \fP' }\]
of admissible blowups, $\chi_{\rho\circ \tau}=\chi_\tau\circ \tau^*\chi(\rho)$. 
\end{definition}

Note that the condition $\chi_{\rho\circ \tau}=\chi_\tau\circ \tau^*\chi(\rho)$ can be checked on the underlying $\cO$-modules, hence does not depend on the congruence level of any of the terms. There is, of course, an analogous category of $\sD^{(0)}_{\ca{B}_{P_{\bullet}}}$-modules, consisting of projective systems of $\sD^{(k_{\fP'})}_{P'_n}$-modules on each $\fP'$, 

The major difficulty in working with this category comes from the fact that  the varying rings $\sD^{(k_\fP',0)}_{\fP'}$ do \emph{not} form a sheaf of rings on $\ca{B}_{\fP}$, thus a $\sD^{(0)}_{\cB_{\fP}}$-module is not strictly speaking a module over a sheaf of rings on a site in any natural way. Despite this, it is nevertheless easy to extend the formalism of $K$-flat and $K$-injective resolutions to the category of $\sD^{(0)}_{\cB_{\fP}}$-modules, and therefore do homological algebra as if we were just working in the category of modules over a sheaf of rings on $\cB_\fP$. For example:

\begin{lemma} The category of left $\sD^{(0)}_{\cB_{\fP}}$-modules is an abelian category with enough injectives.
\end{lemma} 

\begin{proof}

\end{proof}

I will denote by $\bD(\sD^{(0)}_{\cB_{\fP}})$ the derived category of $\sD^{(0)}_{\cB_{\fP}}$-modules, and $\bD_\qc(\sD^{(0)}_{\cB_{\fP}})$ its full subcategory on quasi-coherent objects, that is, objects $\cM$ such that each $\cM_{\fP'}$ is quasi-coherent as a complex of $\cO_{\fP'}$-modules. For such objects, the maps $\chi_\rho$ extend to morphisms
\[ \bL\hat{\rho}^*\cM_{\fP'}\to \cM_{\fP''} \]
of quasi-coherent complexes of $\sD^{(k_\rho,0)}$-modules. There is of course a similar category $\bD_\qc(\sD^{(0)}_{\cB^{(\bullet)}_{\fP}})$ of inductive systems of such objects, as well as its localisation $\LD_{\Q,\qc}(\sD^{(0)}_{\cB_{\fP}})$. Exactly as in the case of $\cO_\fP$-modules, I can then define a functor
\[ \bL\hat{\pi}^* \from \bD_\qc(\sD^{(0)}_{\fP^{(\bullet)}}) \to \bD_\qc(\sD^{(0)}_{\cB^{(\bullet)}_{\fP}}). \]
Since $\sD^{(k_\fP',0)}_{\fP'\Q}= \sp_{\fP'*}\sD_{\fP_K}$, I can also define a functor
\[ \bL\sp^*\from \bD_\qc(\sD^{(0)}_{\cB^{(\bullet)}_{\fP}}) \to \bD(\sD_{\fP}), \]
which, roughly speaking, applies $\sp^{-1}_{\fP'}$ on each admissible blowup $\fP'$, then takes the colimit over both $I_{\fP}^{\rm op}$ and $\N$, and then finally tensors with $\Q$.

\begin{definition} The functor
\[  \bL\hat{\sp}^*=\bL\hat{\sp}^*_\fP \from \LD_{\Q,\qc}(\widehat{\sD}^{(\bullet)}_{{\fP}})  \to \bD(\sD_{\fP_K})\]
is defined to be the composite
\[ \bD_\qc(\widehat{\sD}^{(\bullet)}_{\fP})\overset{\rm forget}{\lto} \bD_\qc(\sD^{(0)}_{\fP^{(\bullet)}}) \overset{\bL\hat{\pi}^*}{\lto} \bD_\qc(\sD^{(0)}_{\cB^{(\bullet)}_{\fP}}) \overset{\bL\sp^*}{\lto} \bD(\sD_{\fP_K}). \]
I then define the shifted functor $\sp^!=\sp_\fP^!:=\bL\hat{\sp}^*[-\dim \fP]$. 
\end{definition} 

It is straightforward to check that $\bL\hat{\sp}^*$ (and therefore the shifted version $\sp^!$) descends to the localisation $\LD_{\Q,\qc}(\widehat{\sD}^{(\bullet)}_{{\fP}})$ of $\bD_\qc(\widehat{\sD}^{(\bullet)}_{\fP})$.

\subsection{Compatibility with de\thinspace Rham pushforwards}

Now let $u\from \fP\to\fQ$ be a smooth morphism of smooth formal $\cV$-schemes. Since $u$ is in particular flat, the strict transform of an admissible blowup $\fQ'\to\fQ$ is just the fibre product $\fQ'\times_{\fQ}\fP$. 

\begin{lemma} \label{lemma: k p' q'} Let $u\from \fP\to \fQ$ be a smooth morphism of smooth formal schemes, and $\fQ'\to \fQ$ an admissible blowup. Then $k_{\fQ'\times_{\fQ}\fP}\leq k_{\fQ'}$. 
\end{lemma}

\begin{proof}
Straightforward calculation.
\end{proof}

I now consider the functor $I_\fr{Q}\to I_\fr{P}$ given by taking the fibre product over $\fr{Q}$ with $\fr{P}$, and restrict the diagram $\ca{B}_\fr{P}$ along this functor to produce a diagram $\ca{B}_{\fr{P}/\fr{Q}}\from I_\fr{Q}\to {\bf FSch}$ taking $\fr{Q}'$ to $\fr{Q'}\times_\fr{Q} \fr{P}$. I then consider the following variant of Definition \ref{defn: D-mods on blowups}. 

\begin{definition} \label{defn: D-mods on blowups 2} A (left) $\sD^{(0)}_{\cB_{\fP/\fQ}}$-module consists of:
\begin{itemize}
\item for each admissible blowup $\fQ'\to \fQ$, a (left) $\sD^{(k_{\fQ'},0)}_{\fQ'\times_\fQ \fP}$-module $\cM_{\fQ'}$;
\item for each morphism $\rho\from \fQ''\to \fQ'$, a morphism $\chi_\rho\from \rho^*\cM_{\fQ'} \to \cM_{\fQ''}$ of (left) $\sD^{(k_{\rho},0)}_{\fQ''\times_\fQ\fP}$-modules,
\end{itemize}
subject to the condition that, for any commutative triangle
\[ \xymatrix{ \fQ''' \ar[rr]^-\tau  \ar[dr]_-{\rho\circ \tau}&& \fQ'' \ar[dl]^-\rho  \\ & \fQ' }\]
of admissible blowups, $\chi_{\rho\circ \tau}=\chi_\tau\circ \tau^*\chi(\rho)$. 
\end{definition}

I can then define the transfer module $\sD^{(0)}_{\ca{B}_\fP\to\ca{B}_\fQ}$ as a left $\sD^{(0)}_{\ca{B}_{\fP/\fQ}}$-module by taking its restriction to each $\pi\from \fQ'\times_\fQ\fP \to\fP$ to be $\pi^*\sD^{(k_{\fQ'},0)}_{\fP\to\fQ}$. The usual side switching operations therefore give a right $\sD^{(0)}_{\ca{B}_{\fP/\fQ}}$-module $\sD^{(0)}_{\ca{B}_\fQ\ot \ca{B}_\fP}$, which also comes with the usual structure of a `left $u^{-1}\sD^{(0)}_{\ca{B}_\fQ}$-module on $\ca{B}_{\fP/\fQ}$', in a sense that I leave to the reader to make precise. I can therefore define a functor
\begin{align*}
\label{eqn: u+ blowups} u_+ \from \bD(\sD^{(0)}_{\ca{B}^{(\bullet)}_{\fP/\fQ}} ) &\to \bD(\sD^{(0)}_{\ca{B}^{(\bullet)}_{\fQ}} ) \\
\cM &\mapsto \bR u_* (  \sD^{(0)}_{\ca{B}_\fQ\ot \ca{B}_\fP} \otimes^{\bL}_{\sD^{(0)}_{\ca{B}_{\fP}}} \cM) \nonumber
\end{align*} 
in the usual way. Finally, I define
\[ u_+\from  \bD(\sD^{(0)}_{\ca{B}^{(\bullet)}_{\fP}} ) \to \bD(\sD^{(0)}_{\ca{B}^{(\bullet)}_{\fQ}} ) \]
by composing with the`restriction' functor $ \bD(\sD^{(0)}_{\ca{B}^{(\bullet)}_{\fP}} )\to  \bD(\sD^{(0)}_{\ca{B}^{(\bullet)}_{\fP/\fQ}} )$, which exists by Lemma \ref{lemma: k p' q'}. 

Since $u$ is smooth, the Spencer resolutions for $\sD^{(0)}_{\fQ\ot \fP}$ and $\widehat{\sD}^{(m)}_{\fQ\ot \fP}$ show that the natural map
\[ \sD^{(0)}_{\fQ\ot\fP}\otimes_{\sD^{(0)}_\fP} \widehat{\sD}^{(m)}_\fP \rightarrow \widehat{\sD}^{(m)}_{\fQ\ot\fP}\]
is an isogeny \cite[\S3.5.5]{Ber02}. Hence the diagram
\[ \xymatrix{  \LD_{\Q,\qc}(\widehat{\sD}^{(\bullet)}_{\fP}) \ar[d]_{u_+} \ar[r]& \LD_{\Q,\qc}(\sD^{(0)}_{\fP})\ar[d]^{u_+^{(0)}} \\ \LD_{\Q,\qc}(\widehat{\sD}^{(\bullet)}_{\fQ}) \ar[r] & \LD_{\Q,\qc}(\sD^{(0)}_{\fQ}) } \]
commutes up to natural isomorphism. Now using the facts that $\sD^{(0)}_{\ca{B}_\fQ\ot \ca{B}_\fP,\Q}=\pi^*\sD^{(0)}_{\fQ\ot \fP,\Q}$, and $\sp^{-1}_{\ca{B}_{\fP/\fQ}} \sD^{(0)}_{\ca{B}_\fQ\ot \ca{B}_\fP,\Q}  =\sD_{\fQ_K \ot \fP_K}$, I obtain a natural base change map
\begin{equation*}
 \bL\hat{\sp}_{\fQ}^* \circ u^{(0)}_+ \to u_+ \circ \bL\hat{\sp}_\fP^* \end{equation*}
of functors $\LD_{\Q,\qc}(\sD^{(0)}_{\fP}) \to \bD(\sD_{\fQ_K})$. Note that since $u$ is smooth, say of relative dimension $d$, the functor $u_+$ on the RHS here is what I have previously called $\bR u_{\dR*}[d]$.

\begin{proposition} \label{prop: base change sp^* u_* de Rham} Let $u\from \fP\to\fQ$ be a smooth morphism of smooth formal schemes of relative dimension $d$, and $\cM\in\LD_{\Q,\qc}(\sD^{(0)}_{\fP})$. Then the base change map
\[ \bL\hat{\sp}^*_\fQ u_{+}^{(0)}\cM \rightarrow \mathbf{R}u_{\dR*}\bL\hat{\sp}^*_\fP\cM [d]  \]
is an isomorphism. Thus the diagram
\[\xymatrix{ \LD_{\Q,\qc}(\widehat{\sD}^{(\bullet)}_{\fP}) \ar[d]_{u_+}\ar[r]^-{\sp_\fP^!} &  \bD(\sD_{\fP_K}) \ar[d]^{\bR u_{\dR*}[2d]} \\ \LD_{\Q,\qc}(\widehat{\sD}^{(\bullet)}_{\fQ}) \ar[r]^-{\sp_\fQ^!} & \bD(\sD_{\fQ_K}) } \]
commutes up to natural isomorphism.
\end{proposition}

\begin{proof}
By taking the stupid filtration on
\[ \sD^{(0)}_{\fQ\ot\fP}\otimes^\bL_{\sD_{\fP}^{(0)}}\cM = \Omega^\bullet_{\fP/\fQ}\otimes_{\cO_{\fP}}\cM,\]
I cam reduce to showing that the base change map
\[ \bL\hat{\sp}^*_\fQ\bR u_*(\Omega^n_{\fP/\fQ}\widehat{\otimes}^\bL_{\cO_{\fP}}\cM )\rightarrow \mathbf{R}u_{*}\bL\hat{\sp}^*_\fP(\Omega^n_{\fP/\fQ}\widehat{\otimes}^\bL_{\cO_{\fP}}\cM) \]
is an isomorphism for each fixed $n$. Thanks to Lemma \ref{lemma: alternative Ru*}, this follows from Theorem \ref{theo: base change sp^* u_*}.   
\end{proof}

For example, taking $\fQ=\spf{\cV}$ shows that the natural map
\[ \bR\Gamma_\dR(\fP,\cM)\to \bR\Gamma_\dR(\fP_K,\bL\hat{\sp}^*_\fP\cM) \]
is an isomorphism in $\bD(K)$. Of course, this case could equally well be deduced slightly more directly from Theorem \ref{theo: base change sp^* u_*}.

\subsection{Constructibility of $\bm{\sp^!}$}

I now come to the first main result concerning $\sp^!$, namely that it sends dual constructible overholonomic complexes on $\fP$ to constructible isocrystals on $\fP_K$.

\begin{theorem} \label{theo: sp^! cons}
Let $\fP$ be a smooth formal scheme, and $\cM\in \DCon_F(\fP)\subset \bD^b_{\hol,F}(\fP)$. Then $\sp^!\cM\in \Isoc_{\cons,F}(\fP)\subset \bD(\sD_{\fP_K})$ is a constructible isocrystal on $\fP$. If $\cM$ is supported on some locally closed subscheme $X\hookrightarrow P$, then so is $\sp^!\cM$.
\end{theorem}

Before diving into the proof, I need a preparatory lemma, and a special case.

\begin{lemma} \label{lemma: reduce sp^! cons finite etale} Let 
\[ \xymatrix{ X'\ar[r]\ar[d]^f & Y'\ar[r]\ar[d]^g & \fP'\ar[d]^u \\ X \ar[r] & Y \ar[r] & \fP } \]
be a morphism of l.p. frames such that:
\begin{itemize}
\item $u$ is smooth and proper, of relative dimension $d$;
$g$ is proper, and $f$ is finite \'etale;
\item $X$ is locally (on $Y$) the complement of a hypersurface in $Y$.
\end{itemize} 
Let $\cM\in \DCon_F(\fP)$ be a dual constructible module supported $X$. If $\sp^!_{\fP'}\mathbf{R}\underline{\Gamma}^\dagger_{X'}u^!\cM$ is a constructible isocrystal supported on $X'$, then $\sp_{\fP}^!\cM$ is a constructible isocrystal supported on $X$.
\end{lemma}

\begin{remark} Of course, if I am considering $\cM$ as an object of $\DCon_F(X,Y)$, then $\mathbf{R}\underline{\Gamma}^\dagger_{X'}u^!\cM$ is simply $f^!\cM\in \DCon_F(X',Y')$.
\end{remark}

\begin{proof} The question is local on $\fP$, so I can therefore suppose that $\fP$ is affine and $X=D(s)$ for some $s\in \Gamma(Y,\cO_Y)$. I am therefore in the situation of Theorem \ref{theo: finite etale general cons}. Write $\cN:=\mathbf{R}\underline{\Gamma}^\dagger_{X'}u^!\cM$, this is a dual constructible module on $\fP'$. By Proposition \ref{prop: base change sp^* u_* de Rham},
\[ \sp^!_\fP u_+\cN = \bR u_{\dR*} \sp^!_{\fP'}\cN[2d] \]
where $\sp^!_{\fP'}\cN$ is a constructible isocrystal on $\fP'$, supported on $X'$. I then consider the commutative diagram
\[\xymatrix{ \tube{X'}_{\fP'}  \ar[r] \ar[d]_{\tube{f}} & \fP'_K \ar[d]^u \\ \tube{X}_\fP \ar[r]^i & \fP_K } \]
Since $u$ is proper, I can identify $\bR u_{\dR!}=\bR u_{\dR*}$ as functors
\[ \bD^b(\sD_{\fP'_K})\to \bD^b(\sD_{\fP_K}).\]
Setting $\sF:=\sp^!_{\fP'}\cN|_{\tube{X'}_{\fP'}} \in\Isoc_\cons(X',Y',\fP')$, I therefore have
\[ \bR u_{\dR*}\sp^!_{\fP'}\cN[2d] = i_!\bR \!\tube{f}_{\dR !} \sF[2d]. \]
Now applying Theorem \ref{theo: finite etale general cons} I deduce that $\sp^!_\fP u_+\cN$ is a constructible isocrystal supported on $X$. Since $\cM$ is a direct summand of $u_+\cN$ (see Remark \ref{rem: D module finite etale direct summand}), I can conclude. 
\end{proof}

\begin{lemma} \label{lemma: sp_* and sp^! for locally free isocrystals} Let $\fP$ be a smooth formal scheme, $D\subset P$ a divisor, and let $j\from U :=P\setminus D\rightarrow P$ be the complementary open immersion. Suppose that $\sF \in \Isoc_F(U,P)$, and let $\cM=\sp_{U!}\sF \in \DCon_F(\fP)$.\footnote{See \S\ref{sec: sp+} and Example \ref{exa: sp in hearts}.} Then $\sp^!\cM \cong j_*\sF\in \Isoc_\cons(\fP)$. 
\end{lemma}

\begin{proof}
Given the shifts involved, an equivalent way of stating the conclusion is that
\[ \bL\hat{\sp}_{\fP}^*\widetilde{\sp}_{U+}\sF \cong j_*\sF\]
in $\bD^b(\sD_{\fP_K})$. Also, since $j_*\sF$ is a module (not just a complex) the claim is local on $\fP$, which I may therefore assume to be affine. I can also assume that there exists $f\in \cO_\fP$ which is a non-zero divisor in $\cO_P$, such that $D=V(f)\cap P$. 

Let $j_m\from V_m\to \fP_K$ be as in the proof of Theorem \ref{theo: jdag in terms of LDqc}, recall that $V_m$ is affinoid. For $m$ large enough, $\sF$ extends to a vector bundle with integrable connection $\sF_m$ on $V_m$, and
\[ j_*\sF \isomto \colim{m}j_{m*}\sF_m \isomto \colim{m} \bR j_{m*} \sF_m.\]
In this case, $\widetilde{\sp}_{U+}\sF$ is explicitly defined in \cite[\S4.4]{Ber96a} by showing that each
\[ \sp_{\fP*} j_{m*} \sF_m \isomto \bR \sp_{\fP*} \bR j_{m*} \sF_m\]
is a coherent $\hat{\ca{B}}_{\fP}^{(m)}(D)$-module, for which the $\sD_\fP^{(0)}$-module structure extends uniquely to a continuous $\widehat{\sD}_\fP^{(m)}$-module structure compatible with that on $\hat{\ca{B}}_{\fP}^{(m)}(D)$. Then 
\[ \widetilde{\sp}_{U+}\sF = \{\sp_{\fP*} j_{m*} \sF_m \}_{m\in \N}\in\LD_{\Q,\qc}(\widehat{\sD}^{(\bullet)}_{\fP}).  \]
It therefore suffices to show that 
\begin{equation} \label{eqn: sp_+ sp^! level m} \bL\hat{\sp}_{\fP}^*\bR\sp_{\fP*}\bR j_{m*} \sF_m\isomto  \bR j_{m*}\sF_m = j_{m*}\sF_m
\end{equation}
in $\bD^b(\sD_{\fP_K})$. Indeed, if this is the case then everything in sight is just a module, rather than a complex, and so showing compatibility in $m$ is straightforward. To prove (\ref{eqn: sp_+ sp^! level m}), let $\pi\from \fP'\to \fP$ be an admissible blowup such that the open immersion $j_m\from V_m\to \fP_K$ arises from an open immersion $j_m\from \fr{V}_m'\to \fP'$ of formal schemes. I then claim that the natural morphism
\[ \bL\hat{\pi}^*\bR\sp_{\fP*}\bR j_{m*} \sF_m \rightarrow \bR\sp_{\fP'*}\bR j_{m*} \sF_m \]
is an isogeny. To see this, since $V_m$ is affinoid, and $\sF_m$ is a vector bundle on $V_m$, it is in fact a direct summand of a trivial vector bundle. Hence I can reduce to the case $\sF_m=\cO_{V_m}$. In this case, the given map can be identified with
\[ \bL\hat{\pi}^*\widehat{\cB}^{(m)}_{\fP}(D)  \rightarrow \bR j_{m*}\cO_{\fr{V}_m'}, \]
which was shown to be an isogeny during the proof of Theorem \ref{theo: jdag in terms of LDqc}. I can now compute
\begin{align*}
\bL\hat{\sp}_{\fP}^*\bR\sp_{\fP*}\bR j_{m*} \sF_m &\isomto \bL\hat{\sp}_{\fP'}^*\bL\hat{\pi}^*\bR\sp_{\fP*}\bR j_{m*} \sF_m \\
&\isomto \bL\hat{\sp}_{\fP'}^*\bR\sp_{\fP'*}\bR j_{m*} \sF_m \\
&\isomto \bR j_{m*}\bL\hat{\sp}_{\fr{V}_m'}^*\bR\sp_{\fr{V}'_m*} \sF_m \\
&\isomto \bR j_{m*} \sF_m \\
\end{align*}
where the first isomorphism follows simply from the definitions, the second is the isogeny proved above, the third follows from Theorem \ref{theo: base change sp^* u_*}, and the last from the fact that  $\sF_m$ is a locally free $\cO_{V_m}$-module.
\end{proof}

\begin{remark} The particular isomorphism $\sp^!\cM\cong j_*\sF$ constructed can be identified as follows: if we set $\fr{U}=\fP\setminus D$, then on $\fr{U}_K$ it is simply the shift of the counit $\sp^*_\fr{U}\sp_{\fr{U}*}\sF\to\sF$  of the adjunction between $\sp^*$ and $\sp_*$ (note that since $\cM\!\!\mid_\fr{U}$ is locally projective, $\bL\hat{\sp}^*_\fr{U} = \sp^*_\fr{U}$ in this case).
\end{remark}

\begin{proof}[Proof of Theorem \ref{theo: sp^! cons}]
The question is local on $\fP$, which I can therefore assume to be affine, in particular $\fP$ admits a locally closed immersion into a smooth and proper formal $\cV$-scheme. This then means that every frame encountered during the proof will be an l.p. frame. By d\'evissage, that is, by Proposition \ref{prop: overhol complexes generically isocrystals}, every $\cM\in \DCon_F(\fP)$ is an iterated extension of objects of the form $\sp_{X!}\sF$ with $X\hookrightarrow P$ a smooth locally closed subscheme, with closure $Y$ in $P$, and $\sF\in \Isoc_F(X,Y)$. I can therefore assume that $\cM$ is of this form.

Now, by de\thinspace Jong's alterations \cite{dJ96}, there exists a projective, generically \'etale morphism $g\from Y'\to Y$ with $Y'$ smooth. By Noetherian induction on $X$, I can therefore assume that $g^{-1}(X)\to X$ is finite \'etale, and that $X$ is the complement of a hypersurface in $Y$. I then have a morphism of frames of the form
\[ \xymatrix{ X'\ar[r]\ar[d]^f & Y'\ar[r]\ar[d]^g & \fP'=\widehat{\P}^d_\fP \ar[d]^u \\ X \ar[r] & Y \ar[r] & \fP } \]
with left hand square Cartesian, $Y'$ smooth, $f$ finite \'etale, and $u$ the projection. Appealing to Lemma \ref{lemma: reduce sp^! cons finite etale}, I can replace $(X,Y,\fP)$ by $(X',Y',\fP')$, in other words I can also assume that $Y$ is smooth. Of course, I have now lost the assumption that $\fP$ is affine, but I can then further localise on $\fP$ to restore it.

Since $\fP$ is affine, and $Y$ is smooth, there exists a closed immersion of smooth formal schemes $\fY\hto\fP$ lifting the closed immersion $Y\hto P$. Now applying Lemma \ref{lemma: reduce sp^! cons finite etale} to the morphism of frames
\[ \xymatrix{ & & \fY\times_\cV \fP \ar[d] \\ X \ar[r]^j & Y \ar[r]^i\ar[ur] & \fP } \]
(where the diagonal map $Y\to \fr{Y}\times_\cV\fr{P}$ is the product of the two given immersions, and the right hand vertical map is the second projection), it suffices to prove the claim with $\fP$ replaced by $\fY\times_\cV \fP$. In other words, I can assume that the closed immersion $\fY\rightarrow \fP$ is a section of a smooth morphism $\pi\from \fP\to \fY$.

Let $c$ be the codimension of $\fY$ in $\fP$. By further localising on $\fP$ if necessary, I can choose functions $t_1\ldots,t_c\in \Gamma(\fP,\cO_\fP)$ such that $\fY=V(t_1,\ldots,t_c)$. For any subset $I\subset \{1,\ldots,c\}$ set $\mathfrak{U}_I:=\cap_{i\in I}D(t_i)$. I also set $\mathcal{N}:=\pi^!\pi_+\mathcal{M}\in \DCon_F(\fr{P})$, thus, by the discussion in \S\ref{sec: sp+}, $\cN=\sp_{\pi^{-1}(X)!}\pi^*\sF$ where $\pi^*\sF\in \Isoc_F(\pi^{-1}(X),P)$. Let $j$ also denote the inclusion $\pi^{-1}(X)\to P$, and write $\sG:=j_*\pi^*\sF\in \Isoc_{\cons,F}(\fP)$.

There is then the following exact sequence of dual constructible $\mathscr{D}^\dagger$-modules on $\fr{P}$:
\[ 0\rightarrow \mathcal{M} \rightarrow \mathcal{N} \rightarrow \bigoplus_i \mathbf{R}\underline{\Gamma}_{U_i}^\dagger \mathcal{N} \rightarrow \ldots \rightarrow \mathbf{R}\underline{\Gamma}_{U_{\{1,\ldots,c\}}}^\dagger \mathcal{N} \rightarrow 0,  \]
whose exactness can be seen by applying the t-exact and conservative family of functors $\mathbf{R}\underline{\Gamma}^\dagger_x$, for $x\in P$ a closed point, together with the inclusion-exclusion principle.

It follows from Theorem \ref{theo: jdag in terms of LDqc} and Lemma \ref{lemma: sp_* and sp^! for locally free isocrystals} that applying $\mathrm{sp}_\fr{P}^!$ to all except the first term in this exact sequence results in the complex
\[ \sG \rightarrow \bigoplus_i j_{U_i}^\dagger \sG \rightarrow \ldots \rightarrow j_{U_{\{1,\ldots,c\}}}^\dagger \sG  \]
of $\sD_{\fP_K}$-modules. I therefore deduce that
\[ \sp_\fP^!\cM\isomto \left[ \sG \rightarrow \bigoplus_i j_{U_i}^\dagger \sG \rightarrow \ldots \rightarrow j_{U_{\{1,\ldots,c\}}}^\dagger \sG\right] \isomto i_!i^*\sG, \]
by using \cite[Proposition 2.1.8]{Ber96b}. This completes the proof.
\end{proof}

\subsection{Consequences} \label{subsec: consequences}

I now gather several important consequences of Theorem \ref{theo: sp^! cons}, starting with:

\begin{corollary} Let $\fP$ be a smooth formal scheme. Then $\sp^!$ induces a functor
\[ \sp^!\from \bD^b_{\hol,F}(\fP)\to \bD^b_{\cons,F}(\fP) \]
which is t-exact for the dual constructible t-structure on the source, and the natural t-structure on the target.
\end{corollary}

Examining the proof of Theorem \ref{theo: sp^! cons}, I have also shown the following. 

\begin{corollary} \label{cor: sp^! sp_! locally free isocrystals} Let $\fP$ be a smooth formal scheme, $i\from X\to P$ be a smooth, locally closed subscheme, with closure $Y$, $\sF\in \Isoc_F(X,Y)$, and $\cM=\sp_{X!}\sF\in \DCon_F(\fP)$. Then $\sp^!\cM \isomto i_!\sF \in \Isoc_{\cons,F}(\fP)$.
\end{corollary}

\begin{remark} It's not straightforward to do explicitly, but the particular isomorphism $\sp^!\cM \isomto i_!\sF \in \Isoc_{\cons,F}(\fP)$ thus constructed can be identified by following through the various steps in the proof of Theorem \ref{theo: sp^! cons} and reducing to the case considered in Lemma \ref{lemma: sp_* and sp^! for locally free isocrystals}.
\end{remark}

The functor $\sp^!$ is compatible with many of the cohomological functors introduced so far. Explicitly, if $ U\hto P\hookleftarrow Z$ are complementary open and closed immersions, and $\cM\in \bD^b_{\hol,F}(\fP)$, then repeatedly applying Theorem \ref{theo: jdag in terms of LDqc} gives rise to an isomorphism
\[ j_U^\dagger \sp^! \cM \isomto  \sp^! \cM(^\dagger Z) \]
in $\bD^b_{\cons,F}(\fP)$, which is natural in $\cM$. Moreover, this is compatible with the unit of the two adjunctions, in the sense that the diagram
\[ \xymatrix{ \sp^! \cM \ar[r] \ar[dr] & j_U^\dagger \sp^!\cM \ar[d] \\ &  \sp^! \cM(^\dagger Z) }  \]
commutes. Similarly, there exists an isomorphism
\[ \bR\underline{\Gamma}_Z^\dagger \sp^!\cM\isomto \sp^!\bR\underline{\Gamma}_Z^\dagger \cM,  \]
natural in $\cM$, which is again compatible with the counit of the two adjunctions in the sense that the diagram
\[ \xymatrix{ \bR\underline{\Gamma}_Z^\dagger \sp^!\cM \ar[r] \ar[d] &  \sp^!\cM \\ \sp^!\bR\underline{\Gamma}_Z^\dagger \cM \ar[ur] &  }  \]
commutes. This implies that for any locally closed subscheme $i\from X\to P$, there is a canonical isomorphism
\[ \sp^!\bR\underline{\Gamma}_X^\dagger \cM \cong i_!i^{-1}\sp^!\cM  \]
in $\bD^b_{\cons,F}(\fP)$.

If now $u\from \fP\to \fQ$ is a smooth and proper morphism of smooth formal schemes, of relative dimension $d$, and $\cM\in \bD^b_{\hol,F}(\fP)$, then Proposition \ref{prop: base change sp^* u_* de Rham} provides a canonical isomorphism
\[ \sp_{\fQ}^! u_+ \cM \isomto \bR u_{\dR*}\sp_\fP^!\cM[2d] \] 
inside $\bD^b(\sD_\fQ)$, in particular the RHS is in fact in $\bD^b_{\cons,F}(\fQ)$.

Compatibility of $\sp^!$ with $u^!$ (and therefore with tensor product) will need to wait until later on, but I can at least record a special case for now. Consider a morphism of l.p. frames
\[ \xymatrix{ X'\ar[d]^f \ar[r] & Y' \ar[r]\ar[d]^g & \fP'\ar[d]^u \\ X \ar[r] & Y \ar[r] & \fP } \]
such that $u$ is smooth and proper, of relative dimension $d$, $g$ is proper, and $f$ is finite \'etale. By the support claim in Theorem \ref{theo: sp^! cons}, I can view $\sp^!_\fP$ as a functor
\[ \bD^b_{\hol,F}(X,Y,\fP)\to \bD^b_{\cons,F}(X,Y,\fP),\]
and similarly for $\sp^!_{\fP'}$. Thanks to Theorem \ref{theo: finite etale general cons}, the isomorphisms
\[ \sp_{\fP}^! u_+ \cM \isomto \bR u_{\dR*}\sp_{\fP'}^!\cM[2d] \]
induce an isomorphism
\[ \chi_f \from \sp_{\fP}^! \circ f_+  \isomto \bR\! \tube{f}_{\dR*} \circ \sp_{\fP'}^!  \]
of functors
\[ \bD^b_{\hol,F}(X',Y',\fP')\to \bD^b_{\cons,F}(X,Y,\fP). \]
The adjunctions $(f^!,f_+)$ and $(f^*,\bR\!\tube{f}_{\dR*})$ mean that I can define a morphism
\[ \psi_f\from f^* \circ \sp_{\fP}^! \to \sp_{\fP'}^! \circ f^! \]
of functors 
\[\bD^b_{\hol,F}(X,Y,\fP)\to \bD^b_{\cons,F}(X',Y',\fP') \]
by taking the adjoint of the morphism
\begin{align*}
\sp_{\fP}^! &\to \bR\!\tube{f}_{\dR*} \circ \sp_{\fP'}^!\circ f^! \\
&\overset{\chi_f^{-1}}{\lto}  \sp_{\fP}^! \circ f_+ \circ f^!
\end{align*}   
obtained by applying $\sp_\fP^!$ to the morphism ${\rm id} \to f_+\circ f^!$.

\begin{proposition} \label{prop: special case sp^! and f^*} The morphism $\psi_f$ is an isomorphism, and gives rise to a commutative diagram
\[ \xymatrix{  {\rm Hom}(f^!\cM,\cN) \ar[r]^-{\cong} \ar[d]_-{\psi_f} & {\rm Hom}(\cM,f_+\cN) \ar[d]^-{\chi_f} \\ {\rm Hom}(f^*\sp_\fP^! \cM,\sp_{\fP'}^!\cN) \ar[r]^-{\cong} & {\rm Hom}(\sp_\fP^! \cM , \bR\!\tube{f}_{\dR*}\sp_{\fP'}^! \cN) } \]
for any $\cM\in \bD^b_{\hol,F}(X,Y)$, $\cN\in\bD^b_{\hol,F}(X',Y')$. 
\end{proposition}

\begin{proof}
The fact that the diagram commutes just follows from the definition of $\psi_f$ via adjunction, the key point is to show that $\psi_f$ is an isomorphism. But again, by compatibility of $\sp^!$ with $\underline{\bR\Gamma}^\dagger$, this reduces to the case when $X$ is a point, which is straightforward.
\end{proof}

\section{Logarithmic variations} \label{sec: logarithmic}

Before I can prove my main result, I will need some elementary logarithmic versions of some of the results proved in previous sections. I have not tried especially hard to frame these results in the most general possible content, but have instead concentrated on giving quick proofs of the specific results that I need.

\subsection{Log convergent isocrystals}

The general theory of log convergent isocrystals was developed in \cite{Shi02}, I will recall here some of the fundamental definitions and results that I need. Throughout I will equip the bases $\spec{k}$ and $\spf{\cV}$ with the trivial log structure.

A fine log variety (over $k$) will mean a fine log scheme over $k$ whose underlying $k$-scheme is a variety. If $X$ is smooth variety over $k$, and $D\subset X$ is a normal crossings divisor, I will let $M(D)$ denote the log structure on $X$ associated to $D$, and $(X,M(D))$ the corresponding log scheme. I will use similar notation $(\fX,M(\fr{D}))$ in the case of a relative crossings divisor $\fr{D}\subset \fX$ inside a smooth formal scheme over $\cV$. 

\begin{definition} Let $(X,M)$ be a fine log variety. Then an \emph{enlargement} of $(X,M)$ is a fine log scheme $(Z,M_Z)\to (X,M)$, together with an exact closed immersion $(Z,M_Z)\to (\fr{T},M_\fr{T})$ into a fine log formal scheme over $\cV$, such that $T_{\rm red}\subset Z$. 
\end{definition}

There is the obvious notion of a morphism of enlargements of $(X,M)$.

\begin{definition} A convergent log isocrystal on $(X,M)$ consists of the following data:
\begin{enumerate}
\item for every enlargement $(\fr{T},M_\fr{T})$ of $(X,M)$, a coherent $\cO_{\fr{T}\Q}$-module $\sF_\fr{T}$;
\item for every morphism of enlargements $g\from (\fr{T}',M_{\fr{T}'})\to (\fr{T},M_\fr{T})$ a `transition' isomorphism $g^*\sF_\fr{T}\to\sF_{\fr{T}'}$ of $\cO_{\fr{T}'\Q}$-modules,
\end{enumerate}
The transition isomorphisms are required to satisfy an appropriate cocycle condition. In fact, Shiho considers two variations, where $\sF$ is considered as a sheaf on $\fr{T}$ with respect to the Zariski and \'etale topologies. He then shows in \cite[Proposition 2.1.21]{Shi02} these give rise to equivalent theories. 
\end{definition}

There is an obvious notion of a morphism of convergent log isocrystals on $(X,M)$, and I will denote the resulting category by $\Isoc^\circ(X,M)$. 

\begin{remark} My choice of notation here reflects the fact that I am using $\Isoc(X)$, and not $\Isoc^\dagger(X)$, to denote the category of overconvergent isocrystals on a variety $X$. Thus, when the log structure $M$ is trivial, $\Isoc^\circ(X)$ is the category of convergent isocrystals on $X$, which has previously been denoted $\Isoc(X,X)$.
\end{remark}

As with the case of ordinary (that is, non-logarithmic) convergent isocrystals, these objects can be described in terms of their `realisations' in the world of rigid analytic geometry. Following Shiho, suppose that there is an exact closed immersion 
\[(X,M)\to (\fP,L) \]
of fine log formal schemes over $\cV$, such that $(\fP,L)$ is log smooth over $\cV$.\footnote{In fact, Shiho works in much greater generality than this, but this special case is technically simpler and will suffice for me.} Then the tube $\tube{X}_\fP$ admits a natural log structure coming from the log structure induced by $L$ on $\fP_K$, and is log smooth over $K$. In \cite[\S2.1,\S2.2]{Shi02} Shiho constructs a functor
\begin{align*} \Isoc^\circ(X,M)&\to  \Coh_\nabla(\tube{X}_\fP,L)
\\ \sF&\mapsto \sF_{\tube{X}_\fP} 
\end{align*}
from the category of log convergent isocrystals on $(X,M)$ to the category $ \Coh_\nabla(\tube{X}_\fP,L)$ of coherent $\cO_{\tube{X}_\fP}$-modules with integrable log connection. In general, it is not clear whether or not this `realisation on $(\fP,L)$' functor is fully faithful, in order to obtain full faithfulness it is necessary to impose local freeness conditions on convergent isocrystals. 

\begin{definition} Let $(X,M)$ be a fine log variety over $k$, and $\sF$ a convergent log isocrystals on $(X,M)$. Then $\sF$ is locally free if, for all enlargements $(\fr{T},M_\fr{T})$ of $(X,M)$, the coherent $\cO_{\fr{T}\Q}$-module $\sF_\fr{T}$ is locally projective. 
\end{definition}

It is then a consequence of \cite[Corollary 2.3.9]{Shi02} that the restriction of the realisation functor to locally free isocrystals is fully faithful. Moreover, if we are given $\sF\in  \Isoc^\circ(X,M)$ we can actually test whether or not it is locally free by taking its realisation on $(\fP,L)$ (this follows from \cite[Proposition 2.2.7]{Shi02}).

\subsubsection{Frobenius structures}

Since log convergent isocrystals are clearly functorial in $(X,M)$, there is a well-defined Frobenius pullback functor $F^*\from \Isoc^\circ(X,M)\to\Isoc^\circ(X,M)$, which makes it possible to talk about Frobenius structures on convergent log isocrystals. The category of convergent log $F$-isocrystals on $(X,M)$, that is, the category of convergent log isocrystals on $(X,M)$ equipped with a Frobenius structure, will be denoted $F\text{-}\Isoc^\circ(X,M)$. In certain cases, the presence of a Frobenius structure is enough to guarantee local freeness.

\begin{lemma} \label{lemma: F structure locally free} Let $X$ be a smooth variety, and $D\subset X$ a normal crossings divisor. Then every convergent log $F$-isocrystal $\sF$ on $(X,M(D))$ is locally free.
\end{lemma}

\begin{proof}
It suffices to prove the claim over some \'etale cover of $X$, so I can therefore assume that the pair $(X,D)$ lifts to a pair $(\fX,\fr{D})$ consisting of a smooth formal scheme $\fX$ over $\cV$ and a strict relative normal crossings divisor $\fr{D}\subset \fX$. In fact, I can assume that $\fX=\spf{R}$ is affine and connected, and that there exists an \'etale morphism $\fX\to \widehat{\A}^d_\cV$, with co-ordinates $z_1,\ldots,z_d$ on the target, such that $\fr{D}$ is (the pullback of) the strict normal crossings divisor $V(z_1\ldots z_c)$ for some $c\leq d$. In particular, the standard lift of Frobenius on $\widehat{\A}^d_\cV$ extends to an endomorphism of $\fX$ lifting the absolute ($q$-power) Frobenius on $X$. Let $\sigma:\from R\to R$ denote the ring homomorphism induced by the Frobenius lift. It is then enough to prove the local freeness and of the realisation of $\sF$ on $\fX_K$. 

The proof of this is a standard fitting ideal argument: it suffices to show that the only ideals $I\subset R_K$ stable under $\sigma$, in the sense that $\sigma(I)R_K=I$, are the zero and unit ideals. Setting $I^+=I\cap R$, this is a $\pi$-adically saturated ideal of $R$, such that $\sigma(I^+)R\subset I^+$. Since $\sigma$ is finite flat, it follows that $\sigma(I_0)R$ is also saturated, hence in fact $\sigma(I^+)R= I^+$. Again, it suffices to show that the only any such ideals are the zero and unit ideals. Let $I_0= I^+\otimes_\cV k$ be the reduction of $I^+$ modulo $\pi$, this is therefore a Frobenius stable ideal in $R_k$, in the sense that $F(I_0)R_k=I_0$, and it is enough to show that the only such ideals are the zero and unit ideals. 

But if $F(I_0)R_k=I_0$ then $I_0=I_0^n$ for all $n\geq 1$, and so $I_0=\bigcap_n I_0^n$. Since $R_k$ is smooth, connected and Noetherian, it follows from Krull's intersection theorem \cite[Corollary 5.4]{Eis95} that either $I_0=\{0\}$ or $I_0=R$.
\end{proof}

\subsection{A vanishing result in log rigid cohomology}

The main result I will need in log rigid cohomology is a certain vanishing result, appearing as Theorem \ref{theo: vanishing in log-rigid cohom} below. It will take a little bit of effort to setup and formulate this result, but this is mostly an issue of keeping track of notation, the underlying geometry is actually rather simple. 

Anyway, Theorem \ref{theo: vanishing in log-rigid cohom} will apply in the situation when I have a smooth formal scheme $\fP$, together with a strict relative normal crossings divisor  $\fr{D}\subset \fP$. In this case, I will let $\fP^\sharp$ denote the log scheme $(\fP,M(\fr{D}))$. It's generic fibre $\fP_K^\sharp$ is therefore a logarithmic analytic variety over $K$, and I will let $\sD_{\fP^\sharp_K}\subset \sD_{\fP_K}$ denote the corresponding sheaf of logarithmic (algebraic) differential operators. For example, if $z_1,\ldots,z_d$ are local co-ordinates on $\fP$, such that $\fr{D}$ is defined by $z_1\ldots z_c=0$, and $\partial_1,\ldots,\partial_d$ are the corresponding derivations, then $\sD_{\fP^\sharp_K}$ is generated as a sub-$\cO_{\fP_K}$-algebra of $\sD_{\fP_K}$ by $z_1\partial_1,\ldots,z_c\partial_c,\partial_{c+1},\ldots,\partial_d$. In particular, any constructible isocrystal on $\fP$ can be viewed a $\sD_{\fP^\sharp_K}$-module via restriction of scalars along $\sD_{\fP_K^\sharp}\to \sD_{\fP_K}$. 

Let $\fr{D}_1,\ldots,\fr{D}_c$ denote the irreducible components of $\fr{D}$, with special fibres $D_1,\ldots,D_c$. For each $J\subset \{1,\ldots,c\}$, set
\begin{align*}
\fr{D}_J=\cap_{j\in J} \fr{D}_j,\;\;&\;\;D_J=\fr{D}_{J,k}=\cap_{j\in J} D_j\\
\fr{D}^{(J)}=\cup_{j\notin J} \fr{D}_j,\;\;&\;\;D^{(J)}=\fr{D}_k^{(J)}=\cup_{j\notin J} D_j.
\end{align*}
The log structure $M\fr{D})$ on $\fP$ can be pulled back to each $\fr{D}_J$, giving rise to a log formal scheme $\fr{D}_J^\sharp$. Similarly, I let $\tube{D_J}_{\fP^\sharp}$ denote the tube of $D_J$ equipped with the log structure obtained by pulling back the log structure $M(\fr{D}_K)$ from $\fP_K$ of the log structure $M(\fr{D}_K)$. I will let $\sD_{\tube{D_J}_\fP^\sharp}$ denote the ring of (algebraic) log differential operators on $\tube{D_J}_\fP^\sharp$, or in other words the restriction of $\sD_{\fP_K^\sharp}$ to $\tube{D_J}_\fP$.

There is also a second is natural log structure on $\fr{D}_J$, namely that induced by the strict normal crossings divisor
\[ \fr{D}_J\cap \fr{D}^{(J)}\subset \fr{D}_J. \]
I will denote the resulting log formal scheme by $\fr{D}_J^\flat=(\fr{D}_J,M(\fr{D}_J\cap \fr{D}^{(J)}))$, and denote the ring of (algebraic) log differential operators on its generic fibre $\fr{D}_{J,K}^\flat$ by $\sD_{\fr{D}_{J,K}^\flat}$.

Now, if $\sF$ is a locally free log convergent isocrystal on the special fibre $P^\sharp$ of $\fP^\sharp$, I can take its realisation as a locally free $\cO_{\fP_K}$-module with integrable log connection, which I will also denote by $\sF$. There is then an induced integrable log connection on the twist $\sF(\fr{D}_K)$ of $\sF$, coming from the canonical log connection on the line bundle $\cO_{\fP_K}(\fr{D}_K)$. In particular, $\sF(\fr{D}_K)$ is a $\sD_{\fP_K^\sharp}$-module, which can then be restricted to any tube $\tube{D_J}_\fP$.

\begin{theorem} \label{theo: vanishing in log-rigid cohom} Let $\sF\in F\text{-}\Isoc^\circ(P^\sharp)$ be a log convergent $F$-isocrystal on $P^\sharp$, and let $\sG\in \Isoc_\cons(\fP)$ a constructible isocrystal on $\fP$. Then, for any $J\subset \{1,\ldots,c\}$,
\[ \bR{\rm Hom}_{\sD_{\tube{D_J}^{\sharp}_\fP}}(\sF(\fr{D}_K),\sG)=0. \]
\end{theorem}  

\begin{remark} The crucial observation of the proof is that the presence of a Frobenius structure forces $\sF$ to have nilpotent residues along $\fr{D}_K$, which then implies that the residues of the twist $\sF(\fr{D}_K)$ along $\fr{D}_K$ are in fact \emph{isomorphisms}. 
\end{remark}

\begin{proof}
The question is local on $\fP$, so I can assume that there exists an \'etale map $\fP\to \widehat{\A}^d_\cV$, with co-ordinates $z_1,\ldots,z_d$ on the target, such that $\fr{D}=V(z_1\ldots z_c)$ for some $c\leq d$, and $\fr{D}_J=V(z_1,\ldots,z_b)$ for $b=\#J \leq c$. The functions $z_{b+1},\ldots,z_d$ induce an \'etale map $\fr{D}_J\to \widehat{\A}^{d-b}_\cV$, and adding in the remaining functions $z_1,\ldots,z_b$ gives an \'etale map $\fr{D}_J\times_\cV \widehat{\A}^b_\cV \to \widehat{\A}^d_\cV$. I now let $\fP'$ be the fibre product
\[ \xymatrix{ \fP'\ar[r] \ar[d]_{u} & \fr{D}_J\times_\cV \widehat{\A}^b_\cV  \ar[d] \\ \fP \ar[r] & \widehat{\A}^d_\cV. } \] 
The closed immersion $\fr{D}_J\to \fP$ canonically extends to a closed immersion $\fr{D}_J\to \fP'$ whose composition with the projection $\fP'\to \fr{D}_J\times_\cV \widehat{\A}^b_\cV$ is the identity on the first factor, and the zero section on the second. This gives rise to a section of the \'etale map $u^{-1}(\fr{D}_J)\to \fr{D}_J$, so there exists an open subspace of $\fP'$ on which $u^{-1}(\fr{D}_J)\isomto \fr{D}_J$. 

The upshot of all of this is that after replacing $\fP$ by this open subspace of $\fP'$, I can assume that there exists a smooth morphism $\fP\to \fr{D}_J$ of which the canonical inclusion is a section. Hence by the weak fibration theorem, there exists an isomorphism
\[ \tube{D_J}_\fP\isomto \D^b_{\fr{D}_{J,K}}(0;1^-) \]
identifying $\fr{D}_{J,K}\to \tube{D_J}_\fP$ is identified with the zero section of $\D^b_{\fr{D}_{J,K}}(0;1^-)$. In fact, this isomorphism can be upgraded to include log structures as follows. As above, let $\fr{D}_J^\flat$ denote the log formal scheme given by $\fr{D}_J$ equipped with the log structure coming from the strict normal crossings divisor $\fr{D}_{J,K}\cap \fr{D}_K^{(J)}$, and let $\fr{D}_{J,K}^\flat$ denote its special fibre. If $z_1,\ldots,z_b$ are co-ordinates on $\D^b_{K}(0;1^-)$, then  $\D^b_{K}(0;1^-)$ is equipped with the log structure coming from the strict normal crossings divisor $V(z_1\ldots z_b)$. Then the above isomorphism can be upgraded to an isomorphism
\[  \tube{D_J}^\sharp_\fP \isomto \fr{D}_{J,K}^\flat\times_K \D^b_{K}(0;1^-) \]
of log analytic varieties. Let $\pi$ denote the first projection $\tube{D_J}_\fr{P}^\sharp\to \fr{D}_{J,K}^\flat$.

I now want to apply \cite[Proposition 2.12]{Shi10} to show that $\sF|_{\tube{D_J}_\fP}$ is an iterated extension of pullbacks of locally free log integrable connections on $\fr{D}_{J,K}^\flat$ along $\pi$, at least after possibly localising to ensure that $\fr{D}_J$ is affine and connected. For Shiho's result to apply, I need to check two conditions: firstly, that the supremum norm on $\Gamma(\fr{D}_{J,K},\cO_{\fr{D}_{J,K}})$ is multiplicative, and secondly, that $\sF$ has nilpotent residues along $\fr{D}_K$. The first of these simply follows from the fact that $\fr{D}_J$ is smooth over $\cV$, and the second from the fact that $\sF$ has a Frobenius structure.

Therefore, I can assume that there exists a locally free log integrable connection $\sF_0$ on $\fr{D}_{J,K}^\flat$ such that $\sF|_{\tube{D_J}_\fP}=\pi^*\sF_0$. In a similar vein, it follows from Theorem \ref{theo: disoc indep} there exists a constructible isocrystal $\sG_0$ on $\fr{D}_J$ such that 
\[ \sG|_{\tube{D_J}_\fP}= \pi^*\sG_0. \]
I now set $\sH_0:=\sG_0 \otimes_{\cO_{\fr{D}_{J,K}}} \sF_0^\vee(-\fr{D}_{J,K}\cap \fr{D}_K^{(J)})$, which is thus a $\sD_{\fr{D}^\flat_{J,K}}$-module, and write $\mathscr{T}$ for the divisor $V(z_1\ldots z_b)\subset \tube{D_J}_\fr{P}$ arising via pullback from $\D^b_K(0;1^-)$. Now I can calculate
\begin{align*}
\bR{\rm Hom}_{\sD_{\tube{D_J}^\sharp_\fP}}(\sF(\fr{D}_K),\sG) &= \bR{\rm Hom}_{\sD_{\tube{D_J}^\sharp_\fP}}(\cO_{\tube{D_J}_\fP}(\mathscr{T}),\pi^{*}\sH_0) \\
&= \bR{\rm Hom}_{\sD_{\tube{D_J}^\sharp_\fP}}(\cO_{\tube{D_J}_\fP},\pi^{*}\sH_0 \otimes_{\cO_{\tube{D_J}_\fP}} \cO_{\tube{D_J}_\fP}(-\mathscr{T}) ) \\
&= \bR{\rm Hom}_{\fr{D}_{J,K}^\flat}(\cO_{\fr{D}_{J,K}},\bR\pi_{\dR*}(\pi^{*}\sH_0 \otimes_{\cO_{\tube{D_J}_\fP}} \cO_{\tube{D_J}_\fP}(-\mathscr{T})) ) .
\end{align*}
I will show that in fact
\[ \bR\pi_{\dR*}(\pi^{*}\sH_0 \otimes_{\cO_{\tube{D_J}_\fP}} \cO_{\tube{D_J}_\fP}(-\mathscr{T})) =0\]
which clearly suffices. To show this, note that since $\sH_0$ is a constructible locally free $\cO_{\fr{D}_{J,K}}$-module, Proposition \ref{prop: weak proj formula for cons mod} implies that, after replacing
\[ \pi\from \tube{D_J}_\fP = \D^b_{\fr{D}_{J,K}}(0;1^-) \to \fr{D}_{J,K}\]
by the projection
\[ \pi_\rho\from \D^b_{\fr{D}_{J,K}}(0;\rho) \to \fr{D}_{J,K}\]
from the closed polydisc of radius $\rho<1$, the projection formula
\[ \bR\pi_{\rho\dR*}(\pi^{*}\sH_0 \otimes_{\cO_{\tube{D_J}_\fP}} \cO_{\tube{D_J}_\fP}(-\mathscr{T})) = \sH_0 \otimes^{\bL}_{\cO_{\fr{D}_{J,K}}}  \bR\pi_{\rho\dR*}\cO_{\tube{D_J}_\fP}(-\mathscr{T}).  \]
holds. Since
\[ \bR\pi_{\dR*}(\pi^{*}\sH_0 \otimes_{\cO_{\tube{D_J}_\fP}} \cO_{\tube{D_J}_\fP}(-\mathscr{T})) = \bR\lim{\rho} \bR\pi_{\rho\dR*}(\pi^{*}\sH_0 \otimes_{\cO_{\tube{D_J}_\fP}} \cO_{\tube{D_J}_\fP}(-\fr{T}_K)), \] it is enough to show that the pro-system
\[ \bR\pi_{\rho\dR*}\cO_{\tube{D_J}_\fP}(-\mathscr{T}) \] 
has zero transition maps. Letting
\[ \pi_{\rho^-}\from \D^b_{\fr{D}_{J,K}}(0;\rho^-) \to \fr{D}_{J,K}\]
denote the projection from the open polydisc, it is then enough to show that
\[ \bR\pi_{\rho^-\dR*}\cO_{\tube{D_J}_\fP}(-\mathscr{T}) =0. \]
To see this, I can replace the log structure on $\fr{D}_{J,K}$ by the trivial one, and the log structure on $\D^b_{\fr{D}_{J,K}}(0;\rho^-)$ by that coming from the strict normal crossings divisor $\mathscr{T}=V(z_1\ldots z_b)$. Replacing $\fr{D}_{J,K}$ be an arbitrary smooth analytic variety $V$ makes it possible to argue via induction and reduce to the case $b=1$, that is, the case of the projection
\[ \pi_{\rho^-} \from \D^1_{V}(0;\rho^-) \to V.\]
for a smooth analytic variety $V$. The target here has the trivial log structure, and the source the log structure coming from the zero section. Thus, arguing locally on $V$, and appealing to Theorems A and B for non-Archimedean quasi-Stein spaces \cite[Satz 2.4]{Kie67b}, I can reduce to calculating the cohomology of the complex
\[ zR\open{\rho^{-1}z} \overset{z\partial_z}{\lto} zR\open{\rho^{-1}z} \]
where $R$ is an affinoid algebra over $K$, and $R\open{\rho^{-1}z}$ is the ring of functions on the open disc of radius $\rho$ over $R$. It is then a straightforward calculation that this cohomology vanishes. 
\end{proof}

\subsection{Rigidification of log $\pmb{\sD^\dagger}$-modules}

For the logarithmic analogue of the $\sD^\dagger$-module side of the picture, I will stick to log formal schemes coming from strict normal crossings divisors on smooth formal schemes. Thus I will let $\fP$ be a smooth formal scheme, $\fr{D}\subset \fP$ a strict normal crossings divisor relative to $\cV$, and $\fP^\sharp=(\fP,M(\fr{D}))$ the corresponding log formal scheme. There is therefore the logarithmic analogue $\sD^\dagger_{\fP^\sharp\Q}$ of Berthelot's ring of overconvergent differential operators, as defined in \cite{Mon02} and \cite{Car09c}. As in the non-logarithmic case,
\[ \sD^\dagger_{\fP^\sharp\Q} := \colim{m} \widehat{\sD}^{(m)}_{\fP^\sharp} \otimes_{\Z} \Q \]
where $\widehat{\sD}^{(m)}_{\fP^\sharp}$ is the $p$-adic completion of the sheaf of level $m$ differential operators on $\fP^\sharp$. The isogeny category $\LD_{\Q,\qc}(\widehat{\sD}^{(\bullet)}_{\fP^\sharp})$ of ind-complexes of $\widehat{\sD}^{(\bullet)}_{\fP^\sharp}$-modules is defined exactly as in the non-logarithmic case, and actually taking the colimit induces an equivalence
\[ \LD^b_{\Q,\coh}(\widehat{\sD}^{(\bullet)}_{\fP^\sharp})  \isomto \bD^b_\coh(\sD^\dagger_{\fP^\sharp\Q}),\]
see \cite[\S1.2]{CT12}. The method of \S\ref{sec: rigidification of D-modules} works \emph{mutatis mutandis} in the logarithmic case to define
\[ \sp^!= \sp^!_{\fP^\sharp} = \bL\hat{\sp}_{\fP^\sharp}^*[-\dim \fP] \from \LD_{\Q,\qc}(\widehat{\sD}^{(\bullet)}_{\fP^\sharp})  \to  \bD(\sD_{\fP_K^\sharp}), \]
where $\fP_K^\sharp$ is the generic fibre of $\fP^\sharp$, considered as a log analytic variety. It follows from the construction that $\sp^!$ commutes with pullback along the morphism $\fP^\sharp\to\fP$. The approach of \S\ref{sec: rigidification of D-modules} can also be followed through in the logarithmic case to prove the following. 

\begin{proposition} \label{prop: base change sp^* u_* log de Rham} For any $\cM\in\LD_{\Q,\qc}(\widehat{\sD}^{(\bullet)}_{\fP^\sharp})$, $\bL\hat{\sp}^*_{\fP^\sharp}$ induces a base change isomorphism
\[  \bR\Gamma_{{\rm log}\text{-}\dR}(\fP^\sharp,\cM) \isomto \bR\Gamma_{{\rm log}\text{-}\dR}(\fP_K^\sharp,\bL\hat{\sp}_\fP^*\cM) \]
which is natural in $\cM$.
\end{proposition}

\begin{remark} The definition of $\bR\Gamma_{{\rm log}\text{-}\dR}(\fP^\sharp,\cM)$ is analogous to the non-logarithmic case, namely if $\cM=\{\cM^{(m)}\}_{m\in \N}$ then
\[\bR\Gamma_{\log\text{-}\dR}(\fP^\sharp,\cM) = \colim{m} \bR\Gamma_{\log\text{-}\dR}(\fP^\sharp,\cM^{(m)}) \otimes_{\Z}\Q. \]
\end{remark}

In will also need a very rudimentary logarithmic version of the specialisation functors considered in \S\ref{sec: sp+}. Suppose, then, that $\sF$ is a locally free convergent log isocrystal on the special fibre $P^\sharp$ of $\fr{P}^\sharp$, viewed as a module with integrable logarithmic connection on $\fr{P}_K^\sharp$. Then $\sp_{\fP*}\sF=\bR\sp_{\fP*}\sF$ is naturally an $\cO_{\fP\Q}$-coherent $\sD^\dagger_{\fP^\sharp\Q}$-module, and is moreover coherent as an $\sD^\dagger_{\fP^\sharp\Q}$-module by \cite[Th\'eor\`eme 4.15]{Car09c}. As in the non-logarithmic case, it is straightforward to verify that there is a canonical isomorphism
\[  \bL\hat{\sp}^*_{\fP^\sharp}\sp_{\fP*}\sF \isomto \sF\]
of $\sD_{\fP_K^\sharp}$-modules.

\section{The overconvergent Riemann--Hilbert correspondence} \label{sec: oc RH}

I can now finally prove the main result of this article. 

\begin{theorem} \label{theo: riemann-hilbert}
Let $\fP$ be a smooth formal scheme. Then the functor
\[ \sp^!\from \bD^b_{\hol,F}(\fP) \to \bD^b_{\cons,F}(\fP) \]
is an equivalence of categories. 
\end{theorem}

As in \S\ref{subsec: independence of the frame}, this reduces, via d\'evissage, to the following derived full faithfulness assertion.

\begin{theorem} Let $\fP$ be a smooth formal scheme, $j\from X \to P$ a locally closed subscheme, smooth over $k$, with closure $Y$ in $P$. Suppose $\sF\in F\text{-}\Isoc(X,Y)$, and set $\cM:=\sp_{X!}\sF\in \DCon_F(\fP)$. Then, for any $\cN\in \DCon_F(\fP)$, the map
\[ \bR\mathrm{Hom}_{\sD^\dagger_{\fP\Q}}(\cM,\cN)\to\bR\mathrm{Hom}_{\sD_{\fP_K}}(\sp^!\cM,\sp^!\cN) \]
is an isomorphism in $\bD(K)$. 
\end{theorem}

\begin{proof} To begin with, the question is local on $\fP$, which I can therefore assume to be affine. As in the proof of Theorem \ref{theo: sp^! cons}, one consequence of this is that every formal scheme appearing in the proof will admit a locally closed embedding into a smooth and proper formal scheme. 

The fact that $\cM$ is supported on $Y$, together with compatibility of $\sp^!$ with the natural morphism
\[ \bR\underline{\Gamma}_{Y}^\dagger \rightarrow \mathrm{id} \]
means that there is a commutative diagram
\[ \xymatrix{  \bR\mathrm{Hom}_{\sD^\dagger_{\fP\Q}}(\cM,\cN) \ar[r] & \bR\mathrm{Hom}_{\sD_{\fP_K}}(\sp^!\cM,\sp^!\cN)  \\ 
&  \bR\mathrm{Hom}_{\sD_{\fP_K}}(\sp^!\cM,\bR\underline{\Gamma}_{Y}^\dagger\sp^!\cN) \ar[u]_-{\cong}
\\ \bR\mathrm{Hom}_{\sD^\dagger_{\fP\Q}}(\cM,\bR\underline{\Gamma}_{Y}^\dagger\cN) \ar[r]\ar[uu]^-{\cong} & \bR\mathrm{Hom}_{\sD_{\fP_K}}(\sp^!\cM, \sp^!\bR\underline{\Gamma}_{Y}^\dagger\cN)\ar[u]_-{\cong} . } \]
Hence, replacing $\cN$ with $\underline{\bR\Gamma}^\dagger\cN$, I can assume that $\cN$ is supported on $Y$.

Now, the semistable reduction theorem for $F$-isocrystals \cite[Theorem 2.4.4]{Ked11} shows that there exists a morphism of pairs $(f,g)\from (\widetilde{X},\widetilde{Y})\rightarrow (X,Y)$ such that:
\begin{itemize}
\item $\widetilde{Y}$ is smooth;
\item $g$ is projective and generically \'etale;
\item $\widetilde{X}=g^{-1}(X)$;
\item $\widetilde{D}:=\widetilde{Y}\setminus \widetilde{X}$ is a strict normal crossings divisor in $\widetilde{Y}$;
\item $f^*\sF\in F\text{-}\Isoc(\widetilde{X},\widetilde{Y})$ extends to a convergent log $F$-isocrystal on the log smooth log scheme $(\widetilde{Y},M(\widetilde{D}))$.
\end{itemize}
Since $g$ is projective, I can extend $(f,g)$ to a morphism of frames
\[ \xymatrix{ \widetilde{X} \ar[r]\ar[d]^f & \widetilde{Y} \ar[r]\ar[d]^g & \widetilde{\fP} \ar[d]^u  \\ X \ar[r] & Y \ar[r] & \fP } \]
where $u$ is smooth and projective. By Noetherian induction on $X$, I am free to replace $X$ with an open subscheme, hence I can assume moreover that $f$ is finite \'etale.

Now, by Proposition \ref{prop: finite etale adjoints dct exact}, $\cN\in \bD^b_{\hol,F}(X,Y,\fP)\subset \bD^b_{\hol,F}(\fP)$ is a direct summand of $f_+f^!\cN$, so I can replace $\cN$ by $f_+f^!\cN$. Then, by Proposition \ref{prop: special case sp^! and f^*}, combined with Proposition \ref{prop: base change sp^* u_* de Rham} and the preceding discussion, there is a commutative diagram
\[ \xymatrix@C=6pt{  \bR\mathrm{Hom}_{\sD^\dagger_{\fP\Q}}(\cM,f_+f^!\cN) \ar[r] & \bR\mathrm{Hom}_{\sD_{\fP_K}}(\sp^!_\fP\cM,\sp^!_\fP f_+f^!\cN)   \ar[r]^-\cong&  \bR\mathrm{Hom}_{\sD_{\fP_K}}(\sp^!_\fP\cM,\bR\tube{f}_{\dR*}\sp^!_{\widetilde{\fP}}f^!\cN) \\ \bR\mathrm{Hom}_{\sD^\dagger_{\widetilde{\fP}\Q}}(f^!\cM,f^!\cN) \ar[r]\ar[u]^-{\cong} & \bR\mathrm{Hom}_{\sD_{\widetilde{\fP}_K}}(\sp^!_{\widetilde{\fP}}f^!\cM,\sp^!_{\widetilde{\fP}}f^!\cN) \ar[r]^-\cong& \bR\mathrm{Hom}_{\sD_{\widetilde{\fP}_K}}(f^*\sp^!_{\fP}\cM,\sp^!_{\widetilde{\fP}}f^!\cN) \ar[u]_-{\cong}. } \]
I can therefore replace $(X,Y,\fP)$ by $(\widetilde{X},\widetilde{Y},\widetilde{\fP})$, in other words, I can assume that $Y$ is smooth, that $D:=Y\setminus X$ is a strict normal crossings divisor, and that $\sF$ extends to a convergent log$F$- isocrystal $\sF^{\sharp}$ on the log scheme $Y^\sharp:=(Y,M(D))$. 

Further localising on $\fP$, I can assume that $Y\hookrightarrow P$ lifts to a closed immersion $i\from \fY\to\fP$ of smooth formal schemes. Set $\fr{Q}=\fr{P}\times_\cV\fY$, and let $\pi \from \fr{Q}\to\fr{P}$ denote the projection, thus $\cM\cong \pi_+\underline{\bR\Gamma}^\dagger_Y\pi^!\cM$ and similarly for $\cN$. Since $\sp^!$ is compatible with de\thinspace Rham pushforward, by Proposition \ref{prop: base change sp^* u_* de Rham}, there is a commutative diagram
\[ \xymatrix{  \bR\mathrm{Hom}_{\sD^\dagger_{\fP\Q}}(\cM,\cN) \ar[r] & \bR\mathrm{Hom}_{\sD_{\fP_K}}(\sp^!_\fP\cM,\sp^!_\fP \cN)    \\ \bR\mathrm{Hom}_{\sD^\dagger_{\fQ\Q}}(\underline{\bR\Gamma}^\dagger_Y\pi^!\cM,\underline{\bR\Gamma}^\dagger_Y\pi^!\cN) \ar[r]\ar[u]^-{\cong} & \bR\mathrm{Hom}_{\sD_{\fQ_K}}(\sp^!_{\fQ}\underline{\bR\Gamma}^\dagger_Y\pi^!\cM,\sp^!_{\fQ}\underline{\bR\Gamma}^\dagger_Y\pi^!\cN)  \ar[u]_-{\cong} } \]
where the vertical maps are isomorphisms by Proposition \ref{prop: indep Ddag} and Theorem \ref{theo: disoc indep} respectively. Hence I can replace $\fr{P}$ by $\fr{Q}$, and by the same argument I can then replace $\fr{Q}$ by $\fr{Y}$. In other words, I can assume that $Y=P$.

With one final localisation on $\fP$ if necessary, I can lift $D\subset P$ to a strict normal crossings divisor $\fr{D}\subset \fP$. I now let $\fr{P}^\sharp:=(\fr{P},M(\fr{D}))$ denote the corresponding formal log-scheme, and set $d=\dim \fP$. Note that in this case $\sF$ is a locally free $\cO_{\tube{X}_\fP}$-module with connection, and $\cM=\sp_{\fP*}j_*\sF[d]$, where $j\from X\to P$ is the given open immersion. The extension $\sF^\sharp$ of $\sF$ is a coherent $\cO_{\fP_K}$-module with integrable log connection, and I set $\cM^{\sharp}:=\sp_{\fP*}\sF^{\sharp}[d]$ to be the shift of the $\cO_{\fP\Q}$-coherent $\sD^{\dagger}_{\fP^\sharp\Q}$-module associated to $\sF^{\sharp}$. Thanks to Lemma \ref{lemma: F structure locally free}, $\sF^{\sharp}$ is locally free. Hence $\cM^{\sharp}[-d]$ is a locally projective $\cO_{\fP\Q}$-module.

The natural morphism
\[ \sF^{\sharp}(\fr{D}_K)\rightarrow j_*\sF \]
of $\sD_{\fP^{\sharp}_K}$-modules induces a morphism
\[ \cM^{\sharp}(\fr{D})\rightarrow \cM \]
of $\sD^{\dagger}_{\fP^\sharp\Q}$-modules, which in turn induces an \emph{isomorphism}
\[\sD^{\dagger}_{\fP\Q} \otimes^{\bL}_{\sD^\dagger_{\fP^{\sharp}\Q}} \cM^{\sharp}(\fr{D})\rightarrow \cM \]
of $\sD^\dagger_{\fP\Q}$-modules by \cite[Theorem 2.2.9]{CT12}. This gives rise to the following commutative diagram
\[ \xymatrix@C=0.9em{  \bR\mathrm{Hom}_{\sD^\dagger_{\fP\Q}}(\cM,\cN) \ar[r]\ar[d]_{\cong} & \bR\mathrm{Hom}_{\sD_{\fP_K}}(\sp^!_\fP\cM,\sp^!_\fP\cN) \ar@{=}[r] & \bR\mathrm{Hom}_{\sD_{\fP_K}}(j_*\sF,\sp^!_\fP\cN)\ar[d] \\ \bR\mathrm{Hom}_{\sD^{\dagger}_{\fP^{\sharp}\Q}}(\cM^{\sharp}(\fr{D}),\cN) \ar[r] & \bR\mathrm{Hom}_{\sD_{\fP^{\sharp}_K}}(\sp^!_\fP\cM^{\sharp}(\fr{D}),\sp^!_\fP\cN) \ar@{=}[r] & \bR\mathrm{Hom}_{\sD_{\fP^{\sharp}_K}}(\sF^{\sharp}(\fr{D}_K),\sp^!_\fP\cN) . } \]
It therefore suffices to show that the two maps
\begin{align*}
\bR\mathrm{Hom}_{\sD_{\fP_K}}(j_*\sF,\sp^!_\fP\cN) &\to \bR\mathrm{Hom}_{\sD_{\fP^{\sharp}_K}}(\sF^{\sharp}(\fr{D}_K),\sp^!_\fP\cN) \\
\bR\mathrm{Hom}_{\sD^{\dagger}_{\fP^{\sharp}\Q}}(\cM^{\sharp}(\fr{D}),\cN) &\to \bR\mathrm{Hom}_{\sD_{\fP^\sharp_K}}(\sp^!_\fP\cM^{\sharp}(\fr{D}),\sp^!_\fP\cN)
\end{align*}
are isomorphisms. The second is a relatively straightforward consequence of the fact that the shift $\cM^{\sharp}(\fr{D})[-d]$ is a locally projective $\cO_{\fP\Q}$-module. Indeed, letting $\cM^*$ denote the $\cO_{\fP\Q}$-dual of $\cM^{\sharp}(\fr{D})$, then $\cM^*[d]$ is a locally projective $\cO_{\fP\Q}$-module, and there is a commutative diagram
\[ \xymatrix{  \bR\mathrm{Hom}_{\sD^{\dagger}_{\fP^{\sharp}\Q}}(\cM^{\sharp}(\fr{D}),\cN) \ar[r]\ar[d]_{\cong} & \bR\mathrm{Hom}_{\sD_{\fP^{\sharp}_K}}(\sp_\fP^!\cM^{\sharp}(\fr{D}),\sp^!_\fP\cN)\ar[d]^{\cong} \\ \bR\mathrm{Hom}_{\sD^{\dagger}_{\fP^{\sharp}\Q}}(\cO_{\fP\Q},\cN \otimes_{\cO_{\fP\Q}} \cM^*) \ar[r] & \bR\mathrm{Hom}_{\sD_{\fP^{\sharp}_K}}(\cO_{\fP_K},\sp^!_\fP\cN \otimes_{\cO_{\fP_K}} \sp_\fP^! \cM^* ) . } \]
The bottom arrow is an isomorphism by Proposition \ref{prop: base change sp^* u_* log de Rham}, and therefore the top arrow is also an isomorphism.

Finally, to see that 
\[ \bR\mathrm{Hom}_{\sD_{\fP_K}}(j_*\sF,\sp^!_\fP\cN) \to \bR\mathrm{Hom}_{\sD_{\fP^{\sharp}_K}}(\sF^{\sharp}(\fr{D}_K),\sp^!_\fP\cN) \]
is an isomorphism, I will (after some set-up) apply Theorem \ref{theo: vanishing in log-rigid cohom}. First, note that the support $\tube{X}_\fP$ of $j_*\sF$ is disjoint from the divisor $\fr{D}_K$. It therefore follows that the `restriction' map
\[ \bR\mathrm{Hom}_{\sD_{\fP_K}}(j_*\sF,\sp^!_\fP\cN) \to \bR\mathrm{Hom}_{\sD_{\fP^{\sharp}_K}}(j_*\sF,\sp^!_\fP\cN) \]
is an isomorphism. Now write $i:D\rightarrow P$ for the natural inclusion, thus the localisation exact sequence
\[ 0 \rightarrow i_!i^*\sF^{\sharp}(\fr{D}_K) \rightarrow \sF^{\sharp}(\fr{D}_K)\rightarrow j_*\sF\rightarrow 0 \]
reduces to proving that
\[ \bR\mathrm{Hom}_{\sD_{\fP^{\sharp}_K}}(i_!i^*\sF^{\sharp}(\fr{D}_K),\sp^!_\fP\cN)= \bR\mathrm{Hom}_{\sD_{\tube{D}^{\sharp}_\fP}}(\sF^{\sharp}(\fr{D}_K),\sp^!_\fP\cN)=0, \]
where I have written $\tube{D}^\sharp_\fP$ for $\tube{D}_\fP$ equipped with the log structure given by pulling back the given log structure on $\fP_K^\sharp$. Let $\fr{D}_1,\ldots,\fr{D}_n$ denote the irreducible components of $\fr{D}$, with special fibres $D_1,\ldots,D_n$. For each $J\subset \{1,\ldots,n\}$ set
\begin{align*}
\fr{D}_J=\cap_{j\in J} \fr{D}_j,\;\;&\;\;D_J=\fr{D}_{J,k}=\cap_{j\in J} D_j \\
\fr{D}^{(J)}=\cup_{j\notin J} \fr{D}_j,\;\;&\;\;D^{(J)}=\fr{D}_k^{(J)}=\cup_{j\notin J} D_j.
\end{align*}
Using the open cover of $\tube{D}_{\fP}$ given by the various tubes $\tube{D_j}_{\fP}$, I can therefore reduce to showing that
\[ \bR\mathrm{Hom}_{\sD_{\tube{D_J}^{\sharp}_\fP}}(\sF^{\sharp}(\fr{D}_K),\sp^!_\fP\cN)=0 \]
for each $J\subset \{1,\ldots,n\}$. Theorem \ref{theo: sp^! cons} says that $\sp^!_\fP\cN$ is a constructible isocrystal on $\fP$, and $\sF^\sharp$ is a log convergent $F$-isocrystal on $P^\sharp$ by assumption. Hence the required vanishing is precisely the content of Theorem \ref{theo: vanishing in log-rigid cohom}.
\end{proof}

Combining this with Proposition \ref{prop: base change sp^* u_* de Rham}, I obtain the following.

\begin{corollary} \label{cor: smooth and proper preserves cons} Let $u\from \fP\to \fQ$ be a smooth and proper morphism of smooth formal schemes. Then the functor $\bR u_{\dR*}$ maps $\bD^b_{\cons,F}(\fP)$ into $\bD^b_{\cons,F}(\fQ)$. 
\end{corollary}

\section{Cohomological operations for constructible isocrystals} \label{sec: cohomological operations}

I can now show compatibility of $\sp^!$ with the various cohomological functors for $\sD^\dagger$-modules and constructible isocrystals. This will then enable me to deduce an analogue of Theorem \ref{theo: riemann-hilbert} for pairs and for varieties.

\subsection{The case of formal schemes}

I have already proved the compatibility of $\sp^!$ with de\thinspace Rham pushfoward along smooth morphisms $u\from \fP\to\fQ$ of smooth formal schemes in Proposition \ref{prop: base change sp^* u_* de Rham}. Similarly, if $\fP$ is a smooth formal scheme, and $i\from X\to P$ is a locally closed immersion, I have already shown in \S\ref{subsec: consequences} that the diagram
\[ \xymatrix{ \bD^b_{\hol,F}(\fP) \ar[r]^-{\sp_{\fP}^!} \ar[d]_{\underline{\bR\Gamma}^\dagger_X} & \bD^b_{\cons,F}(\fP) \ar[d]^{\underline{\bR\Gamma}^\dagger_X=i_!i^{-1}}\\
\bD^b_{\hol,F}(\fP) \ar[r]^-{\sp_\fP^!} & \bD^b_{\cons,F}(\fP). }  \]
commutes up to natural isomorphism, and that this is moreover compatible with the natural `adjunctions' when $X$ is either open or closed in $P$.

There are similar results for pullbacks and tensor products. For the case of pullbacks, suppose that $u\from \fP\to \fQ$ is a (not necessarily smooth) morphism of smooth formal schemes, and that $\cM\in \bD^b_{\hol,F}(\fQ)$. Then I can construct a natural (in $\cM$) morphism
\[ \psi_u\from u^*\sp_{\fQ}^!\cM\to \sp_{\fP}^!u^!\cM \]
as follows. I first use the usual trick of taking the graph two divide into two separate cases: the first when $u$ is smooth, and the second when $u$ is a section of a smooth morphism $\pi\from \fQ\to \fP$. 

In the first case, when $u$ is smooth, of relative dimension $d$, say, giving a morphism  
\[ u^*\sp_{\fQ}^!\cM\to \sp_{\fP}^!u^!\cM \]
is equivalent to giving a morphism
\[ \sp_{\fQ}^!\cM\to \bR u_{\dR*}\sp_{\fP}^!u^!\cM = \sp_{\fQ}^!u_+u^!\cM[-2d].  \]
Since I am not assuming that $u$ is proper, the functor $u_+$ appearing on the RHS should be viewed as taking values in $\LD^b_{\Q,\qc}(\widehat{\sD}^{(\bullet)}_{\fQ})$
It therefore suffices to construct a morphism
\[ \cM \to u_+u^!\cM[-2d] \]
in $\LD^b_{\Q,\qc}(\widehat{\sD}^{(\bullet)}_{\fQ})$. Via the projection formula (which in the generality required here follows from \cite[Proposition 1.2.27]{Car04} as in \cite[\S2.1.4]{Car04}) it suffices to consider the case $\cM=\cO_{\fQ\Q}$. In this case $u^!\cM=\cO_{\fP\Q}[d]$ and the morphism I seek is one of the form
\[ \cO_{\fQ\Q}\to u_+\cO_{\fP\Q}[-d]. \]
After actually taking the colimit in the level $m$, and tensoring with $\Q$ to land inside $\bD(\sD^\dagger_{\fQ\Q})$, this will simply the morphism
\[ \cO_{\fQ\Q} \to \bR u_{\dR*}\cO_{\fP\Q} \]
into the zeroeth relative de\thinspace Rham cohomology group of $\cO_{\fP\Q}$. To see that this morphism lifts to $\LD^b_{\Q,\qc}(\widehat{\sD}^{(\bullet)}_{\fQ})$ simply amounts to making the same construction levelwise in $m$. 

In the second case, suppose that $u$ is a section of a smooth morphism $\pi\from \fQ\to \fP$, of relative dimension $d$, say. Then I can calculate
\begin{align*} \sp_{\fP}^!u^!\cM &=\sp_{\fP}^!\pi_+u_+u^!\cM  \\
&= \sp_\fP^!\pi_+\underline{\bR\Gamma}^\dagger_P \cM \\
&= \bR\pi_{\dR*}\sp_\fQ^!\underline{\bR\Gamma}^\dagger_P \cM [2d] \\
&= \bR\pi_{\dR*}\underline{\bR\Gamma}^\dagger_P \sp_\fQ^! \cM [2d] 
\end{align*}
by compatibility of $\sp^!$ with both de\thinspace Rham pushfowards and the functor $\underline{\bR\Gamma}^\dagger_P$ of sections with overconvergent support. It therefore suffices to construct a morphism 
\[ u^* \to \bR\pi_{\dR*}\underline{\bR\Gamma}^\dagger_P[2d]\]
of functors from $\bD^b_{\cons,F}(\fQ)$ to $\bD^b_{\cons,F}(\fP)$. Let
\[ \tube{\rm id}_u\from \fP_K \to \tube{P}_\fQ,\;\;\;\;\tube{u}_{\rm id}\from \tube{P}_\fQ\to\fQ_K \]
denote the natural inclusions, and
\[ \tube{\rm id}_\pi\from \tube{P}_\fQ\to \fP_K \]
the restriction of $\pi$. Thus $\underline{\bR\Gamma}^\dagger_P=\tube{u}_{\rm id!}\tube{u}^*_{\rm id}$.

\begin{lemma} The `forget supports' map $\bR\tube{\rm id}_{\pi!} \to \bR\pi_* \circ \tube{u}_{\rm id!}$ is an isomorphism. 
\end{lemma} 

\begin{proof} Explicitly, if $\sI$ is an injective sheaf on $\tube{P}_\fQ$, then $\tube{\rm id}_{\pi!}\sI$ consists of sections of $\sI$ with compact support over $\fP_K$, and $\pi_* \circ \tube{u}_{\rm id!}\sI$ consists of sections with compact support over $\fQ_K$.  To show that $\tube{\rm id}_{\pi!}\sI=\pi_* \circ \tube{u}_{\rm id!}\sI$ therefore amounts to showing that if $V\subset \fP_K$ is open, and $S\subset \tube{\rm id}_{\pi}^{-1}(V)$ is a closed subset, then $S\to V$ is proper if and only if $S\to \pi^{-1}(V)$ is proper. Since $S\to V$ is always partially proper, and $\pi^{-1}(V)\to V$ is quasi-compact, this is clear. 

To complete the proof, then, I need to show that $\tube{u}_{\rm id!}\sI$ is a flasque sheaf on $\fQ_K$. But since $\tube{u}_{\rm id!}$ is partially proper, this was demonstrated in the proof of \cite[Corollary 3.4.7]{AL20} (note that the additional hypothesis there that the second morphism $g$ is partially proper is not used to show that $f_!\sI$ is flasque).
\end{proof}

I am therefore looking to construct a morphism
\[ u^* \to \bR\tube{\rm id}_{\pi\dR!}\tube{u}^*_{\rm id}[2d], \]
and in fact an isomorphism between these two functors is provided by Remark \ref{rem: v^* quasi-inverse to u^*}. This therefore gives rise to an \emph{isomorphism}
\[ u^*\sp_{\fQ}^!\cM\isomto \sp_{\fP}^!u^!\cM, \]
which is natural in $\cM$.
%
%

\begin{proposition} \label{prop: u^!, u^* and sp^!} Let $u\from \fP\to \fQ$ be a morphism of smooth formal schemes, and $\cM\in \bD^b_{\hol,F}(\fQ)$. Then the morphism
\[ \psi_u\from u^*\sp_{\fQ}^!\cM\to \sp_{\fP}^!u^!\cM \]
is an isomorphism in $\bD^b_{\hol,F}(\fP)$.
\end{proposition}

\begin{proof}
To show that $\psi_u$ is an isomorphism, I can check that it is so after taking the pullback along any closed immersion $\spf{\cV'}\to \fP$ where $\cV'$ is the ring of integers in a finite extension of $K$. After possibly making a finite base change (for which the result can be easily checked), I can in fact take such closed immersions with $\cV'=\cV$. Thus the general case reduces to the special case when $u$ is a section of a smooth morphism, in which case the given morphism which has already been shown to be an isomorphism.
\end{proof}

To show compatibility of $\sp^!$ with tensor product, recall that in \S\ref{subsec: rigidification O-mod} I constructed, for $\cM,\cN\in \underrightarrow{{\bf LD}}_{\Q,\mathrm{qc}}(\widehat{\sD}^{(\bullet)}_\fr{P})$, a morphism
\[  \mathbf{L}\hat{\mathrm{sp}}_\fr{P}^*\mathcal{M}\otimes_{\mathcal{O}_{\fr{P}_K}}^\mathbf{L} \mathbf{L}\hat{\mathrm{sp}}_\fr{P}^*\mathcal{N} \rightarrow   \mathbf{L}\hat{\mathrm{sp}}_\fr{P}^*(\mathcal{M}\widehat{\otimes}_{\mathcal{O}_{\fr{P}}}^\mathbf{L} \mathcal{N}) \]
in ${\bf D}(\mathcal{O}_{\fr{P}_K})$. It is straightforward, though rather tedious, to upgrade this to a morphism in ${\bf D}(\sD_{\fr{P}_K})$. For $\cM,\cN\in \bD^b_{\hol,F}(\fP)$, I can therefore view it as a morphism
\begin{equation}
\label{eqn: otimes sp^!}
\mathrm{sp}^!\mathcal{M}\otimes_{\mathcal{O}_{\fr{P}_K}} \mathrm{sp}_\fr{P}^!\mathcal{N} \rightarrow   \mathrm{sp}^!(\mathcal{M}\wotimes_{\mathcal{O}_{\fr{P}}} \mathcal{N})
\end{equation} 
in $\bD^b_{\cons,F}(\fP)$. 

\begin{corollary} \label{cor: otimes sp^!} The morphism (\ref{eqn: otimes sp^!}) is an isomorphism.
\end{corollary}

\begin{proof}
Again, this can be checked on stalks, in which case it follows from Proposition \ref{prop: u^!, u^* and sp^!} together with the fact that both pullback functors $u^!$ and $u^*$ commute with the relevant tensor products.
\end{proof}

\subsection{The case of pairs and varieties} \label{sec: coh ops for vars}

In Corollary \ref{cor: smooth and proper preserves cons}, the fact that $u$ is proper implies that $\bR u_{\dR*}=\bR u_{\dR!}$, so I could have equally well phrased the result in terms of compactly supported de\thinspace Rham cohomology. It turns out that this version is the one that generalises to the case of pairs. In fact, I	can use the extra flexibility of constructible isocrystals to remove the need for the formal schemes involved to be immersible in smooth and proper ones. To obtain a functor that is independent of the frame, I also need to include a shift. 

\begin{corollary} \label{cor: cons preserved by morphism of pairs} Let
\[ \xymatrix{  X'\ar[d]^f \ar[r] & Y' \ar[d]^g \ar[r] & \fP'\ar[d]^u \\ X\ar[r] & Y\ar[r] & \fP}  \]
be a morphism of frames, such that $g$ is proper, $\fP$ is smooth around $X$, and $u$ is smooth around $X'$, of relative dimension $d$.

Then the functor $\bR\! \tube{f}_{\dR!}[2d]$ maps $\bD^b_{\cons,F}(X',Y')$ into $\bD^b_{\cons,F}(X,Y)$. The resulting functor only depends on the morphism of pairs $g\from (X',Y')\to (X,Y)$. If $Y$ is proper, then it only depends on the morphism of varieties $f\from X'\to X$.
\end{corollary}

\begin{proof}
Let me first consider the claim that $\bR\! \tube{f}_{\dR!}[2d]$ preserves constructible isocrystals. The key point is that I can use Theorem \ref{theo: disoc indep} repeatedly to replace our morphism of frames by one that is covered by Corollary \ref{cor: smooth and proper preserves cons}. Indeed, the question is local on $X'$, which I can therefore assume to be affine, in particular $f$ is quasi-projective. By Chow's lemma, there is a morphism of varieties $g'\from Y''\to Y'$ such that $g'^{-1}(X')\isomto X'$, and $g\circ g'$ is projective (therefore $g'$ is also projective). Alternately embedding $Y''$ into either $\widehat{\P}_{\fP}^N $ or $\widehat{\P}_{\fP'}^{N'} $, and using Theorem \ref{theo: disoc indep} repeatedly, I can replace the original morphism of frames by one of the form
\[ \xymatrix{  X'\ar[d]^f \ar[r] & Y'' \ar[d] \ar[r] & \widehat{\P}_{\fP}^N \ar[d] \\ X\ar[r] & Y\ar[r] & \fP,}  \]
and therefore apply Corollary \ref{cor: smooth and proper preserves cons}. Now the independence of $u$ (and $g$ when $Y$ is proper) is a simple consequence of Theorem \ref{theo: disoc indep}, using transitivity of $\bR\!\tube{f}_{\dR!}[2d]$ and the usual trick of taking products of frames.
\end{proof}

\begin{definition} If $(f,g)\from (X',Y') \to (X,Y)$ is a morphism of weakly realisable pairs, I denote by 
\[ \bR f_{\flat} \from \bD^b_{\cons,F}(X',Y')\to \bD^b_{\cons,F}(X,Y)  \]
the functor implicitly given by Corollary \ref{cor: cons preserved by morphism of pairs}. If $f\from X'\rightarrow X$ is a morphism of weakly realisable varieties, I similarly denote
\[\bR f_{\flat} \from \bD^b_{\cons,F}(X')\to \bD^b_{\cons,F}(X)  \]
the functor obtained by choosing a suitable compactification of $f$. 
\end{definition}

\begin{remark} Note that whenever $g$ is a closed immersion, and $Y\hookrightarrow \fP$ is a closed immersion into a formal scheme which is smooth around $X$, then $\bR f_\flat$ can be identified with the functor $f_!$ of extension by zero along the induced morphism of tubes
\[ f\from \tube{X'}_\fP \to \tube{X}_\fP. \]
\end{remark}

If $(X,Y)$ is a strongly realisable pair, and $(X,Y,\fP)$ is an l.p. frame enclosing it, I can also deduce from Proposition \ref{prop: base change sp^* u_* de Rham} and Theorem \ref{theo: riemann-hilbert} that the composite (equivalence!)
\[ \bD^b_{\hol,F}(X,Y)\subset \bD^b_{\hol,F}(\fP)\overset{\sp_{\fP}^!}{\longrightarrow} \bD^b_{\cons,F}(\fP) \overset{{\rm res}}{\lto} \bD^b_{\cons,F}(X,Y,\fP)=\bD^b_{\cons,F}(X,Y) \]
does not depend on $\fP$, and can therefore be denoted
\[ \sp_X^!\from \bD^b_{\hol,F}(X,Y) \isomto \bD^b_{\cons,F}(X,Y).\]
Of course, there is a similar functor
\[ \sp_X^!\from \bD^b_{\hol,F}(X) \isomto \bD^b_{\cons,F}(X). \]
whenever $X$ is a strongly realisable variety. Moreover, these functors are t-exact for the dual constructible structure on the source, and the natural t-structure on the target.

\begin{corollary}\label{cor: del kash}
For any strongly realisable pair $(X,Y)$, $\sp_X^!$ induces an equivalence of categories
\[
\sp_X^!\from \DCon_F(X,Y)\isomto \Isoc_{\cons,F}(X,Y). 
\]
For any strongly realisable variety $X$, $\sp_X^!$ induces an equivalence of categories
\[
\sp_X^!\from \DCon_F(X)\isomto \Isoc_{\cons,F}(X). 
\]
\end{corollary}

\begin{remark} \label{rem: duality loc cons} Given that the category $\Isoc_F(X)$ of locally free (overconvergent) isocrystals on $X$ is the $p$-adic analogue of the category $\Loc(X_\et,\Q_\ell)$ of $\ell$-adic local systems on $X$, it may seem a little odd to see it appearing as a full subcategory of $\Isoc_{\cons,F}(X)$, which I am promoting here as a $p$-adic analogue of the category $\DCon(X_\et,\Q_\ell)$ of \emph{dual} constructible sheaves on $X$ (not the category of constructible sheaves on $X$). But something similar does indeed happen in the $\ell$-adic theory. There is a fully faithful embedding
\begin{align*}
\Loc(X_\et,\Q_\ell) &\to \DCon(X_\et,\Q_\ell) \\
\sF &\mapsto \bD_X(\sF^\vee)
\end{align*}
where $(-)^\vee$ is the dual functor for local systems, and $\bD_X$ is the Verdier dual functor. If $X$ is smooth, then $\bD_X(\sF^\vee)$ is simply a shift of $\sF$ by $2\dim X$, but if $X$ is singular this will generally \emph{not} be a shift of a constructible sheaf. This at least goes some way to explaining why the proof of Theorem \ref{theo: rigid and D-module cohomology} below \emph{has} to proceed via duality.
\end{remark}

\begin{corollary} Let $(f,g)\from (X',Y') \to (X,Y)$ be a morphism of strongly realisable pairs. Then the diagram
\[ \xymatrix{ \bD^b_{\hol,F}(X,Y) \ar[r]^{\sp_X^!} \ar[d]_{f^!} & \bD^b_{\cons,F}(X,Y) \ar[d]^{f^*}  \\  \bD^b_{\hol,F}(X',Y') \ar[r]^{\sp_{X'}^!} & \bD^b_{\cons,F}(X',Y') }  \]
commutes up to natural isomorphism. If $g$ is proper, then so does 
\[ \xymatrix{ \bD^b_{\hol,F}(X',Y') \ar[r]^{\sp_{X'}^!} \ar[d]_{f_+} & \bD^b_{\cons,F}(X',Y') \ar[d]^{\bR f_\flat}\\
\bD^b_{\hol,F}(X,Y) \ar[r]^{\sp_X^!} & \bD^b_{\cons,F}(X,Y) }  \]
\end{corollary}

\begin{proof}
The statement for pushforwards follows from Proposition \ref{prop: base change sp^* u_* de Rham}. For pullbacks, it follows from Proposition \ref{prop: u^!, u^* and sp^!} together with the fact that, for any smooth formal scheme $\fP$, and any locally closed immersion $i\from X\rightarrow \fP$, there is an isomorphism
\[ i_!i^{-1}\sp_\fP^! \isomto \sp_\fP^!\bR\underline{\Gamma}^\dagger_X.  \qedhere \]
\end{proof}

There is a similar statement for morphisms of strongly realisable varieties. Finally, I can show compatibility with tensor product using Corollary \ref{cor: otimes sp^!}.

\begin{corollary} Let $(X,Y)$ be a strongly realisable pair. Then there is a natural isomorphism 
\[  \sp_X^!(-) \otimes_{\cO_{X}} \sp_X^! \isomto \sp_X^!(-\wotimes_{\cO_X}-) \]
of functors
\[ \bD^b_{\hol,F}(X,Y)\times \bD^b_{\hol,F}(X,Y) \to \bD^b_{\cons,F}(X,Y). \]
\end{corollary}

Naturally, there is an analogous statement for varieties.

\begin{remark} Via the analogy between $\Isoc_{\cons,F}(X,Y)$ and $\DCon(X_\et,\Q_\ell)$, the standard tensor product on constructible isocrystals corresponds to the dual tensor product
\[ \sF\wotimes_{\Q_\ell}\sG:= \bD_X(\bD_X(\sF)\otimes_{\Q_\ell}\bD_X(\sG)) \]
for dual constructible $\ell$-adic sheaves. Note that this is t-exact for the dual constructible t-structure on $\bD^b_c(X_\et,\Q_\ell)$.
\end{remark}

\subsection{Verdier duality}

Of the six cohomological functors $f_+,f^+,f^!,f_!,\bD_X,\wotimes_{\cO_X}$ defined on $\bD^b_{\hol,F}$, we have seen how to interpret three in terms of $\bD^b_{\cons,F}$, namely $f_+$, $f^!$ and $\wotimes_{\cO_X}$. This begs the question of how to interpret $f^+,f_!$ and $\bD_X$ in the theory of constructible isocrystals. Via the usual relations amongst $f_!,f_+$, $f^!$ and $f^+$, this amounts to describing the duality functor $\bD_X$.

This is a subtle question that I don't know how to answer at the moment. The main reason that I don't yet have a good answer is the following: if $\fP$ is a smooth and proper $\cV$-scheme, say, then the composite functor
\[ \bD^b_{\cons,F}(\fP) \isomfrom \bD^b_{\hol,F}(\fP) \overset{\bD_{\fP}}{\lto} \bD^b_{\hol,F}(\fP)\isomto \bD^b_{\cons,F}(\fP) \]
\emph{cannot} be of a local nature on $\fP_K$. Indeed, if I denote by
\[ \bD_{\fP} \from \bD^b_{\cons,F}(\fP) \to \bD^b_{\cons,F}(\fP) \]
the above composition, then clearly $\bD_{\fP}(\cO_{\fP_K})=\cO_{\fP_K}[-2\dim \fP]$. On the other hand, if $i\from X\to P$ is a smooth closed subscheme, then similarly $\bD_{\fP}(i_!\cO_{\tube{X}_\fP})\cong  i_!\cO_{\tube{X}_\fP}[-2\dim X]$. Since $\cO_{\fP_K}\!\!\mid_{\tube{X}_\fP}=\cO_{\tube{X}_\fP}$, it follows that $\bD_{\fP}$ is not local on $\fP_K$. For example, this rules out any construction of the form $\mathbf{R}\underline{\mathrm{Hom}}_{\mathcal{A}}(-,\mathscr{K})$ where $\ca{A}$ is a sheaf of rings on $\fP_K$ and $\scr{K}$ is a complex of $\ca{A}$-modules. 

\section{Rigid cohomology of varieties} \label{sec: rigid cohomology}

I will finish this article with a comparison theorem between rigid and $\sD^\dagger$-module cohomology, at least for varieties $X$ which admit an immersion $X\hookrightarrow \fP$ into a smooth and proper formal scheme. In order to be able to state it properly, I need to extend (the shifted version of) Caro's functor
\[ \sp_{X+}\from \Isoc_F(X) \to \bD^b_{\hol,F}(X)\subset \bD^b_{\hol,F}(\fP) \]
from the case where $X$ is smooth to the case of arbitrary $X$. In fact, this was achieved in \cite[\S3.8]{Abe19}, by combining descent with the commutation of $\sp_{+}$ with pullback. Note that the functor that I have been denoting $\sp_+$ is denoted $\rho$ in \cite{Abe19}. I then define
\[\sp_{X!}:=\bD_X \circ \sp_{X+} \circ (-)^\vee \from \Isoc_F(X)\to \DCon_F(X)\subset \DCon(\fP)\]
where the functor $(-)^\vee$ is simply the ordinary dual of locally free isocrystals.\footnote{Note that, since I am only considering locally free isocrystals here, this functor coincides with that of the same name define in \cite{AL22} by \cite[Theorem 7.1.1]{AL22}.} Again, by using the fact that $\sp^!$ is compatible with pullback, I can then remove the smoothness hypothesis on $X$ appearing in Corollary \ref{cor: sp^! sp_! locally free isocrystals}. 

\begin{lemma} \label{cor: sp^! sp_! locally free isocrystals 2} Let $\fP$ be a smooth formal scheme, $i\from X\to P$ be a locally closed subscheme, with closure $Y$, $\sF\in \Isoc_F(X,Y)$, and $\cM=\sp_{X!}\sF\in \DCon_F(\fP)$. Then $\sp^!\cM \isomto i_!\sF \in \Isoc_{\cons,F}(\fP)$.
\end{lemma}

I can now state and prove my comparison result between rigid and $\sD^\dagger$-module cohomology.
 
\begin{theorem} \label{theo: rigid and D-module cohomology} Let $f:X\rightarrow \spec{k}$ be a strongly realisable $k$-variety,\footnote{Recall that this means $X$ admits an immersion into a smooth and proper formal $\cV$-scheme.} and $\sF\in \Isoc_F(X)$ a locally free isocrystal on $X$ of Frobenius type.\footnote{See Definition \ref{defn: F type}.} Then $\sp_X^!$ induces an isomorphism
\[  \mathbf{R}\Gamma_\rig(X,\sF) \cong f_+\mathrm{sp}_{X+}\sF \]
in $\bD^b(K)$.
\end{theorem}

\begin{proof} 
Choose an immersion $i:X\hookrightarrow \fr{P}$ of $X$ into a smooth and proper formal scheme. Set $d=\dim \fP$, and let $u:\fr{P}\rightarrow \spf{\mathcal{V}}$ denote the structure morphism. Note that the claimed result is \emph{not} an immediate consequence of Proposition \ref{prop: base change sp^* u_* de Rham}, since the functor $i_!\from \bD^b_{\mathrm{cons},F}(X)\hookrightarrow \bD^b_{\mathrm{cons},F}(\fr{P})$ does not preserve rigid cohomology. Instead I will need to argue via duality.\footnote{Another explanation for why this is necessary is given in Remark \ref{rem: duality loc cons}.}

To start with, a direct calculation shows that $u^!\mathcal{O}_{\spf{\cV}\Q} = \mathcal{O}_{\fr{P}\Q}[d]$. Setting $\cM=\sp_{X+}\sF\in \bD^b_{\hol,F}(X)\subset \bD^b_{\hol,F}(\fP)$, then $\sp_\fP^!(\bD_X\cM) = i_!\sF^\vee\in \bD^b_{\cons,F}(\fP)$ by Lemma \ref{cor: sp^! sp_! locally free isocrystals 2}. I can therefore calculate 
\begin{align*}
\mathbf{R}\Gamma_\rig(X,\sF) &= \mathbf{R}\mathrm{Hom}_{\bD^b_{\cons,F}(X)}(\sF^\vee,\mathcal{O}_{\tube{X}_\fP}) \\
&= \mathbf{R}\mathrm{Hom}_{\bD^b_{\cons,F}(\fP)}(i_!\sF^\vee,i_!i^*\mathcal{O}_{\fP_K}) \\
&= \mathbf{R}\mathrm{Hom}_{\bD^b_{\hol,F}(\fr{P})}(\bD_X\cM,\mathbf{R}\underline{\Gamma}_X^\dagger\mathcal{O}_{\fr{P}\Q}[d]) \\
&= \mathbf{R}\mathrm{Hom}_{\bD^b_{\hol,F}(\fr{P})}(\bD_X\cM,\mathbf{R}\underline{\Gamma}_X^\dagger u^! \mathcal{O}_{\spf{\cV}\Q})
\end{align*}
using Theorem \ref{theo: riemann-hilbert}. Now using the six operations formalism for overholonomic $\sD^\dagger$-modules on varieties, I can rewrite this as
\begin{align*}
\mathbf{R}\mathrm{Hom}_{\bD^b_{\hol,F}(\fr{P})}(\bD_X\cM,\mathbf{R}\underline{\Gamma}_X^\dagger u^! \mathcal{O}_{\spf{\cV}\Q}) &= \mathbf{R}\mathrm{Hom}_{\bD^b_{\hol,F}(X)}(\bD_X\cM,f^!\cO_{\spec{k}}^\dagger) \\
&= \mathbf{R}\mathrm{Hom}_{\bD^b_{\hol,F}(X)}(f^+\cO^\dagger_{\spec{k}},\cM) \\
&= \mathbf{R}\mathrm{Hom}_{\bD^b_{\hol,F}(\spec{k})}(\cO_{\spec{k}}^\dagger,f_+\cM) \\
&= f_+\cM. \qedhere
\end{align*}
\end{proof}

\begin{remark} The analogous calculation for a lisse $\ell$-adic sheaf $\sF\in \Loc(X_\et,\Q_\ell)$ would be that
\begin{align*}
\bR\Gamma_\et(X_{\bar k},\sF) &= \bR{\rm Hom}_{X_{\bar k}}(\sF^\vee,\Q_{\ell,X}) \\
&= \bR{\rm Hom}_{X_{\bar k}}(\bD_X(\Q_{\ell,X}),\bD_X(\sF^\vee)),
\end{align*} 
see Remark \ref{rem: duality loc cons}.
\end{remark}

\bibliographystyle{../Templates/bibsty}
\bibliography{../Templates/lib.bib}

\end{document}

%% file: headart.tex
\usepackage[toc,page]{appendix}
\usepackage[headheight=12.1pt]{geometry}
\usepackage[all,cmtip]{xy}
\usepackage{amsmath}
\usepackage{amsfonts}
\usepackage{amssymb}
\usepackage{mathrsfs}   
\usepackage{amsthm}
\usepackage{xcolor}
\usepackage{cite}
\usepackage{stmaryrd}
\usepackage{graphicx}
\usepackage{pgf}
\usepackage{eepic}
\usepackage{pstricks}
\usepackage{epsfig}
\usepackage{fancyhdr}
\usepackage{tikz,caption,float}
\usepackage[backref=page]{hyperref}

\usepackage{amsbsy}

\numberwithin{equation}{subsection}

\let\realequation\equation
\def\equation{\setcounter{equation}{\arabic{subsubsection}}%
   \refstepcounter{subsubsection}%
   \realequation}

\usepackage{xr,mathtools}

\hypersetup{
     colorlinks   = true,
     citecolor    = red,
     linkcolor = blue,
     urlcolor = purple
}
    
  \newcommand\imCMsym[4][\mathord]{%
  \DeclareFontFamily{U} {#2}{}
  \DeclareFontShape{U}{#2}{m}{n}{
    <-6> #25
    <6-7> #26
    <7-8> #27
    <8-9> #28
    <9-10> #29
    <10-12> #210
    <12-> #212}{}
  \DeclareSymbolFont{CM#2} {U} {#2}{m}{n}
  \DeclareMathSymbol{#4}{#1}{CM#2}{#3}
}
\newcommand\alsoimCMsym[4][\mathord]{\DeclareMathSymbol{#4}{#1}{CM#2}{#3}}

\imCMsym{cmmi}{124}{\CMjmath}
\imCMsym[\mathop]{cmsy}{113}{\CMamalg}
\imCMsym[\mathop]{cmex}{96}{\CMcoprod}
\alsoimCMsym[\mathop]{cmex}{97}{\CMbigcoprod}

\swapnumbers
\theoremstyle{plain}
\newtheorem*{theoremu}{Theorem}
\newtheorem{theorem}[subsubsection]{Theorem}

\newtheorem{proposition}[subsubsection]{Proposition}

\newtheorem{corollary}[subsubsection]{Corollary}

\newtheorem*{corollaryu}{Corollary}
\newtheorem{lemma}[subsubsection]{Lemma}

\theoremstyle{definition}

\newtheorem{definition}[subsubsection]{Definition}

\newtheorem{setup}[subsubsection]{Setup}

\theoremstyle{remark}
\newtheorem{remark}[subsubsection]{Remark}

\newtheorem{example}[subsubsection]{Example}

\newcommand{\N}{{\mathbb N}}
\newcommand{\Z}{{\mathbb Z}}
\newcommand{\Q}{{\mathbb Q}}

\newcommand{\C}{{\mathbb C}}

\renewcommand{\P}{{\mathbb P}}
\newcommand{\A}{{\mathbb A}}
\newcommand{\D}{{\mathbb D}}

\newcommand{\sH}{\mathscr{H}}
\newcommand{\sF}{\mathscr{F}}
\newcommand{\sG}{\mathscr{G}}
\newcommand{\sI}{\mathscr{I}}

\newcommand{\sD}{\mathscr{D}}
\newcommand{\scr}[1]{\mathscr{#1}}

\newcommand{\cO}{\mathcal{O}}

\newcommand{\cI}{\mathcal{I}}
\newcommand{\cM}{\mathcal{M}}
\newcommand{\cN}{\mathcal{N}}
\newcommand{\cB}{\mathcal{B}}

\newcommand{\cV}{\mathcal{V}}
\newcommand{\cH}{\mathcal{H}}
\newcommand{\ca}[1]{\mathcal{#1}}

\newcommand{\fP}{\mathfrak{P}}
\newcommand{\fX}{\mathfrak{X}}
\newcommand{\fY}{\mathfrak{Y}}
\newcommand{\fQ}{\mathfrak{Q}}
\newcommand{\fr}[1]{\mathfrak{#1}}

\newcommand{\LD}{\underrightarrow{{\bf LD}}}
\newcommand{\bD}{{\bf D}}
\newcommand{\bL}{{\bf L}}
\newcommand{\bR}{{\bf R}}

\newcommand{\colim}[1]{\underset{#1}{\mathrm{colim}}\,}
\renewcommand{\lim}[1]{\underset{#1}{\mathrm{lim}}\,}

\newcommand{\id}{\mathrm{id}}
\newcommand{\isomto}{\overset{\cong}{\longrightarrow}}
\newcommand{\isomfrom}{\overset{\cong}{\longleftarrow}}
\newcommand{\from}{\colon}
\renewcommand{\to}{\rightarrow}
\newcommand{\ot}{\leftarrow}
\newcommand{\hto}{\hookrightarrow}
\newcommand{\lto}{\longrightarrow}
\renewcommand{\-}{\text{-}}

\newcommand{\an}{\mathrm{an}}

\newcommand{\cons}{\mathrm{cons}}
\newcommand{\hol}{\mathrm{hol}}
\newcommand{\Isoc}{\mathbf{Isoc}}
\newcommand{\coh}{\mathrm{coh}}
\newcommand{\Hol}{\mathbf{Hol}}
\newcommand{\Loc}{\mathbf{Loc}}
\newcommand{\Con}{\mathbf{Con}}
\newcommand{\DCon}{\mathbf{DCon}}
\newcommand{\wotimes}{\widetilde{\otimes}}
\newcommand{\Coh}{\mathbf{Coh}}
\newcommand{\Mod}{\mathbf{Mod}}
\newcommand{\Sh}{\mathbf{Sh}}
\newcommand{\Perv}{\mathbf{Perv}}

\newcommand{\pow}[1]{\llbracket #1 \rrbracket}

\newcommand{\weak}[1]{\langle #1\rangle^\dagger}
\newcommand{\tate}[1]{\langle #1 \rangle}
\newcommand{\open}[1]{\left\{ #1 \right\}}

\newcommand{\spec}[1]{\mathrm{Spec}\left(#1\right)}
\newcommand{\spa}[1]{\mathrm{Spa}\left(#1\right)}
\newcommand{\spf}[1]{\mathrm{Spf}\left(#1\right)}

\newcommand{\Proj}[2]{\mathbf{Proj}_{#1}\left(#2\right)}

\newcommand{\norm}[1]{\left\vert#1\right\vert}
\newcommand{\tube}[1]{\left]#1\right[}

\newcommand{\rig}{\mathrm{rig}}
\newcommand{\dR}{\mathrm{dR}}

\newcommand{\et}{\mathrm{\acute{e}t}}

\newcommand{\qc}{\mathrm{qc}}

\newcommand{\Tr}{\mathrm{Tr}}
\renewcommand{\sp}{\mathrm{sp}}
\newcommand{\sep}{\mathrm{sep}}